\documentclass[11pt, twoside, leqno]{article}

\usepackage{amssymb}
\usepackage{amsmath}
\usepackage{amsthm}
\usepackage{color}
\usepackage{mathrsfs}

\usepackage{indentfirst}

\usepackage{txfonts}
\usepackage{anysize}

\allowdisplaybreaks

\newcounter{mainthm}

\newtheorem{maintheorem}[mainthm]{Question}

\newtheorem{theorem}{Theorem}[section]
\newtheorem{lemma}[theorem]{Lemma}
\newtheorem{corollary}[theorem]{Corollary}

\newtheorem{proposition}[theorem]{Proposition}
\theoremstyle{definition}
\newtheorem{example}[theorem]{Example}
\newtheorem{remark}[theorem]{Remark}
\newtheorem{definition}[theorem]{Definition}

\numberwithin{equation}{section}

\begin{document}
\title{\bf\Large Matrix-Weighted Besov--Triebel--Lizorkin Spaces of
Optimal Scale: Real-Variable Characterizations,
Invariance on Integrable Index, and Sobolev-Type Embedding
\footnotetext{\hspace{-0.35cm} 2020 {\it Mathematics Subject Classification}.
Primary 46E35; Secondary 47A56, 42B25, 42C40, 46E40, 42B35.
\endgraf {\it Key words and phrases.}
matrix weight, generalized Besov--Triebel--Lizorkin-type space,
$\varphi$-transform, Peetre maximal function,
Littlewood--Paley function, almost diagonal operator,
molecule, wavelet, Sobolev-type embedding.
\endgraf This project is supported by
the National Key Research and Development Program of China
(Grant No. 2020YFA0712900),
the National Natural Science Foundation of China
(Grant Nos. 12431006 and 12371093),
the Fundamental Research Funds
for the Central Universities (Grant No. 2233300008).}}
\date{}
\author{Fan Bu, Dachun Yang\footnote{Corresponding author, E-mail:
\texttt{dcyang@bnu.edu.cn}/{\color{red}{\today}}/Final version.}
,\ Wen Yuan and Mingdong Zhang}
\maketitle
\date{}
\maketitle

\vspace{-0.8cm}

\begin{center}
\begin{minipage}{13.8cm}
{\small {\bf Abstract}\quad
In this article, using growth functions
we introduce generalized matrix-weighted
Besov--Triebel--Lizorkin-type
spaces with matrix $\mathcal{A}_{\infty}$ weights.
We first characterize these spaces, respectively,
in terms of the $\varphi$-transform,
the Peetre-type maximal function,
and the Littlewood--Paley functions.
Furthermore, after establishing the boundedness
of almost diagonal operators on the corresponding
sequence spaces, we obtain the molecular
and the wavelet characterizations of these spaces.
As applications, we find the sufficient and necessary conditions
for the invariance of those Triebel--Lizorkin-type
spaces on the integrable index and also for
the Sobolev-type embedding of all these spaces.
The main novelty exists in that these results are of wide generality,
the growth condition of growth functions is not only sufficient
but also necessary for the boundedness of almost diagonal operators
and hence this new framework of Besov--Triebel--Lizorkin-type
is optimal, some results either are new or improve the known
ones even for known matrix-weighted Besov--Triebel--Lizorkin spaces,
and, furthermore, even in the scalar-valued setting,
all the results are also new.}
\end{minipage}
\end{center}

\vspace{0.2cm}

\tableofcontents

\vspace{0.2cm}

\section{Introduction\label{s-intro}}

Throughout this article, we work in $\mathbb{R}^n$
and, unless necessary, we will not explicitly specify
this underlying space.

Around 1950, Nikol'ski\u{\i} \cite{nik51} and Besov \cite{bes59, bes61}
began to investigate Besov spaces on the Euclidean space $\mathbb{R}^n$.
Later, Triebel--Lizorkin spaces on $\mathbb{R}^n$
were independently studied by Lizorkin \cite{liz72, liz74}
and Triebel \cite{tri73} in the 1970s.
Besov--Triebel--Lizorkin (for short, BTL) spaces unify a variety of
classical function spaces, such as Lebesgue spaces,
Lipschitz spaces, Hardy spaces, and $\operatorname{BMO}$
(the space of all locally integrable functions on $\mathbb{R}^n$
with bounded mean oscillation).
We refer to the monographs \cite{tri83, tri92, tri06}
of Triebel for a systematical treatment of BTL spaces.
In 1990, Frazier and Jawerth in their seminal article
\cite{fj90} thoroughly studied
the homogeneous Triebel--Lizorkin spaces
$\dot{F}^{s}_{p, q}$, particularly the limiting case
$p=\infty$, established their well-known $\varphi$-transform
characterization via their related sequence spaces
$\dot{f}^{s}_{p, q}$. Using this $\varphi$-transform
characterization Frazier and Jawerth \cite{fj90} further gave
several properties of $\dot{F}^{s}_{p, q}$,
such as molecular and atomic characterizations,
duality, interpolation, and trace theorem.
The method used in \cite{fj90} has now become a standard and powerful approach,
which was used to study various BTL spaces in different settings
(see, for example, \cite{bow05,bow07,bow08,bh06,fr04,fr21,syy24,yy10,ysy10}).
Recently, BTL spaces were also generalized to be
associated with various operators (\cite{b20b,bd15,bd21a,
bd21b,bd21c, gkkp16}) and applied to
harmonic analysis and partial differential
equations (\cite{b20a,bbd22,h16a,h16b}).

To study the aforementioned space
$\dot{F}^{s}_{\infty,q}$, Frazier and Jawerth \cite[(5.4)]{fj90}
introduced its corresponding sequence space
$\dot{f}^{s}_{\infty,q}$. Using this and the $\varphi$-transform
characterization of $\dot{F}^{s}_{\infty,q}$, Frazier and Jawerth
\cite[Theorem 5.13]{fj90} showed that the dual space of $\dot{F}^{s}_{1,q}$
is precisely $\dot{F}^{-s}_{\infty,q'}$ and hence
gave another perspective on the well-known duality
between the Hardy space ${\rm H}^1$ and BMO obtained by Fefferman and Stein
in \cite{fs72}.

To introduce the sequence space $\dot{f}^{s}_{\infty,q}$
in \cite{fj90}, we need to first recall some concepts.
Let $\mathbb{Z}$ be the set of all integers
and $\mathcal{D}:=\{Q_{j, k}\}_{j\in\mathbb{Z}, k\in{\mathbb{Z}}^n}
:=\{2^{-j}([0, 1)^n+k)\}_{j\in\mathbb{Z}, k\in{\mathbb{Z}}^n}$
the set of all \emph{dyadic cubes} in $\mathbb{R}^n$.
For any measurable set $E\subset\mathbb{R}^n$,
let $|E|$ denote its Lebesgue measure
and $\mathbf{1}_E$ be its \emph{characteristic function}.
For any measurable set $E\subset\mathbb{R}^n$ with $|E|\in(0,\infty)$,
let $\widetilde{\mathbf{1}}_E:=|E|^{-\frac{1}{2}}\mathbf{1}_E$.
Let $s\in\mathbb{R}$, $p\in(0, \infty)$, and $q\in(0, \infty]$.
Recall that the space $\dot{f}_{\infty, q}^{s}$
is defined to be the set of all $t:=
\{t_Q\}_{Q\in\mathcal{D}}$ in $\mathbb{C}$ such that
\begin{align}\label{eq-f_3}
\|t\|_{\dot{f}_{\infty, q}^{s}}
:=\sup_{P\in\mathcal{D}}\left\{\frac{1}{|P|}
\int_P\sum_{Q\in\mathcal{D}, Q\subset P}
\left[|Q|^{-\frac{s}{n}}\left|t_Q\right|
\widetilde{\mathbf{1}}_Q(x)\right]^q \,dx\right\}^{\frac{1}{q}}
\end{align}
is finite (with the usual modification made if $q=\infty$)
(see \cite[(5.4)]{fj90}) and the space $\dot{f}^{s,\frac{1}{p}}_{p,q}$
is defined to be the set of all
$t:=\{t_Q\}_{Q\in\mathcal{D}}$ in $\mathbb{C}$ such that
\begin{align*}
\|t\|_{\dot{f}^{s,\frac{1}{p}}_{p,q}}
:=\sup_{P\in\mathcal{D}}\left\{\frac{1}{|P|}
\int_P\left(\sum_{Q\in\mathcal{D}, Q\subset P}
\left[|Q|^{-\frac{s}{n}}\left|t_Q\right|
\widetilde{\mathbf{1}}_Q(x)\right]^q\right)^{\frac{p}{q}}
\,dx\right\}^{\frac{1}{p}}
\end{align*}
is finite (with the usual modification made if $q=\infty$)
(see \cite[Definition 3.1]{yy10}).

Observe that, in \eqref{eq-f_3}, when $q\in(0,\infty)$,
$\|t\|^q_{\dot{f}^s_{\infty,q}}$ is equivalent to
the Carleson norm of the measure
$\sum_{Q\in\mathcal{D}}(|Q|^{-\frac{s}{n}-\frac{1}{2}}
|t_Q|)^q|Q| \delta_{(x_Q, \ell(Q))}$
on $\mathbb{R}^n\times(0,\infty)$, where $\delta_{(x, t)}$ is the point mass at $(x, t)\in\mathbb{R}^n\times(0,\infty)$. Furthermore, 
the \emph{invariance} of the sequence space $\dot{f}^{s,\frac{1}{p}}_{p,q}$
on the integrable index $p$ was also given by Frazier and Jawerth 
in \cite[Corollary 5.7]{fj90}, that is, for any $s\in\mathbb{R}$, $p\in(0, \infty)$, 
and $q\in(0, \infty]$, 
\begin{align}\label{eq-f=f}
\dot{f}_{\infty, q}^{s}=\dot{f}^{s,\frac{1}{p}}_{p,q}
\end{align}
with equivalent quasi-norms. Later, Bownik \cite[Theorem 3.6]{bow08} 
further extended \eqref{eq-f=f}
to Triebel--Lizorkin spaces on $\mathbb{R}^n$
associated with general expansive dilations and
the corresponding doubling measures. As pointed out 
by Frazier and Jawerth in \cite[p.\,75]{fj90},
\eqref{eq-f=f} serves an analogue of the John--Nirenberg lemma
on the sequence space level. Bownik \cite[p.\,142]{bow08}
also observed that, to compute the $\dot{f}_{\infty, q}^{s}$-norm,
sometimes \eqref{eq-f=f} can be a useful tool via the computation
of the $\dot{f}^{s,\frac{1}{p}}_{p,q}$-norm 
(see the proof of \cite[Corollary 3.7]{bow08}). Moreover, recently
Bu et al. \cite[Theorem 4.20]{bhyy5} and \cite[Theorem 12.1]{bhyy2} 
used \eqref{eq-f=f} to obtain the sharp boundedness of almost diagonal 
operators on the corresponding matrix-weighted sequence spaces.

To answer an open problem on $Q$ spaces in
\cite{dx04}, motivated by \eqref{eq-f=f},
Yang et al. \cite{yy08,yy10,ysy10}
introduced and thoroughly studied BTL spaces
on $\mathbb{R}^n$ mixed with the structure of Morrey spaces,
which are now called Besov--Triebel--Lizorkin-type (for short, BTL-type) spaces.
The finer structure of BTL-type spaces
enables them further to unify classical BTL spaces,
Morrey spaces, and $Q$ spaces, which hence
gives a positive answer to the open problem
in \cite{dx04}. For more studies on $Q$ spaces and BTL-type spaces,
we refer to \cite{ghs21,ghs23,ht23,x06,x19,yy13,yyz14,yhmsy15,yhsy15a,yhsy15b}.
Another important class of generalized BTL spaces
on $\mathbb{R}^n$, associated with Morrey spaces, was also well developed.
Indeed, to study semilinear heat equations and the Navier--Stokes equation,
Kozono and Yamazaki \cite{ky94} introduced the Besov--Morrey spaces
on $\mathbb{R}^n$, which mix the structures
of both Besov spaces and Morrey spaces.
Following this idea, Tang and Xu \cite{tx05}
investigated the Triebel--Lizorkin--Morrey spaces on $\mathbb{R}^n$.
Later, Sawano \cite{saw08,saw09,saw10a,saw10b}
and Sawano and Tanaka \cite{st07,st09} further systematically
studied Besov--Triebel--Lizorkin--Morrey spaces.
In particular, very recently Haroske et al. \cite{hl23, hlms23,hms24}
studied another generalization of BTL-type spaces,
which are associated with positive growth functions defined
on $(0, \infty)$; these BTL-type spaces are of wide generality.
All these spaces prove useful in harmonic analysis
and partial differential equations (see, for example,
\cite{hms22,hs12,hs13,ky94,lxy14,ly13,maz03,tri13,tri14}),
which naturally lead to the following question
about the invariance of matrix-weighted BTL-spaces
on the integrable index $p$.

\begin{maintheorem}\label{q1} In the matrix-weighted setting,
does there exist a natural generalization
of the aforementioned $\operatorname{BTL}$-type spaces so
that a counterpart of \eqref{eq-f=f} holds?
\end{maintheorem}

The main target of this article is to try to answer Question \ref{q1}.
Recall that, to investigate the prediction
theory of multivariate stochastic processes,
Wiener and Masani \cite[Section 4]{wm58} used
the matrix-weighted Lebesgue space $L^2(W)$.
In the study of both multivariate random stationary
processes and the invertibility of Toeplitz operators,
Treil and Volberg \cite{tv97} determined the matrix $\mathcal{A}_2$ condition
and showed that the Hilbert transform is bounded on
$L^2(W)$ over ${\mathbb R}$ if and only if $W$ satisfies the matrix
$\mathcal{A}_2$ condition.  Subsequently, based on
Bellman function arguments, Nazarov and Treil \cite{nt97}
formulated the matrix $\mathcal{A}_p$ condition for any $p\in(1, \infty)$
and showed that the Hilbert transform
is bounded on $L^{p}(W)$ over ${\mathbb R}$
if and only if $W$ satisfies the matrix $\mathcal{A}_p$ condition.
Via a method related to the classical Littlewood--Paley theory,
Volberg \cite{vol97} also obtained
the same results as those in \cite{nt97} for any $p\in(1,\infty)$.
Since then, the study of $L^p(W)$ attracts more and more attention.
In 2017, using the convex body domination, Nazarov et al. \cite{nptv17}
established the $L^2(W)$-norm inequality over $\mathbb{R}^n$
of Calder\'on--Zygmund operators with upper bound
$C[W]^{\frac{3}{2}}_{\mathcal{A}_2}$.
Furthermore, Domelevo et al. \cite{dptv24}
showed that this exponent $\frac{3}{2}$ is indeed sharp.
Recently, Bownik and Cruz-Uribe \cite{bc22}
established the Jones factorization theorem and the Rubio de Francia
extrapolation theorem for matrix $\mathcal{A}_p$
weights with $p\in(1,\infty)$. We also refer to
\cite{dhl20, dly21, gol03, llor23, llor24, nie25} for more studies on $L^p(W)$.

Another significant aspect on the study of matrix weights is to
develop the real-variable theory of matrix-weighted function spaces.
First, for any $s\in\mathbb{R}$ and $q\in(0,\infty]$,
matrix-weighted Besov spaces $\dot{B}^{s}_{p, q}(W)$ on
$\mathbb{R}^n$ were studied by Roudenko \cite{rou03, rou04}
for any $p\in(1, \infty)$ and $W\in\mathcal{A}_{p}$
and by Frazier and Roudenko \cite{fr04} for any $p\in(0, 1]$
and $W\in\mathcal{A}_{p}$. Later, for any $s\in\mathbb{R}$,
$p\in(0,\infty)$, $q\in(0,\infty]$, and $W\in\mathcal{A}_{p}$,
Frazier and Roudenko \cite{fr21} developed the
matrix-weighted Triebel--Lizorkin space
$\dot{F}^{s}_{p, q}(W)$ on $\mathbb{R}^n$ and established the
Littlewood--Paley theory of $L^p(W)$ with $p\in(1, \infty)$,
which further induces the identification of $\dot{F}^{0}_{p, 2}(W)$
and $L^p(W)$.
Recently, for any $A\in\{B, F\}$, $s\in\mathbb{R}$, $\tau\in[0,\infty)$,
$p\in(0,\infty)$, $q\in(0,\infty]$, and $W\in\mathcal{A}_{p}$,
Bu et al. \cite{bhyy1,bhyy2,bhyy3}
introduced the matrix-weighted BTL-type space
$\dot{A}^{s,\tau}_{p,q}(W)$ on $\mathbb{R}^n$
and established its real-variable theory,
such as the $\varphi$-transform characterization,
both the molecular and the wavelet characterizations, and the
boundedness of pseudo-differential, trace, and Calder\'on--Zygmund
operators; see also the survey \cite{byyz25}. We also refer to
Xu et al. \cite{bx24a,bx24b,bx24c,wgx24} and Yang et al. \cite{bcyy25,byy23,lyy24a,lyy24b,wyy23}
for more studies on function spaces
associated with matrix $\mathcal{A}_p$ weights.
On the one hand, for any $p\in(1,\infty)$,
Volberg \cite{vol97} introduced the matrix $\mathcal{A}_{p,\infty}$ class
on $\mathbb R$, which is a larger class than the matrix $\mathcal{A}_{p}$
class and corresponds to the Muckenhoupt $A_{\infty}$ class in the scalar-valued setting.
Later, for any $p\in(0,\infty)$, Bu et al. \cite{bhyy4}
obtained several equivalent characterizations
of matrix $\mathcal{A}_{p,\infty}$ weights on $\mathbb{R}^n$ and investigated
their fundamental properties such as
the self-improvement property and the reverse H\"{o}lder inequality.
It is also worth pointing out, for any $p\in(0,\infty)$ and
$W\in\mathcal{A}_{p,\infty}$, Bu et al. \cite{bhyy4} introduced
the concepts of upper and lower dimensions of $W$
and used them to obtain sharp estimates of
the corresponding reducing operators of order $p$ for $W$.
For any $A\in\{B, F\}$, $s\in\mathbb{R}$, $\tau\in[0,\infty)$,
$p\in(0,\infty)$, $q\in(0,\infty]$, and
$W\in\mathcal{A}_{p,\infty}$, Bu et al. \cite{bhyy5}
also developed a complete real-variable
theory of the inhomogeneous matrix-weighted BTL-type space
$A^{s,\tau}_{p,q}(W)$ on $\mathbb{R}^n$.

In this article, to answer the above Question \ref{q1},
we introduce generalized matrix-weighted
BTL-type spaces $\dot{A}_{p,q}^{s, \upsilon}(W)$ on $\mathbb{R}^n$,
here and thereafter,
$$(A, a)\in\{(B, b), (F, f)\},$$
$s\in\mathbb{R}$, $p\in(0, \infty)$, $q\in(0,\infty]$,
$W\in\mathcal{A}_{p,\infty}$,
and $\upsilon$ is a growth function.
We first establish the $\varphi$-transform
characterization of $\dot{A}_{p,q}^{s, \upsilon}(W)$. Then we
characterize $\dot{A}_{p,q}^{s, \upsilon}(W)$, respectively,
in terms of the Peetre-type maximal function and
the Littlewood--Paley functions. Furthermore,
after establishing the boundedness of almost diagonal
operators on $\dot{a}_{p,q}^{s, \upsilon}(W)$
(the generalized matrix-weighted BTL-type sequence spaces),
we obtain the molecular and the wavelet
characterizations of $\dot{A}_{p,q}^{s, \upsilon}(W)$.
As applications, we give the sufficient and necessary conditions for
the invariances of $\dot{f}_{p,q}^{s,\upsilon_{1/p, W}}(W)$
and $\dot{F}_{p,q}^{s,\upsilon_{1/p, W}}(W)$
on the integrable index $p$ in the matrix-weighted setting
and hence answer the above Question \ref{q1}.
Moreover, we also find the sufficient and necessary condition for
the Sobolev-type embedding of $\dot{A}_{p,q}^{s, \upsilon}(W)$.
Finally, we compare results obtained in this article
with corresponding known results.
Compared to scalar weights, the lack of the
separability between matrix weights and vector-valued functions
brings some essential difficulties when one tries to directly
establish the $\varphi$-transform characterization of
$\dot{A}_{p,q}^{s, \upsilon}(W)$ by following
the well-known approach originally created in \cite{fj90}.
We employ some ideas from \cite{fr04,fr21,rou03}
and some fundamental and important properties of
matrix $\mathcal{A}_{p, \infty}$ weights obtained in \cite{bhyy4}
to circumvent these difficulties.
More precisely, letting $\mathbb{A}$ be a sequence of reducing
operators of order $p$ for $W$, we first prove
$\dot{a}_{p,q}^{s, \upsilon}(W)=
\dot{a}_{p,q}^{s, \upsilon}(\mathbb{A})$
(the generalized averaging BTL-type sequence space)
and $\dot{A}_{p,q}^{s, \upsilon}(W)=
\dot{A}_{p,q}^{s, \upsilon}(\mathbb{A})$
(the generalized averaging BTL-type space)
and then establish the $\varphi$-transform characterization
of $\dot{A}^{s,\upsilon}_{p,q}(\mathbb{A})$.
Based on these, we finally obtain the $\varphi$-transform characterization
of $\dot{A}^{s,\upsilon}_{p,q}(W)$.
Using the molecular and the wavelet characterizations
of $\dot{A}^{s,\upsilon}_{p,q}(\mathbb{A})$, we also
obtain the boundedness of trace, pseudo-differential,
and Calder\'{o}n--Zygmund operators on $\dot{A}^{s,\upsilon}_{p,q}(W)$,
which are presented in \cite{yyz24} to limit the length of this article.

The \emph{novelty} of these results lies in the following four aspects.
Firstly, the wide generality of growth functions
guarantees that spaces $\dot{A}^{s,\upsilon}_{p,q}(W)$ include
not only matrix-weighted BTL-type spaces $\dot{A}^{s,\tau}_{p,q}(W)$
studied in \cite{bhyy1,bhyy2,bhyy3,bhyy5} and particularly
matrix-weighted BTL spaces $\dot{A}^{s}_{p,q}(W)$ studied in
\cite{fr04,fr21,rou03,rou04} (see Subsection \ref{s-fs-rela})
but also the space $\dot{F}_{p,q}^{s,\upsilon_{1/p, W}}(W)$
that is necessary for studying Question \ref{q1}.
In addition, spaces $\dot{A}^{s,\upsilon}_{p,q}(W)$
we study are associated with the matrix $\mathcal{A}_{p, \infty}$
class, which is a larger weight class than the
matrix $\mathcal{A}_{p}$ class (see \cite{bhyy4}).
All these enable all our results, such as the $\varphi$-transform
characterization, the boundedness of almost diagonal operators, and
the molecular and the wavelet characterizations, to be of wide generality.
Secondly, the growth condition of growth functions
is optimal in the following sense: If $\upsilon$ is a growth function,
then almost diagonal operators are bounded on $\dot{a}^{s,\upsilon}_{p,q}(W)$
and, conversely, if a positive function $\upsilon$ defined on $\mathcal{D}$
which is almost increasing, then the boundedness of almost diagonal operators
on $\dot{a}^{s,\upsilon}_{p,q}(W)$ guarantees that $\upsilon$ is indeed
a growth function (see Proposition \ref{prop-necessity}).
Thus, this framework of BTL-type spaces is optimal.
Thirdly, for any $s\in\mathbb{R}$, $p\in(0, \infty)$,
$q\in(0,\infty]$, any growth function $\upsilon$,
any matrix $\mathcal{A}_{p,\infty}$ weight $W$,
and any sequence $\mathbb{A}$ of positive definite matrices,
we show that $\dot{a}_{p,q}^{s, \upsilon}(W)=
\dot{a}_{p,q}^{s, \upsilon}(\mathbb{A})$ if and only if
$\mathbb{A}$ is a sequence of reducing
operators of order $p$ for $W$ (see Theorem \ref{thm-a(A)=a(W)}).
Even for classical spaces $\dot{a}_{p,q}^{s}(W)$
and $\dot{a}_{p,q}^{s}(\mathbb{A})$, this conclusion is also new.
In particular, for spaces $\dot{f}_{p,2}^{0}(W)$
and $\dot{f}_{p,2}^{0}(\mathbb{A})$, Theorem \ref{thm-a(A)=a(W)}
not only completes the corresponding results
in \cite[p.\,454, 2) of Remarks]{vol97}, but also
answers an open question therein (see Remark \ref{rmk-a(W)=a(A)} for
the details). Fourthly, we answer Question \ref{q1} by giving a
sufficient and necessary condition such that
a counterpart of \eqref{eq-f=f} holds in the matrix-weighted setting
(see Theorems \ref{thm-3=4} and \ref{thm-3=4-F}).
It is worth pointing out that we show \eqref{eq-f=f}
usually fails for Besov-type spaces
(see Proposition \ref{prop-3=4-b}), which also answers
an open question posed in \cite[p.\,464]{yy10}.

The organization of the remainder of this article is as follows.

In Section \ref{s-fs}, we first give a brief review of
matrix weights and introduce generalized
matrix-weighted BTL-type spaces $\dot{A}_{p,q}^{s, \upsilon}(W)$
(see Definition \ref{MWBTL}).
Then we recall the concept of matrix $\mathcal{A}_{p,\infty}$ weights
and introduce the concept of growth functions. Based on these,
we present the $\varphi$-transform characterization
of $\dot{A}_{p,q}^{s, \upsilon}(W)$
(see Theorem \ref{thm-phitransMWBTL}). Next,
we discuss the basic properties of growth functions
on indices and the nontriviality of function spaces related
to growth functions. Finally, we give several
examples of growth functions and clarify the relations of
$\dot{A}_{p,q}^{s, \upsilon}(W)$ with some known spaces.

In Section \ref{s-pf}, we are dedicated to proving Theorem \ref{thm-phitransMWBTL}.
To this end, we first introduce
averaging BTL-type spaces $\dot{A}_{p,q}^{s, \upsilon}(\mathbb{A})$
and their corresponding sequence spaces
$\dot{a}_{p,q}^{s, \upsilon}(\mathbb{A})$.
Then we show $\dot{a}_{p,q}^{s, \upsilon}(W)
=\dot{a}_{p,q}^{s, \upsilon}(\mathbb{A})$ and
$\dot{A}_{p,q}^{s, \upsilon}(W)
=\dot{A}_{p,q}^{s, \upsilon}(\mathbb{A})$,
where $\mathbb{A}$ is a sequence of reducing operators
of order $p$ for $W$ (see Theorems \ref{thm-a(A)=a(W)}
and \ref{thm-A(W)=A(A)}). By this and
establishing the $\varphi$-transform characterization
of $\dot{A}_{p,q}^{s, \upsilon}(\mathbb{A})$
(see Theorem \ref{thm-phitansaverMWBTL}),
we finally prove Theorem \ref{thm-phitransMWBTL}.

In Section \ref{s-ec}, we aim to
characterize $\dot{A}_{p,q}^{s, \upsilon}(W)$, respectively,
via the Peetre-type maximal function
and the Littlewood--Paley functions
(see Theorems \ref{thm-Pee-cha} and \ref{thm-G-L-cha}).
To this end, we make full use of
Proposition \ref{prop-dct-gh}, which can be regarded as
the discrete Littlewood--Paley $g^*_{\lambda}$-function
characterization of $\dot{a}^{s,\upsilon}_{p,q}$.

In Section \ref{s-cmw}, we aim to establish the
molecular and the wavelet
characterizations of $\dot{A}_{p,q}^{s, \upsilon}(W)$
by the boundedness of almost diagonal
operators on $\dot{a}_{p,q}^{s, \upsilon}(W)$.
To this end, we first establish the
boundedness of almost diagonal
operators on $\dot{a}_{p,q}^{s, \upsilon}$
(see Theorem \ref{thm-bound-ad}).
Using this result and the idea of the proof of
\cite[Theorem 4.19]{bhyy5}, we then obtain
the boundedness of almost diagonal
operators on $\dot{a}_{p,q}^{s, \upsilon}(W)$.
Finally, we establish the molecular and the wavelet
characterizations of $\dot{A}_{p,q}^{s, \upsilon}(W)$
(see Theorems \ref{moledecomp} and \ref{Dauwav decomp}).

In Section \ref{s-app}, as applications,
we give the sufficient and necessary conditions
for the invariances of $\dot{f}_{p,q}^{s,\upsilon_{1/p, W}}(W)$
and $\dot{F}_{p,q}^{s,\upsilon_{1/p, W}}(W)$ with $\upsilon_{1/p, W}$
as in \eqref{eq-tau_W} (replaced $\tau$ by $1/p$) on the
integrable index $p$, which is the counterpart of \eqref{eq-f=f}
in the matrix-weighted setting
(see Theorems \ref{thm-3=4} and \ref{thm-3=4-F})
and consequently answers Question \ref{q1}. Finally,
we give the sufficient and necessary conditions
for the Sobolev-type embedding of $\dot{A}_{p,q}^{s, \upsilon}(W)$
by working at the level of sequence spaces (see Theorem \ref{thm-sobolev-B}).

At the end of this introduction, we make some conventions on notation.
Let $\mathbb{N}:=\{1,2,\dots\}$ and $\mathbb{Z}_+:=\mathbb{N}\cup\{0\}$.
All the cubes $Q\subset\mathbb{R}^n$ in this article
are always assumed to have edges parallel to the coordinate
axes. For any cube $Q\subset\mathbb{R}^n$,
let $c_Q$ be its \emph{center}, $\ell(Q)$
be its \emph{edge length}, and $j_Q:=-\log_{2}\ell(Q)$.
For any cube $Q\subset\mathbb{R}^n$ and any $r\in(0,\infty)$,
let $rQ$ be the cube with the same center as $Q$
and the edge length $r\ell(Q)$.
Let $\varphi$ be a complex-valued function defined on $\mathbb{R}^n$.
For any $j\in\mathbb{Z}$ and $x\in\mathbb{R}^n$,
let $\widetilde{\varphi}(x):=\overline{\varphi(-x)}$
and $\varphi_j(x):=2^{jn}\varphi(2^j x)$.
For any $j\in\mathbb{Z}$, $k\in\mathbb{Z}^n$,
and $Q:=Q_{j, k}\in\mathcal{D}$, let
$x_Q:=2^{-j}k$ and, for any $x\in\mathbb{R}^n$, let
\begin{align}\label{eq-phi_Q}
\varphi_Q(x):=2^{\frac{jn}{2}}\varphi\left(2^j x-k\right)
=|Q|^{\frac{1}{2}}\varphi_j\left(x-x_Q\right).
\end{align}
For any $p,q\in\mathbb{R}$,
let $p\wedge q:=\min\{p, q\}$ and $p\vee q:=\max\{p, q\}$.
Let $\mathbf{0}$ denote the \emph{origin}
of $\mathbb{R}^n$ or $\mathbb{C}^m$.
For any measurable set $E\subset\mathbb{R}^n$
with $|E|\in(0,\infty)$ and any
measurable function $f$ on $\mathbb{R}^n$, let
$\fint_E f(x)\,dx:=\frac{1}{|E|}\int_{E}f(x)\,dx$.
For any $p\in(0,\infty]$ and any measurable set
$E\subset\mathbb{R}^n$, the \emph{Lebesgue space} $L^p(E)$
is defined to be the set of all complex-valued
measurable functions $f$ on $E$ such that
\begin{align*}
\|f\|_{L^p(E)}:=
\begin{cases}
\displaystyle{\left[\int_E|f(x)|^p\, dx\right]^{\frac{1}{p}}}
& \text{if } p \in(0, \infty),\\
\displaystyle{\mathop{\operatorname{ess\,sup}}_{x\in E}|f(x)|}
& \text{if } p=\infty
\end{cases}
\end{align*}
is finite. Based on the agreement we made at the beginning of the article,
we simply denote the space
$L^p(\mathbb{R}^n)$ by $L^p$. For any $p\in(0,\infty)$,
let $p':=\frac{p}{p-1}$ if $p\in(1, \infty)$
and let $p':=\infty$ if $p\in(0,1]$ be the \emph{conjugate index} of $p$.
For any $x\in\mathbb{R}^n$ and $r\in(0,\infty)$, let
\begin{align*}
B(x,r):=\left\{y\in\mathbb{R}^n:\ \left|x-y\right|<r\right\}.
\end{align*}
The symbol $C$ denotes a positive constant
that is independent of the main
parameters involved, but may vary from line to line.
The symbol $A\lesssim B$ means that $A\leq CB$ for
some positive constant $C$, while $A\sim B$ means $A\lesssim B\lesssim A$.
Finally, when we prove a theorem (and the like),
in its proof we always use the same symbols as in
the statement itself of that theorem (and the like).

\section{Generalized Matrix-Weighted Function Spaces
$\dot{A}_{p,q}^{s, \upsilon}(W)$\label{s-fs}}	

This section contains three subsections.
In Subsection \ref{s-fs-def},
we introduce generalized matrix-weighted BTL-type
spaces $\dot{A}_{p,q}^{s, \upsilon}(W)$
and the concept of growth functions.
Next, we give the $\varphi$-transform characterization
of $\dot{A}_{p,q}^{s, \upsilon}(W)$.
In Subsection \ref{s-fs-prop},
we discuss some fundamental properties of growth functions
on indices and the nontriviality of function spaces related
to growth functions. Finally, in Subsection \ref{s-fs-rela},
we present several examples of growth functions
and clarify the relations of $\dot{A}_{p,q}^{s, \upsilon}(W)$
with some known spaces, which further implies that the spaces
$\dot{A}_{p,q}^{s, \upsilon}(W)$ in this article are of wide generality.

\subsection{Definition of $\dot{A}_{p,q}^{s, \upsilon}(W)$}\label{s-fs-def}

We start with some basic knowledge about matrices.
In what follows, we always use $m\in\mathbb{N}$ to
denote the dimension of vectors or the order of square matrices.
Let $M_m(\mathbb{C})$ be the set of all $m\times m$
complex-valued matrices. For any $A\in M_m(\mathbb{C})$,
we call $A$ a \emph{unitary matrix} if $A^*A=I_m$,
where $A^*$ is the conjugate transpose of $A$ and
$I_m$ is the identity matrix of order $m$.
A matrix $A\in M_m(\mathbb{C})$ is
said to be \emph{positive definite}
if, for any $\vec z\in\mathbb{C}^m\setminus
\{\mathbf{0}\}$, $\vec z^*A\vec z>0$
and to be \emph{positive semidefinite} if,
for any $\vec z\in\mathbb{C}^m$,
$\vec{z}^*A\vec{z}\geq0$ (see, for example,
\cite[(7.1.1a) and (7.1.1b)]{hj13}).
For any given positive definite matrix $A\in M_m(\mathbb{C})$,
there exists a unitary matrix $U\in M_m(\mathbb{C})$ such that
\begin{align}\label{eq-W^a}
A=U\operatorname{diag}\left(\lambda_1,\ldots,\lambda_m\right)U^{*},
\end{align}
where $\{\lambda_i\}_{i=1}^m$ in $(0,\infty)$ are
all the eigenvalues of $A$ (see, for example,
\cite[Theorems 2.5.6 and 7.2.1]{hj13}).
For any $\alpha\in\mathbb{R}$, let
$A^{\alpha}:=U\operatorname{diag}
(\lambda^{\alpha}_1,\ldots,\lambda^{\alpha}_m)U^{*}$.
We should point out that $A^\alpha$
is independent of $U$ and hence
well defined (see, for example, \cite[p.\,408]{hj94}).

A \emph{scalar weight} is a locally integrable function
on $\mathbb{R}^n$
that takes values in $(0, \infty)$ almost everywhere
(see, for example, \cite[p.\,499]{gra14a}).
Let $D_m(\mathbb{C})$ be the set of all $m\times m$
positive semidefinite complex-valued matrices.
A matrix-valued function
$W:\ \mathbb{R}^n\to D_m(\mathbb{C})$ is called
a \emph{matrix weight} if, for almost every $x\in\mathbb{R}^n$,
$W(x)$ is positive definite and all the
entries of $W$ are locally integrable functions on
$\mathbb{R}^n$ (see, for example, \cite{nt97,vol97}).
It is obvious that, when $m=1$,
a matrix weight reduces to a scalar weight.
Notice that, for any matrix weight
$W:\ \mathbb{R}^n\to D_m(\mathbb{C})$
and any $\alpha\in\mathbb{R}$,
$W^{\alpha}$ is a matrix-valued function
whose entries are all measurable functions on $\mathbb{R}^n$
(see, for example, \cite[Lemma 2.3.5]{rs95}).

We next present some concepts and notation.
Let $\mathcal{S}$ be the set of
all Schwartz functions on $\mathbb{R}^n$
equipped with the well-known topology
determined by a countable family of norms
(see, for example, \cite[Proposition 8.2]{fol99})
and $\mathcal{S}'$ be the set of
all tempered distributions on $\mathbb{R}^n$
equipped with the weak-$\ast$ topology.
For any $f\in L^1$, the \emph{Fourier transform}
$\widehat{f}$ of $f$ is defined by setting,
for any $\xi\in\mathbb{R}^n$,
$\widehat{f}(\xi):=\int_{\mathbb{R}^n}f(x)
e^{-i x\cdot \xi}\,dx$, where $i=\sqrt{-1}$.
The above normalization of the Fourier transform
used in \cite[p.\,165]{fj90}, \cite[p.\,4]{fjw91},
and \cite[p.\,452]{yy10} enables us to directly
apply several results from these works.
Alternatively, if a different normalization is used,
such as one with a $2\pi$ factor in the exponent,
corresponding adjustments need to be made in various other formulas.
Furthermore, for any $f\in\mathcal{S}'$,
the \emph{Fourier transform} $\widehat{f}$
of $f$ is defined by setting,
for any $\varphi\in\mathcal{S}$,
$\langle\widehat f,\varphi\rangle
:=\langle f,\widehat{\varphi}\rangle$.
For any function $f$ on $\mathbb{R}^n$,
its \emph{support} $\operatorname{supp}f$
is defined to be closure of the set
$\{x\in\mathbb{R}^n:\,f(x)\neq0\}$ in $\mathbb{R}^n$.
Suppose that $\varphi\in\mathcal{S}$ satisfies
\begin{align}\label{cond1}
\operatorname{supp}\widehat{\varphi}\subset
\left\{\xi\in\mathbb{R}^n:\ \frac{1}{2}\leq|\xi|\leq2\right\}
\text{\ and\ }\min\left\{\left|\widehat{\varphi}(\xi)\right|:\
\frac{3}{5}\leq|\xi|\leq\frac{5}{3}\right\}>0.
\end{align}
By \cite[Lemma (6.9)]{fjw91},
there exists $\psi\in\mathcal{S}$ satisfying \eqref{cond1} such that,
for any $\xi\in\mathbb{R}^n\backslash\{\mathbf{0}\}$,
\begin{align}\label{cond3}
\sum_{j\in\mathbb{Z}}\overline{\widehat{\varphi}\left(2^j\xi\right)}
\widehat{\psi}\left(2^j\xi\right)=1.
\end{align}
Assume that $A\in\{B, F\}$ and $p,q\in(0,\infty]$.
For any sequence $\{f_j\}_{j\in\mathbb{Z}}$ of measurable
functions on $\mathbb{R}^n$, let
\begin{align}\label{eq-LA}
\|\{f_j\}_{j\in\mathbb{Z}}\|_{L\dot{A}_{p, q}}:=
\begin{cases}
\displaystyle\|\{f_j\}_{j\in\mathbb{Z}}\|_{l^q(L^p)}
:=\left(\sum_{j\in\mathbb{Z}}\|f_j\|
_{L^p}^q\right)^{\frac{1}{q}}
& \text{if } A=B, \\
\displaystyle\|\{f_j\}_{j\in\mathbb Z}\|_{L^p(l^q)}
:=\left\|\left(\sum_{j\in\mathbb{Z}}|f_j|^q
\right)^{\frac{1}{q}}\right\|_{L^p}
& \text{if } A=F
\end{cases}
\end{align}
(with the usual modification made if $q=\infty$).
In what follows, for any $j_0\in\mathbb{Z}$,
let $\mathbf{1}_{j\geq j_0}:=\mathbf{1}_{[j_0,\infty)}(j)$.
Furthermore, for any positive function $\upsilon$
defined on $\mathcal{D}$,
let $L\dot{A}_{p, q}^{\upsilon}$ be the set of
all sequences $\{f_j\}_{j\in\mathbb{Z}}$
of measurable functions on $\mathbb{R}^n$ such that
\begin{align}\label{LA_nu}
\|\{f_j\}_{j\in\mathbb{Z}}\|_{L\dot{A}_{p, q}^{\upsilon}}
:=\sup_{P\in\mathcal{D}}\frac{1}{\upsilon(P)}\|\{f_j\mathbf{1}_P
\mathbf{1}_{j\geq j_P}\}_{j\in\mathbb Z}\|_{L\dot{A}_{p, q}}<\infty.
\end{align}
As in \cite{fjw91}, let
\begin{align*}
\mathcal{S}_{\infty}:=\left\{\varphi\in\mathcal{S}
:\ \int_{\mathbb{R}^n}\varphi(x)x^{\gamma}\,dx=0
\ {\rm for\ any}\ \gamma\in{\mathbb{Z}}_{+}^n\right\},
\end{align*}
and equip $\mathcal{S}_{\infty}$ with the same topology as
$\mathcal{S}$.
We denote the dual space of $\mathcal{S}_{\infty}$
by $\mathcal{S}_{\infty}'$, equipped with the weak-$\ast$ topology.
It is well known that $\mathcal{S}_{\infty}'=\mathcal{S}'/\mathcal{P}$
as topological spaces, where $\mathcal{P}$ is
the set of all polynomials on $\mathbb{R}^n$
(see \cite[Propostion 1.1.3]{gra14b} and
\cite[Proposition 8.1]{ysy10} for more details).
For any $\vec{f}:=(f_1, \dots, f_m)^{T}
\in(\mathcal{S}'_{\infty})^m$ [or $(\mathcal{S}')^m$]
and $\varphi\in\mathcal{S}_{\infty}$ (or $\mathcal{S}$),
let $\varphi*\vec{f}:=(\varphi*f_1, \dots, \varphi*f_m)^{T}$.
We now give the definition of generalized matrix-weighted BTL-type spaces.
\begin{definition}\label{MWBTL}
Let $s\in\mathbb{R}$, $p\in(0, \infty)$,
$q\in(0, \infty]$, and $W$ be a matrix weight.
Suppose that $\upsilon$ is a positive function defined on $\mathcal{D}$
and $\varphi\in{\mathcal{S}}$ satisfies \eqref{cond1}.
The \emph{generalized matrix-weighted Besov-type space}
$\dot{B}_{p,q}^{s, \upsilon}(W,\varphi)$ and the
\emph{generalized matrix-weighted Triebel--Lizorkin-type space}
$\dot{F}_{p,q}^{s, \upsilon}(W,\varphi)$ are respectively defined to be the sets
of all $\vec{f}\in(\mathcal{S}'_{\infty})^m$ such that
\begin{align*}
\left\|\vec{f}\right\|_{\dot{A}_{p,q}^{s, \upsilon}(W, \varphi)}	
:=\left\|\left\{2^{js}\left|W^{\frac1{p}}
\left(\varphi_j*\vec{f}\right)\right|\right\}
_{j\in\mathbb{Z}}\right\|_{L\dot{A}_{p, q}^{\upsilon}}<\infty,
\end{align*}
where $A\in\{B, F\}$ and
$\|\cdot\|_{L\dot{A}_{p, q}^{\upsilon}}$ is as in \eqref{LA_nu}.
\end{definition}

To develop a real-variable theory
of $\dot{A}_{p,q}^{s,\upsilon}(W, \varphi)$,
it is natural to add some assumptions on $\upsilon$ and $W$.
We first recall that   a scalar weight $w\in A_\infty$
if and only if
\begin{align}\label{eq-w-const}
\left[w\right]_{A_{\infty}}&:=
\sup_{\operatorname{cube} Q\subset\mathbb{R}^n}\fint_Qw(x)\,dx
\exp\left(\fint_Q \log\left(\left[w(x)
\right]^{-1}\right)\, dx\right)<\infty
\end{align}
(see, for example, \cite[Definition 7.3.1 and Theorem 7.3.3]{gra14a}
for more equivalent definitions of $A_\infty$).
In the matrix-weighted setting, a natural counterpart of
scalar $A_{\infty}$ class is the
matrix $\mathcal{A}_{p,\infty}$ class,
which was originally introduced in \cite[(2.2)]{vol97}
for any $p\in(1,\infty)$.
The following equivalent definition of
matrix $\mathcal{A}_{p,\infty}$ weights
for any $p\in(0,\infty)$ was established in
\cite[Proposition 3.7]{bhyy4}.
For any $A\in M_m(\mathbb{C})$,
the \emph{operator norm} $\|A\|$ of $A$ is defined by
setting $\|A\|:=\sup_{\vec z\in\mathbb{C}^m,
|\vec z|=1}|A\vec{z}|$.

\begin{definition}
Let $p\in(0,\infty)$. A matrix weight
$W$ is called an $\mathcal{A}_{p,\infty}(\mathbb{R}^n,
\mathbb{C}^m)$-\emph{matrix weight},
denoted by $W\in\mathcal{A}_{p,\infty}(\mathbb{R}^n,
\mathbb{C}^m)$, if $W$ satisfies that,
for any cube $Q\subset\mathbb{R}^n$,
\begin{align*}
\max\left\{\log\left(\fint_Q\left\|W^{\frac{1}{p}}(x)
W^{-\frac{1}{p}}(\cdot)\right\|^p\,dx\right), 0\right\}\in L^1(Q)
\end{align*}
and
\begin{align*}
[W]_{\mathcal{A}_{p,\infty}(\mathbb{R}^n,\mathbb{C}^m)}
:=\sup_{\operatorname{cube} Q\subset\mathbb{R}^n}
\exp\left(\fint_Q\log\left(\fint_Q
\left\|W^{\frac{1}{p}}(x)W^{-\frac{1}{p}}(y)
\right\|^p\,dx\right)\,dy\right)<\infty.
\end{align*}
When no confusion arises, we simply write $W\in\mathcal{A}_{p,\infty}$.
\end{definition}

By \eqref{eq-w-const}, we find that, for any $p\in(0,\infty)$,
$\mathcal{A}_{p,\infty}(\mathbb{R}^n,\mathbb{C})=A_\infty$.
Next, we introduce the concept of growth functions.

\begin{definition}\label{def-grow-func}
Let $\delta_1, \delta_2, \omega\in\mathbb{R}$.
A positive function $\upsilon:\ \mathcal{D}\to (0, \infty)$
is called a \emph{$(\delta_1, \delta_2; \omega)$-order
growth function} if there exists a positive constant
$C$ such that, for any $Q, R\in\mathcal{D}$,	
\begin{align*}
\frac{\upsilon(Q)}{\upsilon(R)}\leq C
\left[1+\frac{|x_Q-x_R|}
{\ell(Q)\vee \ell(R)}\right]^{\omega}
\begin{cases}
\displaystyle{\left(\frac{|Q|}{|R|}\right)^{\delta_1}}
& \text{if } \ell(Q) \leq \ell(R), \\
\displaystyle{\left(\frac{|Q|}{|R|}\right)^{\delta_2}}
& \text{if } \ell(R)<\ell(Q).
\end{cases}
\end{align*}
The set of all $(\delta_1, \delta_2; \omega)$-order
growth functions is denoted by
$\mathcal{G}(\delta_1, \delta_2; \omega)$.
\end{definition}

With the above assumptions on $W$ and $\upsilon$,
we can give the $\varphi$-transform
characterization of $\dot{A}_{p,q}^{s, \upsilon}(W, \varphi)$.
To this end, we first recall the definition of
$\varphi$-transforms (see, for example,
\cite{fj90}) and introduce the sequence spaces
$\dot{a}_{p,q}^{s, \upsilon}(W)$
corresponding to $\dot{A}_{p,q}^{s, \upsilon}(W, \varphi)$.

In what follows, for any $f\in\mathcal{S}'$ (or $\mathcal{S}'_{\infty}$)
and $\varphi\in\mathcal{S}$ (or $\mathcal{S}_{\infty}$), let
$\langle f, \varphi\rangle:=f(\overline{\varphi})$,
where $f(\cdot)$ denotes the dual action.
Let $\varphi, \psi\in\mathcal{S}$ satisfy \eqref{cond1}.
The \emph{$\varphi$-transform} $S_{\varphi}$ is
defined by setting, for any $\vec{f}:=(f_1,\dots,f_m)^T
\in(\mathcal{S}'_{\infty})^m$, $S_{\varphi}\vec{f}
:=\{\langle \vec{f},\varphi_Q\rangle\}_{Q\in\mathcal{D}}
:=\{(\langle f_1, \varphi_Q\rangle,\dots,\langle f_m,
\varphi_Q\rangle)^T\}_{Q\in\mathcal{D}}$.
The \emph{inverse $\varphi$-transform} $T_{\psi}$
is defined by setting, for any $\vec{t}
:=\{\vec{t}_Q\}_{Q\in\mathcal{D}}$ in ${\mathbb{C}}^m$,
$T_{\psi}\vec{t}:=\sum_{Q\in\mathcal{D}}\vec{t}_Q\psi_{Q}$ if
this summation converges in $(\mathcal{S}'_{\infty})^m$.
Here, and thereafter, for any $j\in\mathbb{Z}$, let
$\mathcal{D}_j:=\{Q\in\mathcal{D}:\ \ell(Q)=2^{-j}\}$
be the set of all \emph{dyadic cubes in $\mathbb{R}^n$
at the $j$th level}.

\begin{definition}
Let $s\in\mathbb{R}$, $p\in(0, \infty)$,
$q\in(0, \infty]$, $W$ be a matrix weight,
and $\upsilon$ a positive function defined on $\mathcal{D}$.
The \emph{generalized matrix-weighted Besov-type sequence space}
$\dot{b}_{p,q}^{s,\upsilon}(W)$ and the \emph{generalized matrix-weighted
Triebel--Lizorkin-type sequence space}
$\dot{f}_{p,q}^{s,\upsilon}(W)$ are respectively defined to be the sets of
all $\vec{t}:=\{\vec{t}_Q\}_{Q\in\mathcal{D}}$ in ${\mathbb{C}}^m$ such that
\begin{align*}	
\left\|\vec{t}\right\|_{\dot{a}_{p,q}^{s,\upsilon}(W)}	
:=\left\|\left\{2^{js}\left|W^{\frac1{p}}
\vec{t}_j\right|\right\}_{j\in\mathbb{Z}}
\right\|_{L\dot{A}_{p, q}^{\upsilon}}<\infty,
\end{align*}
where $(A, a)\in\{(B, b), (F, f)\}$,
$\|\cdot\|_{L\dot{A}_{p, q}^{\upsilon}}$
is as in \eqref{LA_nu}, and, for any
$j\in\mathbb{Z}$ and $Q\in\mathcal{D}$,
$\widetilde{\mathbf{1}}_Q:=|Q|^{-\frac{1}{2}}\mathbf{1}_Q$
and
\begin{align}\label{vect_j}
\vec{t}_j:=\sum_{Q\in\mathcal{D}_j}
\widetilde{\mathbf{1}}_Q\vec{t}_Q.
\end{align}
\end{definition}

We now state the $\varphi$-transform characterization
of $\dot{A}_{p,q}^{s, \upsilon}(W, \varphi)$ as follows, whose proof is given
in the next section.

\begin{theorem}\label{thm-phitransMWBTL}
Let $(A, a)\in\{(B, b), (F, f)\}$, $s\in\mathbb{R}$,
$p\in(0, \infty)$, $q\in(0,\infty]$,
and $W\in\mathcal{A}_{p,\infty}$.
Assume that $\upsilon\in\mathcal{G}(\delta_1, \delta_2; \omega)$ with
\begin{align}\label{eq-delta1<0}
\delta_2\in[0, \infty),\ \delta_1\in(-\infty, \delta_2],
\text{\ and\ }\omega\in[0, n(\delta_2-\delta_1)]
\end{align}
and $\varphi, \psi\in\mathcal{S}$ both satisfy \eqref{cond1}.
Then the following statements hold.
\begin{itemize}
\item[{\rm (i)}] 	The maps $S_{\varphi}:\
\dot{A}_{p,q}^{s,\upsilon}(W, \widetilde{\varphi})
\to\dot{a}_{p,q}^{s,\upsilon}(W)$
and $T_{\psi}:\ \dot{a}_{p,q}^{s,\upsilon}(W)\to
\dot{A}_{p,q}^{s,\upsilon}(W, \varphi)$ are bounded,
where $\widetilde{\varphi}(x):=\overline{\varphi(-x)}$
for any $x\in\mathbb{R}^n$.
Moreover, if $\varphi, \psi$ further satisfy \eqref{cond3},
then $T_\psi \circ S_{\varphi}$ is the identity
on $\dot{A}_{p,q}^{s,\upsilon}(W, \widetilde{\varphi})
=\dot{A}_{p,q}^{s,\upsilon}(W, \varphi)$.
\item[{\rm (ii)}] If $\varphi^{(1)}, \varphi^{(2)}\in\mathcal{S}$
both satisfy \eqref{cond1}, then
$\dot{A}_{p,q}^{s,\upsilon}(W, \varphi^{(1)})=
\dot{A}_{p,q}^{s,\upsilon}(W, \varphi^{(2)})$
with equivalent quasi-norms.
\end{itemize}
\end{theorem}

\begin{remark}
\begin{itemize}
\item[{\rm (i)}] Suppose that $\upsilon$ is
a positive function  on $\mathcal{D}$.
In Proposition \ref{prop-necessity},
we prove that, under a mild assumption on $\upsilon$,
the growth condition on $\upsilon$ as in
Definition \ref{def-grow-func} is necessary
for the boundedness of
almost diagonal operators on $\dot{a}_{p,q}^{s,\upsilon}(W)$.
On the other hand, after discussing the nontriviality
of function spaces related to growth functions
in the next subsection, we can justify that,
to study $\dot{A}_{p,q}^{s,\upsilon}(W,\varphi)$
and $\dot{a}_{p,q}^{s,\upsilon}(W)$, the assumptions that
$\delta_1,\delta_2,\omega$ satisfy \eqref{eq-delta1<0} and
$\upsilon\in\mathcal{G}(\delta_1, \delta_2; \omega)$
are reasonable (see Remark \ref{rmk-reason}).

\item[{\rm (ii)}] Let all the symbols be the same as in
Theorem \ref{thm-phitransMWBTL}.
By Theorem \ref{thm-phitransMWBTL}(ii), we find that
the space $\dot{A}_{p,q}^{s,\upsilon}(W, \varphi)$
is independent of the choice of $\varphi$.
Thus, we simply write  $\dot{A}_{p,q}^{s,\upsilon}(W)$ instead of $\dot{A}_{p,q}^{s,\upsilon}(W, \varphi)$.
\end{itemize}
\end{remark}

\subsection{Properties of Growth Functions $\upsilon$
and Nontriviality of $\dot{A}_{p,q}^{s,\upsilon}(W)$}\label{s-fs-prop}

We first give a proposition that determines the
essential ranges of indices for which
growth functions make sense.

\begin{proposition}\label{prop-grow-func}
Let $\delta_1,\widetilde{\delta}_1,\delta_2,
\widetilde{\delta}_2,\omega,\widetilde{\omega}\in\mathbb{R}$.		
Then the following statements hold.
\begin{itemize}
\item[{\rm (i)}] The set
$\mathcal{G}(\delta_1, \delta_2; \omega)\neq\emptyset$
if and only if $\delta_2\geq \delta_1$ and
$\omega\geq0$.
\item[{\rm (ii)}]
Assume that $\delta_2\geq \delta_1$ and $0\leq s<t<\infty$.
Then $\mathcal{G}(\delta_1, \delta_2; s)
\subset\mathcal{G}(\delta_1, \delta_2; t)$.
Moreover,
\begin{align}\label{eq-G(s)-G(t)}
\mathcal{G}(\delta_1, \delta_2; s)
=\mathcal{G}(\delta_1, \delta_2; t)
\text{\ if\ and\ only\ if\ }
s\geq n(\delta_2-\delta_1).
\end{align}
\item[{\rm (iii)}] Suppose that
$0\leq\omega\leq n(\delta_2-\delta_1)$ and
$0\leq\widetilde{\omega}\leq n(\widetilde{\delta}_2
-\widetilde{\delta}_1)$. Then
$\mathcal{G}(\delta_1, \delta_2; \omega)
=\mathcal{G}(\widetilde{\delta}_1,
\widetilde{\delta}_2; \widetilde{\omega})$
if and only if $\delta_1=\widetilde{\delta}_1,
\delta_2=\widetilde{\delta}_2$,
and $\omega=\widetilde{\omega}$.
\end{itemize}
\end{proposition}
\begin{proof}
We first prove the sufficiency of (i).
To this end, let $\delta_2\geq \delta_1$, $\omega\geq0$,
and, for any $Q\in\mathcal{D}$,
$\upsilon(Q):=|Q|^{\delta_1}$.
From Definition \ref{def-grow-func},
it follows that $\upsilon\in\mathcal{G}(\delta_1, \delta_1; 0)
\subset\mathcal{G}(\delta_1, \delta_2; \omega)$,
which completes the proof of the sufficiency of (i).
Next, we show the necessity of (i).
Assuming $\mathcal{G}(\delta_1, \delta_2; \omega)\neq\emptyset$,
we can pick one $\upsilon\in\mathcal{G}(\delta_1, \delta_2; \omega)$.
If $\delta_2<\delta_1$, by the growth condition of $\upsilon$,
we find that, for any $Q, R\in\mathcal{D}$ with
$x_Q=x_R=\mathbf{0}$ and $\ell(Q)\geq\ell(R)$,
\begin{align*}
1=\frac{\upsilon(Q)}{\upsilon(R)}
\frac{\upsilon(R)}{\upsilon(Q)}
\lesssim\left(\frac{|Q|}{|R|}\right)^{\delta_2-\delta_1},
\end{align*}
which induces a contradiction as $\ell(Q)\to\infty$
and further implies that $\delta_2\geq \delta_1$. If $\omega<0$,
from the growth condition of $\upsilon$ again, we infer that,
for any $Q, R\in\mathcal{D}$ with $\ell(Q)=\ell(R)$,
\begin{align*}
\left[1+\frac{|x_Q-x_R|}{\ell(Q)}\right]^{-\omega}\lesssim
\frac{\upsilon(Q)}{\upsilon(R)}\lesssim\left[1+\frac{|x_Q-x_R|}
{\ell(Q)}\right]^{\omega},
\end{align*}
which also induces a contradiction as $|x_Q-x_R|\to\infty$
and further implies that $\omega\geq0$.
This finishes the proof of the necessity of (i)
and hence (i).

We now prove (ii). By Definition \ref{def-grow-func},
we find that $\mathcal{G}(\delta_1, \delta_2; s)
\subset\mathcal{G}(\delta_1, \delta_2; t)$.
To show the necessity of \eqref{eq-G(s)-G(t)},
it suffices to prove, when $s<n(\delta_2-\delta_1)$,
$\mathcal{G}(\delta_1, \delta_2; s)
\subsetneqq\mathcal{G}(\delta_1, \delta_2; t)$.
To this end, we only need to find one
$\upsilon$ such that $\upsilon\in\mathcal{G}(\delta_1, \delta_2; t)$
but $\upsilon\notin\mathcal{G}(\delta_1, \delta_2; s)$.
For any $Q\in\mathcal{D}$, let $\upsilon(Q):=[|x_Q|+\ell(Q)]^{r}
|Q|^{\delta_2-\frac{r}{n}}$, where $r:=\min\{t, n(\delta_2-\delta_1)\}$.
Applying the construction of $\upsilon$,
the triangle inequality of $|\cdot|$,
and Definition \ref{def-grow-func},
we obtain, for any $Q, R\in\mathcal{D}$,
\begin{align*}
\frac{\upsilon(Q)}{\upsilon(R)}&=
\left[\frac{|x_{Q}|+\ell(Q)}{|x_{R}|+\ell(R)}\right]^{r}
\left(\frac{|Q|}{|R|}\right)^{\delta_2-\frac{r}{n}}
\leq\left\{\frac{|x_{R}|+[\ell(Q)\vee\ell(R)]
}{|x_{R}|+\ell(R)}+\frac{|x_{Q}-x_{R}|}{\ell(R)}\right\}^{r}
\left(\frac{|Q|}{|R|}\right)^{\delta_2-\frac{r}{n}}\\
&\leq\left[\frac{\ell(Q)\vee\ell(R)
}{\ell(R)}+\frac{|x_{Q}-x_{R}|}{\ell(R)}\right]^{r}
\left(\frac{|Q|}{|R|}\right)^{\delta_2-\frac{r}{n}}\\
&=\left[1+\frac{|x_Q-x_R|}
{\ell(Q)\vee \ell(R)}\right]^{r}
\left[\frac{\ell(Q)\vee\ell(R)}{\ell(R)}\right]^{r}	
\left(\frac{|Q|}{|R|}\right)^{\delta_2-\frac{r}{n}}\\
&=\left[1+\frac{|x_Q-x_R|}
{\ell(Q)\vee \ell(R)}\right]^{r}
\begin{cases}
\displaystyle{\left(\frac{|Q|}{|R|}
\right)^{\delta_2-\frac{r}{n}}}
& \text{if } \ell(Q) \leq \ell(R), \\
\displaystyle{\left(\frac{|Q|}{|R|}\right)^{\delta_2}}
& \text{if } \ell(R) < \ell(Q),
\end{cases}
\end{align*}
which further implies that $\upsilon\in\mathcal{G}
(\delta_2-\frac{r}{n}, \delta_2; r)
\subset\mathcal{G}(\delta_1, \delta_2; t)$. We next prove
$\upsilon\notin\mathcal{G}(\delta_1, \delta_2; s)$.
If $\upsilon\in\mathcal{G}(\delta_1, \delta_2; s)$,
using the construction and the growth condition of $\upsilon$,
we conclude that, for any $Q, R\in\mathcal{D}$ with
$\ell(Q)=\ell(R)$ and $x_R=\mathbf{0}$,
\begin{align*}
\left[1+\frac{|x_{Q}|}{\ell(R)}\right]^{r}
=\left[\frac{|x_{Q}|+\ell(Q)}{\ell(R)}\right]^{r}
=\frac{\upsilon(Q)}{\upsilon(R)}\lesssim
\left[1+\frac{|x_{Q}|}{\ell(R)}\right]^{s},
\end{align*}
which contradicts the assumption
$s<r$ when $|x_Q|$ is sufficiently large
and further implies that
$\upsilon\notin\mathcal{G}(\delta_1, \delta_2; s)$.
This finishes the proof of the necessity of \eqref{eq-G(s)-G(t)}.
To show the sufficiency of \eqref{eq-G(s)-G(t)},
it suffices to prove, when $s\geq n(\delta_2-\delta_1)$,
$\mathcal{G}(\delta_1, \delta_2; s)=
\mathcal{G}(\delta_1, \delta_2; n[\delta_2-\delta_1])$.
The inclusion $\mathcal{G}(\delta_1, \delta_2; n[\delta_2-\delta_1])
\subset\mathcal{G}(\delta_1, \delta_2; s)$
follows from Definition \ref{def-grow-func}.
We now show the converse inclusion. To this end,
for any given $\upsilon\in\mathcal{G}
(\delta_1, \delta_2; s)$ and for any $Q, R\in\mathcal{D}$,
choose $j\in\mathbb{Z}$ such that
\begin{align}\label{eq-1}
\ell(Q)+\ell(R)+|x_Q-x_R|\leq 2^{-j}<
2\left[\ell(Q)+\ell(R)+\left|x_Q-x_R\right|\right].
\end{align}
By the triangle inequality of $|\cdot|$, \eqref{eq-1},
and the basic property of dyadic cubes in $\mathbb{R}^n$,
there exist unique $Q', R'\in\mathcal{D}_j$ such that
$Q\subset Q'$, $R\subset R'$,
$|x_{Q'}-x_{R'}|\leq|x_{Q'}-x_{Q}|+|x_{Q}-x_{R}|+|x_{R'}-x_{R}|
\lesssim\ell(Q')+2^{-j}+\ell(R')\sim\ell(R')$, and
\begin{align}\label{eq-lR'}
\ell(R')=2^{-j}\sim\left[\ell(Q)+\ell(R)+|x_Q-x_R|\right]
\sim\left\{\left[\ell(Q)\vee\ell(R)\right]+|x_Q-x_R|\right\}.
\end{align}
From these and the assumption $\upsilon
\in\mathcal{G}(\delta_1, \delta_2; s)$,
we infer that, for any $Q, R\in\mathcal{D}$,
\begin{align}\label{eq-v_QR}
\frac{\upsilon(Q)}{\upsilon(R)}&
=\frac{\upsilon(Q)}{\upsilon(Q')}
\frac{\upsilon(Q')}{\upsilon(R')}
\frac{\upsilon(R')}{\upsilon(R)}
\lesssim\left(\frac{|Q|}{|Q'|}\right)^{\delta_1}
\left[1+\frac{|x_{Q'}-x_{R'}|}{\ell(R')}\right]^{s}
\left(\frac{|R'|}{|R|}\right)^{\delta_2}\\
&\lesssim\left[\frac{\ell(R')}{\ell(Q)\vee
\ell(R)}\right]^{n(\delta_2-\delta_1)}
\left(\frac{|Q|}{|Q|\vee |R|}\right)^{\delta_1}
\left(\frac{|Q|\vee|R|}{|R|}\right)^{\delta_2}\notag\\
&\sim\left[1+\frac{|x_Q-x_R|}
{\ell(Q)\vee \ell(R)}\right]^{n(\delta_2-\delta_1)}
\begin{cases}
\displaystyle{\left(\frac{|Q|}{|R|}\right)^{\delta_1}}
& \text{if } \ell(Q) \leq \ell(R), \\
\displaystyle{\left(\frac{|Q|}{|R|}\right)^{\delta_2}}
& \text{if } \ell(R) < \ell(Q)\notag,
\end{cases}
\end{align}
which further implies that
$\upsilon\in\mathcal{G}(\delta_1, \delta_2; n[\delta_2-\delta_1])$
and hence $\mathcal{G}(\delta_1, \delta_2; s)\subset
\mathcal{G}(\delta_1, \delta_2; n[\delta_2-\delta_1])$.
This finishes the proof of the sufficiency
of \eqref{eq-G(s)-G(t)} and hence (ii).

Finally, we prove (iii). The sufficiency is
obvious. We next show the necessity.
If $\delta_1<\widetilde{\delta}_1$,
for any $Q\in\mathcal{D}$,
let $\upsilon(Q):=|Q|^{\delta_1}$.
Using Definition \ref{def-grow-func}, we conclude that
$\upsilon\in\mathcal G(\delta_1,\delta_2;\omega)
\setminus\mathcal G(\widetilde{\delta}_1,
\widetilde{\delta}_2;\widetilde{\omega})$,
which induces to a contradiction.
Thus, we obtain $\delta_1\geq\widetilde{\delta}_1$.
By the symmetry, we find that $\delta_1
\leq\widetilde{\delta}_1$
and hence $\delta_1=\widetilde{\delta}_1$.
A similar argument gives $\delta_2=\widetilde{\delta}_2$.
From $\delta_1=\widetilde{\delta}_1$,
$\delta_2=\widetilde{\delta}_2$, and the just proved (ii),
it follows that $\omega=\widetilde{\omega}$,
which completes the proof of (iii) and
hence Proposition \ref{prop-grow-func}.
\end{proof}

Before presenting the nontriviality
of function spaces related to growth functions,
we first recall some estimates and properties
that are frequently used in this article.
In what follows, for any $k:=(k_1,\dots,k_n)\in\mathbb{Z}^n$,
let $\|k\|_{\infty}:=\max\{|k_1|,\dots,|k_n|\}$.
The following lemma gives some basic estimates
of growth functions; we omit the details.

\begin{lemma}\label{lem-grow-est}
Let $\delta_2\in\mathbb{R}$, $\delta_1\in(-\infty, \delta_2]$,
$\omega\in[0,n(\delta_2-\delta_1)]$, and
$\upsilon\in\mathcal{G}(\delta_1, \delta_2; \omega)$.
Then the following statements hold.
\begin{itemize}
\item[{\rm (i)}] There exists a positive constant
$C$ such that, for any $Q,P\in\mathcal{D}$ with
$Q\subset P$,
\begin{align*}
\frac{1}{C}\left(\frac{|Q|}{|P|}\right)^{\delta_2}\leq
\frac{\upsilon(Q)}{\upsilon(P)}\leq C
\left(\frac{|Q|}{|P|}\right)^{\delta_1}.
\end{align*}
Moreover, if $\delta_1\in[0,\delta_2]$, then,
for any $Q,P\in\mathcal{D}$ with $Q\subset P$,
$\upsilon(Q)\leq C\upsilon(P)$.
\item[{\rm (ii)}] For any $P\in\mathcal{D}$ and
$k\in\mathbb{Z}^n$, one has $\upsilon(P)\sim(1+|k|)^{\omega}
\upsilon(P+k\ell(P))$, where the positive
equivalence constants are independent of $P$ and $k$.
\item[{\rm (iii)}] For any $P\in\mathcal{D}$ and
$k\in\mathbb{Z}^n$ with $\|k\|_{\infty}\leq1$,
one has $\upsilon(P)\sim\upsilon(P+k\ell(P))$, where the positive
equivalence constants are independent of $P$ and $k$.
\end{itemize}
\end{lemma}

We next recall some fundamental properties of
dyadic cubes in $\mathbb{R}^n$ by the following
two lemmas. The proofs of these properties are
well known or by some simple geometrical
observations and computations; we omit the details.

\begin{lemma}\label{lem-3P-sum1}
\begin{itemize}
\item[{\rm (i)}] For any $j\in\mathbb{Z}$, the cubes in  $ \mathcal{D}_j$
are mutually disjoint  and
$\mathbb{R}^n=\bigcup_{Q\in\mathcal{D}_j}Q$.
\item[{\rm (ii)}] For any $j\in\mathbb{Z}$ and $x\in\mathbb{R}^n$,
$\sum_{Q\in\mathcal{D}_j}\mathbf{1}_Q(x)=1$.
\item[{\rm (iii)}] For any $P\in\mathcal{D}$,
$3P=\bigcup_{\{k\in\mathbb{Z}^n:\ \|k\|_{\infty}\leq1\}} [P+k\ell(P)]$ and
$(3P)^\complement=\bigcup_{\{k\in\mathbb{Z}^n:\ \|k\|_{\infty}\geq2\}}[P+k\ell(P)]$.
\item[{\rm (iv)}] Let $j\in\mathbb{Z}$ and $P\in\mathcal{D}_j$.
There exist $\{Q_\eta\}_{\eta=1}^{2^n}$ in $\mathcal{D}_{j+1}$ such that
$P=\bigcup_{\eta=1}^{2^n}Q_\eta$. For any $i\in(-\infty,j]\cap\mathbb{Z}$,
there also exists a unique $P_i\in\mathcal{D}_{i}$ such that
$P\subset P_i$.
\end{itemize}
\end{lemma}

\begin{lemma}\label{lem-xy-QR-P}
\begin{itemize}
\item[{\rm (i)}] For any $i,j\in\mathbb{Z}$, $R\in\mathcal{D}_i$,
$Q\in\mathcal{D}_j$, $x\in Q$, and $y\in R$,
$1+2^{i\wedge j}|x-y|\sim 1+2^{i\wedge j}|x_Q-x_R|$,
where the positive equivalence constants depend only on $n$.
\item[{\rm (ii)}] For any $j\in\mathbb{Z}$, $Q,R\in\mathcal{D}_j$,
$x\in Q$, and $y\in R$, $1+2^{j}|x-y|\sim 1+2^{j}|x_Q-x_R|$,
where the positive equivalence constants depend only on $n$.
\item[{\rm (iii)}]
For any $P\in\mathcal{D}$, $k\in\mathbb{Z}^n$
with $\|k\|_{\infty}\geq2$, $j\in[j_P, \infty)\cap\mathbb{Z}$,
$x\in P$, and $y\in P+k\ell(P)$,
\begin{align*}
1+2^j|x-y|\sim(1+|k|)2^{j-j_P},
\end{align*}
where the positive equivalence constants depend only on $n$.
\end{itemize}
\end{lemma}

Finally, we present the nontriviality of the above
introduced function spaces related to growth functions.

\begin{proposition}\label{prop-LA}
Let $A\in\{B, F\}$ and $p,q\in(0,\infty]$.
Then the following assertions hold.
\begin{itemize}
\item[{\rm (i)}] If $\delta_2\in(-\infty, 0)$,
$\delta_1\in(-\infty, \delta_2]$, and
$\omega\in[0,n(\delta_2-\delta_1)]$, then,
for any $\upsilon\in\mathcal{G}(\delta_1, \delta_2; \omega)$,
the space $L\dot{A}_{p, q}^{\upsilon}$ is trivial,
that is, $L\dot{A}_{p, q}^{\upsilon}$ only contains
the sequences of measurable functions on $\mathbb{R}^n$
whose each component equals to $0$ almost everywhere.
\item[{\rm (ii)}] If
\begin{align}\label{eq-delta1>0}
\delta_2\in[0, \infty),\ \delta_1\in[0,\delta_2],
\text{ and } \omega\in[0, n(\delta_2-\delta_1)],
\end{align}
then, for any $\upsilon\in\mathcal{G}(\delta_1, \delta_2; \omega)$,
the space $L\dot{A}_{p, q}^{\upsilon}$ is nontrivial.
\item[{\rm (iii)}] If $\delta_2\in[0, \infty)$, $\delta_1\in(-\infty, 0)$,
and $\omega\in[0, n(\delta_2-\delta_1)]$, then there exist
$\upsilon_1\in\mathcal{G}(\delta_1, \delta_2; \omega)$
and $\upsilon_2\in\mathcal{G}(\delta_1, \delta_2; \omega)
\backslash\mathcal{G}(0, \delta_2; n\delta_2)$
such that $L\dot{A}_{p, q}^{\upsilon_1}$ and
$L\dot{A}_{p, q}^{\upsilon_2}$ are respectively trivial and nontrivial .
\end{itemize}
\end{proposition}

\begin{proof}
We first prove (i). Applying \eqref{LA_nu}, we obtain,
for any sequence $\{f_j\}_{j\in\mathbb{Z}}$ of measurable
functions on $\mathbb{R}^n$ whose each component equals
to $0$ almost everywhere, $\|\{f_j\}_{j\in\mathbb{Z}}
\|_{L\dot{A}_{p, q}^{\upsilon}}=0$ and hence
$\{f_j\}_{j\in\mathbb{Z}}\in L\dot{A}_{p, q}^{\upsilon}$.
Assume that $\{f_j\}_{j\in\mathbb{Z}}$
is a sequence of measurable functions on $\mathbb{R}^n$ such that
$\|\{f_j\}_{j\in\mathbb{Z}}\|_{L\dot{A}_{p, q}^{\upsilon}}<\infty$.
Given $j\in\mathbb{Z}$ and $P\in\mathcal{D}$,
for any $i\in(-\infty, j_P\wedge j]\cap\mathbb{Z}$,
let $P_i\in\mathcal{D}_i$ be as in Lemma \ref{lem-3P-sum1}(iv).
Using this, \eqref{LA_nu}, and the assumption $\|\{f_j\}_{j\in\mathbb{Z}}\|_{L\dot{A}_{p, q}^{\upsilon}}<\infty$,
we conclude that
\begin{align}\label{eq-vPi}
\frac{1}{\upsilon(P_i)}\|f_j\mathbf{1}_{P}\|_{L^p}
\leq\frac{1}{\upsilon(P_i)}\|\{f_j\mathbf{1}_{P_i}
\mathbf{1}_{j\geq i}\}_{j\in\mathbb Z}\|_{L\dot{A}_{p, q}}
\leq\|\{f_j\}_{j\in\mathbb{Z}}\|_{L\dot{A}_{p, q}^{\upsilon}}<\infty.
\end{align}
For any $i\in(-\infty, j_P\wedge j]\cap\mathbb{Z}$,
by Lemma \ref{lem-grow-est}(i) combined with
$Q$ and $P$ replaced, respectively, by $P$ and $P_i$
and with the assumption $\delta_2\in(-\infty, 0)$,
we find that
\begin{align*}
\upsilon(P_i)=\upsilon(P)
\frac{\upsilon(P_i)}{\upsilon(P)}\lesssim
\upsilon(P)2^{(j_P-i)n\delta_2}\to0
\text{\ as\ } i\to-\infty,
\end{align*}
which, together with \eqref{eq-vPi}, further implies that,
for any given $j\in\mathbb{Z}$ and $P\in\mathcal{D}$,
$\|f_j\mathbf{1}_{P}\|_{L^p}=0$ and hence,
for almost every $x\in P$, $f_j(x)=0$.
Applying this and the arbitrariness of $j\in\mathbb{Z}$
and $P\in\mathcal{D}$, we obtain, for any $j\in\mathbb{Z}$
and almost every $x\in\mathbb{R}^n$, $f_j(x)=0$.
This finishes the proof of (i).

Next, we show (ii). For any $j\in\mathbb{Z}$, let
$f_j:=\mathbf{1}_{Q_{0,\mathbf{0}}}$ if $j=0$
and let $f_j:=0$ otherwise. Using this construction, \eqref{LA_nu},
and Lemma \ref{lem-grow-est}(i) with $Q$ replaced by
$Q_{0,\mathbf{0}}$ and with the assumption
$\delta_1\in[0, \delta_2]$, we conclude that
\begin{align*}
\|\{f_j\}_{j\in\mathbb{Z}}\|_{L\dot{A}_{p, q}^{\upsilon}}
=\sup_{P\in\mathcal{D}, P\supset Q_{0,\mathbf{0}}}
\frac{1}{\upsilon(P)}\sim
\frac{1}{\upsilon(Q_{0,\mathbf{0}})}<\infty
\end{align*}
and hence $\{f_j\}_{j\in\mathbb{Z}}\in L\dot{A}_{p, q}^{\upsilon}$,
which completes the proof of (ii).

Finally, we prove (iii). To construct the desired $\upsilon_1$ and $\upsilon_2$,
suppose that $\widetilde{\delta}_2\in(\delta_1,0)$ and
$\widetilde{\delta}_1\in[\delta_1,\widetilde{\delta}_2]$.
By Proposition \ref{prop-grow-func}(i), we find that
$\mathcal{G}(\widetilde{\delta}_1, \widetilde{\delta}_2; 0)\neq\emptyset$.
From this, the just proved (i), and Definition \ref{def-grow-func},
we infer that, for any $\upsilon_1\in\mathcal{G}(\widetilde{\delta}_1,
\widetilde{\delta}_2;0)\subset\mathcal{G}(\delta_1, \delta_2; \omega)$,
$L\dot{A}_{p, q}^{\upsilon_1}$ is trivial. Let $\alpha\in(\delta_1, 0)$
and $\beta\in[0, \delta_2]$. For any $Q\in\mathcal{D}$, let
$\upsilon_2(Q):=[\ell(Q)]^{\beta}$ if $\ell(Q)\geq1$
and let $\upsilon_2(Q):=[\ell(Q)]^{\alpha}$	if
$\ell(Q)<1$. By this and Definition \ref{def-grow-func},
it is easy to verify that
$\upsilon_2\in\mathcal{G}(\alpha, \beta; 0)
\subset\mathcal{G}(\delta_1, \delta_2; \omega)$;
we omit the details.
We next show $\upsilon_2\notin\mathcal{G}(0, \delta_2; n\delta_2)$.
If $\upsilon_2\in\mathcal{G}(0, \delta_2; n\delta_2)$,
using the construction of $\upsilon_2$ and
Lemma \ref{lem-grow-est}(i), we conclude that,
for any $Q,P\in\mathcal{D}$ with $Q\subset P$, $\ell(P)\geq1$,
and $\ell(Q)<1$, $\frac{[\ell(Q)]^{\alpha}}{[\ell(P)]^{\beta}}=
\frac{\upsilon_2(Q)}{\upsilon_2(P)}\lesssim1$.
Since $\alpha\in(\delta_1, 0)$, letting $\ell(Q)\to0$,
we obtain a contradiction
and hence $\upsilon_2\notin\mathcal{G}(0, \delta_2; n\delta_2)$.
To prove that $L\dot{A}_{p, q}^{\upsilon_2}$
is nontrivial, let $\{f_j\}_{j\in\mathbb{Z}}$ be
as in the just proved (ii). From this, \eqref{LA_nu},
the definition of $\upsilon_2$,
and the assumption $\beta\in[0, \delta_2]$, we deduce that
\begin{align*}
\|\{f_j\}_{j\in\mathbb{Z}}\|_{L\dot{A}_{p, q}^{\upsilon_2}}
=\sup_{P\in\mathcal{D}, P\supset Q_{0,\mathbf{0}}}
\frac{1}{\upsilon_2(P)}=\sup_{P\in\mathcal{D}, P\supset
Q_{0,\mathbf{0}}}\frac{1}{[\ell(P)]^{\beta}}\leq 1,
\end{align*}
which further implies that $\{f_j\}_{j\in\mathbb{Z}}\in
L\dot{A}_{p, q}^{\upsilon_2}$ and hence $L\dot{A}_{p, q}^{\upsilon_2}$
is nontrivial. This finishes the proof of (iii)
and hence Proposition \ref{prop-LA}.
\end{proof}

\begin{remark}\label{rmk-reason}
By Proposition \ref{prop-grow-func},
we find that Proposition \ref{prop-LA}
covers all the ranges of indices where
growth functions are meaningful.
From this and Proposition \ref{prop-LA},
we infer that, to study
$\dot{A}_{p,q}^{s,\upsilon}(W,\varphi)$,
the assumptions that $\delta_1,\delta_2,\omega$
satisfy \eqref{eq-delta1<0} and
$\upsilon\in\mathcal{G}(\delta_1, \delta_2; \omega)$
are reasonable.
\end{remark}

\subsection{Relations of $\dot{A}_{p,q}^{s, \upsilon}(W)$
with Known Function Spaces}\label{s-fs-rela}

Before discussing the relations of $\dot{A}_{p,q}^{s, \upsilon}(W)$
with several known function spaces, we first recall
a key property of the scalar $A_\infty$ class,
which also serves as one of the motivations
for introducing the concept of growth functions.
For its proof, we refer to, for example,
\cite[(7.2.1) and Proposition 7.2.8]{gra14a}.
In what follows, for any scalar weight $w$
and any measurable set $E\subset\mathbb{R}^n$,
let $w(E):=\int_{E} w(x)\, dx$.
\begin{proposition}\label{prop-RDweight}
If $w\in A_\infty$, then there exist $p\in[1, \infty)$,
$\delta\in(0,1)$, and a positive constant $C$ such that,
for any cube $Q\subset\mathbb{R}^n$
and any measurable set $A\subset Q$,
\begin{align*}
\frac{1}{C}\left(\frac{|A|}{|Q|}\right)^{p}
\leq\frac{w(A)}{w(Q)}\leq C\left(\frac{|A|}{|Q|}\right)^\delta.
\end{align*}
\end{proposition}

We now present some examples of growth functions,
which naturally appear in the
study of function spaces (see, for example, \cite{hl23,sdh20,yy08,yy10,yy13,yyz14}).

\begin{example}\label{examp}
\begin{itemize}
\item[{\rm (i)}]Let $\tau\in\mathbb{R}$ and $\upsilon$
be a positive function defined on $\mathcal{D}$.
Applying Definition \ref{def-grow-func},
we conclude that, for any $Q\in\mathcal{D}$,
$\upsilon(Q)\sim{|Q|}^{\tau}$ with positive
equivalence constants independent of $Q$ if and only if
$\upsilon\in\mathcal{G}(\tau, \tau; 0)$.
\item[{\rm (ii)}]
Let $p\in(0, \infty)$ and $\mathcal{G}_p$
be the set of all nondecreasing functions
$g:\ (0, \infty)\to(0, \infty)$ such that,
for any $t_1, t_2\in(0,\infty)$ with $t_1\leq t_2$,
$g(t_1){t_1}^{-\frac{n}{p}}\geq g(t_2){t_2}^{-\frac{n}{p}}$
(see, for example, \cite[(1.2)]{nns16}
and \cite[Definition 2.7]{hl23}).
Let $g\in\mathcal{G}_p$ and, for any $Q\in\mathcal{D}$,
$\upsilon(Q):=g(\ell(Q))$. By the definition of $\mathcal{G}_p$,
one directly obtains $\upsilon\in\mathcal{G}(0, \frac1{p}; 0)$.
For more details about $\mathcal{G}_p$,
we refer to \cite[Example 2.9]{hl23} and \cite[Section 12.1.2]{sdh20}.
\item[{\rm (iii)}]
Let $w\in A_\infty$ and, for any $Q\in\mathcal{D}$,
$\upsilon(Q):=w(Q)$. Then there exist $p\in[1, \infty)$ and
$\delta\in(0,1)$ such that $\upsilon\in\mathcal{G}(\delta, p; n(p-\delta))$.
In general, let $\widetilde{\upsilon}$
be a positive function defined on the set of
all cubes in $\mathbb{R}^n$.
Assume that there exist $\delta_1, \delta_2\in\mathbb{R}$
with $\delta_2\geq\delta_1$ and a positive constant $C$
such that, for any cubes $Q, R\subset\mathbb{R}^n$,
\begin{align}\label{e11}
\frac{\widetilde{\upsilon}(Q)}{\widetilde{\upsilon}(R)}\leq
C\begin{cases}
\displaystyle{\left(\frac{|Q|}{|R|}\right)^{\delta_1}}
& \text{if } Q\subset R, \\
\displaystyle{\left(\frac{|Q|}{|R|}\right)^{\delta_2}}
& \text{if } R\subset Q.
\end{cases}
\end{align}
By the geometrical property of $\mathbb{R}^n$,
we find that, for any cubes $Q, R\subset\mathbb{R}^n$,
there exists a cube $P\subset\mathbb{R}^n$
such that $Q\cup R\subset P$ and
\begin{align}\label{eq-lP}
\ell(P)\sim\left[\ell(Q)+\ell(R)+|x_Q-x_R|\right]
\sim\left\{\left[\ell(Q)\vee\ell(R)\right]+|x_Q-x_R|\right\}.
\end{align}
Repeating an argument used in the proof of \eqref{eq-v_QR}
with $Q'$, $R'$, and \eqref{eq-lR'} replaced,
respectively, by $P$, $P$, and \eqref{eq-lP},
we conclude that, for any cubes $Q, R\subset\mathbb{R}^n$,
\begin{align}\label{eq-v}
\frac{\widetilde{\upsilon}(Q)}{\widetilde{\upsilon}(R)}
&\lesssim\left[1+\frac{|x_Q-x_R|}
{\ell(Q)\vee \ell(R)}\right]^{n(\delta_2-\delta_1)}
\begin{cases}
\displaystyle{\left(\frac{|Q|}{|R|}\right)^{\delta_1}}
& \text{if } \ell(Q) \leq \ell(R), \\
\displaystyle{\left(\frac{|Q|}{|R|}\right)^{\delta_2}}
& \text{if } \ell(R) < \ell(Q).
\end{cases}
\end{align}
For any cube $Q\subset\mathbb{R}^n$, let
$\widetilde{\upsilon}(Q):=w(Q)$.
By Proposition \ref{prop-RDweight},
we find that $\widetilde{\upsilon}$
satisfies \eqref{e11} with $\delta_1=p$ and
$\delta_2=\delta$, where $p$ and $\delta$ are the same
as in Proposition \ref{prop-RDweight}.
This, together with \eqref{eq-v} and
Definition \ref{def-grow-func}, further implies that
$\upsilon:=\widetilde{\upsilon}|_{\mathcal{D}}
\in\mathcal{G}(\delta, p; n[p-\delta])$.
\end{itemize}
\end{example}

Based on the above examples of growth functions,
we now clarify the relations of
$\dot{A}_{p,q}^{s, \upsilon}(W)$ with some known spaces.
It is worth pointing out that all the growth functions
in the following examples are $(\delta_1, \delta_2; \omega)$-order
growth functions for some $\delta_1, \delta_2, \omega$
satisfying \eqref{eq-delta1>0} and hence \eqref{eq-delta1<0}.
Thus, all the results in this article hold for the spaces
in all the following examples, in which we always suppose
that $A\in\{B, F\}$, $s\in\mathbb{R}$, $p\in(0, \infty)$,
$q\in(0,\infty]$, and $W\in\mathcal{A}_{p,\infty}$.
We start with unweighted function spaces.

\begin{example}\label{exam-BTL}
Let  $m=1$ (the scalar-valued case), $W\equiv1$, $\tau\in[0, \infty)$,
and, for any $Q\in\mathcal{D}$, $\upsilon(Q):=|Q|^\tau$.
The space $\dot{A}_{p,q}^{s, \upsilon}(W)$
is exactly the BTL-type space $\dot{A}_{p,q}^{s, \tau}$
introduced in \cite[Definition 1.1]{yy10}.
Furthermore, $\dot{A}_{p,q}^{s, 0}$
is precisely the well-known BTL space $\dot{A}_{p,q}^{s}$.
\end{example}

\begin{example}
Let $p\in(0, \infty)$ and $\varphi\in\mathcal{G}_p$, where
$\mathcal{G}_p$ is as in Example \ref{examp}(ii).
Let $m=1$, $W\equiv1$, and, for any $Q\in\mathcal{D}$,
$\upsilon(Q):=\varphi(\ell(Q))$.
The space $\dot{A}_{p,q}^{s, \upsilon}(W)$ is
exactly the homogeneous variant of $A_{p,q}^{s, \varphi}$,
which was introduced in \cite[Definition 4.1]{hl23}.
\end{example}

Next, we consider weighted function spaces.

\begin{example}
Let $m=1$, $W:=w\in A_\infty$, $\tau\in[0, \infty)$,
and, for any $Q\in\mathcal{D}$, $\upsilon(Q):=|Q|^{\tau}$.
The space  $\dot{A}_{p,q}^{s, \upsilon}(w)$ reduces
to the space $\dot{A}_{p,q}^{s, \tau}(w)$ introduced in
\cite[Definition 1]{tang13}.
In particular, the space $\dot{A}_{p,q}^{s, 0}(w)$
becomes the weighted BTL space $\dot{A}_{p, q}^{s}(w)$
introduced in \cite[p.\,583]{bui82}. Moreover, for any $Q\in\mathcal{D}$,
let $\upsilon(Q):=[w(Q)]^{\frac{1}{p}}$.
The space $\dot{F}_{p,p}^{s, \upsilon}(w)$ coincides
with the space $\dot{F}_{\infty,p}^{s}(w)$ introduced in \cite[(1.3)]{bt00}.
\end{example}

\begin{example}\label{exam-MWBTL}
For any $Q\in\mathcal{D}$, let $\upsilon(Q):=1$.
The space $\dot{B}_{p,q}^{s, \upsilon}(W)$
is precisely the matrix-weighted Besov space $\dot{B}_{p,q}^{s}(W)$
introduced in \cite[Definition 1.1]{rou03} for any $p\in(1, \infty)$
and \cite[p.\,1227, Definition]{fr04} for any $p\in(0, 1]$.
The space $\dot{F}_{p,q}^{s, \upsilon}(W)$
coincides with the matrix-weighted
Triebel--Lizorkin space $\dot{F}_{p,q}^{s}(W)$
introduced in \cite[p.\,489, (i)]{fr21}.
Moreover, Frazier and Roudenko \cite[Theorems 4.1 and 4.2]{fr21}
also proved that, for any $p\in(1, \infty)$
and any matrix $\mathcal{A}_{p}$ weight
$W$ (see, for instance, \cite[p.\,490]{fr21}
for the definition of the matrix $\mathcal{A}_p$ class),
$\dot{F}_{p, 2}^{0}(W)=L^p(W)$ with equivalent norms,
where $L^p(W)$ is the well-known matrix-weighted
Lebesgue space (see, for example, \cite[p.\,450]{vol97}).
\end{example}
\begin{example}\label{exam-MWBTLtype}
Let $\tau\in[0, \infty)$ and, for any $Q\in\mathcal{D}$,
$\upsilon(Q):=|Q|^{\tau}$.
The space $\dot{A}_{p,q}^{s, \upsilon}(W)$
coincides with the matrix-weighted BTL-type
space $\dot{A}_{p,q}^{s, \tau}(W)$
introduced in \cite[Definition 3.5]{bhyy1}.
Furthermore, the space $\dot{A}_{p,q}^{s, 0}(W)$
is exactly the matrix-weighted BTL
space $\dot{A}_{p,q}^{s}(W)$ in Example \ref{exam-MWBTL}.
\end{example}

Finally, we present a class of growth functions that lead to
new matrix-weighed BTL-type spaces.

\begin{example}\label{exam-tau_W}
Let $\tau\in[0,\infty)$ and, for any $Q\in\mathcal{D}$, let
\begin{align}\label{eq-tau_W}
\upsilon_{\tau, W}(Q):=\left[\int_{Q}\|W(x)\|\,dx\right]^{\tau}.
\end{align}
From \cite[Lemma 5.3]{bhyy4}, it follows that
$\|W\|\in A_{\infty}$. This, combined with
Example \ref{examp}(iii) and Definition \ref{def-grow-func},
further implies that
there exist $\delta_1,\delta_2,\omega$
satisfying \eqref{eq-delta1>0} such that
$\upsilon_{\tau, W}\in\mathcal{G}(\delta_1, \delta_2; \omega)$.
This new space $\dot{A}_{p,q}^{s, \upsilon_{\tau, W}}(W)$
seems more compatible with matrix weights
than the space $\dot{A}_{p,q}^{s, \tau}(W)$
in Example \ref{exam-MWBTLtype}.
Moreover, $\dot{A}_{p,q}^{s, \upsilon_{\tau, W}}(W)$
is necessary to study Question \ref{q1}
(see Theorems \ref{thm-3=4} and \ref{thm-3=4-F}).
This also serves as one of the main motivations
for us to introduce growth functions as in
Definition \ref{def-grow-func}.
\end{example}

\section{Averaging Spaces and Proof of Theorem \ref{thm-phitransMWBTL}\label{s-pf}}

In this section, we are dedicated to proving Theorem \ref{thm-phitransMWBTL}
by the following two subsections. In Subsection \ref{s-pf-e},
we first introduce averaging spaces $\dot{A}_{p,q}^{s, \upsilon}(\mathbb{A})$ along with their corresponding
sequence spaces $\dot{a}_{p,q}^{s, \upsilon}(\mathbb{A})$
and then show $\dot{A}_{p,q}^{s, \upsilon}(W)=
\dot{A}_{p,q}^{s, \upsilon}(\mathbb{A})$ and
$\dot{a}_{p,q}^{s, \upsilon}(W)=\dot{a}_{p,q}^{s, \upsilon}(\mathbb{A})$,
where $\mathbb{A}$ is a sequence of reducing operators of order $p$ for $W$.
In Subsection \ref{s-pf-pf},
by first establishing the $\varphi$-transform
characterization of $\dot{A}_{p,q}^{s, \upsilon}(\mathbb{A})$,
we then show Theorem \ref{thm-phitransMWBTL}.

\subsection{Coincidence of Matrix-Weighted Spaces
and Averaging Spaces}\label{s-pf-e}

We start with the concept of reducing operators,
which was originally introduced by Volberg
in \cite[(3.1)]{vol97} and
plays a key role in the study of matrix weights.

\begin{definition}\label{def-red-ope}
Let $p\in(0,\infty)$ and $W$ be a matrix weight.
A sequence $\{A_Q\}_{Q\in\mathcal{D}}$ of positive definite
matrices is called \emph{a sequence of reducing operators
of order $p$ for $W$} if, for any $Q\in\mathcal{D}$
and $\vec{z}\in\mathbb{C}^m$,
\begin{equation}\label{eq-red-ope}
\left|A_Q\vec z\right|
\sim\left[\fint_Q\left|W^{\frac{1}{p}}(x)\vec{z}\right|^p
\,dx\right]^{\frac{1}{p}},
\end{equation}
where the positive equivalence constants
are independent of $Q$ and $\vec{z}$.
\end{definition}

The existence of reducing operators is guaranteed
by \cite[Proposition 1.2]{gol03} for any $p\in(1, \infty)$
and \cite[p.\,1237]{fr04} for any $p\in(0, 1]$.
Observe that \eqref{eq-red-ope} shows that
there exists a relation between
matrix weights and special sequences of
positive definite matrices.
Motivated by this, we introduce the
averaging spaces, which can be shown to
coincide with matrix-weighted spaces.

\begin{definition}\label{averBTL}
Let $s\in\mathbb{R}$, $p, q\in(0,\infty]$,
and $\mathbb{A}:=\{A_Q\}_{Q\in\mathcal{D}}$
be a sequence of positive definite matrices.
Assume that $\delta_1,\delta_2,\omega$ satisfy \eqref{eq-delta1<0},
$\upsilon\in\mathcal{G}(\delta_1, \delta_2; \omega)$,
and $\varphi\in\mathcal{S}$ satisfies \eqref{cond1}.
The \emph{generalized averaging Besov-type space}
$\dot{B}_{p,q}^{s, \upsilon}(\mathbb{A},\varphi)$ and,
when $p\in(0, \infty)$, the
\emph{generalized averaging Triebel--Lizorkin-type space}
$\dot{F}_{p,q}^{s, \upsilon}(\mathbb{A}, \varphi)$
are respectively defined to be the sets of all
$\vec{f}\in(\mathcal{S}'_{\infty})^m$ such that
\begin{align*}	
\left\|\vec{f}\right\|_{\dot{A}_{p,q}^{s, \upsilon}(\mathbb{A}, \varphi)}	
:=\left\|\left\{2^{js}\left|\mathbb{A}_j
\left(\varphi_j*\vec{f}\right)\right|\right\}_{j\in\mathbb{Z}}
\right\|_{L\dot{A}_{p, q}^{\upsilon}}<\infty,
\end{align*}
where $A\in\{B, F\}$, $\|\cdot\|_{L\dot{A}_{p, q}^{\upsilon}}$ is as
in \eqref{LA_nu}, and, for any $j\in\mathbb{Z}$,		
\begin{align}\label{A_j}
\mathbb{A}_j:=\sum_{Q\in\mathcal{D}_j}\mathbf{1}_QA_Q.
\end{align}
\end{definition}

\begin{definition}\label{averBTLseq}
Let $s\in\mathbb{R}$, $p, q\in(0,\infty]$,
and $\mathbb{A}:=\{A_Q\}_{Q\in\mathcal{D}}$
be a sequence of positive definite matrices.
Suppose that $\delta_1,\delta_2,\omega$
satisfy \eqref{eq-delta1<0} and
$\upsilon\in\mathcal{G}(\delta_1, \delta_2; \omega)$.
The \emph{generalized averaging
Besov-type sequence space}
$\dot{b}_{p,q}^{s,\upsilon}(\mathbb{A})$ and,
when $p\in(0, \infty)$,
the \emph{generalized averaging
Triebel--Lizorkin-type sequence space}
$\dot{f}_{p,q}^{s,\upsilon}(\mathbb{A})$
are respectively defined to be the sets of
all $\vec{t}:=\{\vec{t}_Q\}_{Q\in\mathcal{D}}$
in ${\mathbb{C}}^m$ such that
\begin{align*}
\left\|\vec{t}\right\|
_{\dot{a}_{p,q}^{s,\upsilon}(\mathbb{A})}	
:=\left\|\left\{2^{js}\left|\mathbb{A}_j
\vec{t}_j\right|\right\}
_{j\in\mathbb{Z}}\right\|_{L\dot{A}
_{p, q}^{\upsilon}}<\infty,
\end{align*}
where $(A, a)\in\{(B, b), (F, f)\}$,
$\|\cdot\|_{L\dot{A}_{p, q}^{\upsilon}}$ is as in \eqref{LA_nu},
and, for any $j\in\mathbb{Z}$, $\mathbb{A}_j$ and $\vec{t}_j$
are as, respectively, in \eqref{A_j} and \eqref{vect_j}.
\end{definition}
\begin{remark}\label{rmk-a(A)-a}
Let all the symbols be the same as in
Definitions \ref{averBTL} and \ref{averBTLseq}.
If $\mathbb{A}:=\{I_m\}_{Q\in\mathcal{D}}$,
where $I_m$ is the identity matrix of order $m$,
we simply denote $\dot{A}_{p,q}^{s, \upsilon}(\mathbb{A}, \varphi)$
and $\dot{a}_{p,q}^{s,\upsilon}(\mathbb{A})$, respectively,
by $\dot{A}_{p,q}^{s, \upsilon}(\mathbb{C}^m, \varphi)$ and
$\dot{a}_{p,q}^{s,\upsilon}(\mathbb{C}^m)$.
Furthermore, when $m=1$, we denote
$\dot{A}_{p,q}^{s, \upsilon}(\mathbb{C}^m, \varphi)$ and
$\dot{a}_{p,q}^{s,\upsilon}(\mathbb{C}^m)$, respectively, by
$\dot{A}_{p,q}^{s, \upsilon}(\varphi)$ and
$\dot{a}_{p,q}^{s,\upsilon}$. Observe that, for any $\vec{t}:=
\{\vec{t}_Q\}_{Q\in\mathcal{D}}$ in ${\mathbb{C}}^m$,
\begin{align}\label{eq-a(A)-a}
\left\|\vec{t}\right\|
_{\dot{a}_{p,q}^{s,\upsilon}(\mathbb{A})}
=\left\|\left\{A_Q\vec{t}_Q\right\}
_{Q\in\mathcal{D}}\right\|_{\dot{a}_{p,q}^{s,\upsilon}
(\mathbb{C}^m)}=\left\|\left\{\left|A_Q\vec{t}_Q
\right|\right\}_{Q\in\mathcal{D}}
\right\|_{\dot{a}_{p,q}^{s,\upsilon}}.
\end{align}
\end{remark}

Before giving the main results of this subsection,
we first present some symbols and concepts.
For any scalar weight $w$, let
$$[w]_{A_{\infty}}^*:=\sup_{\operatorname{cube} Q\subset\mathbb{R}^n}\frac{1}{w(Q)}
\int_{Q}\mathscr{M}(w\mathbf{1}_Q)(x)\,dx$$
[see \eqref{eq-HL} for the definition of $\mathscr{M}$].
By Lebesgue's differentiation theorem, we find that,
for any scalar weight $w$, $[w]_{A_{\infty}}^*\in[1, \infty]$.
Let $p\in(0, \infty)$, $W\in
\mathcal{A}_{p,\infty}$, and
$O_m$ be the zero matrix in $M_m(\mathbb{C})$.
As pointed out in \cite[Lemma 5.3]{bhyy4},
for any $M\in M_m(\mathbb{C})\setminus \{O_m\}$,
the positive function $w_M:=\|W^{\frac{1}{p}}M\|^p
\in A_{\infty}$ and $[W]^{\operatorname{sc}}_{\mathcal{A}_{p,\infty}}:=
\sup_{M\in M_m(\mathbb{C})\setminus \{O_m\}}[w_M]_{A_{\infty}}^*
\lesssim[W]_{\mathcal{A}_{p,\infty}}$,
where the implicit positive constant is independent of $W$.
Let
\begin{align}\label{eq-r(W)}
r(W):=1+\left(2^{n+1}\left[W\right]^{\operatorname{sc}}
_{\mathcal{A}_{p,\infty}}-1\right)^{-1}.
\end{align}
We say that
a function $\upsilon:\ \mathcal{D}\to (0, \infty)$
is an \emph{almost increasing function} if
there exists a positive constant $C$ such that,
for any $Q, P\in\mathcal{D}$ with $Q\subset P$,
$\upsilon(Q)\leq C\upsilon(P)$.

The main results of this subsection are the following
two theorems. The first theorem establishes the coincidence of
$\dot{a}_{p,q}^{s, \upsilon}(W)$ and
$\dot{a}_{p,q}^{s, \upsilon}(\mathbb{A})$.
The second theorem gives the coincidence of
$\dot{A}_{p,q}^{s,\upsilon}(W, \varphi)$ and
$\dot{A}_{p,q}^{s,\upsilon}(\mathbb{A}, \varphi)$.

\begin{theorem}\label{thm-a(A)=a(W)}
Let $a\in\{b, f\}$, $s\in\mathbb{R}$,
$p\in(0, \infty)$, $q\in(0,\infty]$,
and $\upsilon$ be an almost increasing function.
Assume that $W\in \mathcal{A}_{p,\infty}$ and
$\mathbb{A}:=\{A_Q\}_{Q\in\mathcal{D}}$ is a sequence of
positive definite matrices. Then the following
statements are mutually equivalent.
\begin{itemize}
\item[{\rm (i)}] There exists some $r\in\left[p, pr(W)\right]$
such that, for any $Q\in\mathcal{D}$ and $\vec{z}\in\mathbb{C}^m$,
\begin{align}\label{eq-a}
\left|A_Q\vec{z}\right|\sim\left[\fint_{Q}
\left|W^{\frac{1}{p}}(x)\vec{z}\right|^r\, dx\right]^{\frac{1}{r}},
\end{align}
where $r(W)$ is as in \eqref{eq-r(W)} and
the positive equivalence constants are independent of
$Q$ and $\vec{z}$ but may depend on
$[W]_{\mathcal{A}_{p,\infty}}$.
\item[{\rm (ii)}] \eqref{eq-a} with $r$ replaced by $p$ holds,
that is, $\mathbb{A}$ is a sequence of reducing operators
of order $p$ for $W$.
\item[{\rm (iii)}] $\dot{a}_{p,q}^{s,\upsilon}(W)=
\dot{a}_{p,q}^{s,\upsilon}(\mathbb{A})$
with equivalent quasi-norms.
\end{itemize}
\end{theorem}

\begin{remark}\label{rmk-a(W)=a(A)}
\begin{itemize}
\item[{\rm (i)}] Observe that $\{|A_Q\vec{t}_Q
|\}_{Q\in\mathcal{D}}$ in \eqref{eq-a(A)-a} is a
sequence in $[0,\infty)$ and hence
Theorem \ref{thm-a(A)=a(W)}(iii) can
reduce some problems in the matrix-weighted setting
to problems in the unweighted setting
(see, for example, Theorem \ref{thm-sobolev-B} for
the study of the Sobolev-type embedding
of $\dot{A}_{p,q}^{s,\upsilon}(W)$
and \cite{fr04, rou04} for the study of the
duality of matrix-weighted Besov spaces).
To the best of our knowledge, even in the scalar-valued setting,
Theorem \ref{thm-a(A)=a(W)} is also new.
\item[{\rm (ii)}] In Theorem \ref{thm-a(A)=a(W)},
for any $Q\in\mathcal{D}$, let $\upsilon(Q):=1$.
Then spaces $\dot{a}_{p,q}^{s,\upsilon}(W)$ and
$\dot{a}_{p,q}^{s,\upsilon}(\mathbb{A})$ in
Theorem \ref{thm-a(A)=a(W)} are respectively the classical spaces
$\dot{a}_{p,q}^{s}(W)$ and $\dot{a}_{p,q}^{s}(\mathbb{A})$.
In particular, Volberg \cite[p.\,454, 2) of Remarks]{vol97}
pointed out that, when $p\in[2, \infty)$ and
Theorem \ref{thm-a(A)=a(W)}(ii) is satisfied, then
\begin{align}\label{eq-f(W)=f(A)}
\dot{f}_{p,2}^{0}(W)=\dot{f}_{p,2}^{0}(\mathbb{A})
\end{align}
with equivalent norms, that is,
Theorem \ref{thm-a(A)=a(W)}(iii) holds.
We now extend \eqref{eq-f(W)=f(A)} to any $p\in(0,\infty)$
and show that Theorem \ref{thm-a(A)=a(W)}(ii) is also necessary
to guarantee \eqref{eq-f(W)=f(A)} for any $p\in(0,\infty)$,
which also answers an open question
in \cite[p.\,454, 2) of Remarks]{vol97}.
Moreover, Volberg \cite[p.\,454, 2) of Remarks]{vol97} also
posed a question that, to make \eqref{eq-f(W)=f(A)} hold, whether
the index $p$ in Theorem \ref{thm-a(A)=a(W)}(ii) can be changed into
an index different from $p$.
We give an affirmative answer in Theorem \ref{thm-a(A)=a(W)}(i)
to this question by proving that the index $p$ in Theorem \ref{thm-a(A)=a(W)}(ii)
can be replaced by any index in $[p, pr(W)]$ with $r(W)\in (1,\infty)$.
\end{itemize}	
\end{remark}

\begin{theorem}\label{thm-A(W)=A(A)}
Let $A\in\{B, F\}$, $s\in\mathbb{R}$, $p\in(0, \infty)$,
and $q\in(0,\infty]$. Suppose that $\delta_1,\delta_2,\omega$
satisfy \eqref{eq-delta1<0},
$\upsilon\in\mathcal{G}(\delta_1, \delta_2; \omega)$,
and $\varphi\in\mathcal{S}$ satisfies \eqref{cond1}.
Assume that $W\in \mathcal{A}_{p,\infty}$ and
$\mathbb{A}:=\{A_Q\}_{Q\in\mathcal{D}}$
is a sequence of reducing operators of order $p$ for $W$.
Then $\dot{A}_{p,q}^{s,\upsilon}(W, \varphi)=
\dot{A}_{p,q}^{s,\upsilon}(\mathbb{A}, \varphi)$
with equivalent quasi-norms.
\end{theorem}

To prove Theorem \ref{thm-a(A)=a(W)}, we need some lemmas.
The following lemma is a part of \cite[Lemma 2.10]{bhyy1}.
\begin{lemma}\label{lem-red-ope}
Let $p \in(0,\infty)$, $W$ be a matrix weight,
and $\{A_Q\}_{Q\in\mathcal{D}}$
be a sequence of reducing operators of order $p$ for $W$.
Then, for any $Q\in\mathcal{D}$ and $M\in M_m(\mathbb{C})$,
\begin{align*}
\|A_QM\|\sim\left[\fint_Q\left\|W^{\frac{1}{p}}(x)
M\right\|^p\,dx\right]^{\frac{1}{p}},
\end{align*}
where the positive equivalence constants
are independent of $Q$ and $M$.
\end{lemma}

The next lemma can be regarded
as a suitable substitute of the
Fefferman--Stein vector-valued inequality
in the matrix-weighted setting.
In the case where $L\dot{F}_{p, q}^{\upsilon}
=L\dot{F}_{p, q}=L^p(l^q)$,
Lemma \ref{lem-E_j}(ii) originates from
\cite[Corollary 3.8]{fr21} for the
matrix $\mathcal{A}_{p}$ classes
and was later extended
to matrix $\mathcal{A}_{p, \infty}$ classes
in \cite[Corollary 5.8]{bhyy4}.

\begin{lemma}\label{lem-E_j}
Let $p\in(0, \infty)$, $q\in(0,\infty]$,
and $\upsilon$ be a positive function defined on $\mathcal{D}$.
Suppose that $W$ is a matrix weight
and $\{A_Q\}_{Q\in\mathcal{D}}$
is a sequence of reducing operators of order $p$ for $W$.
For any $j\in\mathbb{Z}$, let
\begin{align}\label{eq-gamma_j}
\gamma_j:=\sum_{Q\in\mathcal{D}_j}\mathbf{1}_Q
\left\|W^{\frac1{p}}A^{-1}_Q\right\|.
\end{align}
Then the following statements hold.
\begin{itemize}
\item[{\rm (i)}] For any $\{t_Q\}_{Q\in\mathcal{D}}$
in $\mathbb{C}$,
\begin{align}\label{eq-E_j-f_j-B}
\left\|\left\{\gamma_j\sum_{Q\in\mathcal{D}_j}\widetilde{\mathbf{1}}_Q
t_Q\right\}_{j\in\mathbb{Z}}\right\|_{L\dot{B}_{p, q}^{\upsilon}}
\sim\left\|\left\{\sum_{Q\in\mathcal{D}_j}\widetilde{\mathbf{1}}_Q
t_Q\right\}_{j\in\mathbb{Z}}\right\|_{L\dot{B}_{p, q}^{\upsilon}},
\end{align}
where $\widetilde{\mathbf{1}}_Q:=|Q|^{-\frac{1}{2}}\mathbf{1}_Q$, $\|\cdot\|_{L\dot{B}_{p, q}^{\upsilon}}$ is
as in \eqref{LA_nu}, and the positive equivalence constants
are independent of $\{t_Q\}_{Q\in\mathcal{D}}$.
\item[{\rm (ii)}] If $W\in\mathcal{A}_{p,\infty}$,
then there exists a positive constant $C$ such that,
for any $\{t_Q\}_{Q\in\mathcal{D}}$ in $\mathbb{C}$,
\begin{align}\label{eq-E_j-f_j-F}
\left\|\left\{\gamma_j\sum_{Q\in\mathcal{D}_j}
\widetilde{\mathbf{1}}_Q t_Q\right
\}_{j\in\mathbb{Z}}\right\|_{L\dot{F}_{p, q}^{\upsilon}}
\lesssim\left\|\left\{\sum_{Q\in\mathcal{D}_j}
\widetilde{\mathbf{1}}_Q t_Q\right\}_{j\in\mathbb{Z}}
\right\|_{L\dot{F}_{p, q}^{\upsilon}},
\end{align}
where $\|\cdot\|_{L\dot{F}_{p, q}^{\upsilon}}$ is
as in \eqref{LA_nu} and the positive constant is
independent of $\{t_Q\}_{Q\in\mathcal{D}}$.
\end{itemize}
\end{lemma}

\begin{proof}
We first prove (i). Applying \eqref{eq-gamma_j},
both (i) and (iv) of Lemma \ref{lem-3P-sum1},
and Lemma \ref{lem-red-ope}
with $M$ replaced by $A^{-1}_Q$ for any
$Q\in\mathcal{D}$, we obtain,
for any $P\in\mathcal{D}$, $j\in[j_P, \infty)\cap\mathbb{Z}$,
and $\{t_Q\}_{Q\in\mathcal{D}}$ in $\mathbb{C}$,
\begin{align*}
\left\|\gamma_j
\sum_{Q\in\mathcal{D}_j}\widetilde{\mathbf{1}}_Q
t_Q\mathbf{1}_{P}\right\|_{L^p}
&=\left[\int_{P}\left|\gamma_j(x)\right|^p
\sum_{Q\in\mathcal{D}_j}\left|\widetilde{\mathbf{1}}_Q(x)
t_Q\right|^p\,dx\right]^{\frac{1}{p}}\\
&=\left[\sum_{Q\in\mathcal{D}_j, Q\subset P}
\int_{Q}\left\|W^{\frac1{p}}(x)A^{-1}_Q\right\|^p\,dx
|t_Q|^p|Q|^{-\frac{p}{2}}\right]^{\frac{1}{p}}
\sim\left[\sum_{Q\in\mathcal{D}_j, Q\subset P}
|t_Q|^p|Q|^{-\frac{p}{2}+1}\right]^{\frac{1}{p}}\\
&=\left[\int_{P}
\sum_{Q\in\mathcal{D}_j}\left|\widetilde{\mathbf{1}}_Q(x)
t_Q\right|^p\,dx\right]^{\frac{1}{p}}
=\left\|\sum_{Q\in\mathcal{D}_j}\widetilde{\mathbf{1}}_Q
t_Q\mathbf{1}_{P}\right\|_{L^p}.
\end{align*}
By taking the $l^q$ quasi-norm on its both sides
with respect to
$j\in[j_P, \infty)\cap\mathbb{Z}$ and \eqref{eq-LA},
we find that, for any $P\in\mathcal{D}$
and $\{t_Q\}_{Q\in\mathcal{D}}$ in $\mathbb{C}$,
\begin{align*}
\left\|\left\{\gamma_j\sum_{Q\in\mathcal{D}_j}
\widetilde{\mathbf{1}}_Q
t_Q\mathbf{1}_{P}\mathbf{1}_{j\geq j_P}\right\}_{j\in\mathbb{Z}}
\right\|_{L\dot{B}_{p, q}}\sim
\left\|\left\{\sum_{Q\in\mathcal{D}_j}
\widetilde{\mathbf{1}}_Q
t_Q\mathbf{1}_{P}\mathbf{1}_{j\geq j_P}
\right\}_{j\in\mathbb{Z}}\right\|_{L\dot{B}_{p, q}}.
\end{align*}
Dividing its both sides by $\upsilon(P)$, then
taking the supremum over all $P\in\mathcal{D}$,
and using \eqref{LA_nu},
we conclude that \eqref{eq-E_j-f_j-B} holds.
This finishes the proof of (i). Next, we show (ii).
For any $P\in\mathcal{D}$ and
$\{t_Q\}_{Q\in\mathcal{D}}$ in $\mathbb{C}$,
from \cite[Corollary 5.8]{bhyy4}
with $\{f_j\}_{j\in\mathbb{Z}}$ replaced by
$\{\sum_{Q\in\mathcal{D}_j}\widetilde{\mathbf{1}}_Q
t_Q\mathbf{1}_{P}\mathbf{1}_{j\geq j_P}\}_{j\in\mathbb{Z}}$,
we infer that
\begin{align*}
\left\|\left\{\gamma_j\sum_{Q\in\mathcal{D}_j}\widetilde{\mathbf{1}}_Q
t_Q\mathbf{1}_{P}\mathbf{1}_{j\geq j_P}\right\}_{j\in\mathbb{Z}}
\right\|_{L\dot{F}_{p, q}}\lesssim
\left\|\left\{\sum_{Q\in\mathcal{D}_j}\widetilde{\mathbf{1}}_Q
t_Q\mathbf{1}_{P}\mathbf{1}_{j\geq j_P}
\right\}_{j\in\mathbb{Z}}\right\|_{L\dot{F}_{p, q}}.
\end{align*}
Dividing its both sides by $\upsilon(P)$, then
taking the supremum over all $P\in\mathcal{D}$,
and using \eqref{LA_nu}, we obtain \eqref{eq-E_j-f_j-F},
which completes the proof of (ii) and hence Lemma \ref{lem-E_j}.
\end{proof}

The following lemma gives a characterization
of the $\dot{a}_{p,q}^{s,\upsilon}$-norm  via
sequences of sparse sets.

\begin{lemma}\label{lem-aLA-equi}
Let $(A, a)\in\{(B, b), (F, f)\}$, $s\in\mathbb{R}$,
$p, q\in(0,\infty]$ ($p<\infty$ if $a=f$), and
$\upsilon$ be a positive function defined on $\mathcal{D}$.
Assume that $\varepsilon\in(0,1]$ and $\{E_Q\}_{Q\in\mathcal{D}}$
is a sequence of measurable sets with $E_Q \subset Q$
and $|E_Q|\geq\varepsilon|Q|$ for any $Q\in\mathcal{D}$.
Then, for any $t:=\{t_Q\}_{Q\in\mathcal{D}}$ in $\mathbb{C}$,
\begin{align}\label{eq-aLA-equi}
\left\|t\right\|_{\dot{a}_{p,q}^{s,\upsilon}}
\sim\left\|\left\{2^{js}\sum_{Q\in\mathcal{D}_j}
\widetilde{\mathbf{1}}_{E_Q}t_Q\right\}
_{j\in\mathbb{Z}}\right\|_{L\dot{A}_{p, q}^{\upsilon}},
\end{align}
where $\widetilde{\mathbf{1}}_{E_Q}:=|E_Q|^{-\frac{1}{2}}
\mathbf{1}_{E_Q}$, $\|\cdot\|_{L\dot{A}_{p, q}^{\upsilon}}$ is
as in \eqref{LA_nu}, and the positive equivalence
constants are independent of $t$.
\end{lemma}

\begin{proof}
The case where $a=b$ follows from
an argument similar to that used in the proof of Lemma \ref{lem-E_j}(i)
with $\gamma_j$ and Lemma \ref{lem-red-ope} therein
replaced, respectively, by $\sum_{Q\in\mathcal{D}_j}\mathbf{1}_{E_Q}$
and the assumption that,
for any $Q\in\mathcal{D}$, $\varepsilon|Q|\leq|E_Q|\leq|Q|$;
we omit the details. Next, we consider the case where $a=f$.
For any $P\in\mathcal{D}$ and $t:=\{t_Q\}
_{Q\in\mathcal{D}}$ in $\mathbb{C}$,
applying \cite[Proposition 2.7]{fj90} to $t$
restricted to $P$, which equals $t$
on the dyadic cubes contained in $P$ and equals $0$
otherwise, we obtain
\begin{align*}
\left\|\left\{2^{js}\left(\sum_{Q\in\mathcal{D}_j}
\widetilde{\mathbf{1}}_{Q}t_Q\right)
\mathbf{1}_{P}\mathbf{1}_{j\geq j_P}\right\}
_{j\in\mathbb{Z}}\right\|_{L\dot{F}_{p, q}}
\sim\left\|\left\{2^{js}\left(\sum_{Q\in\mathcal{D}_j}
\widetilde{\mathbf{1}}_{E_Q}t_Q\right)
\mathbf{1}_{P}\mathbf{1}_{j\geq j_P}\right\}
_{j\in\mathbb{Z}}\right\|_{L\dot{F}_{p, q}}.
\end{align*}
Dividing both sides by $\upsilon(P)$, then taking the
supremum over all $P\in\mathcal{D}$,
and using \eqref{LA_nu} and the definition of
$\|\cdot\|_{\dot{f}_{p,q}^{s,\upsilon}}$, we conclude that,
for any $t:=\{t_Q\}_{Q\in\mathcal{D}}$ in $\mathbb{C}$,
\eqref{eq-aLA-equi} holds for the case where $a=f$.
This finishes the proof of Lemma \ref{lem-aLA-equi}.
\end{proof}

In general, for any $p\in(0, \infty)$ and
any matrix weight $W$, there exists a
sequence $\mathbb{A}:=\{A_Q\}_{\operatorname{cube} Q}$ of
positive definite matrices such that, for any
cube $Q\subset\mathbb{R}^n$ and any $\vec{z}$,
\eqref{eq-red-ope} holds (see, for example,
\cite[Definition 2.7]{bhyy4} and \cite[Definition 2.8]{bhyy1}).
The sequence $\mathbb{A}$ is  called a sequence of \emph{reducing
operators of order $p$ for $W$}.
The next is precisely \cite[Corollary 3.9]{bhyy4}.

\begin{lemma}\label{lem-AW-1-set}
Let $p\in(0,\infty)$, $W\in\mathcal{A}_{p,\infty}$, and
$\{A_Q\}_{\operatorname{cube} Q}$ be a sequence of reducing operators
of order $p$ for $W$. Then there exists a positive constant
$C$ such that, for any cube $Q\subset\mathbb{R}^n$
and any $M\in(0,\infty)$,
\begin{align}\label{eq-AW-1-set}
\left|\left\{y\in Q:\ \left\|A_Q W^{-\frac{1}{p}}(y)
\right\|^p\geq e^M\right\}\right|
\leq\frac{\log(C[W]_{\mathcal{A}_{p, \infty}})}{M}|Q|.
\end{align}
\end{lemma}

\begin{remark}
Volberg \cite[p.\,454, Remark]{vol97} pointed out that,
in the scalar-valued setting,   for any $p\in(0,\infty)$,
$\mathcal{A}_{p, \infty}$ reduces to $A_{\infty}$, and
\eqref{eq-AW-1-set} reveals a characteristic property of the
scalar weight $w\in A_{\infty}$, that is, the set where $w$ is much
smaller than its average is small. To see this, observe that,
in the scalar-valued setting,
$\{[\fint_{Q}w(x)\,dx]^{\frac{1}{p}}\}_{\operatorname{cube} Q}$ is
exactly a sequence of reducing operators of order $p$ for $w$.
Applying this, we conclude that \eqref{eq-AW-1-set} can be read as,
for any cube $Q\subset\mathbb{R}^n$ and any $M\in(0,\infty)$,
\begin{align*}
\left|\left\{y\in Q:\ \fint_{Q}w(x)\,dx\geq e^M w(y)\right\}\right|
\leq\frac{\log(C[w]_{A_{\infty}})}{M}|Q|,
\end{align*}
which is an equivalent definition of $w\in A_{\infty}$
(see, for example, \cite[Theorem 7.3.3]{gra14a})
and hence illustrates the aforementioned property.
\end{remark}

We now present a lemma to compute the $\dot{a}_{p,q}^{s,\upsilon}(W)$-norm
of single-pointed sequences. For any $Q,R\in\mathcal{D}$, let
\begin{align}\label{eq-e-QR}
\mathbf{1}_{Q=R}:=
\begin{cases}
1
&\text{if } Q=R, \\
0
&\text{otherwise}.
\end{cases}	
\end{align}

\begin{lemma}\label{lem-onepoint}
Let $a\in\{b, f\}$, $s\in\mathbb{R}$,
$p\in(0, \infty)$, $q\in(0,\infty]$, and $W$ be a matrix weight.
Suppose that $\upsilon$ is a positive function on $\mathcal{D}$.
Then the following three statements are mutually equivalent.
\begin{itemize}
\item[{\rm (i)}] $\upsilon$ is an almost increasing function.
\item[{\rm (ii)}] For any $Q\in\mathcal{D}$ and $\vec{z}
\in\mathbb{C}^m$, $\{\mathbf{1}_{Q=R}\vec{z}\}_{R\in\mathcal{D}}
\in\dot{a}_{p,q}^{s,\upsilon}(W)$ and
\begin{align}\label{eq-1z}
\|\{\mathbf{1}_{Q=R}\vec{z}\}_{R\in\mathcal{D}}
\|_{\dot{a}_{p,q}^{s,\upsilon}(W)}
\sim\frac{2^{j_Q(s +\frac{n}{2})}}{\upsilon(Q)}
\left[\int_{Q}\left|W^{\frac{1}{p}}(x)
\vec{z}\right|^p\, dx\right]^{\frac{1}{p}},
\end{align}
where the positive equivalence constants are
independent of $Q$ and $\vec{z}$.
\item[{\rm (iii)}] There exists
$\vec{u}\in\mathbb{C}^m\setminus\{\mathbf{0}\}$ such that,
for any $Q\in\mathcal{D}$,
$\{\mathbf{1}_{Q=R}\vec{u}\}_{R\in\mathcal{D}}
\in\dot{a}_{p,q}^{s,\upsilon}(W)$ and
\eqref{eq-1z} holds with $\vec{z}$ replaced by
$\vec{u}$.
\end{itemize}
\end{lemma}
\begin{proof}
We first prove (i)$\ \Longrightarrow\ $(ii).
Using \eqref{eq-e-QR}, the definitions  of $\|\cdot\|_{\dot{a}_{p,q}^{s, \upsilon}(W)}$, almost increasing functions, and matrix weights,
we conclude that, for any $Q\in\mathcal{D}$ and $\vec{z}\in\mathbb{C}^m$,
\begin{align*}
&\|\{\mathbf{1}_{Q=R}\vec{z}\}_{R\in\mathcal{D}}
\|_{\dot{a}_{p,q}^{s,\upsilon}(W)}\\
&\quad=\sup_{P\in\mathcal{D}, P\supset Q}
\frac{2^{j_Q(s +\frac{n}{2})}}{\upsilon(P)}
\left[\int_{Q}\left|W^{\frac{1}{p}}(x)
\vec{z}\right|^p\, dx\right]^{\frac{1}{p}}
\sim\frac{2^{j_Q(s +\frac{n}{2})}}{\upsilon(Q)}
\left[\int_{Q}\left|W^{\frac{1}{p}}(x)
\vec{z}\right|^p\, dx\right]^{\frac{1}{p}}<\infty,
\end{align*}
which completes the proof of (i)$\ \Longrightarrow\ $(ii).

The implication (ii)$\ \Longrightarrow\ $(iii) is obvious.
We next show (iii)$\ \Longrightarrow\ $(i).
From the definition of $\|\cdot\|_{\dot{a}_{p,q}^{s, \upsilon}(W)}$
and (iii), it follows that, for any $Q,P\in\mathcal{D}$
with $Q\subset P$,
\begin{align*}
\frac{2^{j_Q(s +\frac{n}{2})}}{\upsilon(P)}
\left[\int_{Q}\left|W^{\frac{1}{p}}(x)
\vec{u}\right|^p\, dx\right]^{\frac{1}{p}}
&\leq\sup_{R\in\mathcal{D}, R\supset Q}
\frac{2^{j_Q(s +\frac{n}{2})}}{\upsilon(R)}
\left[\int_{Q}\left|W^{\frac{1}{p}}(x)
\vec{u}\right|^p\, dx\right]^{\frac{1}{p}}\\
&=\left\|\left\{\mathbf{1}_{Q=R}\vec{u}\right\}
_{R\in\mathcal{D}}\right\|_{\dot{a}_{p,q}^{s,\upsilon}(W)}
\sim\frac{2^{j_Q(s +\frac{n}{2})}}{\upsilon(Q)}
\left[\int_{Q}\left|W^{\frac{1}{p}}(x)
\vec{u}\right|^p\, dx\right]^{\frac{1}{p}},
\end{align*}
which, together with the definition of matrix weights
and the assumption $\vec{u}\neq\mathbf{0}$,
further implies that $\int_{Q}|W^{\frac{1}{p}}(x)
\vec{u}|^p\, dx\in(0,\infty)$ and hence $\upsilon(Q)\lesssim\upsilon(P)$.
Thus, $\upsilon$ is an almost increasing function.
This finishes the proof of (iii)$\ \Longrightarrow\ $(i)
and hence Lemma \ref{lem-onepoint}.
\end{proof}

\begin{remark}\label{rmk-onepoint}
Let all the symbols be the same as in Lemma \ref{lem-onepoint}.
By Lemma \ref{lem-grow-est}(i), we find that,
if $\delta_1,\delta_2,\omega$ satisfy \eqref{eq-delta1>0} and
$\upsilon\in\mathcal{G}(\delta_1, \delta_2; \omega)$,
then $\upsilon$ is an almost increasing function and hence
Lemma \ref{lem-onepoint}(ii) holds for $\dot{a}_{p,q}^{s,\upsilon}(W)$.
\end{remark}

We now show Theorem \ref{thm-a(A)=a(W)}.

\begin{proof}[Proof of Theorem \ref{thm-a(A)=a(W)}]
We first prove the equivalence (i)$\ \Longleftrightarrow\ $(ii).
Since (ii)$\ \Longrightarrow\ $(i) is trivial, it suffices to
show (i)$\ \Longrightarrow\ $(ii). For this purpose,
assume that $r\in[p, pr(W)]$ such that \eqref{eq-a} holds.
From \cite[Proposition 5.6]{bhyy4}, we infer that,
for any $Q\in\mathcal{D}$ and $\vec{z}\in\mathbb{C}^m$,
\begin{align}\label{eq-rev-hold}
\left[\fint_{Q}\left|W^{\frac{1}{p}}(x)\vec{z}
\right|^{r}\, dx\right]^{\frac{1}{r}}
\lesssim\left[\fint_{Q}\left|W^{\frac{1}{p}}(x)
\vec{z}\right|^{p}\, dx\right]^{\frac{1}{p}}.
\end{align}
Applying H\"{o}lder's inequality, we obtain
the converse estimate of \eqref{eq-rev-hold},
which, combined with \eqref{eq-a} and \eqref{eq-rev-hold},
further implies that,
for any $Q\in\mathcal{D}$ and $\vec{z}\in\mathbb{C}^m$,
\begin{align*}
\left|A_Q\vec{z}\right|\sim\left[\fint_{Q}\left|W^{\frac{1}{p}}(x)
\vec{z}\right|^{r}\, dx\right]^{\frac{1}{r}}
\sim\left[\fint_{Q}\left|W^{\frac{1}{p}}(x)
\vec{z}\right|^{p}\, dx\right]^{\frac{1}{p}}.
\end{align*}
This finishes the proof of (i)$\ \Longrightarrow\ $(ii)
and hence the equivalence (i)$\ \Longleftrightarrow\ $(ii).

Next, we prove (ii)$\ \Longrightarrow\ $(iii).
By the definitions of both $\vec{t}_j$ [see \eqref{vect_j}]
and $\gamma_j$ [see \eqref{eq-gamma_j}] and Lemma \ref{lem-3P-sum1}(ii),
we find that, for any $j\in\mathbb{Z}$,
$\vec{t}:=\{\vec{t}_Q\}_{Q\in\mathcal{D}}$
in $\mathbb{C}^m$, and $x\in\mathbb{R}^n$,
\begin{align}\label{eq-W-A}
\left|W^{\frac1{p}}(x)\vec{t}_j(x)\right|
&=\sum_{Q\in\mathcal{D}_j}\mathbf{1}_{Q}(x)\left|W^{\frac1{p}}(x)
\sum_{R\in\mathcal{D}_j}\widetilde{\mathbf{1}}_R(x)\vec{t}_R\right|
=\sum_{Q\in\mathcal{D}_j}\widetilde{\mathbf{1}}_{Q}(x)
\left|W^{\frac1{p}}(x)A^{-1}_QA_Q\vec{t}_Q\right|\\
&\leq\sum_{Q\in\mathcal{D}_j}\widetilde{\mathbf{1}}_{Q}(x)
\left\|W^{\frac1{p}}(x)A^{-1}_Q\right\|\left|A_Q\vec{t}_Q\right|
=\gamma_j(x)\sum_{Q\in\mathcal{D}_j}\widetilde{\mathbf{1}}_{Q}(x)
\left|A_Q\vec{t}_Q\right|.\nonumber
\end{align}
For any $\vec{t}:=\{\vec{t}_Q\}
_{Q\in\mathcal{D}}$ in $\mathbb{C}^m$,
using the definitions of
$\|\cdot\|_{\dot{a}_{p,q}^{s,\upsilon}(W)}$ and
$\|\cdot\|_{\dot{a}_{p,q}^{s,\upsilon}}$,
\eqref{eq-W-A}, Lemma \ref{lem-E_j}
with $\{t_Q\}_{Q\in\mathcal{D}}$
replaced by $\{2^{j_Q s}|A_Q\vec{t}_Q|\}_{Q\in\mathcal{D}}$,
and \eqref{eq-a(A)-a}, we conclude that
\begin{align}\label{eq-a(W)-a(A)}
\left\|\vec{t}\right\|_{\dot{a}_{p,q}^{s,\upsilon}(W)}
&=\left\|\left\{2^{js}\left|W^{\frac1{p}}
\vec{t}_j\right|\right\}_{j\in\mathbb{Z}}
\right\|_{L\dot{A}_{p, q}^{\upsilon}}\leq
\left\|\left\{2^{js}\gamma_j\sum_{Q\in \mathcal{D}_j}\widetilde{\mathbf{1}}_Q
\left|A_Q\vec{t}_Q\right|\right\}_{j\in\mathbb{Z}}
\right\|_{L\dot{A}_{p, q}^{\upsilon}}\\
&\lesssim\left\|\left\{2^{js}\sum_{Q\in\mathcal{D}_j}
\widetilde{\mathbf{1}}_Q\left|A_Q\vec{t}_Q
\right|\right\}_{j\in\mathbb{Z}}
\right\|_{L\dot{A}_{p, q}^{\upsilon}}
=\left\|\left\{\left|A_Q\vec{t}_Q
\right|\right\}_{Q\in\mathcal{D}}\right\|_{\dot{a}_{p,q}^{s,\upsilon}}
=\left\|\vec{t}\right\|_{\dot{a}_{p,q}^{s,\upsilon}(\mathbb{A})}.\nonumber
\end{align}
To complete the proof of (ii)$\ \Longrightarrow\ $(iii),
we only need to show the reverse estimate of \eqref{eq-a(W)-a(A)}.
By Lemma \ref{lem-AW-1-set}, we can find
$L\in(0, \infty)$ such that, for any $Q\in\mathcal{D}$,
$E_Q:=\{x\in Q:\ \|A_{Q} W^{-\frac{1}{p}}(x)\|<L\}$ satisfies that
$E_Q\subset Q$ and $\frac{|Q|}{2}\leq|E_Q|\leq|Q|$.
For any $\vec{t}:=\{\vec{t}_Q\}_{Q\in\mathcal{D}}$ in $\mathbb{C}^m$,
applying \eqref{eq-a(A)-a}, the construction of
$E_Q$ for any $Q\in\mathcal{D}$, Lemma \ref{lem-aLA-equi}
with $\{t_Q\}_{Q\in\mathcal{D}}$ replaced by
$\{|A_Q\vec{t}_Q|\}_{Q\in\mathcal{D}}$, and the definition of
$\|\cdot\|_{\dot{a}_{p,q}^{s,\upsilon}(W)}$, we obtain
\begin{align*}
\left\|\vec{t}\right\|_{\dot{a}_{p,q}^{s,\upsilon}(\mathbb{A})}
&=\left\|\left\{\left|A_Q \vec{t}_Q \right|\right\}
_{Q\in\mathcal{D}}\right\|_{\dot{a}_{p,q}^{s,\upsilon}}
\sim\left\|\left\{2^{js}\sum_{Q\in\mathcal{D}_j}\widetilde{\mathbf{1}}_{E_Q}
\left|A_QW^{-\frac{1}{p}}W^{\frac{1}{p}}\vec{t}_Q\right|\right\}
_{j\in\mathbb{Z}}\right\|_{L\dot{A}_{p, q}^{\upsilon}}\\
&\leq\left\|\left\{2^{js}\sum_{Q\in\mathcal{D}_j}\widetilde{\mathbf{1}}_{E_Q}
\left\|A_QW^{-\frac{1}{p}}\right\|\left|W^{\frac{1}{p}}\vec{t}_Q
\right|\right\}
_{j\in\mathbb{Z}}\right\|_{L\dot{A}_{p, q}^{\upsilon}}\\
&\lesssim\left\|\left\{2^{js}\sum_{Q\in\mathcal{D}_j}
\widetilde{\mathbf{1}}_{E_Q}
\left|W^{\frac{1}{p}}\vec{t}_Q\right|\right\}
_{j\in\mathbb{Z}}\right\|_{L\dot{A}_{p, q}^{\upsilon}}
\lesssim\left\|\left\{2^{js}\sum_{Q\in\mathcal{D}_j}
\widetilde{\mathbf{1}}_{Q}
\left|W^{\frac{1}{p}}\vec{t}_Q\right|\right\}
_{j\in\mathbb{Z}}\right\|_{L\dot{A}_{p, q}^{\upsilon}}
=\left\|\vec{t}\right\|_{\dot{a}_{p,q}^{s,\upsilon}(W)},
\end{align*}
which proves the reverse estimate
of \eqref{eq-a(W)-a(A)} and hence (ii)$\ \Longrightarrow\ $(iii).

Finally, we prove (iii)$\ \Longrightarrow\ $(ii).
To this end, for any $Q,R\in\mathcal{D}$,
let $\mathbf{1}_{Q=R}$ be as in \eqref{eq-e-QR}.
If (iii) holds, by the
assumption that $\upsilon$ is an almost increasing function,
Lemma \ref{lem-onepoint}, and the definition of
$\|\cdot\|_{\dot{a}_{p, q}^{s, \upsilon}(\mathbb{A})}$,
we conclude that,
for any $Q\in\mathcal{D}$ and $\vec{z}\in\mathbb{C}^m$,
\begin{align*}
\frac{2^{j_Q(s +\frac{n}{2})}}{\upsilon(Q)}
\left[\int_{Q}\left|W^{\frac{1}{p}}(x)
\vec{z}\right|^p\, dx\right]^{\frac{1}{p}}
&\sim\left\|\left\{\mathbf{1}_{Q=R}\vec{z}\right\}
_{R\in\mathcal{D}}\right\|_{\dot{a}_{p,q}^{s, \upsilon}(W)}
\sim\left\|\left\{\mathbf{1}_{Q=R}\vec{z}\right\}
_{R\in\mathcal{D}}\right\|_{\dot{a}_{p, q}^{s, \upsilon}(\mathbb{A})}\\
&=\sup_{P\in\mathcal{D}, P\supset Q}
\frac{2^{j_Q(s +\frac{n}{2})}}{\upsilon(P)}
\left[\int_{Q}\left|A_Q\vec{z}\right|^p\, dx\right]^{\frac{1}{p}}\\
&\sim\frac{2^{j_Q(s+\frac{n}{2}-\frac{n}{p})}}{\upsilon(Q)}
\left|A_Q \vec{z}\right|,
\end{align*}
which further implies that \eqref{eq-a} holds.
This finishes the proof of (iii)$\ \Longrightarrow\ $(ii)
and hence Theorem \ref{thm-a(A)=a(W)}.
\end{proof}

Notice that, in   the proof of (ii)$\ \Longrightarrow\ $(iii)
in Theorem \ref{thm-a(A)=a(W)},
there is no need to suppose that $\upsilon$ is an
almost increasing function. Thus, we obtain the following corollary.

\begin{corollary}\label{cor-a(A)=a(W)}
Let $a\in\{b, f\}$, $s\in\mathbb{R}$,
$p\in(0, \infty)$, $q\in(0,\infty]$, and
$\upsilon$ be a positive function on $\mathcal{D}$.
If $W\in\mathcal{A}_{p,\infty}$ and $\mathbb{A}$
is a sequence of reducing operators of order $p$ for $W$, then,
for any $\vec{t}:=\{\vec{t}_Q\}_{Q\in\mathcal{D}}$
in ${\mathbb{C}}^m$, $\|\vec{t}\|
_{\dot{a}_{p,q}^{s,\upsilon}(W)}\sim
\|\vec{t}\|_{\dot{a}_{p,q}^{s,\upsilon}(\mathbb{A})}$,
where the positive equivalence
constants are independent of $\vec{t}$.
\end{corollary}

To prove Theorem \ref{thm-A(W)=A(A)},
we need to further establish several lemmas.
Recall that the \emph{Hardy--Littlewood maximal operator}
$\mathscr{M}$ is defined by setting, for any measurable function $f$
on $\mathbb{R}^n$ and any $x\in\mathbb{R}^n$,
\begin{align}\label{eq-HL}
\mathscr{M}(f)(x):=\sup_{\operatorname{cube}
Q\ni x}\fint_{Q}|f(x)|\, dx.
\end{align}

The following lemma is an immediate corollary
of the boundedness of $\mathscr{M}$ on
$L^p$ for any $p\in(1,\infty]$
(see, for example, \cite[Theorem 2.1.6]{gra14a})
and the well-known Fefferman--Stein vector-valued
inequality established in \cite[Theorem 1]{fs71}.
Its proof follows from the proof of \cite[Lemma 3.12]{bhyy1}
with minor modifications; we omit the details.

\begin{lemma}\label{lem-FS-ineq}
Let $A\in\{B, F\}$. Assume that, when $A=B$,
$p\in(1,\infty]$ and $q\in(0,\infty]$ and,
when $A=F$, $p\in(1, \infty)$ and $q\in(1,\infty]$.
Then there exists a positive constant $C$ such that,
for any sequence $\{f_j\}_{j\in\mathbb{Z}}$
of measurable functions on $\mathbb{R}^n$,
$\|\{\mathscr{M}(f_j)\}_{j\in\mathbb{Z}}\|_{L\dot{A}_{p, q}}
\leq C\|\{f_j\}_{j\in\mathbb{Z}}\|_{L\dot{A}_{p, q}}$,
where $\|\cdot\|_{L\dot{A}_{p, q}}$ is as in \eqref{eq-LA}.
\end{lemma}

The next lemma follows from some basic calculations;
we omit the details.
\begin{lemma}\label{lem-sum-k}
If $\lambda\in(n,\infty)$, then, for any $j\in\mathbb{Z}_+$
and $x\in\mathbb{R}^n$,
\begin{align*}
\sum_{k\in{\mathbb{Z}}^n}\left(1+|x-2^{-j}k|
\right)^{-\lambda}\sim2^{jn},
\end{align*}
where the positive equivalence constants
depend only on $\lambda$ and $n$.
\end{lemma}

Let $(A, a)\in\{(B, b), (F, f)\}$,
$p, q\in(0, \infty]$, and
\begin{align}\label{eq-zeta}
\Gamma_{p,q}:=
\begin{cases}
\displaystyle p	
&\text{if } A=B \text{ or } a=b,\\
\displaystyle p\wedge q
&\text{if } A=F \text{ or } a=f.
\end{cases}
\end{align}
We have the following observation.

\begin{proposition}\label{prop-gh}
Let $A\in\{B, F\}$ and $p, q\in(0, \infty]$
($p<\infty$ if $A=F$). Suppose that $\delta_1,\delta_2,\omega$
satisfy \eqref{eq-delta1<0},
$\upsilon\in\mathcal{G}(\delta_1, \delta_2; \omega)$,
$r\in(0, \Gamma_{p,q})$, and $\lambda\in(n+[\omega\wedge
n(\delta_2-\frac{1}{p})_+]r,\infty)$.
Then there exists a positive constant $C$
such that, for any sequences $\{g_j\}_{j\in\mathbb{Z}}$
and $\{h_j\}_{j\in\mathbb{Z}}$ of measurable functions
on $\mathbb{R}^n$ satisfying that,
for any $j\in\mathbb{Z}$ and $x\in\mathbb{R}^n$,
\begin{align*}
\left|g_j(x)\right|^r\leq\int_{\mathbb{R}^n}
\frac{2^{jn}}{(1+2^j|x-y|)^\lambda}\left|h_j(y)\right|^r\,dy,
\end{align*}
$\|\{g_j\}_{j\in\mathbb{Z}}\|_{L\dot{A}_{p, q}^{\upsilon}}\leq C
\|\{h_j\}_{j\in\mathbb{Z}}\|_{L\dot{A}_{p, q}^{\upsilon}}$,
where $\|\cdot\|_{L\dot{A}_{p, q}^{\upsilon}}$ is as in \eqref{LA_nu}.
\end{proposition}

\begin{proof}
Let $P\in\mathcal{D}$. By the quasi-triangle inequality of $|\cdot|^{\frac{1}{r}}$, we find that,
for any $j\in\mathbb{Z}$ and $x\in\mathbb{R}^n$,
\begin{align}\label{eq-g-rho-t}
|g_j(x)|\lesssim\left[\int_{3P}
\frac{2^{jn}}{(1+2^j|x-y|)^\lambda}\left|h_j(y)\right|^r\,dy
\right]^{\frac{1}{r}}+\left[\int_{(3P)^{\complement}}
\cdots\right]^{\frac{1}{r}}=:\rho_j(x)+t_j(x).
\end{align}
Multiplying both sides by $\mathbf{1}_{P}\mathbf{1}
_{j\geq j_P}$, taking $\|\cdot\|_{L\dot{A}_{p,q}}$
[see \eqref{eq-LA}], and then applying the quasi-triangle inequality
of $\|\cdot\|_{L\dot{A}_{p, q}}$, we obtain
\begin{align}\label{eq-g-t}
\left\|\left\{g_j\mathbf{1}_{P}\mathbf{1}
_{j\geq j_P}\right\}_{j\in\mathbb{Z}}
\right\|_{L\dot{A}_{p, q}}\lesssim
\left\|\left\{\rho_j\mathbf{1}_{P}\mathbf{1}
_{j\geq j_P}\right\}_{j\in\mathbb{Z}}
\right\|_{L\dot{A}_{p, q}}+
\left\|\left\{t_j\mathbf{1}_{P}\mathbf{1}
_{j\geq j_P}\right\}_{j\in\mathbb{Z}}
\right\|_{L\dot{A}_{p, q}}.
\end{align}

We first estimate $\|\{\rho_j\mathbf{1}_{P}
\mathbf{1}_{j\geq j_P}\}_{j\in\mathbb{Z}}\|_{L\dot{A}_{p, q}}$.
Obviously,
for any $j\in\mathbb{Z}$ and $x\in\mathbb{R}^n$,
\begin{align}\label{eq-Rn-Ball}
\mathbf{1}_{B(x, 2^{-j})}+\sum_{l\in\mathbb{N}}
\mathbf{1}_{B(x,2^{l-j})\setminus B(x,2^{l-1-j})}=1
\end{align}
and, for any $l\in\mathbb{N}$ and $y\in
B(x,2^{l-j})\setminus B(x,2^{l-1-j})$, $1+2^j|x-y|\sim 2^l$.
From these, \eqref{eq-Rn-Ball}, the assumption $\lambda\in(n, \infty)$,
and the definitions of both $\mathscr{M}$ [see \eqref{eq-HL}]
and $\rho_j$ [see \eqref{eq-g-rho-t}],
we deduce that, for any $j\in[j_P, \infty)\cap\mathbb{Z}$ and $x\in P$,
\begin{align}\label{eq-rho}
\left[\rho_j(x)\right]^r&=\int_{\mathbb{R}^n}
\frac{2^{jn}}{(1+2^j|x-y|)^\lambda}\left|h_j(y)\right|^r\mathbf{1}_{3P}(y)\,dy\\
&=\int_{B(x, 2^{-j})}\frac{2^{jn}}{(1 + 2^j |x-y|)^\lambda}\left|h_j(y)\right|^r
\mathbf{1}_{3P}(y)\,dy+\sum_{l\in\mathbb{N}}
\int_{B(x,2^{l-j})\setminus B(x,2^{l-1-j})}\cdots\nonumber\\
&\lesssim\int_{B(x, 2^{-j})}2^{jn}\left|h_j(y)\right|^r
\mathbf{1}_{3P}(y)\,dy+\sum_{l\in\mathbb{N}}2^{-l\lambda}
\int_{B(x,2^{l-j})}2^{jn}\left|h_j(y)\right|^r\mathbf{1}_{3P}(y)\,dy\nonumber\\
&\sim\sum_{l\in\mathbb{Z}_+}2^{l(n-\lambda)}\fint_{B(x,2^{l-j})}
\left|h_j(y)\right|^r\mathbf{1}_{3P}(y)\,dy
\lesssim\mathscr{M}\left(\left|h_j\right|^r\mathbf{1}_{3P}\right)(x)\nonumber
\end{align}
and similarly
\begin{align}\label{eq-t-1}
\left[t_j(x)\right]^r\lesssim\sum_{l\in\mathbb{Z}_+}
2^{l(n-\lambda)}\fint_{B(x,2^{l-j})}\left|h_j(y)\right|^r
\mathbf{1}_{(3P)^{\complement}}(y)\,dy.
\end{align}
Using \eqref{eq-rho}, a convexification of
$\|\cdot\|_{L\dot{A}_{p, q}}$ with respect to the index $r$,
and Lemma \ref{lem-FS-ineq}
with the assumption that $r\in(0, \Gamma_{p,q})$
and with $p,q$, and $\{f_j\}_{j\in\mathbb{Z}}$ replaced, respectively,
by $\frac{p}{r}, \frac{q}{r}$, and $\{|h_j|^r\mathbf{1}_{3P}
\mathbf{1}_{j\geq j_P}\}_{j\in\mathbb{Z}}$,
we conclude that
\begin{align*}
\left\|\left\{\rho_j\mathbf{1}_{P}
\mathbf{1}_{j\geq j_P}\right\}_{j\in\mathbb{Z}}
\right\|_{L\dot{A}_{p, q}}&\lesssim\left\|\left\{
\left[\mathscr{M}\left(\left|h_j\right|^r\mathbf{1}_{3P}\right)
\right]^{\frac{1}{r}}\mathbf{1}_{P}
\mathbf{1}_{j\geq j_P}\right\}_{j\in\mathbb{Z}}
\right\|_{L\dot{A}_{p, q}}\\
&\leq\left\|\left\{\mathscr{M}\left(\left|h_j\right|^r\mathbf{1}_{3P}
\mathbf{1}_{j\geq j_P}\right)\right\}_{j\in\mathbb{Z}}
\right\|^{\frac{1}{r}}_{L\dot{A}_{\frac{p}{r}, \frac{q}{r}}}\\
&\lesssim\left\|\left\{\left|h_j\right|^r\mathbf{1}_{3P}
\mathbf{1}_{j\geq j_P}\right\}_{j\in\mathbb{Z}}
\right\|^{\frac{1}{r}}_{L\dot{A}_{\frac{p}{r}, \frac{q}{r}}}
=\left\|\left\{\left|h_j\right|\mathbf{1}_{3P}\mathbf{1}_{j\geq j_P}
\right\}_{j\in\mathbb{Z}}\right\|_{L\dot{A}_{p, q}},\nonumber
\end{align*}
which, together with Lemma \ref{lem-3P-sum1}(iii),
the quasi-triangle inequality of
$\|\cdot\|_{L\dot{A}_{p, q}}$, the definition of
$\|\cdot\|_{L\dot{A}_{p, q}^{\upsilon}}$,
and Lemma \ref{lem-grow-est}(iii), further implies that
\begin{align}\label{eq-rho-v}
\left\|\left\{\rho_j\mathbf{1}_{P}
\mathbf{1}_{j\geq j_P}\right\}_{j\in\mathbb{Z}}
\right\|_{L\dot{A}_{p, q}}
&\lesssim\sum_{k\in\mathbb{Z}^n,\|k\|_{\infty}\leq1}
\left\|\left\{\left|h_j\right|\mathbf{1}_{[P+k\ell(P)]}\mathbf{1}_{j\geq j_P}
\right\}_{j\in\mathbb{Z}}\right\|_{L\dot{A}_{p, q}}\\
&\leq\sum_{k\in\mathbb{Z}^n,\|k\|_{\infty}\leq1}
\left\|\left\{h_j\right\}_{j\in\mathbb{Z}}
\right\|_{L\dot{A}_{p, q}^{\upsilon}}\upsilon(P+k\ell(P))
\sim\left\|\left\{h_j\right\}_{j\in\mathbb{Z}}
\right\|_{L\dot{A}_{p, q}^{\upsilon}}\upsilon(P)\nonumber.
\end{align}
This gives the desired estimate of $\|\{\rho_j\mathbf{1}_{P}
\mathbf{1}_{j\geq j_P}\}_{j\in\mathbb{Z}}\|_{L\dot{A}_{p, q}}$.

Next, we estimate $\|\{t_j\mathbf{1}_{P}
\mathbf{1}_{j\geq j_P}\}_{j\in\mathbb{Z}}
\|_{L\dot{A}_{p, q}}$ by considering the following
two cases for $n(\delta_2-\frac{1}{p})_+$ and $\omega$.

\emph{Case (1)} $n(\delta_2-\frac{1}{p})_+<\omega$.
In this case, for any $i,j\in\mathbb{Z}$ and $x\in\mathbb{R}^n$,
let
\begin{align}\label{eq-f_ij}
f_{i,j}(x):=\fint_{B(x,2^{i})}\left|h_j(y)\right|^r
\mathbf{1}_{(3P)^{\complement}}(y)\,dy.
\end{align}
Let $\kappa:=\frac{p}{r}\wedge\frac{q}{r}\wedge1$.
By a convexification of
$\|\cdot\|_{L\dot{A}_{p, q}}$ with respect to the index $r$,
\eqref{eq-t-1},  \eqref{eq-f_ij}, and the triangle inequality
of $\|\cdot\|^{\kappa}_{L\dot{A}_{\frac{p}{r}, \frac{q}{r}}}$,
we find that
\begin{align}\label{eq-t-f}
\left\|\left\{t_j\mathbf{1}_{P}\mathbf{1}
_{j\geq j_P}\right\}_{j\in\mathbb{Z}}
\right\|^r_{L\dot{A}_{p, q}}
&=\left\|\left\{|t_j|^r\mathbf{1}_{P}
\mathbf{1}_{j\geq j_P}\right\}_{j\in\mathbb{Z}}
\right\|_{L\dot{A}_{\frac{p}{r}, \frac{q}{r}}}
\lesssim\left\|\left\{\sum_{l\in\mathbb{Z}_+}2^{l(n-\lambda)}
f_{l-j,j}\mathbf{1}_{P}\mathbf{1}
_{j\geq j_P}\right\}_{j\in\mathbb{Z}}
\right\|_{L\dot{A}_{\frac{p}{r}, \frac{q}{r}}}\\
&\leq\left[\sum_{l\in\mathbb{Z}_+}2^{l(n-\lambda)\kappa}
\left\|\left\{f_{l-j,j}\mathbf{1}_{P}\mathbf{1}
_{j\geq j_P}\right\}_{j\in\mathbb{Z}}
\right\|^{\kappa}_{L\dot{A}_{\frac{p}{r},
\frac{q}{r}}}\right]^{\frac{1}{\kappa}}\nonumber.
\end{align}
From \eqref{eq-f_ij} and the basic property of
$\mathbb{R}^n$, we infer that, for any $l\in\mathbb{Z}_{+}$,
$j\in(j_P+l, \infty)\cap\mathbb{Z}$, and $x\in P$,
$B(x, 2^{l-j})\subset3P$ and hence $f_{l-j,j}(x)=0$.
Applying this and the triangle inequality
of $\|\cdot\|^{\kappa}_{L\dot{A}_{\frac{p}{r}, \frac{q}{r}}}$
again, we obtain, for any $l\in\mathbb{Z}_{+}$,
\begin{align}\label{eq-f-f}
\left\|\left\{f_{l-j,j}\mathbf{1}_{P}\mathbf{1}
_{j\geq j_P}\right\}_{j\in\mathbb{Z}}
\right\|^{\kappa}_{L\dot{A}_{\frac{p}{r}, \frac{q}{r}}}
&=\left\|\left\{f_{l-j,j}\mathbf{1}_{P}\mathbf{1}_{j_P+l\geq j\geq j_P}
\right\}_{j\in\mathbb{Z}}\right\|^{\kappa}
_{L\dot{A}_{\frac{p}{r}, \frac{q}{r}}}\\
&\leq\sum_{j=j_P}^{j_P+l}\left\|\left\{f_{l-i,i}
\mathbf{1}_{P}\mathbf{1}_{i=j}\right\}_{i\in\mathbb{Z}}
\right\|^{\kappa}_{L\dot{A}_{\frac{p}{r}, \frac{q}{r}}}
=\sum_{j=j_P}^{j_P+l}\left\|
f_{l-j,j}\right\|^{\kappa}_{L^{\frac{p}{r}}(P)}\nonumber,
\end{align}
where the last equality follows from the fact that,
for any sequence $\{f_i\}_{i\in\mathbb{Z}}$ of
measurable functions on $\mathbb{R}^n$ with only
a non-zero component $f_{j}$ for some
$j\in\mathbb{Z}$, $\|\{f_i\}_{i\in\mathbb{Z}}\|
_{L\dot{A}_{\frac{p}{r}, \frac{q}{r}}}=\|f_{j}\|_{L^{\frac{p}{r}}}$.
For any $l\in\mathbb{Z}_{+}$ and $j\in\{j_P,\dots,j_P+l\}$,
let $P_{j-l}$ be as in Lemma \ref{lem-3P-sum1}(iv)
with $i$ replaced by $j-l$. Clearly, for any $x\in P$,
$B(x,2^{l-j})\subset3P_{j-l}$ and $|B(x,2^{l-j})|
\sim|P_{j-l}|$. From this, \eqref{eq-f_ij},
H\"{o}lder's inequality with the assumption
$r\in(0, \Gamma_{p,q})$, Lemma \ref{lem-3P-sum1}(iii)
with $P$ replaced by $P_{j-l}$,
and the quasi-triangle inequality of $\|\cdot\|_{L^p}$,
we infer that, for any $l\in\mathbb{Z}_{+}$,
\begin{align*}
\sum_{j=j_P}^{j_P+l}\left\|
f_{l-j,j}\right\|^{\kappa}_{L^{\frac{p}{r}}(P)}
&\leq\sum_{j=j_P}^{j_P+l}\left\|
\left[\fint_{B(\cdot, 2^{l-j})}|h_j(y)|^p
\mathbf{1}_{(3P)^{\complement}}(y)\,
dy\right]^{\frac{r}{p}}
\right\|^{\kappa}_{L^{\frac{p}{r}}(P)}\\
&\lesssim\sum_{j=j_P}^{j_P+l}
\left[\int_{3P_{j-l}}|h_j(y)|^p\,
dy\right]^{\frac{r\kappa}{p}}
2^{-(l-j+j_P)\frac{n}{p}r\kappa}\nonumber\\
&\lesssim\sum_{j=j_P}^{j_P+l}
\sum_{k\in\mathbb{Z}^n,\|k\|_{\infty}\leq1}
\left[\int_{P_{j-l}+k\ell(P_{j-l})}|h_j(y)|^p\,
dy\right]^{\frac{r\kappa}{p}}
2^{-(l-j+j_P)\frac{n}{p}r\kappa},\nonumber
\end{align*}
which, together with the definition of
$\|\cdot\|_{L\dot{A}_{p, q}^{\upsilon}}$,
Lemma \ref{lem-grow-est}(iii) with $P$ replaced by
$P_{j-l}$, and Lemma \ref{lem-grow-est}(i) with
$Q$ and $P$ replaced, respectively, by $P$ and $P_{j-l}$,
further implies that
\begin{align}\label{eq-f-h}
\sum_{j=j_P}^{j_P+l}\left\|
f_{l-j,j}\right\|^{\kappa}_{L^{\frac{p}{r}}(P)}
&\leq\left\|\left\{h_j\right\}_{j\in\mathbb{Z}}
\right\|^{r\kappa}_{L\dot{A}_{p, q}^{\upsilon}}
\sum_{j=j_P}^{j_P+l}\sum_{k\in\mathbb{Z}^n,
\|k\|_{\infty}\leq1}
[\upsilon(P_{j-l}+k\ell(P_{j-l}))]^{r\kappa}	
2^{-(l-j+j_P)\frac{n}{p}r\kappa}\\
&\sim\left\|\left\{h_j\right\}_{j\in\mathbb{Z}}
\right\|^{r\kappa}_{L\dot{A}_{p, q}^{\upsilon}}
\sum_{j=j_P}^{j_P+l}[\upsilon(P_{j-l})]^{r\kappa}
2^{-(l-j+j_P)\frac{n}{p}r\kappa}\nonumber\\
&\lesssim\left\|\left\{h_j\right\}_{j\in\mathbb{Z}}
\right\|^{r\kappa}_{L\dot{A}_{p, q}^{\upsilon}}
[\upsilon(P)]^{r\kappa}\sum_{j=j_P}^{j_P+l}	
2^{(l-j+j_P)(\delta_2-\frac{1}{p})nr\kappa}\nonumber\\
&\leq\left\|\left\{h_j\right\}_{j\in\mathbb{Z}}
\right\|^{r\kappa}_{L\dot{A}_{p, q}^{\upsilon}}
[\upsilon(P)]^{r\kappa}(1+l)2^{nl(\delta_2-\frac{1}{p})_+ r\kappa}\nonumber.
\end{align}
Combining  \eqref{eq-f-f}, \eqref{eq-f-h},  \eqref{eq-t-f}, and the assumption $\lambda\in(n+n(\delta_2-\frac{1}{p})_+r,\infty)$,
we conclude that
\begin{align}\label{eq-t-v-3}
\left\|\left\{t_j\mathbf{1}_{P}
\mathbf{1}_{j\geq j_P}\right\}_{j\in\mathbb{Z}}
\right\|_{L\dot{A}_{p, q}}&\lesssim
\left\{\sum_{l\in\mathbb{Z}_+}
2^{-l[\lambda-n-n(\delta_2-\frac{1}{p})_+ r]\kappa}
(1+l)\right\}^{\frac{1}{r\kappa}}
\left\|\left\{h_j\right\}_{j\in\mathbb{Z}}
\right\|_{L\dot{A}_{p, q}^{\upsilon}}\upsilon(P)\\
&\sim\left\|\left\{h_j\right\}_{j\in\mathbb{Z}}
\right\|_{L\dot{A}_{p, q}^{\upsilon}}\upsilon(P)\nonumber,
\end{align}
which completes the estimation of $\|\{t_j\mathbf{1}_{P}
\mathbf{1}_{j\geq j_P}\}_{j\in\mathbb{Z}}
\|_{L\dot{A}_{p, q}}$ in this case.

\emph{Case (2)} $n(\delta_2-\frac{1}{p})_+\geq\omega$.
In this case, by \eqref{eq-g-rho-t}, Lemma \ref{lem-3P-sum1}(iii),
and Lemma \ref{lem-xy-QR-P}(iii),
we find that, for any $j\in[j_P, \infty)\cap\mathbb{Z}$ and $x\in P$,
\begin{align*}
\left[t_j(x)\right]^r&=\sum_{k\in\mathbb{Z}^n, \|k\|_{\infty}\geq2}
\int_{P+k\ell(P)}\frac{2^{jn}}{(1+2^j|x-y|)^\lambda}|h_j(y)|^r\,dy\\\nonumber
&\sim2^{(j-j_P)(n-\lambda)}\sum_{k\in\mathbb{Z}^n, \|k\|_{\infty}\geq2}
(1+|k|)^{-\lambda}\fint_{P+k\ell(P)}|h_j(y)|^r\, dy.
\end{align*}
Using this, H\"{o}lder's inequality,
the definition of $\|\cdot\|_{L\dot{A}_{p, q}^{\upsilon}}$,
Lemma \ref{lem-grow-est}(ii), and Lemma \ref{lem-sum-k} together with
$\lambda\in(n+\omega r,\infty)$, $j=0$, and $x=\mathbf{0}$, we conclude that,
for any $j\in[j_P, \infty)\cap\mathbb{Z}$ and $x\in P$,
\begin{align*}
\left[t_j(x)\right]^r&\lesssim2^{(j-j_P)(n-\lambda)}|P|^{-\frac{r}{p}}
\sum_{k\in\mathbb{Z}^n, \|k\|_{\infty}\geq2}(1+|k|)^{-\lambda}
\left[\int_{P+k\ell(P)}\left|h_j(y)\right|^p\,dy
\right]^\frac{r}{p}\\
&\leq2^{(j-j_P)(n-\lambda)}|P|^{-\frac{r}{p}}
\left\|\left\{h_j\right\}_{j\in\mathbb{Z}}
\right\|^r_{L\dot{A}_{p, q}^{\upsilon}}
\sum_{k\in\mathbb{Z}^n, \|k\|_{\infty}\geq2}(1+|k|)^{-\lambda}
\left[\upsilon(P+k\ell(P))\right]^r\\
&\sim2^{(j-j_P)(n-\lambda)}|P|^{-\frac{r}{p}}
\left\|\left\{h_j\right\}_{j\in\mathbb{Z}}
\right\|^r_{L\dot{A}_{p, q}^{\upsilon}}[\upsilon(P)]^r
\sum_{k\in\mathbb{Z}^n, \|k\|_{\infty}\geq2}(1+|k|)^{-\lambda+\omega r}\\
&\sim2^{(j-j_P)(n-\lambda)}|P|^{-\frac{r}{p}}
\left\|\left\{h_j\right\}_{j\in\mathbb{Z}}
\right\|^r_{L\dot{A}_{p, q}^{\upsilon}}[\upsilon(P)]^r.	
\end{align*}
By this and the assumption
$\lambda\in(n, \infty)$, we obtain
\begin{align}\label{eq-t-v-1}
\left\|\left\{t_j\mathbf{1}_{P}\mathbf{1}
_{j\geq j_P}\right\}_{j\in\mathbb{Z}}
\right\|_{L\dot{A}_{p, q}}&\lesssim
\left[\sum_{j=j_P}^{\infty}2^{
\frac{(j-j_P)(n-\lambda)}{r}q}\right]^{\frac{1}{q}}
\left\|\left\{h_j\right\}_{j\in\mathbb{Z}}
\right\|_{L\dot{A}_{p, q}^{\upsilon}}\upsilon(P)
\sim\left\|\left\{h_j\right\}_{j\in\mathbb{Z}}
\right\|_{L\dot{A}_{p, q}^{\upsilon}}\upsilon(P).
\end{align}
This gives the desired estimate of $\|\{t_j\mathbf{1}_{P}
\mathbf{1}_{j\geq j_P}\}_{j\in\mathbb{Z}}
\|_{L\dot{A}_{p, q}}$ in this case.

Collecting the estimates \eqref{eq-rho-v},
\eqref{eq-t-v-3}, and \eqref{eq-t-v-1}
together to \eqref{eq-g-t},
we show that, for any given $P\in\mathcal{D}$,
\begin{align*}
\frac{1}{\upsilon(P)}\left\|\left\{g_j\mathbf{1}_{P}
\mathbf{1}_{j\geq j_P}\right\}_{j\in\mathbb{Z}}
\right\|_{L\dot{A}_{p, q}}
\lesssim\left\|\left\{h_j\right\}_{j\in\mathbb{Z}}
\right\|_{L\dot{A}_{p, q}^{\upsilon}},
\end{align*}
where the implicit positive constant is independent of
$\{g_j\}_{j\in\mathbb{Z}}$, $\{h_j\}_{j\in\mathbb{Z}}$,
and $P$. Taking the supremum over all $P\in\mathcal{D}$
on its left-hand side and using the definition of
$\|\cdot\|_{L\dot{A}_{p, q}^{\upsilon}}$,
we obtain $\|\{g_j\}_{j\in\mathbb{Z}}\|
_{L\dot{A}_{p, q}^{\upsilon}}\lesssim
\|\{h_j\}_{j\in\mathbb{Z}}\|_{L\dot{A}_{p, q}^{\upsilon}}$.
This finishes the proof of Proposition \ref{prop-gh}.
\end{proof}

Let $r\in(0, \infty]$ and $\lambda\in(0,\infty)$.
For any $t:=\{t_Q\}_{Q\in\mathcal{D}}$
in $\mathbb{C}$, we define the \emph{majorant sequence}
$t^{*}_{r, \lambda}:=\{t^{*}_{r, \lambda, Q}\}_{Q\in\mathcal{D}}$
of $t$ by setting, for any $Q\in\mathcal{D}$,
\begin{align}\label{eq-lamda*}
t^{*}_{r, \lambda, Q}:=
\left[\sum_{R\in\mathcal{D},\ell(R)=\ell(Q)}
\frac{|t_R|^{r}}{\{1+[\ell(R)]^{-1}
|x_Q-x_R|\}^{\lambda r}}\right]^{\frac{1}{r}}
\end{align}
with the usual modification made if $r=\infty$
(see, for example, \cite[p.\,48]{fj90}).
Observe that the majorant sequence $t^{*}_{r, \lambda}$
can be interpreted as the action of the discrete
Littlewood--Paley $g^*_{\lambda}$-function (see Definition \ref{def-LPfunc})
on $t$ or, when $r=1$, as the action of
a special almost diagonal operator (see Definition \ref{DEFadope})
on $t$. Based on this observation, we give the following proposition,
which can be regarded as the discrete Littlewood--Paley
$g^*_{\lambda}$-function characterization of sequence spaces $\dot{a}_{p,q}^{s,\upsilon}$
and hence is of independent interest.

\begin{proposition}\label{prop-dct-gh}
Let $a\in\{b, f\}$, $s\in\mathbb{R}$, and $p, q\in(0, \infty]$
($p<\infty$ if $a=f$). Assume that $\delta_1,\delta_2,\omega$
satisfy \eqref{eq-delta1<0} and
$\upsilon\in\mathcal{G}(\delta_1, \delta_2; \omega)$. If
$r\in(0, \infty]$ and $\lambda\in(\frac{n}{r\wedge\Gamma_{p,q}}
+[\omega\wedge n(\delta_2-\frac{1}{p})_+],\infty)$,
where $\Gamma_{p,q}$ is as in \eqref{eq-zeta},
then, for any $t:=\{t_Q\}_{Q\in\mathcal{D}}$
in $\mathbb{C}$, $\|t\|_{\dot{a}_{p,q}^{s,\upsilon}}
\sim\|{t}^{*}_{r, \lambda}\|_{\dot{a}_{p,q}^{s,\upsilon}}$,
where the positive equivalence
constants are independent of $t$.
\end{proposition}

\begin{proof}
Applying \eqref{eq-lamda*}, we obtain,
for any $Q\in\mathcal{D}$ and
$t:=\{t_Q\}_{Q\in\mathcal{D}}$ in $\mathbb{C}$,
$|t_Q|\leq t^{*}_{r, \lambda, Q}$ and hence
$\|t\|_{\dot{a}_{p,q}^{s,\upsilon}}\leq
\|{t}^{*}_{r, \lambda}\|_{\dot{a}_{p,q}^{s,\upsilon}}$.
Thus, to finish the proof, it suffices
to show that, for any $t:=\{t_Q\}_{Q\in\mathcal{D}}$
in $\mathbb{C}$, $\|{t}^{*}_{r, \lambda}\|_{\dot{a}_{p,q}^{s,\upsilon}}
\lesssim\|t\|_{\dot{a}_{p,q}^{s,\upsilon}}$.
We next prove this by considering the following
two cases for $\Gamma_{p,q}$ and $r$.

\emph{Case (1)} $\Gamma_{p,q}>r$. In this case, by \eqref{eq-lamda*},
the definition of $t_j$
[see \eqref{vect_j}], both (i) and (ii) of Lemma \ref{lem-3P-sum1},
and Lemma \ref{lem-xy-QR-P}(ii), we find that,
for any $j\in\mathbb{Z}$,
$t:=\{t_Q\}_{Q\in\mathcal{D}}$ in $\mathbb{C}$, and $x\in\mathbb{R}^n$,
\begin{align*}
\left|\sum_{Q\in\mathcal{D}_j}\widetilde{\mathbf{1}}_Q(x)
t^{*}_{r, \lambda, Q}\right|^r&=\sum_{Q\in\mathcal{D}_j}
\left[\widetilde{\mathbf{1}}_Q(x)\right]^r\sum_{R\in\mathcal{D}_j}
\frac{|t_R|^{r}}{\{1+[\ell(R)]^{-1}|x_Q-x_R|\}^{\lambda r}}\\
&\sim2^{jn}\sum_{Q\in\mathcal{D}_j}\mathbf{1}_Q(x)
\sum_{R\in\mathcal{D}_j}\int_{R}\frac1{(1+2^j|x-y|)^{\lambda r}}
\left|t_j(y)\right|^{r}\,dy\\
&=2^{jn}\int_{\mathbb{R}^n}\frac1{(1+2^j|x-y|)^{\lambda r}}
\left|t_j(y)\right|^{r}\,dy.
\end{align*}
For any $t:=\{t_Q\}_{Q\in\mathcal{D}}$
in $\mathbb{C}$, applying the definition of
$\|\cdot\|_{\dot{a}_{p,q}^{s,\upsilon}}$
and  Proposition \ref{prop-gh}
with the assumptions on $r$ and $\lambda$ and
with $\{g_j\}_{j\in\mathbb{Z}}$ and
$\{h_j\}_{j\in\mathbb{Z}}$ replaced, respectively, by
$\{2^{js}\sum_{Q\in\mathcal{D}_j}\widetilde{\mathbf{1}}_Q
t^{*}_{r, \lambda, Q}\}_{j\in\mathbb{Z}}$ and
$\{2^{js}t_j\}_{j\in\mathbb{Z}}$, we obtain
\begin{align*}
\left\|{t}^{*}_{r, \lambda}\right\|
_{\dot{a}_{p,q}^{s,\upsilon}}
=\left\|\left\{2^{js}\sum_{Q\in\mathcal{D}_j}
\widetilde{\mathbf{1}}_Q t^{*}_{r, \lambda,Q}
\right\}_{j\in\mathbb{Z}}\right\|
_{L\dot{A}_{p, q}^{\upsilon}}\lesssim
\left\|\left\{2^{js}t_j
\right\}_{j\in\mathbb{Z}}
\right\|_{L\dot{A}_{p, q}^{\upsilon}}=
\left\|t\right\|_{\dot{a}_{p,q}^{s,\upsilon}},
\end{align*}
which completes the proof of this case.

\emph{Case (2)} $\Gamma_{p,q}\leq r$. In this case,
the assumption on $\lambda$ is precisely $\lambda\in(\frac{n}{\Gamma_{p,q}}
+[\omega\wedge n(\delta_2-\frac{1}{p})_+],\infty)$.
Based on this, we can pick $\rho\in(0, \Gamma_{p,q})$
such that $\lambda\in(\frac{n}{\rho}+[\omega\wedge
n(\delta_2-\frac{1}{p})_+], \infty)$.
Using the monotonicity of $l^q$ on $q$,
we conclude that, for any $Q\in\mathcal{D}$ and
$t:=\{t_Q\}_{Q\in\mathcal{D}}$ in $\mathbb{C}$,
$t^{*}_{r, \lambda,Q}\leq t^{*}_{\rho, \lambda,Q}$.
From this and the just proved Case (1) with $r$
replaced by $\rho$, we infer that,
for any $t:=\{t_Q\}_{Q\in\mathcal{D}}$
in $\mathbb{C}$, $\|{t}^{*}_{r, \lambda}\|_{\dot{a}_{p,q}^{s,\upsilon}}
\leq\|{t}^{*}_{\rho, \lambda}\|_{\dot{a}_{p,q}^{s,\upsilon}}
\lesssim\|t\|_{\dot{a}_{p,q}^{s,\upsilon}}$.
This finishes the proof of this case and
hence Proposition \ref{prop-dct-gh}.
\end{proof}

To present a sharp estimate of reducing operators
established in \cite{bhyy4}, we need to recall the following concepts
introduced in \cite[Definition 6.2]{bhyy4}.

\begin{definition}
Let $p\in(0,\infty)$ and $d\in\mathbb{R}$.
A matrix weight $W$ is said to have
\emph{$\mathcal{A}_{p,\infty}$-lower dimension $d$}
if there exists a positive constant $C$ such that,
for any $t\in[1,\infty)$ and any cube $Q\subset\mathbb{R}^n$,
\begin{align*}
\exp\left(\fint_{\lambda Q}\log\left(\fint_Q
\left\|W^{\frac{1}{p}}(x)W^{-\frac{1}{p}}(y)
\right\|^p\,dx\right)\,dy\right)
\leq C{\lambda}^d.
\end{align*}
A matrix weight $W$ is said to have
\emph{$\mathcal{A}_{p,\infty}$-upper dimension $d$}
if there exists a positive constant $C$ such that,
for any $\lambda\in[1,\infty)$ and any cube $Q\subset\mathbb{R}^n$,
\begin{align*}
\exp\left(\fint_Q\log\left(\fint_{\lambda Q}
\left\|W^{\frac{1}{p}}(x)W^{-\frac{1}{p}}(y)
\right\|^p\,dx\right)\,dy\right)\leq C\lambda^d.
\end{align*}
\end{definition}

Let $p\in(0,\infty)$. Using \cite[Propositions 6.4(ii)
and 6.5(ii)]{bhyy4}, we conclude that,
for any $W\in\mathcal{A}_{p,\infty}$, there exist $d_1\in[0,n)$ and $d_2\in[0,\infty)$ such that
$W$ has $\mathcal{A}_{p,\infty}$-lower dimension $d_1$
and $\mathcal{A}_{p,\infty}$-upper dimension $d_2$.
By \cite[Proposition 6.4(i) and 6.5(i)]{bhyy4},
we find that the $\mathcal{A}_{p,\infty}$-lower and the
$\mathcal{A}_{p,\infty}$-upper dimensions are both nonnegative.
Based on these facts, for any $W\in\mathcal{A}_{p,\infty}$, let
\begin{align}\label{eq-low-dim}
d_{p,\infty}^{\mathrm{lower}}(W):=
\inf\left\{d\in[0,n):\ W\text{ has }\mathcal{A}_{p,\infty}
\text{-lower dimension } d\right\}
\end{align}
and
\begin{align}\label{eq-upp-dim}
d_{p,\infty}^{\mathrm{upper}}(W)
:=\inf\left\{d\in[0,\infty):\ W\text{ has }\mathcal{A}_{p,\infty}
\text{-upper dimension } d\right\}.
\end{align}
Furthermore, let
\begin{align*}
[\![d_{p,\infty}^{\mathrm{lower}}(W),n):=
\begin{cases}
[d_{p,\infty}^{\mathrm{lower}}(W),n)
&\text{if } W\text{ has }\mathcal{A}_{p,\infty}
\text{-lower dimension }d_{p,\infty}^{\mathrm{lower}}(W),\\
(d_{p,\infty}^{\mathrm{lower}}(W),n)
&\text{otherwise}
\end{cases}
\end{align*}
and
\begin{align*}
[\![d_{p,\infty}^{\mathrm{upper}}(W),\infty):=
\begin{cases}
[d_{p,\infty}^{\mathrm{upper}}(W),\infty)
&\text{if } W\text{ has }\mathcal{A}_{p,\infty}
\text{-upper dimension }d_{p,\infty}^{\mathrm{upper}}(W),\\
(d_{p,\infty}^{\mathrm{upper}}(W),\infty)
&\text{otherwise}.
\end{cases}
\end{align*}

For any $p\in(0,\infty)$, any $W\in\mathcal{A}_{p,\infty}$,
and any sequence $\{A_Q\}_{Q\in\mathcal{D}}$
of reducing operators of order $p$ for $W$,
the sharp estimate of $\|A_QA^{-1}_{R}\|$
was established in \cite[Lemma 6.8(i)]{bhyy4} as follows.

\begin{lemma}\label{growEST}
Let $p \in(0,\infty), W \in\mathcal{A}_{p,\infty}$,
and $\{A_Q\}_{Q\in\mathcal{D}}$
be a sequence of reducing operators of order $p$ for $W$.
If $\beta_1 \in \llbracket d_{p, \infty}^{\mathrm{lower}}(W),\infty)$
and $\beta_2\in \llbracket d_{p, \infty}^{\mathrm{upper}}(W),\infty)$,
then there exists a positive constant $C$ such that,
for any $Q, R\in\mathcal{D}$,
\begin{align}\label{eq-st-double}
\left\|A_Q A_R^{-1}\right\|^p\leq C\max
\left\{\left[\frac{\ell(R)}{\ell(Q)}\right]^{\beta_1},
\left[\frac{\ell(Q)}{\ell(R)}\right]^{\beta_2}\right\}
\left[1+\frac{|x_Q-x_R|}{\ell(Q)\vee\ell(R)}\right]^{\beta_1+\beta_2}.
\end{align}
\end{lemma}

The following concepts were introduced
in \cite[Definition 1.3]{rou04} and
\cite[Definition 2.1]{fr21}.

\begin{definition}\label{def-doub-seq}
Let $\beta_1, \beta_2, \beta_3 \in [0, \infty)$ and
$\mathbb{A}:=\{A_Q\}_{Q\in\mathcal{D}}$ be a
sequence of positive definite matrices.
Then $\mathbb{A}$ is said to be
\begin{itemize}
\item [{\rm (i)}]
\emph{strongly doubling of order $(\beta_1, \beta_2)$}
if there exists a positive constant $C$ such that,
for any $Q, R\in\mathcal{D}$, \eqref{eq-st-double} holds;
\item [{\rm (ii)}]
\emph{weakly doubling of order $\beta_3$}
if there exists a positive constant $C$ such that,
for any $Q, R \in\mathcal{D}$ with $\ell(Q)=\ell(R)$,
\begin{align*}
\left\|A_Q A_R^{-1}\right\|^p\leq C
\left\{1+[\ell(R)]^{-1}|x_Q-x_R|\right\}^{\beta_3}.
\end{align*}
\end{itemize}
\end{definition}

Let $p\in(0,\infty)$, $W\in\mathcal{A}_{p,\infty}$,
and $\mathbb{A}:=\{A_Q\}_{Q\in\mathcal{D}}$ be a sequence
of reducing operators of order $p$ for $W$.
By Lemma \ref{growEST}, we find that, for any
$\beta_1 \in \llbracket d_{p, \infty}^{\mathrm{lower}}(W),\infty)$
and $\beta_2\in \llbracket d_{p, \infty}^{\mathrm{upper}}(W),\infty)$,
$\mathbb{A}$ is strongly doubling of order $(\beta_1, \beta_2)$,
where $d^{\operatorname{lower}}_{p, \infty}(W)$ and
$d^{\operatorname{ upper}}_{p, \infty}(W)$ are as, respectively, in
\eqref{eq-low-dim} and \eqref{eq-upp-dim}.
Next, we present an important technical lemma,
which in the case where $r\in(0,1]$ was given by
\cite[(2.8)]{fr21}.

\begin{lemma}\label{lem-suptosum}
Let $\varphi\in\mathcal{S}$ satisfy \eqref{cond1},
$r\in(0, \infty)$, $\lambda\in\mathbb{R}$, and $\beta\in[0,\infty)$.
If $\{A_Q\}_{Q\in\mathcal{D}}$ is weakly doubling of order $\beta$,
then there exists a positive constant $C$
such that, for any $j\in\mathbb{Z}$,
$k\in\mathbb{Z}^n$, and $\vec{f}\in(\mathcal{S}'_{\infty})^m$,
\begin{align}\label{eq-A-phi-f}
\sup_{x\in Q_{j, k}}\left|A_{Q_{j, k}}
\left(\varphi_j*\vec{f}\right)(x)\right|^{r}\leq C
\sum_{l\in{\mathbb{Z}}^n}\frac{2^{jn}}{(1+|k-l|)^{\lambda}}
\int_{Q_{j, l}}\left|A_{Q_{j, l}}
\left(\varphi_j*\vec{f}\right)(y)\right|^{r}\,dy.
\end{align}
\end{lemma}

\begin{proof}
Since the right-hand side of \eqref{eq-A-phi-f}
decreases as $\lambda$ increases, to prove the present lemma,
we only need to consider the case where $\lambda\in(0, \infty)$.
To this end, we consider the following two cases for $r$.

\emph{Case (1)} $r\in(0, 1]$. In this case,
by \cite[(2.8)]{fr21} with $A$ and $A(R-r)$
replaced, respectively, by $r$ and $\lambda$,
we find that \eqref{eq-A-phi-f} holds,
which completes the proof of the present lemma
in this case.

\emph{Case (2)} $r\in(1, \infty)$.
In this case, let $\tau\in(\lambda\vee n, \infty)$.
Applying the just proved Case (1) with $r=1$,
we obtain, for any $j\in\mathbb{Z}$, $k\in\mathbb{Z}^n$,
and $\vec{f}\in(\mathcal{S}'_{\infty})^m$,
\begin{align*}
\sup_{x\in Q_{j, k}}\left|A_{Q_{j, k}}
\left(\varphi_j*\vec{f}\right)(x)\right|\lesssim
\sum_{l\in{\mathbb{Z}}^n}\frac{1}{(1+|k-l|)^{\tau}}
\fint_{Q_{j, l}}\left|A_{Q_{j, l}}
\left(\varphi_j*\vec{f}\right)(y)\right|\,dy.
\end{align*}
Using this, the assumption $\tau\in(\lambda\vee n, \infty)$,
Lemma \ref{lem-sum-k} combined with $x=0$ and $j=0$,
and H\"{o}lder's inequality,
we conclude that, for any $j\in\mathbb{Z}$, $k\in\mathbb{Z}^n$,
and $\vec{f}\in(\mathcal{S}'_{\infty})^m$,
\begin{align*}
\sup_{x\in Q_{j, k}}\left|A_{Q_{j, k}}
\left(\varphi_j*\vec{f}\right)(x)\right|
&\leq\left[\sum_{l\in{\mathbb{Z}}^n}
\frac{1}{(1+|k-l|)^{\tau}}\right]^{\frac{1}{r'}}
\left[\sum_{l\in{\mathbb{Z}}^n}
\frac{1}{(1+|k-l|)^{\tau}}
\fint_{Q_{j, l}}\left|A_{Q_{j, l}}
\left(\varphi_j*\vec{f}\right)(y)
\right|^r\,dy\right]^{\frac{1}{r}}\\
&\lesssim\left[\sum_{l\in{\mathbb{Z}}^n}
\frac{2^{jn}}{(1+|k-l|)^{\lambda}}
\int_{Q_{j, l}}\left|A_{Q_{j, l}}
\left(\varphi_j*\vec{f}\right)(y)
\right|^r\,dy\right]^{\frac{1}{r}}.
\end{align*}
This finishes the proof of \eqref{eq-A-phi-f} in this case
and hence Lemma \ref{lem-suptosum}.
\end{proof}

Observe that, applying \cite[Theorem 2.3.21]{gra14a}, we find that,
for any $f\in\mathcal{S}'$ with $\widehat{f}$ having compact support,
where the definition of the support of $\widehat{f}$ can be found
in \cite[Definition 2.3.16]{gra14a},
$f$ is an infinitely differential function
on $\mathbb{R}^n$. In what follows, for any $\vec{f}:=(f_1,\dots,f_m)^{T}
\in(\mathcal{S}')^m$ and any set $K\subset \mathbb{R}^n$,
we say that $\operatorname{supp}\widehat{{\vec{f}}}\subset K$
if, for any $i\in\{1,\dots, m\}$, $\operatorname{supp}\widehat{f_i}\subset K$.

The following lemma is a homogeneous variant of \cite[Lemma 3.15]{bhyy5},
which can be proved by a slight modification on
the proof of \cite[Lemma 3.15]{bhyy5}; we omit the details.

\begin{lemma}\label{lem-ab-equi}
Let $\beta_1, \beta_2\in[0,\infty)$ and
$\{A_Q\}_{Q\in\mathcal{D}}$
be strongly doubling of order $(\beta_1,\beta_2)$.
Suppose that $r\in(0,\infty)$, $\lambda\in(\frac{n}{r},\infty)$,
and $\gamma\in\mathbb{Z}_+$ is sufficiently large.
Then, for any $j\in\mathbb{Z}$, $\vec{f}\in(\mathcal{S}')^m$
with $\operatorname{supp}\widehat{{\vec{f}}}
\subset\{\xi\in\mathbb{R}^n:\ |\xi|\leq2^{j+1}\}$, and
$Q\in\mathcal{D}_j$, $t^*_{r, \lambda,Q}
\sim u^*_{r, \lambda,Q}$,
where the positive equivalence constants are independent of
$j, \vec{f}$, and $Q$ and, for any $Q\in\mathcal{D}$,
$t^*_{r, \lambda,Q}$ and $u^*_{r, \lambda,Q}$ are as
in \eqref{eq-lamda*} with
\begin{align*}
t:=\left\{|Q|^{\frac{1}{2}}\sup_{y\in Q}
\left|A_Q\vec{f}(y)\right|\right\}_{Q\in\mathcal{D}}\ \
\text{and}\ \ u:=\left\{{|Q|}^{\frac{1}{2}}
\max_{\genfrac{}{}{0pt}{}{P\in\mathcal{D}_{j_Q+\gamma}}{P\subset Q}}
\inf_{y\in P}\left|A_P\vec{f}(y)\right|\right\}_{Q\in\mathcal{D}}.
\end{align*}
\end{lemma}

Let $\varphi\in\mathcal{S}_{\infty}$ and
$\mathbb{A}:=\{A_Q\}_{Q\in\mathcal{D}}$
be a sequence of positive definite matrices.
For any $\gamma\in\mathbb{Z}_+$
and $\vec{f}\in(\mathcal{S}'_{\infty})^m$, let
\begin{align}\label{supf}
\sup_{\mathbb{A}, \varphi}\left(\vec{f}\right):=
\left\{\sup_{\mathbb{A}, \varphi, Q}\left(\vec{f}\right)\right\}_{Q\in\mathcal{D}}
:=\left\{|Q|^{\frac{1}{2}}\sup_{y\in Q}
\left|A_Q\left(\varphi_{j_Q}*\vec{f}\right)(y)\right|\right\}_{Q\in\mathcal{D}}
\end{align}
and
\begin{align}\label{inff}
\inf_{\mathbb{A}, \varphi, \gamma}\left(\vec{f}\right)
&:=\left\{\inf_{\mathbb{A}, \varphi, Q, \gamma}
\left(\vec{f}\right)\right\}_{Q\in\mathcal{D}}:=\left\{|Q|^{\frac{1}{2}}
\max_{\genfrac{}{}{0pt}{}{P\in\mathcal{D}_{j_Q+\gamma}}{P\subset Q}}
\inf_{y\in P}\left|A_P\left(\varphi_{j_Q}*\vec{f}\right)(y)
\right|\right\}_{Q\in\mathcal{D}}.
\end{align}	

Based on Lemma \ref{lem-ab-equi}, we establish the following equivalences.

\begin{lemma}\label{A(A)=a}
Let $(A, a)\in\{(B, b), (F, f)\}$, $s\in\mathbb{R}$,
and $p, q\in(0, \infty]$ ($p<\infty$ if $A=F$).
Assume that $\delta_1,\delta_2,\omega$ satisfy \eqref{eq-delta1<0},
$\upsilon\in\mathcal{G}(\delta_1, \delta_2; \omega)$,
and $\varphi\in\mathcal{S}$ satisfies \eqref{cond1}.
Suppose that $\beta_1, \beta_2\in[0,\infty)$,
$\mathbb{A}:=\{A_Q\}_{Q\in\mathcal{D}}$
is strongly doubling of order $(\beta_1,\beta_2)$,
and $\gamma\in\mathbb{Z}_+$ is sufficiently large
as in Lemma \ref{lem-ab-equi}. Then,
for any $\vec{f}\in(\mathcal{S}'_{\infty})^m$,
$\|\vec{f}\|_{\dot{A}_{p,q}^{s,\upsilon}
(\mathbb{A}, \varphi)}\sim\|\sup_{\mathbb{A}, \varphi}
(\vec{f})\|_{\dot{a}_{p,q}^{s,\upsilon}}\sim
\|\inf_{\mathbb{A}, \varphi, \gamma}
(\vec{f})\|_{\dot{a}_{p,q}^{s,\upsilon}}$,
where all the positive equivalence constants are
independent of $\vec{f}$.
\end{lemma}

\begin{proof}
We begin with proving the first equivalence in the present lemma.
By \eqref{A_j}, \eqref{supf}, and the definitions of
$\|\cdot\|_{\dot{A}_{p,q}^{s,\upsilon}(\mathbb{A}, \varphi)}$
and $\|\cdot\|_{\dot{a}_{p,q}^{s,\upsilon}}$, we find that,
for any $\vec{f}\in(\mathcal{S}'_{\infty})^m$,
\begin{align}\label{eq-f-sup}
\left\|\vec{f}\right\|_{\dot{A}_{p,q}^{s,\upsilon}
(\mathbb{A}, \varphi)}
&=\left\|\left\{2^{js}\left|\mathbb{A}_j
\left(\varphi_j*\vec{f}\right)\right|\right\}_{j\in\mathbb{Z}}
\right\|_{L\dot{A}_{p, q}^{\upsilon}}
=\left\|\left\{2^{js}\left|\sum_{Q\in\mathcal{D}_j}\mathbf{1}_QA_Q
\left(\varphi_j*\vec{f}\right)\right|\right\}_{j\in\mathbb{Z}}
\right\|_{L\dot{A}_{p, q}^{\upsilon}}\\
&\leq\left\|\left\{2^{js}\sum_{Q\in\mathcal{D}_j}
\widetilde{\mathbf{1}}_Q\sup_{\mathbb{A}, \varphi, Q}
\left(\vec{f}\right)\right\}_{j\in\mathbb{Z}}
\right\|_{L\dot{A}_{p, q}^{\upsilon}}
=\left\|\sup_{\mathbb{A}, \varphi}
\left(\vec{f}\right)\right\|_{\dot{a}_{p,q}^{s,\upsilon}}.\nonumber
\end{align}

Next, we show the reverse estimate of \eqref{eq-f-sup}.
To this end, let $r\in(0, \Gamma_{p,q})$ and
$\lambda\in(n+[\omega\wedge n(\delta_2-\frac{1}{p})_+]r,\infty)$,
where $\Gamma_{p,q}$ is as in \eqref{eq-zeta}.
Using \eqref{supf}, Lemma \ref{lem-suptosum},
the definition of $\mathbb{A}_j$ [see \eqref{A_j}], and
Lemmas \ref{lem-3P-sum1}(i) and \ref{lem-xy-QR-P}(ii),
we conclude that, for any $j\in\mathbb{Z}$,
$\vec{f}\in(\mathcal{S}'_{\infty})^m$, and $x\in\mathbb{R}^n$,
\begin{align}\label{eq-sup-A_j}
\left|\sum_{Q\in\mathcal{D}_j}\widetilde{\mathbf{1}}_{Q}(x)
\sup_{\mathbb{A},\varphi, Q}\left(\vec{f}\right)\right|^{r}
&=\sum_{Q\in\mathcal{D}_j}\mathbf{1}_{Q}(x)
\sup_{y\in Q}\left|A_{Q}\left(\varphi_j
*\vec{f}\right)(y)\right|^{r}\\
&\lesssim\sum_{Q\in\mathcal{D}_j}\mathbf{1}_{Q}(x)
\sum_{R\in\mathcal{D}_j}
\frac{2^{jn}}{(1+2^j|x_Q-x_R|)^{\lambda}}\int_{R}\left|A_{R}
\left(\varphi_j*\vec{f}\right)(y)\right|^{r}\,dy\nonumber\\
&\sim\sum_{R\in\mathcal{D}_j}
\int_{R}\frac{2^{jn}}{(1+2^j|x-y|)^{\lambda}}\left|\mathbb{A}_j(y)
\left(\varphi_j*\vec{f}\right)(y)\right|^{r}\,dy\nonumber\\
&=\int_{\mathbb{R}^n}\frac{2^{jn}}{(1+2^j|x-y|)^{\lambda}}
\left|\mathbb{A}_j(y)\left(\varphi_j*\vec{f}\right)(y)\right|^{r}\,dy.\nonumber
\end{align}
For any $\vec{f}\in(\mathcal{S}'_{\infty})^m$,
by \eqref{eq-sup-A_j}, the definitions of
$\|\cdot\|_{\dot{a}_{p,q}^{s,\upsilon}}$ and
$\|\cdot\|_{\dot{A}_{p,q}^{s,\upsilon}(\mathbb{A}, \varphi)}$,
and Proposition \ref{prop-gh} with $\{g_j\}_{j\in\mathbb{Z}}$ and $\{h_j\}_{j\in\mathbb{Z}}$
replaced, respectively, by $\{2^{js}\sum_{Q\in\mathcal{D}_j}
\widetilde{\mathbf{1}}_Q\sup_{\mathbb{A}, \varphi,Q}
(\vec{f})\}_{j\in\mathbb{Z}}$ and $\{2^{js}|\mathbb{A}_j
(\varphi_j*\vec{f})|\}_{j\in\mathbb{Z}}$ and with the
aforementioned assumptions on $r$ and $\lambda$, we find that
\begin{align*}
\left\|\sup_{\mathbb{A}, \varphi}
\left(\vec{f}\right)\right\|
_{\dot{a}_{p,q}^{s,\upsilon}}
&=\left\|\left\{2^{js}\sum_{Q\in\mathcal{D}_j}
\widetilde{\mathbf{1}}_Q\sup_{\mathbb{A}, \varphi,Q}
\left(\vec{f}\right)\right\}_{j\in\mathbb{Z}}
\right\|_{L\dot{A}_{p, q}^{\upsilon}}\\
&\lesssim\left\|\left\{2^{js}\left|\mathbb{A}_j
\left(\varphi_j*\vec{f}\right)\right|
\right\}_{j\in\mathbb{Z}}
\right\|_{L\dot{A}_{p, q}^{\upsilon}}=
\left\|\vec{f}\right\|_{\dot{A}_{p,q}^{s,
\upsilon}(\mathbb{A}, \varphi)},
\end{align*}
which completes the proof of reverse estimate of \eqref{eq-f-sup}
and hence the first equivalence.

Finally, we prove the second equivalence in the present lemma.
For any $\vec{f}\in(\mathcal{S}'_{\infty})^m$
and $Q\in\mathcal{D}$, applying Lemma \ref{lem-ab-equi}
with $\vec{f}$ replaced by $\varphi_{j_Q}*\vec{f}$
and with the assumptions on $r$ and $\lambda$, we obtain
\begin{align*}
\left[\sup_{\mathbb{A}, \varphi}
\left(\vec{f}\right)\right]_{r, \frac{\lambda}{r},Q}^{*}
\sim\left[\inf_{\mathbb{A},\varphi,\gamma}
\left(\vec{f}\right)\right]_{r, \frac{\lambda}{r},Q}^{*},
\end{align*}
which, together with Proposition \ref{prop-dct-gh}
on $\sup_{\mathbb{A}, \varphi}(\vec{f})$
and $\inf_{\mathbb{A}, \varphi, \gamma}
(\vec{f})$ with the aforementioned assumptions on $r$ and $\lambda$ again,
further implies that
\begin{align*}
\left\|\sup_{\mathbb{A}, \varphi}
\left(\vec{f}\right)\right\|
_{\dot{a}_{p,q}^{s,\upsilon}}
\sim\left\|\left[\sup_{\mathbb{A}, \varphi}
\left(\vec{f}\right)\right]_{r, \frac{\lambda}{r}}^{*}\right\|
_{\dot{a}_{p,q}^{s,\upsilon}}\sim
\left\|\left[\inf_{\mathbb{A},\varphi,\gamma}
\left(\vec{f}\right)\right]_{r, \frac{\lambda}{r}}^{*}\right\|
_{\dot{a}_{p,q}^{s,\upsilon}}\sim
\left\|\inf_{\mathbb{A}, \varphi, \gamma}
\left(\vec{f}\right)\right\|
_{\dot{a}_{p,q}^{s,\upsilon}}.
\end{align*}
This finishes the proof of the second equivalence
and hence Lemma \ref{A(A)=a}.
\end{proof}

\begin{remark}\label{rem-equi-seq}
By checking the proof of Lemma \ref{A(A)=a} very
carefully, we find that, when $\mathbb{A}$ is only weakly doubling,
the first equivalence in Lemma \ref{A(A)=a} in this case also holds.
\end{remark}

Now, we give the proof of Theorem \ref{thm-A(W)=A(A)}.

\begin{proof}[Proof of Theorem \ref{thm-A(W)=A(A)}]
We first prove that, for any $\vec{f}\in (\mathcal{S}'_{\infty})^m$,
$\|\vec{f}\|_{\dot{A}_{p,q}^{s,\upsilon}(W, \varphi)}
\lesssim\|\vec{f}\|_{\dot{A}_{p,q}^{s,\upsilon}(\mathbb{A}, \varphi)}$.
For this purpose, applying Lemmas \ref{growEST} and \ref{A(A)=a},
we only need to show that, for any $\vec{f}\in (\mathcal{S}'_{\infty})^m$,
\begin{align}\label{eq-f-supf}
\left\|\vec{f}\right\|_{\dot{A}_{p,q}^{s,\upsilon}(W, \varphi)}
\lesssim\left\|\sup_{\mathbb{A}, \varphi}\left(\vec{f}\right)\right\|
_{\dot{a}_{p,q}^{s,\upsilon}},
\end{align}
where $\sup_{\mathbb{A}, \varphi}(\vec{f})$
is as in \eqref{supf}. By Lemma \ref{lem-3P-sum1}(ii),
the definition of $\gamma_j$ [see \eqref{eq-gamma_j}],
and \eqref{supf}, we find that, for any $j\in\mathbb{Z}$,
$\vec{f}\in(\mathcal{S}'_{\infty})^m$, and $x\in\mathbb{R}^n$,
\begin{align}\label{eq-ineqau-equa}
\left|W^{\frac1{p}}(x)\left(\varphi_j*\vec{f}\right)(x)\right|
&=\sum_{Q\in\mathcal{D}_j}\mathbf{1}_Q(x)\left|W^{\frac1{p}}(x)
A^{-1}_QA_Q\left(\varphi_j*\vec{f}\right)(x)\right|\\
&\leq\sum_{Q\in\mathcal{D}_j}\mathbf{1}_{Q}(x)
\left\|W^{\frac1{p}}(x)A^{-1}_Q\right\|\left|A_Q
\left(\varphi_j*\vec{f}\right)(x)\right|\nonumber\\
&\leq\gamma_j(x)\sum_{Q\in\mathcal{D}_j}\widetilde{\mathbf{1}}_{Q}(x)
\sup_{\mathbb{A}, \varphi, Q}\left(\vec{f}\right).\nonumber
\end{align}
For any $\vec{f}\in (\mathcal{S}'_{\infty})^m$,
from \eqref{eq-ineqau-equa}, the definitions of
$\|\cdot\|_{\dot{A}_{p,q}^{s,\upsilon}(W, \varphi)}$
and $\|\cdot\|_{\dot{a}_{p,q}^{s,\upsilon}}$,
and Lemma \ref{lem-E_j} with $\{t_Q\}_{Q\in\mathcal{D}}$
replaced by $\{2^{j_Q s}\sup_{\mathbb{A},
\varphi,Q}(\vec{f})\}_{Q\in\mathcal{D}}$, we infer that
\begin{align*}
\left\|\vec{f}\right\|_{\dot{A}_{p,q}^{s,\upsilon}(W, \varphi)}
&=\left\|\left\{2^{js}\left|W^{\frac1{p}}
\left(\varphi_j*\vec{f}\right)\right|\right\}_{j\in\mathbb{Z}}
\right\|_{L\dot{A}_{p, q}^{\upsilon}}
\leq\left\|\left\{2^{js}\gamma_j\sum_{Q\in\mathcal{D}_j}\widetilde{\mathbf{1}}_{Q}
\sup_{\mathbb{A}, \varphi, Q}\left(\vec{f}\right)\right\}_{j\in\mathbb{Z}}
\right\|_{L\dot{A}_{p, q}^{\upsilon}}\\
&\lesssim\left\|\left\{2^{js}\sum_{Q\in\mathcal{D}_j}
\widetilde{\mathbf{1}}_{Q}\sup_{\mathbb{A}, \varphi, Q}
\left(\vec{f}\right)\right\}_{j\in\mathbb{Z}}
\right\|_{L\dot{A}_{p, q}^{\upsilon}}
=\left\|\sup_{\mathbb{A}, \varphi}\left(\vec{f}\right)\right\|
_{\dot{a}_{p,q}^{s,\upsilon}},
\end{align*}
which further implies that \eqref{eq-f-supf} holds.

Next, we prove that,
for any $\vec{f}\in (\mathcal{S}'_{\infty})^m$,
$\|\vec{f}\|_{\dot{A}_{p,q}^{s,\upsilon}(\mathbb{A}, \varphi)}
\lesssim\|\vec{f}\|_{\dot{A}_{p,q}^{s,\upsilon}(W, \varphi)}$.
To this end, using Lemma \ref{A(A)=a},
we only need to show, for any $\vec{f}\in(\mathcal{S}'_{\infty})^m$,
\begin{align}\label{eq-inf-f}
\left\|\inf_{\mathbb{A}, \varphi, \gamma}
\left(\vec{f}\right)\right\|_{\dot{a}_{p,q}^{s,\upsilon}}\lesssim
\left\|\vec{f}\right\|_{\dot{A}_{p,q}^{s,\upsilon}(W, \varphi)},
\end{align}
where both $\gamma\in\mathbb{Z}_+$ and
$\inf_{\mathbb{A}, \varphi, \gamma}(\vec{f})$
are as in Lemma \ref{A(A)=a}.	
To obtain \eqref{eq-inf-f}, for any $Q\in\mathcal{D}$ and
$\vec{f}\in (\mathcal{S}'_{\infty})^m$, from \eqref{inff},
we infer that there exists $R_Q\in\mathcal{D}_{j_Q+\gamma}$
satisfying $R_Q\subset Q$ and
\begin{align}\label{eq-inf}
\inf_{\mathbb{A}, \varphi, Q, \gamma}\left(\vec{f}\right)
=|Q|^{\frac{1}{2}}\inf_{y\in R_Q}\left|A_{R_Q}
\left(\varphi_{j_Q}*\vec{f}\right)(y)\right|.
\end{align}
Applying Lemma \ref{lem-AW-1-set}, we conclude that there exists
$L\in(0, \infty)$ such that, for any $Q\in\mathcal{D}$,
$$E_Q:=\left\{x\in R_Q:\ \left\|A_{R_Q} W^{-\frac{1}{p}}(x)\right\|<L\right\}$$ satisfies that
\begin{align}\label{eq-E_Q-R_Q-Q}
E_Q\subset R_Q\subset Q\text{\ and\ }
|E_Q|\geq\frac{|R_Q|}{2}=2^{-(\gamma n+1)}|Q|.
\end{align}
By \eqref{eq-inf} and the above choice of $E_Q$ for any $Q\in\mathcal{D}$,
we find that, for any $Q\in\mathcal{D}$,
$\vec{f}\in(\mathcal{S}'_{\infty})^m$, and $x\in E_Q$,
\begin{align*}
\inf_{\mathbb{A}, \varphi, Q, \gamma}\left(\vec{f}\right)
&\leq|Q|^{\frac{1}{2}}\left|A_{R_Q}
W^{-\frac{1}{p}}(x)W^{\frac{1}{p}}(x)
\left(\varphi_{j_Q}*\vec{f}\right)(x)\right|\\
&\leq|Q|^{\frac{1}{2}}\left\|A_{R_Q} W^{-\frac{1}{p}}(x)\right\|
\left|W^{\frac{1}{p}}(x)\left(\varphi_{j_Q}*\vec{f}\right)(x)\right|
\lesssim|Q|^{\frac{1}{2}}\left|W^{\frac{1}{p}}(x)
\left(\varphi_{j_Q}*\vec{f}\right)(x)\right|,
\end{align*}
which, together with \eqref{eq-E_Q-R_Q-Q},
Lemma \ref{lem-3P-sum1}(ii), Lemma \ref{lem-aLA-equi}
with $t$ replaced by $\inf_{\mathbb{A}, \varphi, \gamma}(\vec{f})$,
and the definition of $\|\cdot\|_{\dot{A}_{p,q}^{s,\upsilon}(W, \varphi)}$,
further implies that, for any $\vec{f}\in (\mathcal{S}'_{\infty})^m$,
\begin{align*}
\left\|\inf_{\mathbb{A}, \varphi, \gamma}\left(\vec{f}\right)
\right\|_{\dot{a}_{p,q}^{s,\upsilon}}
&\sim\left\|\left\{2^{js}\sum_{Q\in\mathcal{D}_j}
\widetilde{\mathbf{1}}_{E_Q}\inf_{\mathbb{A},
\varphi, Q, \gamma}\left(\vec{f}\right)\right\}
_{j\in\mathbb{Z}}\right\|_{L\dot{A}_{p, q}^{\upsilon}}
\lesssim\left\|\left\{2^{js}\sum_{Q\in\mathcal{D}_j}\mathbf{1}_Q
\left|W^{\frac{1}{p}}\left(\varphi_{j}*\vec{f}\right)\right|
\right\}_{j\in\mathbb{Z}}\right\|_{L\dot{A}_{p, q}^{\upsilon}}\\
&=\left\|\left\{2^{js}\left|W^{\frac{1}{p}}
\left(\varphi_{j}*\vec{f}\right)\right|\right\}_{j\in\mathbb{Z}}
\right\|_{L\dot{A}_{p, q}^{\upsilon}}
=\left\|\vec{f}\right\|_{\dot{A}_{p,q}^{s,\upsilon}(W, \varphi)},
\end{align*}
which completes the proof of \eqref{eq-inf-f}
and hence Theorem \ref{thm-A(W)=A(A)}.
\end{proof}

\subsection{Proof of Theorem \ref{thm-phitransMWBTL}}\label{s-pf-pf}

Observe that, in Subsection \ref{s-pf-e},
we obtain $\dot{A}_{p,q}^{s, \upsilon}(W)=
\dot{A}_{p,q}^{s, \upsilon}(\mathbb{A})$ and
$\dot{a}_{p,q}^{s, \upsilon}(W)=\dot{a}_{p,q}^{s, \upsilon}(\mathbb{A})$,
where $\mathbb{A}$ is a sequence of reducing operators
of order $p$ for $W$. If we can establish
the $\varphi$-transform characterization of $\dot{A}_{p,q}^{s,\upsilon}(\mathbb{A})$
for any strongly doubling sequence $\mathbb{A}$,
then Theorem \ref{thm-phitransMWBTL} [that is, the $\varphi$-transform
characterization of $\dot{A}_{p,q}^{s, \upsilon}(W)$] naturally holds.
Based on this idea, we present the following
first main result of this subsection, which
gives the $\varphi$-transform characterization of
$\dot{A}_{p,q}^{s, \upsilon}(\mathbb{A})$.

\begin{theorem}\label{thm-phitansaverMWBTL}
Let $(A, a)\in\{(B, b), (F, f)\}$, $s\in\mathbb{R}$,
and $p, q\in(0,\infty]$ ($p<\infty$ if $A=F$).
Assume that $\delta_1,\delta_2,\omega$ satisfy \eqref{eq-delta1<0},
$\upsilon\in\mathcal{G}(\delta_1, \delta_2; \omega)$, and
$\varphi, \psi\in\mathcal{S}$ satisfy \eqref{cond1}.
Suppose that $\beta_1, \beta_2\in[0,\infty)$ and
$\mathbb{A}$ is strongly doubling of order
$(\beta_1,\beta_2)$. Then the following statements hold.
\begin{itemize}
\item[{\rm (i)}] The maps $S_{\varphi}:\
\dot{A}_{p,q}^{s,\upsilon}(\mathbb{A}, \widetilde{\varphi})
\to\dot{a}_{p,q}^{s,\upsilon}(\mathbb{A})$
and $T_{\psi}:\ \dot{a}_{p,q}^{s,\upsilon}(\mathbb{A})\to
\dot{A}_{p,q}^{s,\upsilon}(\mathbb{A}, \varphi)$ are bounded,
where $\widetilde{\varphi}(x):=\overline{\varphi(-x)}$
for any $x\in\mathbb{R}^n$.
Moreover, if $\varphi, \psi$ further satisfy \eqref{cond3},
then $T_\psi \circ S_{\varphi}$ is the identity
on $\dot{A}_{p,q}^{s,\upsilon}(\mathbb{A}, \widetilde{\varphi})
=\dot{A}_{p,q}^{s,\upsilon}(\mathbb{A}, \varphi)$.
\item[{\rm (ii)}] If $\varphi^{(1)}, \varphi^{(2)}\in\mathcal{S}$
both satisfy \eqref{cond1}, then
$\dot{A}_{p,q}^{s,\upsilon}(\mathbb{A},
\varphi^{(1)})=\dot{A}_{p,q}^{s,\upsilon}(\mathbb{A}, \varphi^{(2)})$
with quasi-norms.
\end{itemize}
\end{theorem}

\begin{remark}
Let all the symbols be the same
as in Theorem \ref{thm-phitansaverMWBTL}.
From Theorem \ref{thm-phitansaverMWBTL}(ii), we infer that
the space $\dot{A}_{p,q}^{s,\upsilon}(\mathbb{A}, \varphi)$
is independent of the choice of $\varphi$.
Hence, we can simply write $\dot{A}_{p,q}^{s,\upsilon}(\mathbb{A})$ instead of $\dot{A}_{p,q}^{s,\upsilon}(\mathbb{A}, \varphi)$.
\end{remark}

Before proving Theorem \ref{thm-phitansaverMWBTL},
we first show that, in Theorem \ref{thm-phitansaverMWBTL},
the operator $T_{\psi}$ is well-defined. To this end,
for any $\phi\in\mathcal{S}$ and $N\in\mathbb{N}$, let
\begin{align}\label{eq-S_N}
\|\phi\|_{\mathcal{S}_N}:=
\sup_{\gamma\in\mathbb{Z}^n_{+},\, |\gamma|\leq N}
\sup_{x\in\mathbb{R}^n}|\partial^{\gamma}
\phi(x)|(1+|x|)^{n+N+|\gamma|},
\end{align}
where, for any multi-index $\gamma:=(\gamma_1,\dots,\gamma_n)
\in\mathbb{Z}^n_{+}$, $|\gamma|:=\sum_{i=1}^{n}|\gamma_i|$.
The following result was established in \cite[Corollary 3.32]{bhyy1}.
\begin{lemma}\label{lem-psi-phi}
Let $\psi,\phi\in\mathcal{S}_{\infty}$.
If $N\in\mathbb{N}$, then, for any $Q, R\in\mathcal{D}$,
\begin{align*}
\left|\langle\psi_Q, \phi_R\rangle\right|
\lesssim\left\|\psi\right\|_{\mathcal{S}_{N+1}}
\left\|\phi\right\|_{\mathcal{S}_{N+1}}\left[\min
\left\{\left[\frac{\ell(R)}{\ell(Q)}\right],
\left[\frac{\ell(Q)}{\ell(R)}\right]\right\}
\right]^{N+\frac{n}{2}}\left[1+
\frac{|x_Q-x_R|}{\ell(Q)\vee\ell(R)}\right]^{-(N+n)},
\end{align*}
where the implicit positive constant is independent of $Q$ and $R$
and where $\|\cdot\|_{\mathcal{S}_{N+1}}$ is as in \eqref{eq-S_N}.
\end{lemma}

The next lemma shows that,
in Theorem \ref{thm-phitansaverMWBTL},
the operator $T_{\psi}$ is well-defined.

\begin{lemma}\label{well-define}
Let $a\in\{b, f\}$, $s\in\mathbb{R}$, and
$p, q\in(0,\infty]$ ($p<\infty$ if $a=f$).
Assume that $\delta_1,\delta_2,\omega$ satisfy \eqref{eq-delta1<0},
$\upsilon\in\mathcal{G}(\delta_1, \delta_2; \omega)$,
and $\psi\in\mathcal{S}$ satisfies \eqref{cond1}.
Suppose that $\beta_1, \beta_2\in[0,\infty)$ and
$\mathbb{A}:=\{A_Q\}_{Q\in\mathcal{D}}$
is strongly doubling of order
$(\beta_1,\beta_2)$. Then, for any $\vec{t}
:=\{\vec{t}_Q\}_{Q\in\mathcal{D}}
\in\dot{a}_{p,q}^{s,\upsilon}(\mathbb{A})$,
$\sum_{Q\in\mathcal{D}}\vec{t}_Q\psi_Q$
converges in $(\mathcal{S}'_{\infty})^m$.
More precisely, if
\begin{align*}
N\in\left(\max\left\{\beta_2-n\delta_1
-\frac{n}{p}-s,\
\beta_1+n\delta_2-\frac{n}{p}+s,\
\beta_1+\beta_2+\omega\right\},\infty
\right)\cap\mathbb{N},
\end{align*}
then there exists a positive constant $C$ such that,
for any $\phi\in\mathcal{S}_{\infty}$,
\begin{align*}
\sum_{Q\in\mathcal{D}}\left|\vec{t}_Q\right|
\left|\langle\psi_Q, \phi\rangle\right|
\leq C\left\|\vec{t}\right\|
_{\dot{a}_{p,q}^{s,\upsilon}(\mathbb{A})}
\left\|\psi\right\|_{\mathcal{S}_{N+1}}
\left\|\phi\right\|_{\mathcal{S}_{N+1}},
\end{align*}
where $\|\cdot\|_{\mathcal{S}_{N+1}}$ is as in \eqref{eq-S_N}.
\end{lemma}

\begin{proof}
By the definition of
$\|\cdot\|_{\dot{a}_{p,q}^{s,\upsilon}(\mathbb{A})}$,
we obtain, for any $Q\in\mathcal{D}$ and $\vec{t}
:=\{\vec{t}_Q\}_{Q\in\mathcal{D}}
\in\dot{a}_{p,q}^{s,\upsilon}(\mathbb{A})$,
\begin{align}\label{eq-welldefine-1}
\left|\vec{t}_Q\right|\leq
\left\|A^{-1}_Q\right\|\left|A_Q\vec{t}_Q\right|
\leq\left\|A^{-1}_Q\right\|
|Q|^{\frac{s}{n}+\frac{1}{2}-\frac{1}{p}}\upsilon(Q)
\left\|\vec{t}\right\|
_{\dot{a}_{p,q}^{s,\upsilon}(\mathbb{A})}.
\end{align}
Notice that, by the growth condition of $\upsilon$,
for any $Q\in\mathcal{D}$,
\begin{align}\label{eq-welldefine-2}
\upsilon(Q)=\upsilon(Q_{0,\mathbf{0}})
\frac{\upsilon(Q)}{\upsilon(Q_{0,\mathbf{0}})}
\lesssim\max\left\{|Q|^{\delta_1},
|Q|^{\delta_2}\right\}\left[1+\frac{|x_Q|}
{\ell(Q)\vee 1}\right]^{\omega}.
\end{align}
Using the assumption that $\mathbb{A}$ is strongly
doubling of order $(\beta_1,\beta_2)$, we conclude that,
for any $Q\in\mathcal{D}$,
\begin{align}\label{eq-welldefine-3}
\left\|A^{-1}_Q\right\|\leq\left\|
A^{-1}_{Q_{0, \mathbf{0}}}\right\|
\left\|A_{Q_{0, \mathbf{0}}}A^{-1}_Q\right\|
\lesssim\max\left\{\left[\ell(Q)\right]^{\beta_1},
\left[\ell(Q)\right]^{-\beta_2}\right\}
\left[1+\frac{|x_Q|}{\ell(Q)\vee 1}
\right]^{\beta_1+\beta_2}.
\end{align}
From \eqref{eq-phi_Q} and Lemma \ref{lem-psi-phi}, it follows that,
for any $Q\in\mathcal{D}$ and $\phi\in\mathcal{S}_{\infty}$,
\begin{align}\label{eq-welldefine-4}
\left|\langle\psi_Q, \phi\rangle\right|
=\left|\langle\psi_Q, \phi_{Q_{0, \mathbf{0}}}
\rangle\right|&\lesssim\left\|\psi\right\|_{\mathcal{S}_{N+1}}
\left\|\phi\right\|_{\mathcal{S}_{N+1}}\left[\min
\left\{[\ell(Q)]^{-1},\ell(Q)\right\}
\right]^{N+\frac{n}{2}}\\
&\quad\times\left[1+\frac{|x_Q|}{\ell(Q)\vee 1}\right]^{-(N+n)}.\nonumber
\end{align}
Let $\theta:=[N+n-(\beta_1+\beta_2)-\omega]\in(n,\infty)$.
Applying the above four estimates
\eqref{eq-welldefine-1}, \eqref{eq-welldefine-2},
\eqref{eq-welldefine-3}, and \eqref{eq-welldefine-4},
we conclude that, for any $\vec{t}
:=\{\vec{t}_Q\}_{Q\in\mathcal{D}}
\in\dot{a}_{p,q}^{s,\upsilon}(\mathbb{A})$
and $\phi\in\mathcal{S}_{\infty}$,
\begin{align}\label{eq-well-1}
\sum_{Q\in\mathcal{D}}\left|\vec{t}_Q\right|
\left|\langle\psi_Q, \phi\rangle\right|
&\lesssim\left\|\vec{t}\right\|
_{\dot{a}_{p,q}^{s,\upsilon}(\mathbb{A})}
\left\|\psi\right\|_{\mathcal{S}_{N+1}}
\left\|\phi\right\|_{\mathcal{S}_{N+1}}\\
&\quad\times\sum_{Q\in\mathcal{D}}
\left[1+\frac{|x_Q|}{\ell(Q)\vee 1}
\right]^{-\theta}
\begin{cases}
[\ell(Q)]^{s+n-\frac{n}{p}+N-\beta_2+n\delta_1}
& \text{if } \ell(Q)\leq1,\\
[\ell(Q)]^{s-\frac{n}{p}-N+\beta_1+n\delta_2}
& \text{if }
\ell(Q)>1.\nonumber
\end{cases}
\end{align}
Denote the summation on the right-hand side
of \eqref{eq-well-1} by $\Omega$. By the fact
$\mathcal{D}=\{Q_{j,k}:\,j\in\mathbb{Z},k\in{\mathbb{Z}}^n\}$
and Lemma \ref{lem-sum-k} combined with
$x=0$ and $\theta\in(n,\infty)$, we find that
\begin{align}\label{eq-well-2}
&\Omega=\sum_{j\in\mathbb{Z}}\sum_{k\in{\mathbb{Z}}^n}
\left[1+\frac{|x_{Q_{j,k}}|}{\ell(Q_{j,k})\vee 1}
\right]^{-\theta}
\begin{cases}
\left[\ell\left(Q_{j,k}\right)\right]^{s+n-\frac{n}{p}+N-\beta_2+n\delta_1}
& \text{if } j\geq0,\\
\left[\ell\left(Q_{j,k}\right)\right]^{s-\frac{n}{p}-N+\beta_1+n\delta_2}
& \text{if }
j<0\\
\end{cases}\\
&\quad=\sum_{j=0}^{\infty}
2^{-j(s+n-\frac{n}{p}+N-\beta_2+n\delta_1)}
\sum_{k\in{\mathbb{Z}}^n}\left(1+2^{-j}|k|
\right)^{-\theta}+\sum_{j=-\infty}^{-1}
2^{-j(s-\frac{n}{p}
-N+\beta_1+n\delta_2)}\sum_{k\in{\mathbb{Z}}^n}
(1+|k|)^{-\theta}\nonumber\\
&\quad\sim\sum_{j=0}^{\infty}2^{-j(s-\frac{n}{p}+N
-\beta_2+n\delta_1)}+\sum_{j=-\infty}^{-1}
2^{-j(s-\frac{n}{p}-N+\beta_1+n\delta_2)}\sim1,\nonumber
\end{align}
where the last equivalence follows from the choice of $N$
in the present lemma, which guarantees that both two summations
in the penultimate equivalence converge.
From \eqref{eq-well-1} and \eqref{eq-well-2},
we deduce that, for any $\vec{t}
:=\{\vec{t}_Q\}_{Q\in\mathcal{D}}
\in\dot{a}_{p,q}^{s,\upsilon}(\mathbb{A})$
and $\phi\in\mathcal{S}_{\infty}$,
\begin{align*}
\sum_{Q\in\mathcal{D}}\left|\vec{t}_Q\right|
\left|\langle\psi_Q, \phi\rangle\right|
\lesssim\left\|\vec{t}\right\|
_{\dot{a}_{p,q}^{s,\upsilon}(\mathbb{A})}
\left\|\psi\right\|_{\mathcal{S}_{N+1}}
\left\|\phi\right\|_{\mathcal{S}_{N+1}}.
\end{align*}
This finishes the proof of Lemma \ref{well-define}.
\end{proof}

We also need the following Calder\'{o}n reproducing
formulae which can be found in \cite[Lemma 2.1]{yy10}.

\begin{lemma}\label{lem-Cdreproform}
Let $\varphi, \psi\in\mathcal{S}$ satisfy \eqref{cond3}
such that both $\operatorname{supp}\widehat{\varphi}$
and $\operatorname{supp}\widehat{\psi}$
are compact and bounded away from the origin.
Then, for any $f\in\mathcal{S}_\infty$,
\begin{align}\label{eq-Cdreproform}
f=\sum_{j\in\mathbb{Z}}2^{-jn}\sum_{k\in\mathbb{Z}^n}
\left(\widetilde{\varphi}_j*f\right)\left(2^{-j}k\right)
\psi_j\left(\cdot-2^{-j}k\right)
=\sum_{Q\in\mathcal{D}}\left\langle f,\varphi_Q\right\rangle\psi_Q
\end{align}
in $\mathcal{S}_\infty$, where $\widetilde{\varphi}(x):=\overline{\varphi(-x)}$
for any $x\in\mathbb{R}^n$. Moreover, for any $f\in\mathcal{S}'_\infty$,
\eqref{eq-Cdreproform} also holds in $\mathcal{S}'_\infty$.
\end{lemma}

We next recall the estimate established in \cite[Lemma 2.2]{yy08}.
\begin{lemma}\label{lem-yy08}
For any $M \in \mathbb{N}$, there exists a positive constant $C$,
depending on $M$ and $n$, such that, for any $\varphi$,
$\psi\in\mathcal{S}_{\infty}$, $i, j\in\mathbb{Z}$, and $x\in\mathbb{R}^n$,
\begin{align*}
\left|\varphi_j * \psi_i(x)\right|\leq
C\|\varphi\|_{\mathcal{S}_{M+1}}
\|\psi\|_{\mathcal{\mathcal{S}}_{M+1}}
\frac{2^{-(i\vee j) M}}{[2^{-(i\wedge j)}+|x|]^{n+M}},
\end{align*}
where $\|\cdot\|_{\mathcal{\mathcal{S}}_{M+1}}$ is as in \eqref{eq-S_N}
and $\varphi_j(x):=2^{jn} \varphi(2^j x)$ and
$\psi_i(x):=2^{i n}\psi(2^i x)$ for any $x\in\mathbb{R}^n$.
\end{lemma}

We now give the proof of Theorem \ref{thm-phitansaverMWBTL}.

\begin{proof}[Proof of Theorem \ref{thm-phitansaverMWBTL}]
Let $\mathbb{A}:=\{A_Q\}_{Q\in\mathcal{D}}$
be strongly doubling of order $(\beta_1,\beta_2)$.
To prove (i), we first establish the boundedness of $S_{\varphi}:\
\dot{A}_{p,q}^{s,\upsilon}(\mathbb{A}, \widetilde{\varphi})
\to\dot{a}_{p,q}^{s,\upsilon}(\mathbb{A})$.
Using the definition of $S_{\varphi}$,
\cite[Theorem 2.3.20]{gra14a}, and \eqref{supf},
we conclude that, for any $Q\in\mathcal{D}$ and
$\vec{f}\in\dot{A}_{p,q}^{s,\upsilon}
(\mathbb{A},\widetilde{\varphi})$,
\begin{align*}
\left|A_Q\left(S_{\varphi}\vec{f}\right)_Q\right|=
\left|A_Q\left\langle\vec{f},\varphi_Q\right
\rangle\right|=|Q|^{\frac{1}{2}}\left|A_Q
\left(\widetilde{\varphi}_{j_Q}*\vec{f}\right)(x_Q)
\right|\leq\sup_{\mathbb{A},
\widetilde{\varphi}, Q}\left(\vec{f}\right),
\end{align*}
From this, \eqref{eq-a(A)-a},
and Lemma \ref{A(A)=a}, it follows that,
for any $\vec{f}\in\dot{A}_{p,q}^{s,\upsilon}
(\mathbb{A},\widetilde{\varphi})$,
\begin{align*}
\left\|\left\{S_{\varphi}\vec{f}\right\}_{Q\in\mathcal{D}}\right\|
_{\dot{a}_{p,q}^{s,\upsilon}(\mathbb{A})}
=\left\|\left\{\left|A_Q\left(S_{\varphi}\vec{f}\right)_Q
\right|\right\}_{Q\in\mathcal{D}}\right\|
_{\dot{a}_{p,q}^{s,\upsilon}}\leq\left\|\sup_{\mathbb{A},
\widetilde{\varphi}}\left(\vec{f}\right)\right\|
_{\dot{a}_{p,q}^{s,\upsilon}}\sim
\left\|\vec{f}\right\|_{\dot{A}_{p,q}^{s,\upsilon}
(\mathbb{A}, \widetilde{\varphi})},
\end{align*}
which implies the boundedness of $S_{\varphi}$.

Next, we show the boundedness of $T_{\psi}:\
\dot{a}_{p,q}^{s,\upsilon}(\mathbb{A})\to
\dot{A}_{p,q}^{s,\upsilon}(\mathbb{A}, \varphi)$.
To this end,  by the definition of $T_\psi$ and
Lemma \ref{well-define}, we find that,
for any $\vec{t}:=\{\vec{t}_R\}_{R\in\mathcal{D}}
\in\dot{a}_{p,q}^{s,\upsilon}(\mathbb{A})$,
$T_\psi\vec{t}=\sum_{R\in\mathcal{D}}\vec{t}_R\psi_{R}
\in(\mathcal{S}'_{\infty})^m$. Applying this, \eqref{eq-phi_Q},
and the fact that $\varphi, \psi\in\mathcal{S}$
satisfy \eqref{cond1}, we conclude that,
for any $j\in\mathbb{Z}$ and $\vec{t}
:=\{\vec{t}_R\}_{R\in\mathcal{D}}
\in\dot{a}_{p,q}^{s,\upsilon}(\mathbb{A})$,
\begin{align}\label{eq-g-j}
g_j:&=\sum_{Q\in\mathcal{D}_j}\mathbf{1}_Q
\left|A_Q\left[\varphi_j*\left(T_\psi\vec{t}
\right)\right]\right|\\
&=\sum_{Q\in\mathcal{D}_j}\mathbf{1}_Q
\left|A_Q\left(\varphi_j*\sum_{R\in\mathcal{D}}
\vec{t}_R\psi_{R}\right)\right|
=\sum_{Q\in\mathcal{D}_j}\mathbf{1}_Q\left|\sum_{i=j-1}^{j+1}
\sum_{R\in\mathcal{D}_i} A_Q\vec{t}_R
\left(\varphi_j*\psi_R\right)\right|\nonumber\\
&\leq\sum_{Q\in\mathcal{D}_j}\mathbf{1}_Q\sum_{i=j-1}^{j+1}
\sum_{R\in\mathcal{D}_i}\left\|A_Q A_R^{-1}\right\|
\left|A_R\vec{t}_R\right|\left|\left(\varphi_j *\psi_R\right)\right|\nonumber.
\end{align}
From the assumption that $\mathbb{A}$
is strongly doubling of order $(\beta_1,\beta_2)$
[and hence satisfies \eqref{eq-st-double}],
we infer that, for any $j\in\mathbb{Z}$,
$i\in\{j-1,j,j+1\}$, $Q\in\mathcal{D}_j$,
and $R\in\mathcal{D}_i$, $\ell(Q)\sim\ell(R)$ and hence
\begin{align}\label{eq-AQAr}
\left\|A_Q A_R^{-1}\right\|
&\lesssim\max\left\{
\left[\frac{\ell(R)}{\ell(Q)}\right]^{\beta_1},
\left[\frac{\ell(Q)}{\ell(R)}\right]^{\beta_2}\right\}
\left[1+\frac{|x_Q-x_R|}{\ell(R)\vee\ell(Q)}\right]^{\beta_1+\beta_2}\\
&\sim\left\{1+[\ell(R)]^{-1}|x_Q-x_R|\right\}^{\beta_1+\beta_2}.\nonumber
\end{align}
Let $M\in(\frac{n}{1\wedge\Gamma_{p,q}}-n+\omega+\beta_1+\beta_2,
\infty)\cap\mathbb{N}$.
By \eqref{eq-phi_Q} and Lemma \ref{lem-yy08},
we find that, for any $j\in\mathbb{Z}$, $i\in\{j-1,j,j+1\}$,
$R\in\mathcal{D}_i$, and $x\in\mathbb{R}^n$,
\begin{align}\label{eq-phi-R}
\left|\left(\varphi_j * \psi_R\right)(x)\right|
&=|R|^{\frac{1}{2}}\left|\left(\varphi_j * \psi_i\right)\left(x-x_R\right)\right|
\lesssim|R|^{\frac{1}{2}}\left\|\varphi\right\|_{\mathcal{S}_{M+1}}
\left\|\psi\right\|_{\mathcal{S}_{M+1}}\frac{ 2^{-(i \vee j) M}}
{[2^{-(i \wedge j)}+|x-x_R|]^{n+M}}\\
&\sim|R|^{-\frac{1}{2}}\frac{1}
{\{1+[\ell(R)]^{-1}|x-x_R|\}^{n+M}}\nonumber.
\end{align}
Let $\eta:=n+M-\beta_1-\beta_2$. Combining \eqref{eq-AQAr}
and \eqref{eq-phi-R} with \eqref{eq-g-j} and then applying
Lemma \ref{lem-xy-QR-P}(i) together with $y=x_R$
and with the fact that $\ell(Q)\sim \ell(R)$,
we obtain, for any $j\in\mathbb{Z}$,
$\vec{t}:=\{\vec{t}_R\}_{R\in\mathcal{D}}
\in\dot{a}_{p,q}^{s,\upsilon}(\mathbb{A})$, and $x\in\mathbb{R}^n$,
\begin{align*}
g_j(x)&\lesssim\sum_{Q\in\mathcal{D}_j}\mathbf{1}_Q(x)
\sum_{i=j-1}^{j+1} 2^{\frac{in}{2}}\sum_{R\in\mathcal{D}_i}
\frac{\{1+[\ell(R)]^{-1}|x_Q-x_R|
\}^{\beta_1+\beta_2}}{\{1+[\ell(R)]^{-1}|x-x_R|\}^{n+M}}
\left|A_R\vec{t}_R\right|\\
&\sim\sum_{Q\in\mathcal{D}_j}\mathbf{1}_Q(x)
\sum_{i=j-1}^{j+1}2^{\frac{in}{2}}
\sum_{R\in\mathcal{D}_i}\frac{|A_R\vec{t}_R|}
{\{1+[\ell(R)]^{-1}|x-x_R|\}^{\eta}},\nonumber
\end{align*}	
which, together with Lemma \ref{lem-3P-sum1}(i),
Lemma \ref{lem-xy-QR-P}(ii) with $y=x_R$, and \eqref{eq-lamda*},
further implies that, for any $x\in\mathbb{R}^n$ and $i\in\mathbb{Z}$,
there exists a unique $Q(x,i)\in\mathcal{D}_i$ such that $x\in Q(x,i)$
and hence $1+[\ell(R)]^{-1}|x-x_R|\sim1+[\ell(R)]^{-1}|x_{Q(x,i)}-x_R|$
for any $R\in\mathcal{D}_i$ and consequently
\begin{align}\label{eq-g-lambda}
g_j(x)&\lesssim\sum_{i=j-1}^{j+1} 2^{\frac{in}{2}}
\sum_{R\in\mathcal{D}_i}\frac{|A_R\vec{t}_R|}
{\{1+[\ell(R)]^{-1}|x_{Q(x,i)}-x_R|\}^{\eta}}
=\sum_{i=j-1}^{j+1} 2^{\frac{in}{2}}
|t|^{*}_{1, \eta,Q(x,i)}\\
&=\sum_{i=j-1}^{j+1}\sum_{Q\in\mathcal{D}_{i}}
\widetilde{\mathbf{1}}_Q(x)
|t|^{*}_{1, \eta,Q}
=\sum_{i=-1}^{1}\sum_{Q\in\mathcal{D}_{j+i}}
\widetilde{\mathbf{1}}_Q(x)
|t|^{*}_{1, \eta,Q},\nonumber
\end{align}
where $|t|:=\{|A_R\vec{t}_R|\}_{R\in\mathcal{D}}$
and, for any $Q\in\mathcal{D}$,
$|t|^{*}_{1, \eta,Q}$ is as in \eqref{eq-lamda*}.
For any $\vec{t}:=\{\vec{t}_R\}_{R\in\mathcal{D}}
\in\dot{a}_{p,q}^{s,\upsilon}(\mathbb{A})$,
using \eqref{eq-g-lambda}, the definitions
of $\|\cdot\|_{L\dot{A}_{p, q}^{\upsilon}}, |t|$,
and $\|\cdot\|_{\dot{a}_{p,q}^{s,\upsilon}}$,
the quasi-triangle inequality of
$\|\cdot\|_{L\dot{A}_{p, q}^{\upsilon}}$,
the growth condition of $\upsilon$, Proposition \ref{prop-dct-gh}
with the assumption $\eta\in(\frac{n}
{1\wedge\Gamma_{p,q}}+\omega, \infty)$ and with $t$
replaced by $|t|$, and \eqref{eq-a(A)-a},
we conclude that
\begin{align*}
\left\|T_\psi\vec{t}\right\|
_{\dot{A}_{p,q}^{s,\upsilon}(\mathbb{A}, \varphi)}
:&=\left\|\left\{2^{js}g_j\right\}_{j\in\mathbb{Z}}
\right\|_{L\dot{A}_{p, q}^{\upsilon}}
\lesssim\left\|\left\{2^{js}\sum_{i=-1}^{1}
\sum_{Q\in\mathcal{D}_{j+i}}\widetilde{\mathbf{1}}_Q
|t|^{*}_{1, \eta,Q}\right\}_{j\in\mathbb{Z}}
\right\|_{L\dot{A}_{p, q}^{\upsilon}}\\
&\lesssim\sum_{i=-1}^{1}\left\|\left\{2^{js}
\sum_{Q\in\mathcal{D}_{j+i}}\widetilde{\mathbf{1}}_Q
|t|^{*}_{1, \eta, Q}\right\}_{j\in\mathbb{Z}}
\right\|_{L\dot{A}_{p, q}^{\upsilon}}\sim\left\|\left\{2^{js}
\sum_{Q\in\mathcal{D}_{j}}\widetilde{\mathbf{1}}_Q
|t|^{*}_{1, \eta,Q}\right\}_{j\in\mathbb{Z}}\right\|_{L\dot{A}_{p, q}^{\upsilon}}\\
&=\left\|\,\left|t\right|^*_{1, \eta}\,\right\|_{\dot{a}_{p,q}^{s,\upsilon}}
\sim\left\|\,\left|t\right|\,\right\|_{\dot{a}_{p,q}^{s,\upsilon}}
=\left\|\left\{\left|A_Q\vec{t}_Q\right|\right\}_{Q\in\mathcal{D}}\right\|_{\dot{a}_{p,q}^{s,\upsilon}}
=\left\|\vec{t}\right\|_{\dot{a}_{p,q}^{s,\upsilon}(\mathbb{A})},			
\end{align*}
which implies the boundedness of $T_{\psi}$.
Moreover, if $\varphi, \psi$ further satisfy \eqref{cond3},
from Lemma \ref{lem-Cdreproform}, it follows that
$T_\psi \circ S_{\varphi}$ is the identity on
$\mathcal{S}'_\infty$. Applying the just proved boundedness of
both $S_{\varphi}$ and $T_\psi$, we obtain the identity
$T_\psi \circ S_{\varphi}$ is bounded from
$\dot{A}_{p,q}^{s,\upsilon}(\mathbb{A}, \widetilde{\varphi})$
to $\dot{A}_{p,q}^{s,\upsilon}(\mathbb{A}, \varphi)$.
By the symmetry, the identity
$T_\psi \circ S_{\varphi}$ is also bounded from
$\dot{A}_{p,q}^{s,\upsilon}(\mathbb{A}, \varphi)$
to $\dot{A}_{p,q}^{s,\upsilon}(\mathbb{A}, \widetilde{\varphi})$.
Based on these, we conclude that $\dot{A}_{p,q}^{s,\upsilon}
(\mathbb{A}, \widetilde{\varphi})
=\dot{A}_{p,q}^{s,\upsilon}(\mathbb{A}, \varphi)$
with equivalent quasi-norms, which completes the proof of (i).

Finally, we prove (ii). For any $i\in\{1, 2\}$, there exists $\psi^{(i)}\in\mathcal{S}$ satisfying \eqref{cond1}
such that $\varphi^{(i)}, \psi^{(i)}$
satisfy \eqref{cond3} (see, for example, \cite[Lemma (6.9)]{fjw91}). From this, Lemma \ref{lem-Cdreproform},
and the just proved (i), we deduce that,
for any $\vec{f}\in({\mathcal{S}}'_{\infty})^m$,
\begin{align*}
\left\|\vec{f}\right\|_{\dot{A}_{p,q}^{s,\upsilon}
(\mathbb{A}, \varphi^{(1)})}=\left\|T_{\psi^{(2)}}
\circ S_{\varphi^{(2)}}\vec{f}\right\|
_{\dot{A}_{p,q}^{s,\upsilon}
(\mathbb{A}, \varphi^{(1)})}\lesssim\left\|
S_{\varphi^{(2)}}\vec{f}\right\|
_{\dot{a}_{p,q}^{s,\upsilon}
(\mathbb{A})}\lesssim\left\|\vec{f}\right\|
_{\dot{A}_{p,q}^{s,\upsilon}
(\mathbb{A}, \widetilde{\varphi^{(2)}})}
\sim\left\|\vec{f}\right\|
_{\dot{A}_{p,q}^{s,\upsilon}
(\mathbb{A}, \varphi^{(2)})}.
\end{align*}
By the symmetry, we also obtain
$\|\vec{f}\|_{\dot{A}_{p,q}^{s,\upsilon}
(\mathbb{A}, \varphi^{(2)})}\lesssim
\|\vec{f}\|_{\dot{A}_{p,q}^{s,\upsilon}
(\mathbb{A}, \varphi^{(1)})}$.
This finishes the proof of (ii) and hence
Theorem \ref{thm-phitansaverMWBTL}.
\end{proof}

The following conclusion can be proved by using
Theorem \ref{thm-phitansaverMWBTL}
and a standard argument (see, for example,
\cite[Proposition 3.13]{syy24}); we omit the details.

\begin{corollary}\label{cor-completeA(A)}
Let $A\in\{B, F\}$, $s\in\mathbb{R}$, and
$p, q\in(0,\infty]$ ($p<\infty$ if $A=F$).
Assume that $\delta_1,\delta_2,\omega$ satisfy \eqref{eq-delta1<0} and
$\upsilon\in\mathcal{G}(\delta_1, \delta_2; \omega)$.
Suppose that $\beta_1, \beta_2\in[0,\infty)$ and
$\mathbb{A}:=\{A_Q\}_{Q\in\mathcal{D}}$
is strongly doubling of order $(\beta_1,\beta_2)$.
Then $\dot{A}_{p,q}^{s,\upsilon}(\mathbb{A})$
equipped with $\|\cdot\|_{\dot{A}_{p,q}^{s,\upsilon}(\mathbb{A})}$
is a complete quasi-normed space.	
\end{corollary}

Finally, we show Theorem \ref{thm-phitransMWBTL}.
\begin{proof}[Proof of Theorem \ref{thm-phitransMWBTL}]
We first prove (i). To do this,
let $\mathbb{A}:=\{A_Q\}_{Q\in\mathcal{D}}$
be a sequence of reducing operators of order $p$ for $W$.
Next, we show the boundedness of $S_{\varphi}:\
\dot{A}_{p,q}^{s,\upsilon}(W, \widetilde{\varphi})
\to\dot{a}_{p,q}^{s,\upsilon}(W)$.
By Theorems \ref{thm-phitansaverMWBTL}(i) and
\ref{thm-A(W)=A(A)} and Corollary \ref{cor-a(A)=a(W)},
we find that, for any $\vec{f}\in\dot{A}_{p,q}^{s,\upsilon}
(W,\widetilde{\varphi})$,
\begin{align*}
\left\|\left\{S_{\varphi}\vec{f}\right\}_{Q\in\mathcal{D}}\right\|
_{\dot{a}_{p,q}^{s,\upsilon}(W)}\sim\left\|
\left\{S_{\varphi}\vec{f}\right\}_{Q\in\mathcal{D}}\right\|
_{\dot{a}_{p,q}^{s,\upsilon}(\mathbb{A})}
\lesssim\left\|\vec{f}\right\|_{\dot{A}_{p,q}^{s,\upsilon}
(\mathbb{A},\widetilde{\varphi})}\sim\left\|\vec{f}\right\|
_{\dot{A}_{p,q}^{s,\upsilon}(W, \widetilde{\varphi})},
\end{align*}
which implies the boundedness of $S_{\varphi}$.
Then we prove the boundedness of
$T_{\psi}:\ \dot{a}_{p,q}^{s,\upsilon}(W)\to
\dot{A}_{p,q}^{s,\upsilon}(W, \varphi)$.
Applying Theorems \ref{thm-phitansaverMWBTL}(i)
and \ref{thm-A(W)=A(A)} and Corollary \ref{cor-a(A)=a(W)} again,
we conclude that, for any $\vec{t}\in\dot{a}
_{p,q}^{s,\upsilon}(W)$,
\begin{align*}
\left\|T_\psi\vec{t}\right\|
_{\dot{A}_{p,q}^{s,\upsilon}(W, \varphi)}
\sim\left\|T_\psi\vec{t}\right\|
_{\dot{A}_{p,q}^{s,\upsilon}(\mathbb{A}, \varphi)}
\lesssim\left\|\vec{t}\right\|_{\dot{a}_{p,q}^{s,\upsilon}
(\mathbb{A})}\sim\left\|\vec{t}\right\|
_{\dot{a}_{p,q}^{s,\upsilon}(W)},			
\end{align*}
which establishes the boundedness of $T_{\psi}$.
Moreover, if $\varphi, \psi$ further satisfy \eqref{cond3},
from Lemma \ref{lem-Cdreproform} and Theorems \ref{thm-phitansaverMWBTL}(i)
and \ref{thm-A(W)=A(A)}, it follows that $T_\psi \circ S_{\varphi}$
is the identity on $\dot{A}_{p,q}^{s,\upsilon}(W, \widetilde{\varphi})
=\dot{A}_{p,q}^{s,\upsilon}(\mathbb{A}, \widetilde{\varphi})
=\dot{A}_{p,q}^{s,\upsilon}(\mathbb{A}, \varphi)
=\dot{A}_{p,q}^{s,\upsilon}(W, \varphi)$, which completes the proof of (i).
We next show (ii). From Theorems \ref{thm-A(W)=A(A)} and
\ref{thm-phitansaverMWBTL}(ii), we infer that
$\dot{A}_{p,q}^{s,\upsilon}(W, \varphi^{(1)})=\dot{A}_{p,q}^{s,\upsilon}(\mathbb{A}, \varphi^{(1)})=\dot{A}_{p,q}^{s,\upsilon}(\mathbb{A}, \varphi^{(2)})=\dot{A}_{p,q}^{s,\upsilon}(W, \varphi^{(2)})$
all with equivalent quasi-norms. This finishes the proof of (ii) and
hence Theorem \ref{thm-phitransMWBTL}.
\end{proof}

The following result can be proved by Corollary
\ref{cor-completeA(A)} and Theorem \ref{thm-A(W)=A(A)};
we omit the details.

\begin{corollary}
Let $A\in\{B, F\}$, $s\in\mathbb{R}$, $p\in(0, \infty)$,
$q\in(0,\infty]$, and $W\in\mathcal{A}_{p,\infty}$.
Assume that $\delta_1,\delta_2,\omega$ satisfy
\eqref{eq-delta1<0} and
$\upsilon\in\mathcal{G}(\delta_1, \delta_2; \omega)$.
Then $\dot{A}_{p,q}^{s,\upsilon}(W)$
equipped with $\|\cdot\|_{\dot{A}_{p,q}^{s,\upsilon}(W)}$
is a complete quasi-normed space.	
\end{corollary}

\section{Peetre-Type Maximal Function
and \\ Littlewood--Paley Function Characterizations
of $\dot{A}^{s,\upsilon}_{p,q}(W)$\label{s-ec}}

In this section, we aim to obtain equivalent characterizations
of $\dot{A}^{s,\upsilon}_{p,q}(W)$, respectively,
in terms of the Peetre-type maximal functions (Subsection \ref{s-ec-pee})
and the Littlewood--Paley functions (Subsection \ref{s-ec-g}).
To this end, we make full use of the discrete
Littlewood--Paley $g^*_{\lambda}$-function
characterization of $\dot{a}^{s,\upsilon}_{p,q}$ in
Proposition \ref{prop-dct-gh}.

\subsection{Peetre-Type Maximal Function Characterization\label{s-ec-pee}}

To establish the Peetre-type maximal function
characterization of $\dot{A}^{s,\upsilon}_{p,q}(W)$,
we first present two Peetre-type maximal functions
in the matrix-weighted setting
introduced in \cite[(3.1)]{wyy23} and \cite[(3.4)]{lyy24a}.
Let $p,\eta\in(0,\infty)$, $\varphi\in\mathcal{S}_{\infty}$,
$W\in\mathcal{A}_{p,\infty}$, and $\mathbb{A}:=\{A_Q\}
_{Q\in\mathcal{D}}$ be a sequence of positive
definite matrices. For any $j\in\mathbb{Z}$ and
$\vec{f}\in(\mathcal{S}'_{\infty})^m$,
the \emph{matrix-weighted Peetre-type maximal function}
$(\varphi_j^*\vec{f})_{W,p,\eta}$
and the \emph{averaging Peetre-type maximal function}
$(\varphi_j^*\vec{f})_{\mathbb{A},\eta}$ are defined,
respectively, by setting, for any $x\in\mathbb{R}^n$,
\begin{align}\label{eq-def-pee}
&\left(\varphi_j^* \vec{f}\right)_{W,p,\eta}(x):=
\sup_{y\in\mathbb{R}^n}\frac{|W^{\frac{1}{p}}(x)
(\varphi_j *\vec{f})(y)|}{(1+2^j|x-y|)^{\eta}}\ \text{\ and\ }\\
&\left(\varphi_j^* \vec{f}\right)_{\mathbb{A},\eta}(x)
:=\sum_{Q\in\mathcal{D}_j}\mathbf{1}_Q(x)\sup_{y\in\mathbb{R}^n}
\frac{|A_Q(\varphi_j*\vec{f})(y)|}{(1+2^j|x-y|)^{\eta}}.\nonumber
\end{align}

The following lemma gives a useful
equality related to averaging Peetre-type maximal functions.
\begin{lemma}\label{lem-phi-A}
Let $\eta\in(0,\infty)$, $\varphi\in\mathcal{S}_{\infty}$,
and $\mathbb{A}:=\{A_Q\}
_{Q\in\mathcal{D}}$ be a sequence of positive
definite matrices. For any $j\in\mathbb{Z}$,
$\vec{f}\in(\mathcal{S}'_{\infty})^m$, and $x\in\mathbb{R}^n$,
\begin{align*}
\left(\varphi_j^*\vec{f}\right)_{\mathbb{A},\eta}(x)
\sim\sum_{Q\in\mathcal{D}_j}\mathbf{1}_Q(x)
\left(\varphi_j^*\vec{f}\right)_{\mathbb{A},\eta}(x_Q),
\end{align*}
where the positive equivalence constants are independent
of $j,\vec{f}$, and $x$.
\end{lemma}
\begin{proof}
By \eqref{eq-def-pee} and by Lemma \ref{lem-xy-QR-P}(ii) twice,
we find that, for any $j\in\mathbb{Z}$,
$\vec{f}\in(\mathcal{S}'_{\infty})^m$, and $x\in\mathbb{R}^n$,
\begin{align*}
\left(\varphi_j^*\vec{f}\right)_{\mathbb{A},\eta}(x)
\sim\sum_{Q\in\mathcal{D}_j}\mathbf{1}_Q(x)
\sup_{y\in\mathbb{R}^n} \frac{|A_Q(\varphi_j
*\vec{f})(y)|}{(1+2^j|x_Q-y|)^{\eta}}
=\sum_{Q\in\mathcal{D}_j}\mathbf{1}_Q(x)
\left(\varphi_j^*\vec{f}\right)_{\mathbb{A},\eta}(x_Q),
\end{align*}
which completes the proof of Lemma \ref{lem-phi-A}.
\end{proof}

In the following definition, we first recall
the concepts of doubling matrix weights
and doubling exponents (see, for example, \cite[p.\,1230]{fr04}).
Moreover, we also introduce the concept of weakly doubling exponents.

\begin{definition}\label{def-double-W}
Let $p\in(0, \infty)$.
\begin{itemize}
\item[{\rm (i)}] A matrix weight $W$ is
called a \emph{doubling matrix weight of order $p$}
if there exists a positive constant $C$ such that,
for any cube $Q\subset\mathbb{R}^n$ and any $\vec{z}\in\mathbb{C}^m$,
\begin{align}\label{eq-double}
\int_{2Q}\left|W^{\frac{1}{p}}(x)\vec{z}\right|^p\, dx
\leq C \int_{Q}\left|W^{\frac{1}{p}}(x)\vec{z}\right|^p\, dx.
\end{align}
The \emph{doubling exponent} $\beta_{p}(W)$ of
$W$ of order $p$ is defined by setting
\begin{align*}
\beta_{p}(W):=\min\left\{\beta\in(0, \infty):\
\eqref{eq-double}\text{\ holds\ with\ }C=2^{\beta}\right\}.
\end{align*}
\item[{\rm (ii)}] Let $W\in \mathcal{A}_{p,\infty}$ and
$\mathbb{A}$ be a sequence of reducing operators
of order $p$ for $W$. The \emph{weakly doubling exponent}
$\alpha_{p}(W)$ of $W$ of order $p$
is defined by setting
\begin{align}\label{eq-beta-p-weak}
\alpha_{p}(W):=\inf\left\{\beta\in[0, \infty):\
\mathbb{A}{\rm\ is\ weakly\ doubling\ of\ order}\ \beta\right\}.
\end{align}
\end{itemize}
\end{definition}

\begin{remark}
In Definition \ref{def-double-W}(ii),
using \cite[Proposition 6.1]{bhyy4}, we conclude that
\eqref{eq-beta-p-weak} is independent of the choice of
$\mathbb{A}$ and hence well-defined.
\end{remark}

We now establish the Peetre-type maximal
function characterization
of $\dot{A}^{s,\upsilon}_{p,q}(W)$.

\begin{theorem}\label{thm-Pee-cha}
Let $A\in\{B, F\}$, $s\in\mathbb{R}$,
$p\in(0, \infty)$, $q\in(0,\infty]$,
$W\in \mathcal{A}_{p,\infty}$, and
$\mathbb{A}:=\{A_Q\}_{Q\in\mathcal{D}}$
be a sequence of reducing operators of order $p$ for $W$.
Suppose that $\delta_1,\delta_2,\omega$ satisfy \eqref{eq-delta1<0},
$\upsilon\in\mathcal{G}(\delta_1, \delta_2; \omega)$,
and $\varphi\in\mathcal{S}$ satisfies \eqref{cond1}.
If $\eta\in(\frac{n}{\Gamma_{p,q}}+\alpha_{p}(W)+
[\omega\wedge n(\delta_2-\frac{1}{p})_+], \infty)$,
where $\Gamma_{p,q}$ and $\alpha_{p}(W)$
are as, respectively, in \eqref{eq-zeta} and \eqref{eq-beta-p-weak},
then, for any $\vec{f}\in(\mathcal{S}'_{\infty})^m$,
\begin{align*}
\left\|\vec{f}\right\|_{\dot{A}^{s,\upsilon}_{p,q}(W)}
\sim\left\|\left\{2^{js}\left(\varphi_j^* \vec{f}
\right)_{W,p,\eta}\right\}_{j\in\mathbb{Z}}
\right\|_{L\dot{A}_{p, q}^{\upsilon}}
\sim\left\|\left\{2^{js}\left(\varphi_j^* \vec{f}\right)
_{\mathbb{A},\eta}\right\}_{j\in\mathbb{Z}}
\right\|_{L\dot{A}_{p, q}^{\upsilon}},
\end{align*}
where all the positive equivalence constants are
independent of $\vec{f}$.
\end{theorem}

\begin{proof}
To prove the present theorem, it suffices to show,
for any $\vec{f}\in(\mathcal{S}'_{\infty})^m$,
\begin{align}\label{eq-thm-pee-cha}
\left\|\vec{f}\right\|_{\dot{A}^{s,\upsilon}_{p,q}(W)}
\lesssim\left\|\left\{2^{js}\left(\varphi_j^* \vec{f}
\right)_{W,p,\eta}\right\}_{j\in\mathbb{Z}}
\right\|_{L\dot{A}_{p, q}^{\upsilon}}
\lesssim\left\|\left\{2^{js}\left(\varphi_j^* \vec{f}\right)
_{\mathbb{A},\eta}\right\}_{j\in\mathbb{Z}}
\right\|_{L\dot{A}_{p, q}^{\upsilon}}
\lesssim\left\|\vec{f}\right\|_{\dot{A}^{s,\upsilon}_{p,q}(W)}.
\end{align}
We begin with proving the first inequality in \eqref{eq-thm-pee-cha}.
Using \eqref{eq-def-pee} and the definition of
$\|\cdot\|_{\dot{A}^{s,\upsilon}_{p,q}(W)}$, we conclude that,
for any $j\in\mathbb{Z}$, $\vec{f}\in(\mathcal{S}'_{\infty})^m$,
and $x\in\mathbb{R}^n$,
\begin{align*}
\left|W^{\frac{1}{p}}(x)\left(\varphi_j *\vec{f}\right)(x)\right|
\leq\sup_{y\in\mathbb{R}^n}\frac{|W^{\frac{1}{p}}(x)
(\varphi_j *\vec{f})(y)|}{(1+2^j|x-y|)^{\eta}}
=\left(\varphi_j^* \vec{f}\right)_{W,p,\eta}(x)
\end{align*}
and hence
\begin{align*}
\left\|\vec{f}\right\|_{\dot{A}^{s,\upsilon}_{p,q}(W)}
=\left\|\left\{2^{js}\left|W^{\frac{1}{p}}\left(\varphi_j *\vec{f}
\right)\right|\right\}_{j\in\mathbb{Z}}
\right\|_{L\dot{A}_{p, q}^{\upsilon}}
\leq\left\|\left\{2^{js}\left(\varphi_j^*
\vec{f}\right)_{W,p,\eta}\right\}_{j\in\mathbb{Z}}
\right\|_{L\dot{A}_{p, q}^{\upsilon}},
\end{align*}
which implies that the first inequality holds.

Next, we prove the second inequality in \eqref{eq-thm-pee-cha}.
Applying \eqref{eq-def-pee}, Lemma \ref{lem-3P-sum1}(ii),
the definition of $\gamma_j$ [see \eqref{eq-gamma_j}],
and Lemma \ref{lem-phi-A}, we obtain, for any $j\in\mathbb{Z}$,
$\vec{f}\in(\mathcal{S}'_{\infty})^m$, and $x\in\mathbb{R}^n$,
\begin{align}\label{eq-W*-gamma}
\left(\varphi_j^* \vec{f}\right)_{W,p,\eta}(x)
&=\sum_{Q\in\mathcal{D}_j}\mathbf{1}_{Q}(x)\sup_{y\in\mathbb{R}^n}
\frac{|W^{\frac{1}{p}}(x)A_{Q}^{-1}A_{Q}
(\varphi_j*\vec{f})(y)|}{(1+2^j|x-y|)^{\eta}}\\
&\leq\sum_{Q\in\mathcal{D}_j}\mathbf{1}_{Q}(x)\left\|W^{\frac{1}{p}}(x)
A_{Q}^{-1}\right\|\sup_{y\in\mathbb{R}^n}
\frac{|A_Q(\varphi_j*\vec{f})(y)|}
{(1+2^j|x-y|)^{\eta}}\nonumber\\
&=\gamma_j(x)\left(\varphi_j^{*}\vec{f}\right)_{\mathbb{A},\eta}(x)\sim
\gamma_j(x)\sum_{Q\in\mathcal{D}_j}\mathbf{1}_Q(x)
\left(\varphi_j^*\vec{f}\right)_{\mathbb{A},\eta}(x_Q).\nonumber
\end{align}
For any $\vec{f}\in(\mathcal{S}'_{\infty})^m$,
from \eqref{eq-W*-gamma},
Lemma \ref{lem-E_j} with $\{t_Q\}_{Q\in\mathcal{D}}$
replaced by $\{2^{j_Q(s-\frac{n}{2})}(\varphi_{j_Q}^*\vec{f}
)_{\mathbb{A},\eta}(x_Q)\}_{Q\in\mathcal{D}}$,
and Lemma \ref{lem-phi-A} again, we deduce that
\begin{align*}
\left\|\left\{2^{js}\left(\varphi_j^*
\vec{f}\right)_{W,p,\eta}\right\}_{j\in\mathbb{Z}}
\right\|_{L\dot{A}_{p, q}^{\upsilon}}
&\lesssim\left\|\left\{2^{js}\gamma_j\sum_{Q\in\mathcal{D}_j}\mathbf{1}_Q
\left(\varphi_j^*\vec{f}\right)_{\mathbb{A},\eta}(x_Q)
\right\}_{j\in\mathbb{Z}}\right\|_{L\dot{A}_{p, q}^{\upsilon}}\\
&\lesssim\left\|\left\{2^{js}\sum_{Q\in\mathcal{D}_j}\mathbf{1}_Q
\left(\varphi_j^*\vec{f}\right)_{\mathbb{A},\eta}(x_Q)
\right\}_{j\in\mathbb{Z}}\right\|_{L\dot{A}_{p, q}^{\upsilon}}
\sim\left\|\left\{2^{js}\left(\varphi_j^* \vec{f}\right)
_{\mathbb{A},\eta}\right\}_{j\in\mathbb{Z}}
\right\|_{L\dot{A}_{p, q}^{\upsilon}}.
\end{align*}
This finishes the proof of the second inequality.

Finally, we show the last inequality in \eqref{eq-thm-pee-cha}.
Based on Theorem \ref{thm-A(W)=A(A)},
to prove the last inequality in \eqref{eq-thm-pee-cha},
we only need to show, for any $\vec{f}\in(\mathcal{S}'_{\infty})^m$,
\begin{align}\label{eq-phi-f}
\left\|\left\{2^{js}\left(\varphi_j^* \vec{f}\right)
_{\mathbb{A},\eta}\right\}_{j\in\mathbb{Z}}
\right\|_{L\dot{A}_{p, q}^{\upsilon}}
\lesssim\left\|\vec{f}\right\|_{\dot{A}^{s,\upsilon}
_{p,q}(\mathbb{A})}.
\end{align}
To do this, by the assumption that
$\eta\in(\frac{n}{\Gamma_{p,q}}+\alpha_{p}(W)+
[\omega\wedge n(\delta_2-\frac{1}{p})_+],\infty)$,
we pick $\beta\in(\alpha_{p}(W),\infty)$
such that $\eta\in(\frac{n}{\Gamma_{p,q}}+\beta+
[\omega\wedge n(\delta_2-\frac{1}{p})_+],\infty)$.
Applying the fact that $\beta\in(\alpha_{p}(W),\infty)$
and \eqref{eq-beta-p-weak}, we obtain $\mathbb{A}$
is weakly doubling of order $\beta$. Using this,
Lemma \ref{lem-phi-A}, \eqref{eq-lamda*}, \eqref{supf},
and Lemma \ref{lem-xy-QR-P}(ii) with $x=x_Q$,
we conclude that, for any $j\in\mathbb{Z}$, $Q\in\mathcal{D}_j$,
$\vec{f}\in(\mathcal{S}'_{\infty})^m$, and $x\in Q$,	
\begin{align}\label{eq-phi-sup}
\left(\varphi_j^* \vec{f}\right)_{\mathbb{A},\eta}(x)
&\sim\sup_{y\in\mathbb{R}^n}\frac{|A_{Q}A^{-1}_{R}A_{R}(\varphi_j
*\vec{f})(y)|}{(1+2^j|x_Q-y|)^{\eta}}\leq\sup_{R\in\mathcal{D}_j}
\sup_{y\in R}\frac{\|A_{Q}A^{-1}_{R}\|\,|A_{R}(\varphi_j*\vec{f})(y)|}
{(1+2^j|x_Q-y|)^{\eta}}\\
&\lesssim2^{\frac{jn}{2}}\sup_{R\in\mathcal{D}_j}
\frac{\sup_{\mathbb{A},\varphi,Q}
(\vec{f})}{(1+2^j|x_Q-x_R|)^{\eta-\beta}}
=2^{\frac{jn}{2}}\left[\sup_{\mathbb{A}, \varphi}
\left(\vec{f}\right)\right]_{\infty, (\eta-\beta),Q}^{*},\nonumber
\end{align}
where $\sup_{\mathbb{A}, \varphi}(\vec{f})$
is as in \eqref{supf}.
For any $\vec{f}\in(\mathcal{S}'_{\infty})^m$,	
by \eqref{eq-phi-sup}, the definition of
$\|\cdot\|_{\dot{a}_{p,q}^{s,\upsilon}}$,
Proposition \ref{prop-dct-gh} combined with
$\eta-\beta\in(\frac{n}{\Gamma_{p,q}}+
[\omega\wedge n(\delta_2-\frac{1}{p})_+],\infty)$
and with $t$ replaced by $\sup_{\mathbb{A}, \varphi}(\vec{f})$,
and Lemma \ref{A(A)=a}, we find that
\begin{align*}
\left\|\left\{2^{js}\left(\varphi_j^* \vec{f}\right)
_{\mathbb{A},\eta}\right\}_{j\in\mathbb{Z}}\right\|
_{L\dot{A}_{p, q}^{\upsilon}}&\lesssim
\left\|\left\{2^{js}\sum_{Q\in\mathcal{D}_j}
\widetilde{\mathbf{1}}_{Q}\left[\sup_{\mathbb{A},\varphi}
\left(\vec{f}\right)\right]_{\infty,(\eta-\beta),Q}^{*}
\right\}_{j\in\mathbb{Z}}\right\|
_{L\dot{A}_{p, q}^{\upsilon}}\\
&=\left\|\left[\sup_{\mathbb{A}, \varphi}
\left(\vec{f}\right)\right]_{\infty, (\eta-\beta)}^{*}\right\|
_{\dot{a}_{p,q}^{s,\upsilon}}\sim
\left\|\sup_{\mathbb{A}, \varphi}\left(\vec{f}\right)\right\|
_{\dot{a}_{p,q}^{s,\upsilon}}\sim
\left\|\vec{f}\right\|_{\dot{A}^{s,\upsilon}_{p,q}(\mathbb{A})}.
\end{align*}
This finishes the proof of \eqref{eq-phi-f}
and hence Theorem \ref{thm-Pee-cha}.
\end{proof}

\begin{remark}\label{rmk-pee}
The proof of Theorem \ref{thm-Pee-cha} depends on the
use of averaging spaces and Proposition \ref{prop-dct-gh}.
To the best of our knowledge,
even in the scalar-valued setting, this approach is new.
Recall that, using the rescaled maximal operator
$[\mathcal{M}(|\cdot|^r)]^{\frac{1}{r}}$ for some $r\in(0, \infty)$
and the weighted Fefferman--Stein vector-valued inequality,
where $\mathcal{M}$ is as in \eqref{eq-HL},
Bui \cite[Theorem 2.2]{bui82}  obtained the
Peetre-type maximal function characterization of $\dot{A}^s_{p,q}(w)$ with the scalar weight $w\in A_{\infty}$,
where $A\in\{B, F\}$, $s\in\mathbb{R}$,
$p\in(0, \infty)$, and $q\in(0, \infty]$.
Very recently, Kakaroumpas and Soler i Gibert \cite{ks24}
established the matrix-weighted Fefferman--Stein
vector-valued maximal inequality for
any $p,q\in(1,\infty)$ and $W\in\mathcal{A}_p$.
By this and Calder\'{o}n's reproducing formulae,
one can give another proof of Theorem \ref{thm-Pee-cha}
for any $p,q\in(1,\infty)$ and $W\in\mathcal{A}_p$.
However, in the matrix-weighted setting,
since matrix weights and vector-valued functions are inseparable,
the appropriate substitute of the rescaled maximal operator
$[\mathcal{M}(|\cdot|^r)]^{\frac{1}{r}}$
for some $r\in(0, \infty)$ is still unavailable.
Based on this reason, one can not use
the matrix-weighted Fefferman--Stein
vector-valued maximal inequality to deal with the general case
where $p,q\in(0,\infty)$ and $W\in \mathcal{A}_{p,\infty}$,
as in Theorem \ref{thm-Pee-cha}. Therefore,
the new approach used in the proof of Theorem \ref{thm-Pee-cha} seems necessary.
\end{remark}

\subsection{Littlewood--Paley Function Characterization\label{s-ec-g}}	

We  begin with giving the definitions of  the Lusin area function sequence
and the Littlewood--Paley $g^{*}_{\lambda}$-function  sequence
in the matrix-weighted setting.

\begin{definition}\label{def-LPfunc}
Let $p, \alpha, \lambda\in(0,\infty)$,
$r\in(0,\infty]$, $\varphi\in\mathcal{S}_{\infty}$,
and $W\in \mathcal{A}_{p,\infty}$.
For any $\vec{f}\in(\mathcal{S}'_{\infty})^m$,
the \emph{matrix-weighted Lusin area function sequence
$S_{W, p,\varphi, \alpha}^{r}(\vec{f})$} on $\mathbb{R}^n$
and the \emph{matrix-weighted Littlewood--Paley
$g^{*}_{\lambda}$-function  sequence $g_{W, p, \varphi, r ,\lambda}^{*}(\vec{f})$}
on $\mathbb{R}^n$  are defined, respectively, by setting
\begin{align}\label{eq-LAF}
S_{W, p,\varphi, \alpha}^{r}\left(\vec{f}\right):=
\left\{\left[S_{W, p,\varphi, \alpha}^{r}\left(\vec{f}\right)\right]
_j(\cdot)\right\}_{j\in\mathbb{Z}}
:=\left\{\left[\fint_{B(\cdot,\alpha2^{-j})}
\left|W^{\frac{1}{p}}(\cdot)\left(\varphi_j
*\vec{f}\right)(y)\right|^r\,dy\right]^{\frac{1}{r}}
\right\}_{j\in\mathbb{Z}}
\end{align}
and
\begin{align}\label{eq-LPg}
g_{W, p, \varphi, r ,\lambda}^{*}\left(\vec{f}\right):=
\left\{\left[g_{W, p, \varphi, r ,\lambda}^{*}
\left(\vec{f}\right)\right]_j(\cdot)\right\}_{j\in\mathbb{Z}}
:=\left\{\left[\int_{\mathbb{R}^n}
\frac{2^{jn}|W^{\frac{1}{p}}(\cdot)
(\varphi_j *\vec{f})(y)|^r}{(1+2^j|\cdot-y|
)^{\lambda r}}\,dy\right]^{\frac{1}{r}}
\right\}_{j\in\mathbb{Z}}
\end{align}
(with the usual modification made if $r=\infty$).
\end{definition}

We next establish the Littlewood--Paley function characterization
of $\dot{A}^{s,\upsilon}_{p,q}(W)$.
\begin{theorem}\label{thm-G-L-cha}
Let $A\in\{B, F\}$, $s\in\mathbb{R}$,
$p, \alpha\in(0, \infty)$, $q,r\in(0,\infty]$,
and $W\in\mathcal{A}_{p,\infty}$.
Assume that $\delta_1,\delta_2,\omega$ satisfy \eqref{eq-delta1<0},
$\upsilon\in\mathcal{G}(\delta_1, \delta_2; \omega)$,
and $\varphi\in\mathcal{S}$ satisfies \eqref{cond1}.
If $\lambda\in(\frac{n}{r\wedge\Gamma_{p,q}}
+\alpha_{p}(W)+[\omega\wedge n
(\delta_2-\frac{1}{p})_+], \infty)$,
where $\Gamma_{p,q}$ and $\alpha_{p}(W)$
are as, respectively, in \eqref{eq-zeta} and \eqref{eq-beta-p-weak},
then, for any $\vec{f}\in(\mathcal{S}'_{\infty})^m$,
\begin{align*}
\left\|\vec{f}\right\|_{\dot{A}^{s,\upsilon}_{p,q}(W)}
\sim\left\|\left\{2^{js}\left[S_{W, p,\varphi, \alpha}^{r}
\left(\vec{f}\right)\right]_j\right\}_{j\in\mathbb{Z}}
\right\|_{L\dot{A}_{p, q}^{\upsilon}}
\sim\left\|\left\{2^{js}\left[g_{W, p, \varphi, r ,\lambda}^{*}
\left(\vec{f}\right)\right]_j\right\}_{j\in\mathbb{Z}}
\right\|_{L\dot{A}_{p, q}^{\upsilon}},
\end{align*}
where all the positive equivalence constants
are independent of $\vec{f}$.
\end{theorem}

\begin{proof}
To prove the present theorem, it suffices to show that,
for any $\vec{f}\in(\mathcal{S}'_{\infty})^m$,
\begin{align}\label{eq-thm-G-L-cha}
\left\|\left\{2^{js}\left[S_{W,p,\varphi,\alpha}^{r}
\left(\vec{f}\right)\right]_j\right\}_{j\in\mathbb{Z}}
\right\|_{L\dot{A}_{p, q}^{\upsilon}}
&\lesssim\left\|\left\{2^{js}\left[g_{W, p, \varphi, r ,\lambda}^{*}
\left(\vec{f}\right)\right]_j\right\}_{j\in\mathbb{Z}}
\right\|_{L\dot{A}_{p, q}^{\upsilon}}\\
&\lesssim\left\|\vec{f}\right\|_{\dot{A}^{s,\upsilon}_{p,q}(W)}
\lesssim\left\|\left\{2^{js}\left[S_{W,p,\varphi, \alpha}^{r}
\left(\vec{f}\right)\right]_j\right\}_{j\in\mathbb{Z}}
\right\|_{L\dot{A}_{p, q}^{\upsilon}}.\nonumber
\end{align}
We begin with proving the first inequality in \eqref{eq-thm-G-L-cha}.
Notice that, for any $j\in\mathbb{Z}$, $x\in\mathbb{R}^n$,
and $y\in B(x,\alpha2^{-j})$, $1+2^j|x-y|<1+\alpha$
and $|B(x,\alpha2^{-j})|\sim2^{-jn}$.
By this, \eqref{eq-LAF}, and \eqref{eq-LPg}, we find that,
for any $j\in\mathbb{Z}$,
$\vec{f}\in(\mathcal{S}'_{\infty})^m$, and $x\in\mathbb{R}^n$,
\begin{align*}
\left[S_{W, p, \varphi, \alpha}^{r}
\left(\vec{f}\right)\right]_j(x)
&\lesssim\left[\int_{B(x,\alpha2^{-j})}\frac{2^{jn}|W^{\frac{1}{p}}(x)
(\varphi_j *\vec{f})(y)|^r}{(1+2^j|x-y|)^{\lambda r}}
\,dy\right]^{\frac{1}{r}}\\
&\leq\left[\int_{\mathbb{R}^n}\frac{2^{jn}|W^{\frac{1}{p}}(x)
(\varphi_j *\vec{f})(y)|^r}{(1+2^j|x-y|)^{\lambda r}}\,dy\right]^{\frac{1}{r}}
=\left[g_{W, p, \varphi, r ,\lambda}^{*}\left(\vec{f}\right)\right]_j(x)
\end{align*}	
and hence
\begin{align*}
\left\|\left\{2^{js}\left[S_{W, p, \varphi, \alpha}^{r}
\left(\vec{f}\right)\right]_j\right\}_{j\in\mathbb{Z}}
\right\|_{L\dot{A}_{p, q}^{\upsilon}}
&\lesssim\left\|\left\{2^{js}
\left[g_{W, p, \varphi, r ,\lambda}^{*}
\left(\vec{f}\right)\right]_j\right\}_{j\in\mathbb{Z}}
\right\|_{L\dot{A}_{p, q}^{\upsilon}},
\end{align*}
which implies that the first inequality in \eqref{eq-thm-G-L-cha} holds.

Next, we prove the second inequality in \eqref{eq-thm-G-L-cha}.
To this end, let $\mathbb{A}:=\{A_Q\}_{Q\in\mathcal{D}}$ be
a sequence of reducing operators of order $p$ for $W$.
Applying Lemma \ref{A(A)=a}, Remark \ref{rem-equi-seq},
and Corollary \ref{cor-a(A)=a(W)},
to show the second inequality in \eqref{eq-thm-G-L-cha},
we only need to prove that, for any $\vec{f}\in(\mathcal{S}'_{\infty})^m$,
\begin{align}\label{eq-g-sup}
\left\|\left\{2^{js}\left[g_{W, p, \varphi, r ,\lambda}^{*}
\left(\vec{f}\right)\right]_j\right\}_{j\in\mathbb{Z}}
\right\|_{L\dot{A}_{p, q}^{\upsilon}}
\lesssim\left\|\sup_{\mathbb{A}, \varphi}
\left(\vec{f}\right)\right\|_{\dot{a}_{p,q}^{s,\upsilon}},
\end{align}	
where $\sup_{\mathbb{A}, \varphi}(\vec{f})$
is as in \eqref{supf}. By the assumption
$\lambda\in(\frac{n}{r\wedge\Gamma_{p,q}}
+\alpha_{p}(W)+[\omega\wedge n
(\delta_2-\frac{1}{p})_+],\infty)$, we
can pick $\beta\in(\alpha_{p}(W),\infty)$
such that $\lambda\in(\frac{n}{r\wedge\Gamma_{p,q}}
+\beta+[\omega\wedge n(\delta_2-\frac{1}{p})_+],\infty)$.
From $\beta\in(\alpha_{p}(W),\infty)$ and \eqref{eq-beta-p-weak},
it follows that $\mathbb{A}$ is weakly doubling of order $\beta$.
To obtain \eqref{eq-g-sup}, using \eqref{eq-LPg} and
Lemma \ref{lem-3P-sum1}(i), we conclude that, for any $j\in\mathbb{Z}$,
$Q\in\mathcal{D}_j$, $\vec{f}\in(\mathcal{S}'_{\infty})^m$, and $x\in Q$,
\begin{align*}
\left[g_{W,p,\varphi, r,\lambda}^{*}
\left(\vec{f}\right)\right]_j(x)
&=\left[\int_{\mathbb{R}^n}\frac{2^{jn}
|W^{\frac{1}{p}}(x)(\varphi_j *\vec{f})(y)|^r}
{(1+2^j|x-y|)^{\lambda r}}\,dy\right]^{\frac{1}{r}}\\
&=\left[\sum_{R\in\mathcal{D}_j}
\int_{R}\frac{2^{jn}|W^{\frac{1}{p}}(x)A_{Q}^{-1}A_Q
A^{-1}_{R}A_R(\varphi_j *\vec{f})(y)|^r}
{(1+2^j|x-y|)^{\lambda r}}\,dy\right]^{\frac{1}{r}}\\
&\leq\left\|W^{\frac{1}{p}}(x)A_{Q}^{-1}\right\|
\left[\sum_{R\in\mathcal{D}_j}\int_{R}\frac{2^{jn}
\|A_QA^{-1}_{R}\|^{r}|A_R(\varphi_j *\vec{f})(y)|^r}
{(1+2^j|x-y|)^{\lambda r}}\,dy\right]^{\frac{1}{r}},
\end{align*}
which, together with the just proved conclusion that $\mathbb{A}$
is weakly doubling of order $\beta$,
the definition of $\gamma_j$ [see \eqref{eq-gamma_j}], Lemma \ref{lem-xy-QR-P}(ii), \eqref{supf}, and \eqref{eq-lamda*}
with $\lambda$ replaced by $\sup_{\mathbb{A}, \varphi}(\vec{f})$,
further implies that
\begin{align}\label{eq-g-sup-1}
\left[g_{W,p,\varphi, r,\lambda}^{*}\left(\vec{f}\right)\right]_j(x)
&\lesssim\gamma_j(x)\left[\sum_{R\in\mathcal{D}_j}
\frac{\sup_{y\in R}|A_{R}(\varphi_j*\vec{f})(y)|^r}
{(1+[\ell(R)]^{-1}|x_Q-x_R|)^{(\lambda-\beta)r}}\right]^{\frac{1}{r}}\\
&=\gamma_j(x)2^{\frac{jn}{2}}\left[\sum_{R\in\mathcal{D}_j}
\frac{|\sup_{\mathbb{A}, \varphi, R}(\vec{f})|^{r}}
{(1+[\ell(R)]^{-1}|x_Q-x_R|)^{(\lambda-\beta)r}}
\right]^{\frac{1}{r}}\nonumber\\
&=\gamma_j(x)2^{\frac{jn}{2}}\left[\sup_{\mathbb{A}, \varphi}
\left(\vec{f}\right)\right]_{r, (\lambda-\beta), Q}^{*}.\nonumber
\end{align}

Applying \eqref{eq-g-sup-1}, Lemma \ref{lem-E_j}
with $\{t_Q\}_{Q\in\mathcal{D}}$ replaced by
$\{2^{j_Q s}[\sup_{\mathbb{A}, \varphi}(\vec{f})]_{r,
(\lambda-\beta), Q}^{*}\}_{Q\in\mathcal{D}}$,
the definition of $\|\cdot\|_{\dot{a}_{p,q}^{s,\upsilon}}$,
and Proposition \ref{prop-dct-gh} combined with
$\lambda-\beta\in(\frac{n}{r\wedge\Gamma_{p,q}}+
[\omega\wedge n(\delta_2-\frac{1}{p})_+],\infty)$
and with $t$ replaced by
$\sup_{\mathbb{A}, \varphi}(\vec{f})$, we obtain,
for any $\vec{f}\in(\mathcal{S}'_{\infty})^m$,
\begin{align*}
\left\|\left\{2^{js}\left[g_{W,p,\varphi, r,\lambda}^{*}
\left(\vec{f}\right)\right]_j\right\}_{j\in\mathbb{Z}}
\right\|_{L\dot{A}_{p, q}^{\upsilon}}
&\lesssim\left\|\left\{2^{js}\gamma_j\sum_{Q\in\mathcal{D}_j}
\widetilde{\mathbf{1}}_{Q}\left[\sup_{\mathbb{A}, \varphi}
\left(\vec{f}\right)\right]_{r, (\lambda-\beta), Q}^{*}\right\}_{j\in\mathbb{Z}}
\right\|_{L\dot{A}_{p, q}^{\upsilon}}\\
&\lesssim\left\|\left\{2^{js}\sum_{Q\in\mathcal{D}_j}
\widetilde{\mathbf{1}}_{Q}\left[\sup_{\mathbb{A}, \varphi}
\left(\vec{f}\right)\right]_{r, (\lambda-\beta), Q}^{*}
\right\}_{j\in\mathbb{Z}}\right\|_{L\dot{A}_{p, q}^{\upsilon}}\\
&=\left\|\left[\sup_{\mathbb{A}, \varphi}
\left(\vec{f}\right)\right]_{r, (\lambda-\beta)}^{*}\right\|
_{\dot{a}_{p,q}^{s,\upsilon}}\sim
\left\|{\sup_{\mathbb{A}, \varphi}
\left(\vec{f}\right)}\right\|
_{\dot{a}_{p,q}^{s,\upsilon}},
\end{align*}
which completes the proof of \eqref{eq-g-sup} and
hence the second inequality in \eqref{eq-thm-G-L-cha}.

Finally, we prove the last inequality in \eqref{eq-thm-G-L-cha}.
To this end, using Theorem \ref{thm-A(W)=A(A)}, we only need to show
that, for any $\vec{f}\in(\mathcal{S}'_{\infty})^m$,
\begin{align}\label{eq-f-S}
\left\|\vec{f}\right\|_{\dot{A}^{s,\upsilon}_{p,q}(\mathbb{A})}
\lesssim\left\|\left\{2^{js}\left[S_{W, p, \varphi, \alpha}^{r}
\left(\vec{f}\right)\right]_j\right\}_{j\in\mathbb{Z}}
\right\|_{L\dot{A}_{p, q}^{\upsilon}}.
\end{align}
By \cite[Proposition 4.1]{bhyy4},
we find that there exists some $u\in(0,\infty)$ such that
\begin{align}\label{eq-u}
\sup_{Q\in\mathcal{D}}\fint_{Q}
\left\|A_Q W^{-\frac{1}{p}}(x)\right\|^{u}dx<\infty.
\end{align}
Pick $\gamma\in(0, \min\{u, r\})$,
$\lambda\in(0, \infty)$, and $\eta\in\mathbb{N}$
such that $\frac{\gamma u}{u-\gamma}\in(0, \Gamma_{p,q})$,
$\lambda\in(n+\frac{\gamma u}{u-\gamma}\omega,\infty)$,
and $\alpha\in(\sqrt{n}2^{-\eta},\infty)$.
From Lemma \ref{lem-3P-sum1}(ii),
the definition of $\mathbb{A}_j$ [see \eqref{A_j}],
Lemma \ref{lem-suptosum}, and Lemma \ref{lem-xy-QR-P}(ii)
with $y=x_R$, we infer that, for any $j\in\mathbb{Z}$,
$\vec{f}\in(\mathcal{S}'_{\infty})^m$, and $x\in\mathbb{R}^n$,
\begin{align}\label{eq-Ajgamma}
\left|\mathbb{A}_j(x)\left(\varphi_j*\vec{f}\right)(x)\right|^{\gamma}
&=\sum_{Q\in\mathcal{D}_j}\mathbf{1}_Q(x)
\left|\mathbb{A}_j(x)\left(\varphi_j*\vec{f}\right)(x)\right|^{\gamma}
=\sum_{Q\in\mathcal{D}_j}\mathbf{1}_Q(x)
\left|A_Q\left(\varphi_j*\vec{f}\right)(x)\right|^{\gamma}\\
&\lesssim\sum_{Q\in\mathcal{D}_j}\mathbf{1}_Q(x)\sum_{R\in\mathcal{D}_j}
\frac{2^{jn}}{(1+2^{j}|x_Q-x_R|)^{\lambda}}\int_{R}
\left|A_{R}\left(\varphi_j*\vec{f}\right)(y)\right|^{\gamma}\,dy\nonumber\\
&\sim\sum_{R\in\mathcal{D}_j}\frac{2^{jn}}{(1+2^{j}|x-x_R|)^{\lambda}}
\int_{R}\left|A_{R}\left(\varphi_j*\vec{f}\right)(y)\right|^{\gamma}\,dy\nonumber.
\end{align}
By \eqref{eq-Ajgamma} and Lemmas \ref{lem-3P-sum1}(iv)
and \ref{growEST} and by Lemma \ref{lem-xy-QR-P}(i) twice, we find that,
for any $j\in\mathbb{Z}$, $\vec{f}\in(\mathcal{S}'_{\infty})^m$,
and $x\in\mathbb{R}^n$,
\begin{align}\label{eq-Ajgamma-1}
\left|\mathbb{A}_j(x)\left(\varphi_j*\vec{f}\right)(x)\right|^{\gamma}
&\lesssim\sum_{R\in\mathcal{D}_j}\sum_{P\in\mathcal{D}_{j+\eta},
P\subset R}\frac{2^{jn}}{(1+2^{j}|x-x_R|)^{\lambda}}\int_{P}
\left|A_{R}A^{-1}_{P}A_{P}\left(\varphi_j*
\vec{f}\right)(y)\right|^{\gamma}\,dy\\
&\leq\sum_{R\in\mathcal{D}_j}\sum_{P\in\mathcal{D}_{j+\eta},
P\subset R}\frac{2^{jn}\|A_{R}A^{-1}_{P}
\|}{(1+2^{j}|x-x_R|)^{\lambda}}\int_{P}
\left|A_{P}\left(\varphi_j*\vec{f}\right)(y)\right|^{\gamma}\,dy\nonumber\\
&\lesssim\sum_{P\in\mathcal{D}_{j+\eta}}\frac{2^{jn}}
{(1+2^{j}|x-x_P|)^{\lambda}}\int_{P}
\left|A_{P}\left(\varphi_j*\vec{f}\right)(y)\right|^{\gamma}\,dy.\nonumber
\end{align}
Applying the choice $\alpha\in(\sqrt{n}2^{-\eta},\infty)$,
we obtain, for any $j\in\mathbb{Z}$, $P\in\mathcal{D}_{j+\eta}$,
and $z\in P$, $P\subset B(z, \alpha2^{-j})$
and $|P|\sim|B(z, \alpha2^{-j})|$. Using this,
we conclude that, for any $j\in\mathbb{Z}$,
$\vec{f}\in(\mathcal{S}'_{\infty})^m$,
and $P\in\mathcal{D}_{j+\eta}$,
\begin{align*}
\int_{P}\left|A_{P}\left(\varphi_j*\vec{f}\right)(y)\right|^{\gamma}\,dy
&=\fint_{P}\int_{P}\left|A_{P}W^{-\frac{1}{p}}(z)
W^{\frac{1}{p}}(z)\left(\varphi_j*\vec{f}\right)(y)
\right|^{\gamma}\,dy\,dz\\
&\leq\fint_{P}\int_{P}\left\|A_PW^{-\frac{1}{p}}(z)
\right\|^{\gamma}\left|W^{\frac{1}{p}}(z)\left(\varphi_j*\vec{f}
\right)(y)\right|^{\gamma}\,dy\,dz\nonumber\\
&=\int_{P}\left\|A_PW^{-\frac{1}{p}}(z)\right\|^{\gamma}
\fint_{P}\left|W^{\frac{1}{p}}(z)\left(\varphi_j*\vec{f}
\right)(y)\right|^{\gamma}\,dy\,dz\nonumber\\
&\lesssim\int_{P}\left\|A_PW^{-\frac{1}{p}}(z)\right\|^{\gamma}
\fint_{B(z, \alpha2^{-j})}\left|W^{\frac{1}{p}}(z)\left(\varphi_j*\vec{f}
\right)(y)\right|^{\gamma}\,dy\,dz,\nonumber
\end{align*}
which, together with H\"{o}lder's inequality, \eqref{eq-LAF},
and \eqref{eq-u}, further implies that
\begin{align}\label{eq-A-S-1}
\int_{P}\left|A_{P}\left(\varphi_j*\vec{f}\right)(y)
\right|^{\gamma}\,dy&\lesssim\left[\int_{P}\left\|A_{P}
W^{-\frac{1}{p}}(z)\right\|^{u}\,dz\right]^{\frac{\gamma}{u}}\\
&\quad\times\left\{\int_{P}\left[\fint_{B(z, \alpha2^{-j})}
\left|W^{\frac{1}{p}}(z)\left(\varphi_j*\vec{f}\right)(y)
\right|^{\gamma}\,dy\right]^{\frac{u}{u-\gamma}}\,dz
\right\}^{\frac{u-\gamma}{u}}\nonumber\\
&\lesssim2^{-jn\frac{\gamma}{u}}
\left(\int_{P}\left\{\left[S_{W, p, \varphi, \alpha}^{r}
\left(\vec{f}\right)\right]_j(z)\right\}^{\frac{\gamma u}{u-\gamma}}
\,dz\right)^{\frac{u-\gamma}{u}}.\nonumber
\end{align}
If $\lambda\in(n,\infty)$, by Lemma \ref{lem-sum-k}
with $j=0$, we find that,
for any $j\in\mathbb{Z}$ and $x\in\mathbb{R}^n$,
\begin{align}\label{eq-sum-P}
\sum_{P\in\mathcal{D}_{j+\eta}}
\frac{1}{(1+2^{j}|x-x_P|)^{\lambda}}
\sim\sum_{P\in\mathcal{D}_{j+\eta}}
\frac{1}{(1+2^{j+\eta}|x-x_P|)^{\lambda}}
&=\sum_{k\in\mathbb{Z}^n}
\frac{1}{(1+|2^{j+\eta}x-k|)^{\lambda}}\sim1.
\end{align}
Combining \eqref{eq-A-S-1} with \eqref{eq-Ajgamma-1}
and applying H\"{o}lder's inequality, \eqref{eq-sum-P}
combined with the assumption $\lambda\in(n+\frac{\gamma u}{u-\gamma}\omega,
\infty)$,  Lemma \ref{lem-xy-QR-P}(ii) twice,
and Lemma \ref{lem-3P-sum1}(i) with $j$ replaced by $j+\eta$,
we conclude that, for any $j\in\mathbb{Z}$, $\vec{f}\in(\mathcal{S}'_{\infty})^m$, and $x\in \mathbb{R}^n$,
\begin{align*}
\left|\mathbb{A}_j(x)\left(\varphi_j*\vec{f}\right)(x)\right|^{\gamma}
&\lesssim\sum_{P\in\mathcal{D}_{j+\eta}}
\frac{2^{jn\frac{u-\gamma}{u}}}{{(1+2^{j}|x-x_P|)^{\lambda}}}
\left(\int_{P}\left\{\left[S_{W, p, \varphi, \alpha}^{r}\left(\vec{f}\right)\right]_j(y)
\right\}^{\frac{\gamma u}{u-\gamma}}
\,dy\right)^{\frac{u-\gamma}{u}}\\
&\leq\left[\sum_{P\in\mathcal{D}_{j+\eta}}
\frac{1}{(1+2^{j}|x-x_P|)^{\lambda}}\right]^{\frac{\gamma}{u}}\\
&\quad\times\left[\sum_{P\in\mathcal{D}_{j+\eta}}
\frac{2^{jn}}{(1+2^{j}|x-x_P|)^{\lambda}}
\int_{P}\left\{\left[S_{W, p, \varphi, \alpha}^{r}
\left(\vec{f}\right)\right]_j(y)\right\}^{\frac{\gamma u}
{u-\gamma}}\,dy\right]^{\frac{u-\gamma}{u}}\\
&\sim\left[\sum_{P\in\mathcal{D}_{j+\eta}}\int_{P}
\frac{2^{jn}}{(1+2^{j}|x-y|)^{\lambda}}
\left\{\left[S_{W, p, \varphi, \alpha}^{r}
\left(\vec{f}\right)\right]_j(y)\right\}^{\frac{\gamma u}{u-\gamma}}
\,dy\right]^{\frac{u-\gamma}{u}}\\
&=\left[\int_{\mathbb{R}^n}\frac{2^{jn}}
{(1+2^j|x-y|)^{\lambda}}\left\{\left[S_{W, p, \varphi, \alpha}^{r}
\left(\vec{f}\right)\right]_j(y)
\right\}^{\frac{\gamma u}{u-\gamma}}
\,dy\right]^{\frac{u-\gamma}{u}}.
\end{align*}
For any $\vec{f}\in(\mathcal{S}'_{\infty})^m$, from
the definition of $\|\cdot\|_{\dot{A}^{s,\upsilon}_{p,q}(\mathbb{A})}$,
the above estimate, and Proposition \ref{prop-gh} combined with
assumptions $\frac{\gamma u}{u-\gamma}\in(0, \Gamma_{p,q})$ and
$\lambda\in(n+\frac{\gamma u}{u-\gamma}\omega,\infty)$
and with $\{g_j\}_{j\in\mathbb{Z}}$ and $\{h_j\}_{j\in\mathbb{Z}}$
replaced, respectively, by $\{2^{js}|\mathbb{A}_j(\varphi_j*\vec{f})|\}_{j\in\mathbb{Z}}$
and $\{2^{js}[S_{W, p, \varphi, \alpha}^{r}
(\vec{f})]_j\}_{j\in\mathbb{Z}}$, we infer that
\begin{align*}
\left\|\vec{f}\right\|_{\dot{A}^{s,\upsilon}_{p,q}(\mathbb{A})}
=\left\|\left\{2^{js}\left|\mathbb{A}_j\left(\varphi_j*\vec{f}\right)\right|
\right\}_{j\in\mathbb{Z}}\right\|_{L\dot{A}_{p, q}^{\upsilon}}
\lesssim\left\|\left\{2^{js}\left[S_{W, p, \varphi, \alpha}^{r}
\left(\vec{f}\right)\right]_j\right\}_{j\in\mathbb{Z}}
\right\|_{L\dot{A}_{p, q}^{\upsilon}},
\end{align*}
which completes the proof of the last inequality
in \eqref{eq-f-S} and hence Theorem \ref{thm-G-L-cha}.
\end{proof}

\section{Molecular and Wavelet Characterizations
of $\dot{A}_{p,q}^{s, \upsilon}(W)$\label{s-cmw}}

This section contains two subsections.
In Subsection \ref{s-cmw-ad}, we aim to establish the
the boundedness of almost diagonal operators on
$\dot{a}_{p,q}^{s,\upsilon}(W)$. It is well known that
the molecular and the wavelet characterizations of
function spaces can be reduced to
the boundedness of almost diagonal operators on
their corresponding sequence spaces via $\varphi$-transform characterizations
(see, for example, \cite{bow05, bow07, bh06, bhyy2,
bhyy3, bhyy5, fj90, syy24}). Based on this idea and the boundedness
of almost diagonal operators on $\dot{a}_{p,q}^{s,\upsilon}(W)$,
in Subsection \ref{s-cmw-mw}, we finally
obtain the molecular and the
wavelet characterizations of $\dot{A}_{p,q}^{s,\upsilon}(W)$.

\subsection{Boundedness of Almost Diagonal Operators
on $\dot{a}_{p,q}^{s, \upsilon}(W)$\label{s-cmw-ad}}

We start with some notions.
Let $U:=\{u_{Q,R}\}_{Q,R\in\mathcal{D}}$ in $\mathbb{C}$.
For any sequence $\vec{t}:=\{\vec{t}_R\}_{R\in\mathcal{D}}$
in $\mathbb{C}^m$, we define $U\vec{t}:=\{(U\vec{t})_Q\}_{Q\in\mathcal{D}}$
by setting, for any $Q\in\mathcal{D}$,
$(U\vec{t})_Q:=\sum_{R\in\mathcal{D}}u_{Q,R}\vec{t}_R$
if this summation is absolutely convergent.
Next, we recall the concept of almost diagonal operators
introduced in \cite[Definition 4.1]{bhyy2}, which is
a slight generalization of the traditional one in \cite[(3.1)]{fj90}.

\begin{definition}\label{DEFadope}
Let $D,E,F\in\mathbb{R}$. The infinite matrix $U^{DEF}:=
\{u^{DEF}_{Q,R}\}_{Q, R\in\mathcal{D}}$ is defined by setting,
for any $Q, R\in\mathcal{D}$,
\begin{align}\label{DEFmatrix}
u^{DEF}_{Q,R}:=\left[1+\frac{|x_Q-x_R|}
{\ell(Q)\vee \ell(R)}\right]^{-D}
\begin{cases}
\displaystyle{\left[\frac{\ell(Q)}{\ell(R)}\right]^E}
& \text{if } \ell(Q)\leq\ell(R), \\
\displaystyle{\left[\frac{\ell(R)}{\ell(Q)}\right]^F}
& \text{if } \ell(R)<\ell(Q).
\end{cases}
\end{align}
An infinite matrix $U:=\{u_{Q,R}\}_{Q,R\in\mathcal{D}}$
in $\mathbb{C}$ is said to be $(D, E, F)$-\emph{almost diagonal}
if there exists a positive constant $C$ such that,
for any $Q, R\in\mathcal{D}$,
$|u_{Q,R}| \leq C u^{DEF}_{Q,R}$.
\end{definition}

To establish the boundedness of
almost diagonal operators on $\dot{a}_{p,q}^{s,\upsilon}(W)$,
we first need to prove the boundedness of
almost diagonal operators on $\dot{a}_{p,q}^{s,\upsilon}$
as follows. In the special case where $\dot{a}_{p,q}^{s,\upsilon}=\dot{a}_{p,q}^{s,\tau}$,
Theorem \ref{thm-bound-ad} coincides with \cite[Theorem 4.4(ii)]{bhyy2}
which is sharp and consequently in this sense
Theorem \ref{thm-bound-ad} is also sharp
[see Remark \ref{rmk-com-ad}(i) for more details].

\begin{theorem}\label{thm-bound-ad}
Let $a\in\{b, f\}$, $s\in\mathbb{R}$, $p, q\in(0, \infty]$
($p<\infty$ if $a=f$), and $D,E,F\in\mathbb{R}$. Suppose that $\delta_1,\delta_2,\omega$ satisfy \eqref{eq-delta1<0} and
$\upsilon\in\mathcal{G}(\delta_1, \delta_2; \omega)$. If
\begin{align*}
D>J_{\dot{a}_{p,q}^{s,\upsilon}}+\left[\omega\wedge
n\left(\delta_2-\frac{1}{p}\right)_+\right],\
E>\frac{n}{2}+s+n\left(\delta_2-\frac{1}{p}\right)_{+},
\text{ and }F>J_{\dot{a}_{p,q}^{s,\upsilon}}-\frac{n}{2}-s-n
\left(\delta_1-\frac{1}{p}\right)_+,
\end{align*}
where
\begin{align}\label{J_nu}
J_{\dot{a}_{p,q}^{s,\upsilon}}:=
\begin{cases}
n
&{\rm if}\ \delta_1>\frac1{p}
\ {\rm or}\ (\delta_1, q)=(\frac1{p}, \infty)
\ (\text{``supercritical case''}),\\
\displaystyle\frac{n}{\min \{1, q\}}	
&{\rm if}\ a=f,\  \delta_1=\delta_2=\frac1{p}, \text{ and } q<\infty
\ (\text{``critical case''}),\\
\displaystyle\frac{n}{1\wedge\Gamma_{p,q}}
&{\rm if}\ \delta_1<\frac1{p},\
{\rm or\ if}\ a=b,\ \delta_1=\delta_2=\frac1{p}, \text{ and } q<\infty,\\
&{\rm or\ if}\ \delta_2>\delta_1=\frac1{p} \text{ and } q<\infty
\ (\text{``subcritical case''}).
\end{cases}
\end{align}

then any $(D, E, F)$-almost diagonal operator
is bounded on $\dot{a}_{p,q}^{s,\upsilon}$.
\end{theorem}

The following lemma proves
Theorem \ref{thm-bound-ad} in the subcritical case.
The key idea of its proof is to control almost
diagonal operators in terms of the discrete Littlewood--Paley
$g^*_{\lambda}$-function as in \eqref{eq-lamda*},
whose boundedness on $\dot{a}_{p,q}^{s,\upsilon}$
is exactly established in Proposition \ref{prop-gh}.

\begin{lemma}\label{lem-adbound}
Let $a\in\{b, f\}$, $s\in\mathbb{R}$, and
$p, q\in(0, \infty]$ ($p<\infty$ if $a=f$),
and $D,E,F\in\mathbb{R}$.
Assume that $\delta_1,\delta_2,\omega$ satisfy \eqref{eq-delta1<0}
and $\upsilon\in\mathcal{G}(\delta_1, \delta_2; \omega)$. If
\begin{align*}
D>\frac{n}{1\wedge\Gamma_{p,q}}+\left[\omega\wedge n
\left(\delta_2-\frac{1}{p}\right)_+\right],\ E>\frac{n}{2}+s+
n\left(\delta_2-\frac{1}{p}\right)_{+},\text{ and }
F>\frac{n}{1\wedge\Gamma_{p,q}}-\frac{n}{2}-s,
\end{align*}
where $\Gamma_{p,q}$ is as in \eqref{eq-zeta},	
then any $(D, E, F)$-almost diagonal operator
is bounded on $\dot{a}_{p,q}^{s,\upsilon}$.
\end{lemma}
\begin{proof}
We first point out that, to prove the present theorem,
it is enough to consider the case where $s=0$.
Indeed, suppose that the present theorem holds
for $\dot{a}_{p,q}^{0,\upsilon}$. Then, for any
$s\in\mathbb{R}$ and $t:=\{t\}_{R\in\mathcal{D}}$
in $\mathbb{C}$, let $\widetilde{t}
:=\{[\ell(R)]^{-s}t_R\}_{R\in\mathcal{D}}$.
Observe that, for any $(D, E, F)$-almost diagonal
operator $U:=\{u_{Q,R}\}_{Q, R\in\mathcal{D}}$,
$\widetilde{U}:=\{\widetilde{u}_{Q,R}\}_{Q, R\in\mathcal{D}}
:=\{u_{Q,R}[{\ell(R)}/{\ell(Q)}]^{s}
\}_{Q, R\in\mathcal{D}}$ is $(D, E-s, F+s)$-almost diagonal.
By the fact that $\|t\|_{\dot{a}_{p,q}^{s,\upsilon}}
=\|\widetilde{t}\|_{\dot{a}_{p,q}^{0,\upsilon}}$
for any $t\in\dot{a}_{p,q}^{s,\upsilon}$,
the definition of $\widetilde{U}$,
and the assumption that the present theorem holds
for $\dot{a}_{p,q}^{0,\upsilon}$ and hence $\widetilde{U}$
is bounded on $\dot{a}_{p,q}^{0,\upsilon}$, we find that,
for any $t\in\dot{a}_{p,q}^{s,\upsilon}$,
\begin{align*}
\|Ut\|_{\dot{a}_{p,q}^{s,\upsilon}}=
\left\|\widetilde{\left(Ut\right)}\right\|_{\dot{a}_{p,q}^{0,\upsilon}}
=\left\|\widetilde{U}\,\widetilde{t}\right\|_{\dot{a}_{p,q}^{0,\upsilon}}
\lesssim\left\|\widetilde{t}\right\|_{\dot{a}_{p,q}^{0,\upsilon}}
=\left\|t\right\|_{\dot{a}_{p,q}^{s,\upsilon}},
\end{align*}
which implies the boundedness of $U$ on $\dot{a}_{p,q}^{s,\upsilon}$.

Based on the above discussion and Definition \ref{DEFadope},
to prove the present theorem, we only need to show,
for any $t:=\{t_R\}
_{R\in\mathcal{D}}\in\dot{a}_{p,q}^{0,\upsilon}$,
\begin{align}\label{eq-|B|}
\left\|U^{DEF}|t|\right\|_{\dot{a}_{p,q}^{0,\upsilon}}
\lesssim\|t\|_{\dot{a}_{p,q}^{0,\upsilon}},
\end{align}
where $U^{DEF}$ is as in \eqref{DEFmatrix}
and $|t|:=\{|t_R|\}_{R\in\mathcal{D}}$.
We first establish \eqref{eq-|B|} for the case where
$\Gamma_{p,q}>1$ with $\Gamma_{p,q}$ as in \eqref{eq-zeta}.
From Definition \ref{DEFadope}, the definitions of both $\|\cdot\|_{\dot{a}_{p,q}^{0,\upsilon}}$
and $\|\cdot\|_{L\dot{A}^{\upsilon}_{p,q}}$,
and the quasi-triangle inequality
of $\|\cdot\|_{L\dot{A}_{p,q}}$, we infer that,
for any $t:=\{t_R\}_{R\in\mathcal{D}}
\in\dot{a}_{p,q}^{0,\upsilon}$,
\begin{align}\label{eq-I-II-III}
\left\|U^{DEF}|t|\right\|_{\dot{a}_{p,q}^{0,\upsilon}}
&=\left\|\left\{\sum_{Q\in\mathcal{D}_j}\widetilde{\mathbf{1}}_Q
\sum_{R\in\mathcal{D}}u_{Q,R}^{DEF}|t_{R}|
\right\}_{j\in\mathbb Z}\right\|_{L\dot{A}^{\upsilon}_{p,q}}\\
&=\sup_{P\in\mathcal{D}} \frac{1}{\upsilon(P)}
\left\|\left\{\sum_{Q\in\mathcal{D}_j}\widetilde{\mathbf{1}}_Q
\sum_{R\in\mathcal{D}}u_{Q,R}^{DEF}|t_{R}|
\mathbf{1}_P \mathbf{1}_{j\geq j_P}
\right\}_{j\in\mathbb Z}\right\|_{L\dot{A}_{p,q}}\nonumber\\
&\lesssim\sup_{P\in\mathcal{D}} \frac{1}{\upsilon(P)}
\left\|\left\{\sum_{Q\in\mathcal{D}_j}\widetilde{\mathbf{1}}_Q
\sum_{\genfrac{}{}{0pt}{}{R\in\mathcal{D}}{\ell(R)>\ell(P)}}
u_{Q,R}^{DEF}|t_{R}|
\mathbf{1}_P \mathbf{1}_{j\geq j_P}
\right\}_{j\in\mathbb Z}\right\|_{L\dot{A}_{p, q}}\nonumber\\
&\quad+\sup_{P\in\mathcal{D}}\frac{1}{\upsilon(P)}
\left\|\left\{\sum_{Q\in\mathcal{D}_j}\widetilde{\mathbf{1}}_Q
\sum_{\genfrac{}{}{0pt}{}{R\in\mathcal{D}}
{\ell(Q)\leq\ell(R)\leq\ell(P)}}u_{Q,R}^{DEF}|t_{R}|
\mathbf{1}_P \mathbf{1}_{j\geq j_P}\right\}_{j\in\mathbb Z}
\right\|_{L\dot{A}_{p, q}}\nonumber\\
&\quad+\sup_{P\in\mathcal{D}}\frac{1}{\upsilon(P)}
\left\|\left\{\sum_{Q\in\mathcal{D}_j}\widetilde{\mathbf{1}}_Q
\sum_{\genfrac{}{}{0pt}{}{R\in\mathcal{D}}{\ell(R)<\ell(Q)}}
u_{Q,R}^{DEF}|t_{R}|\mathbf{1}_P \mathbf{1}_{j\geq j_P}
\right\}_{j\in\mathbb Z}\right\|_{L\dot{A}_{p, q}}\nonumber\\
&=:\operatorname{I}+\operatorname{II}+\operatorname{III}.\nonumber
\end{align}

Next, we estimate I, II, and III, respectively, by Steps (1), (2), and (3).

\emph{Step (1)}
For any $i\in\mathbb{Z}$, $t:=\{t_R\}_{R\in\mathcal{D}}$
in $\mathbb{C}$, and $x\in\mathbb{R}^n$, let
\begin{align}\label{eq-h_i}
h_i(x):=\int_{\mathbb{R}^n}\frac{2^{in}
|t_i(y)|}{(1+2^i|x-y|)^D}\,dy,
\end{align}
where, for any $i\in\mathbb{Z}$,
$t_i$ is as in \eqref{vect_j}.
By this, the fact that $\mathcal{D}=\bigcup_{j\in\mathbb{Z}}\mathcal{D}_j$,
\eqref{DEFmatrix}, \eqref{vect_j},
Lemma \ref{lem-3P-sum1}(i) with $j$ replaced by $i$,
and Lemma \ref{lem-xy-QR-P}(i) together with
the fact $\ell(Q)\leq\ell(R)$, we find that, for any
$Q, P\in\mathcal{D}$ with $Q\subset P$,
$t:=\{t_R\}_{R\in\mathcal{D}}$ in $\mathbb{C}$, and $x\in Q$,
\begin{align}\label{eq-u-h}
\sum_{\genfrac{}{}{0pt}{}{R\in\mathcal{D}}{\ell(R)>\ell(P)}}
u_{Q,R}^{DEF}|t_{R}|
&=\sum_{i=-\infty}^{j_P-1}
\sum_{R\in\mathcal{D}_i}\left[\frac{\ell(Q)}{\ell(R)}\right]^E
\frac{|t_R|}{(1+[\ell(R)]^{-1}|x_Q-x_R|)^{D}}\\
&=\sum_{i=-\infty}^{j_P-1}2^{(i-j_Q)E}\sum_{R\in\mathcal{D}_i}
\frac{|t_R|}{(1+2^i|x_{Q}-x_R|)^{D}}\nonumber\\
&\sim\sum_{i=-\infty}^{j_P-1}2^{(i-j_Q)E}2^{-\frac{in}{2}}
\sum_{R\in\mathcal{D}_i}\int_{R}\frac{2^{in}
|t_i(y)|}{(1+2^i|x-y|)^D}\,dy\nonumber\\
&\sim\sum_{i=-\infty}^{j_P-1}2^{(i-j_Q)E}2^{-\frac{in}{2}}
\int_{\mathbb{R}^n}\frac{2^{in}
|t_i(y)|}{(1+2^i|x-y|)^D}\,dy
=\sum_{i=-\infty}^{j_P-1}2^{(i-j_Q)E}2^{-\frac{in}{2}}h_i(x).\nonumber
\end{align}
Applying this, the definition of
$\|\cdot\|_{L\dot{A}_{p,q}}$, and
the triangle inequality of $\|\cdot\|^{r}_{L\dot{A}_{p,q}}$
with $r:=p\wedge q\wedge1$, we obtain,
for any $P\in\mathcal{D}$ and
$t:=\{t_R\}_{R\in\mathcal{D}}$ in $\mathbb{C}$,
\begin{align*}
&\left\|\left\{\sum_{Q\in\mathcal{D}_j}\widetilde{\mathbf{1}}_Q
\sum_{\genfrac{}{}{0pt}{}{R\in\mathcal{D}}{\ell(R)>\ell(P)}}
u_{Q,R}^{DEF}|t_{R}|\mathbf{1}_P
\mathbf{1}_{j\geq j_P}\right\}_{j\in\mathbb{Z}}\right\|_{L\dot{A}_{p,q}}\\
&\quad\sim\left\|\left\{\sum_{i=-\infty}^{j_P-1}
2^{(i-j)(E-\frac{n}{2})}h_i\mathbf{1}_P \mathbf{1}_{j\geq j_P}
\right\}_{j\in\mathbb Z}\right\|_{L\dot{A}_{p, q}}\\
&\quad\leq\left[\sum_{i=-\infty}^{j_P-1}2^{(i-j_P)(E-\frac{n}{2})r}
\left\|\left\{h_i\mathbf{1}_P \mathbf{1}_{j\geq j_P}
\right\}_{j\in\mathbb{Z}}\right\|^r_{L\dot{A}_{p, q}}\right]^{\frac{1}{r}}
=\left[\sum_{i=-\infty}^{j_P-1}2^{(i-j_P)(E-\frac{n}{2})r}
\left\|h_i\right\|^r_{L^p(P)}\right]^{\frac{1}{r}}=:\Omega.
\end{align*}

To estimate $\Omega$,
for any $P\in\mathcal{D}$ and $i\in(-\infty, j_P-1]\cap\mathbb{Z}$,
let $P_i\in\mathcal{D}_i$ be as in Lemma \ref{lem-3P-sum1}(iv).
From Lemma \ref{lem-xy-QR-P}(ii) twice and \eqref{eq-h_i},
it follows that,  for any $P\in\mathcal{D}$, $i\in(-\infty, j_P-1]\cap\mathbb{Z}$, $x\in P$, and $y\in\mathbb{R}^n$,
$(1+2^i|x-y|)^D\sim(1+2^i|x_{P_i}-y|)^D$ and hence $h_i(x)\sim h_i(x_{P_i})$.
By this, the definitions of both $\|\cdot\|_{L\dot{A}_{p, q}}$
and $\|\cdot\|_{L\dot{A}^{\upsilon}_{p, q}}$, and
Lemma \ref{lem-grow-est}(i) with $Q$ and $P$
replaced, respectively, by $P$ and $P_i$,
we find that, for any $P\in\mathcal{D}$ and
$i\in(-\infty, j_P-1]\cap\mathbb{Z}$,
\begin{align*}
\left\|h_i\right\|_{L^p(P)}\sim
\left|h_i(x_{P_i})\right||P|^{\frac{1}{p}}
=2^{(i-j_P)\frac{n}{p}}\left|h_i(x_{P_i})
\right||P_i|^{\frac{1}{p}}\sim
2^{(i-j_P)\frac{n}{p}}\left\|h_i\right\|_{L^p(P_i)}
\end{align*}
and hence
\begin{align*}
\left\|h_i\right\|_{L^p(P)}
&\sim2^{(i-j_P)\frac{n}{p}}\left\|h_i\right\|_{L^p(P_i)}
\leq2^{(i-j_P)\frac{n}{p}}\left\|
\left\{h_j\mathbf{1}_{P_i}\mathbf{1}_{j\geq i}
\right\}_{j\in\mathbb{Z}}\right\|_{L\dot{A}_{p, q}}\\
&\leq2^{(i-j_P)\frac{n}{p}}\left\|\left\{h_j\right\}_{j\in\mathbb{Z}}
\right\|_{L\dot{A}^{\upsilon}_{p, q}}\upsilon(P_i)
\lesssim2^{(i-j_P)(\frac{n}{p}-n\delta_2)}
\left\|\left\{h_j\right\}_{j\in\mathbb{Z}}
\right\|_{L\dot{A}^{\upsilon}_{p, q}}\upsilon(P).
\end{align*}
Using this, the assumption that
$E>\frac{n}{2}+n(\delta_2-\frac{1}{p})_{+}$,
and the definition of $\|\cdot\|_{\dot{a}_{p, q}^{0,\upsilon}}$,
we conclude that
\begin{align*}
&\Omega\lesssim\left[\sum_{i=-\infty}^{j_P-1}
2^{(i-j_P)(E-\frac{n}{2}+\frac{n}{p}-n\delta_2)r}
\right]^{\frac{1}{r}}\left\|\left\{h_j\right\}_{j\in\mathbb{Z}}
\right\|_{L\dot{A}^{\upsilon}_{p, q}}\upsilon(P)\\
&\quad\sim\left\|\left\{h_j\right\}_{j\in\mathbb{Z}}
\right\|_{L\dot{A}^{\upsilon}_{p, q}}\upsilon(P)
\lesssim\upsilon(P)
\left\|\left\{t_j\right\}_{j\in\mathbb{Z}}
\right\|_{L\dot{A}^{\upsilon}_{p, q}}=
\upsilon(P)\left\|t\right\|_{\dot{a}_{p, q}^{0,\upsilon}},
\end{align*}
which further implies that $\operatorname{I}\lesssim
\left\|t\right\|_{\dot{a}_{p, q}^{0,\upsilon}}$.
This gives the desired estimate of $\operatorname{I}$.

\emph{Step (2)}
Applying the same argument as that used to prove \eqref{eq-u-h},
we obtain, for any $Q, P\in\mathcal{D}$ with $Q\subset P$,
$t:=\{t_R\}_{R\in\mathcal{D}}$ in $\mathbb{C}$, and $x\in Q$,
\begin{align*}
\sum_{\genfrac{}{}{0pt}{}{R\in\mathcal{D}}{\ell(Q)\leq\ell(R)\leq\ell(P)}}
u_{Q,R}^{DEF}|t_{R}|
\sim\sum_{i=j_P}^{j_Q}2^{(i-j_Q)E}2^{-\frac{in}{2}}h_i(x),
\end{align*}
where $h_i$ is as in \eqref{eq-h_i}.
By this, we find that,
for any $P\in\mathcal{D}$ and
$t:=\{t_R\}_{R\in\mathcal{D}}
\in\dot{a}_{p, q}^{0,\upsilon}$,
\begin{align*}
&\left\|\left\{\sum_{Q\in\mathcal{D}_j}\widetilde{\mathbf{1}}_Q
\sum_{\genfrac{}{}{0pt}{}{R\in\mathcal{D}}
{\ell(Q)\leq\ell(R)\leq\ell(P)}}
u_{Q,R}^{DEF}|t_{R}|\mathbf{1}_P
\mathbf{1}_{j\geq j_P}\right\}_{j\in\mathbb{Z}}\right\|_{L\dot{A}_{p,q}}\\
&\quad\sim\left\|\left\{\sum_{i=j_P}^{j}2^{(i-j)(E-\frac{n}{2})}
h_i\mathbf{1}_P \mathbf{1}_{j\geq j_P}
\right\}_{j\in\mathbb Z}\right\|_{L\dot{A}_{p, q}}
=\left\|\left\{\sum_{l=0}^{j-j_P}2^{-l(E-\frac{n}{2})}
h_{j-l}\mathbf{1}_P \mathbf{1}_{j\geq j_P}
\right\}_{j\in\mathbb{Z}}\right\|_{L\dot{A}_{p, q}}\\
&\quad=\left\|\left\{\sum_{l=0}^{\infty}2^{-l(E-\frac{n}{2})}
h_{j-l}\mathbf{1}_P \mathbf{1}_{j\geq j_P+l}
\right\}_{j\in\mathbb{Z}}\right\|_{L\dot{A}_{p, q}}=:\Lambda.
\end{align*}
This, combined with \eqref{eq-I-II-III},
the triangle inequality of $\|\cdot\|^{r}_{L\dot{A}_{p,q}}$,
the definitions of $\|\cdot\|_{L\dot{A}^{\upsilon}_{p, q}}$,
$\|\cdot\|_{L\dot{A}_{p, q}}$, and
$\|\cdot\|_{\dot{a}_{p, q}^{0,\upsilon}}$,
the assumption $E-\frac{n}{2}>0$,
and Proposition \ref{prop-gh} together with
$\Gamma_{p,q}>1$ and $D>n+[\omega\wedge n(\delta_2-\frac{1}{p})_+]$
and with $\{g_j\}_{j\in\mathbb{Z}}$ and
$\{h_j\}_{j\in\mathbb{Z}}$ replaced, respectively, by
$\{h_j\}_{j\in\mathbb{Z}}$ and $\{t_j\}_{j\in\mathbb{Z}}$,
further implies that
\begin{align*}
\Lambda&\leq\left[\sum_{l=0}^{\infty}
2^{-l(E-\frac{n}{2})r}\left\|\left\{h_{j-l}
\mathbf{1}_P \mathbf{1}_{j\geq j_P+l}\right\}_{j\in\mathbb Z}
\right\|^{r}_{L\dot{A}_{p, q}}\right]^{\frac{1}{r}}
=\left[\sum_{l=0}^{\infty}2^{-l(E-\frac{n}{2})r}
\left\|\left\{h_j\mathbf{1}_P\mathbf{1}_{j\geq j_P}
\right\}_{j\in\mathbb{Z}}\right\|^{r}_{L\dot{A}_{p, q}}\right]^{\frac{1}{r}}\\
&\sim\left\|\left\{h_j\mathbf{1}_P\mathbf{1}_{j\geq j_P}
\right\}_{j\in\mathbb{Z}}\right\|_{L\dot{A}_{p, q}}
\leq\upsilon(P)\left\|\left\{h_j\right\}_{j\in\mathbb{Z}}
\right\|_{L\dot{A}^{\upsilon}_{p, q}}\lesssim\upsilon(P)
\left\|\left\{t_j\right\}_{j\in\mathbb{Z}}
\right\|_{L\dot{A}^{\upsilon}_{p, q}}=
\upsilon(P)\left\|t\right\|_{\dot{a}_{p, q}^{0,\upsilon}}
\end{align*}
and hence $\operatorname{II}\lesssim\left\|t\right\|_{\dot{a}_{p, q}^{0,\upsilon}}$,
which establishes the desired estimate of II.

\emph{Step (3)}	For any
$i, j\in\mathbb{Z}$, $t:=\{t_R\}_{R\in\mathcal{D}}$
in $\mathbb{C}$, and $x\in\mathbb{R}^n$, let
\begin{align*}
g_{i, j}(x):=\int_{\mathbb{R}^n}\frac{2^{j n}
|t_i(y)|}{(1+2^{j}|x-y|)^D}\,dy,
\end{align*}
where, for any $i\in\mathbb{Z}$,
$t_i$ is as in \eqref{vect_j}.
From this, the fact that $\mathcal{D}=\bigcup_{j\in\mathbb{Z}}\mathcal{D}_j$,
\eqref{DEFmatrix}, \eqref{vect_j},
Lemma \ref{lem-3P-sum1}(i) with $j$ replaced by $i$,
and Lemma \ref{lem-xy-QR-P}(i) with the fact $\ell(R)<\ell(Q)$
in the following calculation, it follows that,
for any $Q\in\mathcal{D}$, $t:=\{t_R\}_{R\in\mathcal{D}}$
in $\mathbb{C}$, and $x\in Q$,
\begin{align*}
\sum_{\genfrac{}{}{0pt}{}{R\in\mathcal{D}}{\ell(R)<\ell(Q)}}
u_{Q,R}^{DEF}|t_{R}|
&=\sum_{i=j_Q+1}^{\infty}\sum_{R\in\mathcal{D}_i}
\left[\frac{\ell(R)}{\ell(Q)}\right]^{F}
\frac{|t_{R}|}{(1+[\ell(Q)]^{-1}|x_Q-x_R|)^{D}}\nonumber\\
&=\sum_{i=j_Q+1}^{\infty}2^{(j_Q-i)F}\sum_{R\in\mathcal{D}_i}
\frac{|t_R|}{(1+2^{j_Q}|x_{Q}-x_R|)^{D}}\nonumber\\
&\sim\sum_{i=j_Q+1}^{\infty}2^{(j_Q-i)(F-n)}2^{-\frac{in}{2}}
\sum_{R\in\mathcal{D}_i}\int_{R}\frac{2^{j_Q n}|t_i(y)|}{(1+2^{j_Q}|x-y|)^D}\,dy\nonumber\\
&=\sum_{i=j_Q+1}^{\infty}2^{(j_Q-i)(F-n)}2^{-\frac{in}{2}}
\int_{\mathbb{R}^n}\frac{2^{j_Q n}|t_i(y)|}{(1+2^{j_Q}|x-y|)^D}\,dy\\
&=\sum_{i=j_Q+1}^{\infty}2^{(j_Q-i)(F-n)}2^{-\frac{in}{2}}g_{i, j_Q}(x)\nonumber.
\end{align*}
This, together with \eqref{eq-I-II-III},
the definitions of $\|\cdot\|_{L\dot{A}^{\upsilon}_{p, q}}$,
$\|\cdot\|_{L\dot{A}_{p, q}}$, and
$\|\cdot\|_{\dot{a}_{p, q}^{0,\upsilon}}$,
the assumption that $F-\frac{n}{2}>0$,
and Proposition \ref{prop-gh} combined with
$\Gamma_{p,q}>1$ and $D>n+[\omega\wedge n(\delta_2-\frac{1}{p})_+]$
and with $\{g_j\}_{j\in\mathbb{Z}}$ and
$\{h_j\}_{j\in\mathbb{Z}}$ replaced, respectively, by
$\{g_{j+l, j}\}_{j\in\mathbb{Z}}$ and
$\{t_{j+l}\}_{j\in\mathbb{Z}}$,
further implies that, for any $P\in\mathcal{D}$
and $t:=\{t_R\}_{R\in\mathcal{D}}
\in\dot{a}_{p, q}^{0,\upsilon}$,
\begin{align*}
&\left\|\left\{\sum_{Q\in\mathcal{D}_j}\widetilde{\mathbf{1}}_Q
\sum_{\genfrac{}{}{0pt}{}{R\in\mathcal{D}}
{\ell(R)<\ell(Q)}}u_{Q,R}^{DEF}|t_{R}|\mathbf{1}_P
\mathbf{1}_{j\geq j_P}\right\}_{j\in\mathbb Z}
\right\|_{L\dot{A}_{p, q}}\\
&\quad\sim\left\|\left\{\sum_{i=j+1}^{\infty}2^{(j-i)(F-\frac{n}{2})}
g_{i, j}\mathbf{1}_P \mathbf{1}_{j\geq j_P}\right\}
_{j\in\mathbb Z}\right\|_{L\dot{A}_{p, q}}
=\left\|\left\{\sum_{l=1}^{\infty}2^{-l(F-\frac{n}{2})}
g_{j+l, j}\mathbf{1}_P \mathbf{1}_{j\geq j_P}
\right\}_{j\in\mathbb Z}\right\|_{L\dot{A}_{p, q}}\\
&\quad\leq\left[\sum_{l=1}^{\infty}2^{-l(F-\frac{n}{2})r}
\left\|\left\{g_{j+l, j}\mathbf{1}_P \mathbf{1}_{j\geq j_P}
\right\}_{j\in\mathbb{Z}}\right\|^{r}
_{L\dot{A}_{p, q}}\right]^{\frac{1}{r}}
\leq\upsilon(P)\left[\sum_{l=1}^{\infty}2^{-l(F-\frac{n}{2})r}
\left\|\left\{g_{j+l, j}\right\}_{j\in\mathbb{Z}}\right\|^{r}
_{L\dot{A}^{\upsilon}_{p, q}}\right]^{\frac{1}{r}}\\
&\quad\lesssim\upsilon(P)\left[\sum_{l=1}^{\infty}2^{-l(F-\frac{n}{2})r}
\left\|\left\{t_{j+l}\right\}_{j\in\mathbb{Z}}\right\|^{r}
_{L\dot{A}^{\upsilon}_{p, q}}\right]^{\frac{1}{r}}
\leq\upsilon(P)\left[\sum_{l=1}^{\infty}2^{-l(F-\frac{n}{2})r}
\left\|\left\{t_{j}\right\}_{j\in\mathbb{Z}}\right\|^{r}
_{L\dot{A}^{\upsilon}_{p, q}}\right]^{\frac{1}{r}}\\
&\quad\sim\upsilon(P)\left\|\left\{t_j\right\}_{j\in\mathbb Z}
\right\|_{L\dot{A}^{\upsilon}_{p, q}}=\upsilon(P)
\left\|t\right\|_{\dot{a}_{p, q}^{0,\upsilon}},
\end{align*}
and hence $\operatorname{III}\lesssim\left\|t\right\|
_{\dot{a}_{p, q}^{0,\upsilon}}$, which gives the desired estimate of III.

Combining \eqref{eq-I-II-III} and the estimates in Steps (1), (2), and (3)
together, we conclude that \eqref{eq-|B|} holds
for the case where $\Gamma_{p,q}>1$. To complete the proof of the present
theorem, it suffices to show \eqref{eq-|B|}
for the case where $\Gamma_{p,q}\leq1$. In this case,
let $U:=\{u_{Q,R}\}_{Q, R\in\mathcal{D}}$
be a $(D, E, F)$-almost diagonal
operator. Fix some $\gamma\in(0, \Gamma_{p,q})$
such that $\widetilde{U}:=\{\widetilde{u}_{Q,R}\}_{Q, R\in\mathcal{D}}:=
\{(|Q|/|R|)^{\frac{1-\gamma}{2}}|u_{Q,R}|^{\gamma}\}_{Q, R\in\mathcal{D}}$
is a $(\gamma D,\gamma E+\frac{n}{2}-\frac{\gamma n}{2},
\gamma F-\frac{n}{2}+\frac{\gamma n}{2})$-almost diagonal
operator satisfying all the hypotheses of the present theorem
for $\dot{a}_{p/\gamma, q/\gamma}^{0, {\upsilon}^{\gamma}}$,
where ${\upsilon}^{\gamma}$ is defined by setting,
for any $Q\in\mathcal{D}$, ${\upsilon}^{\gamma}(Q)
:=[\upsilon(Q)]^{\gamma}$.
For any $t:=\{t_R\}_{R\in\mathcal{D}}
\in\dot{a}_{p, q}^{0, \upsilon}$, let $\widetilde{t}
:=\{\widetilde{t}_R\}_{R\in\mathcal{D}}
:=\{|R|^{\frac{1-\gamma}{2}}|t_R|^{\gamma}\}_{R\in\mathcal{D}}$.
Applying these constructions, the definitions of $U$,
$\|\cdot\|_{\dot{a}_{p, q}^{0, \upsilon}}$, and its convexified version
$\|\cdot\|_{\dot{a}_{p/\gamma, q/\gamma}^{0, {\upsilon}^{\gamma}}}$
with respect to the index $\gamma$,
the monotonicity of the sequence space $l^q$ on $q$,
and \eqref{eq-|B|} for the case where $\Gamma_{p,q}>1$,
we obtain, for any $t:=\{t_R\}_{R\in\mathcal{D}}
\in\dot{a}_{p, q}^{0, \upsilon}$,
\begin{align*}
\|Ut\|_{\dot{a}_{p, q}^{0, \upsilon}}
&=\left\|\left\{\left|\sum_{R\in\mathcal{D}}u_{Q,R}
t_R\right|\right\}_{Q\in\mathcal{D}}
\right\|_{\dot{a}_{p, q}^{0, \upsilon}}
\leq\left\|\left\{|Q|^{\frac{1-\gamma}{2}}
\sum_{R\in\mathcal{D}}\left|u_{Q,R}
t_R\right|^{\gamma}\right\}_{Q\in\mathcal{D}}
\right\|^{\frac{1}{\gamma}}_{\dot{a}_{p/\gamma, q/\gamma}^{0, {\upsilon}^{\gamma}}}\\
&=\left\|\left\{\sum_{R\in\mathcal{D}}
\widetilde{u}_{Q,R}\widetilde{t}_R\right\}_{Q\in\mathcal{D}}
\right\|^{\frac{1}{\gamma}}_{\dot{a}_{p/\gamma, q/\gamma}^{0, {\upsilon}^{\gamma}}}=\left\|\widetilde{U}\,
\widetilde{t}\right\|^{\frac{1}{\gamma}}
_{\dot{a}_{p/\gamma, q/\gamma}^{0, {\upsilon}^{\gamma}}}
\lesssim\left\|\widetilde{t}\right\|^{\frac{1}{\gamma}}
_{\dot{a}_{p/\gamma, q/\gamma}^{0, {\upsilon}^{\gamma}}}
=\|t\|_{\dot{a}_{p, q}^{0, \upsilon}}.
\end{align*}
This finishes the proof of the case where $\Gamma_{p,q}\leq1$
and hence Lemma \ref{lem-adbound}.
\end{proof}

By restricting the indices of growth functions,
we next establish the equivalence between
$\dot{a}_{p,q}^{s,\upsilon}$ and
$\dot{f}_{\infty,q}^{s}$ in \eqref{eq-f_3},
which can be used to improve Lemma \ref{lem-adbound}
and finally obtain Theorem \ref{thm-bound-ad}.

\begin{lemma}\label{lem-ident}
Let $a\in\{b, f\}$, $s\in\mathbb{R}$,
and $p\in(0, \infty]$ ($p<\infty$ if $a=f$).
Assume that $\delta_1,\delta_2,\omega$ satisfy \eqref{eq-delta1<0}
and $\upsilon\in\mathcal{G}(\delta_1, \delta_2; \omega)$.
Then the following statements hold.
\begin{itemize}
\item[{\rm (i)}] If $q\in(0,\infty)$ and $\delta_1\in(1/p, \infty)$,
or $q=\infty$ and $\delta_1\in[1/p, \infty)$, then,
for any $t:=\{t_Q\}_{Q\in\mathcal{D}}$ in $\mathbb{C}$, $\|t\|_{\dot{a}_{p,q}^{s,\upsilon}}
\sim\|\{\frac{t_Q}{\upsilon(Q)}\}
_{Q\in\mathcal{D}}\|_{\dot{f}_{\infty,\infty}^{s-\frac{n}{p}}}$,
where the positive equivalence constants are independent of $t$.
\item[{\rm (ii)}] If $q\in(0,\infty]$ and
$\delta_1=\delta_2=1/p$, then
$\dot{f}_{p,q}^{s,\upsilon}=\dot{f}_{\infty,q}^{s}$
with equivalent quasi-norms.
\end{itemize}
\end{lemma}

\begin{proof}
We first prove (i). By the definitions of
$\|\cdot\|_{\dot{f}_{\infty,\infty}^{s-\frac{n}{p}}}$ [see \eqref{eq-f_3}]
and $\|\cdot\|_{\dot{a}_{p,q}^{s,\upsilon}}$, we find that,
for any $t:=\{t_Q\}_{Q\in\mathcal{D}}$ in $\mathbb{C}$,
\begin{align}\label{eq-u-lambda}
\left\|\left\{\frac{t_Q}{\upsilon(Q)}\right\}
_{Q\in\mathcal{D}}\right\|_{\dot{f}_{\infty,\infty}^{s-\frac{n}{p}}}
=\sup_{P\in\mathcal{D}}\frac{|P|^{-\frac{s}{n}+\frac{1}{p}
-\frac{1}{2}}|t_P|}{\upsilon(P)}\leq
\|t\|_{\dot{a}_{p,q}^{s,\upsilon}}.
\end{align}
Next, we establish the reverse estimate of \eqref{eq-u-lambda}.
To achieve this, from the definition of
$\|\cdot\|_{\dot{a}_{p,q}^{s,\upsilon}}$,
the first equality in \eqref{eq-u-lambda}, and
Lemma \ref{lem-grow-est}(i), we infer that,
for any $t:=\{t_Q\}_{Q\in\mathcal{D}}$ in $\mathbb{C}$,
\begin{align}\label{eq-t-t/v}
\|t\|_{\dot{a}_{p,q}^{s,\upsilon}}
&=\sup_{P\in\mathcal{D}}\frac{1}{\upsilon(P)}
\left\|\left\{\sum_{Q\in\mathcal{D}_j}\widetilde{\mathbf{1}}_Q
|Q|^{-\frac{s}{n}}|t_Q|\mathbf{1}_P
\mathbf{1}_{j\geq j_P}\right\}_{j\in\mathbb Z}\right\|_{L\dot{A}_{p, q}}\\
&\leq\left\|\left\{\frac{t_Q}{\upsilon(Q)}\right\}
_{Q\in\mathcal{D}}\right\|_{\dot{f}_{\infty,\infty}^{s-\frac{n}{p}}}
\sup_{P\in\mathcal{D}}\left\|\left\{\sum_{Q\in\mathcal{D}_j}
\mathbf{1}_Q|Q|^{-\frac{1}{p}}\frac{\upsilon(Q)}{\upsilon(P)}
\mathbf{1}_P \mathbf{1}_{j\geq j_P}
\right\}_{j\in\mathbb Z}\right\|_{L\dot{A}_{p, q}}\nonumber\\
&\lesssim\left\|\left\{\frac{t_Q}{\upsilon(Q)}\right\}
_{Q\in\mathcal{D}}\right\|_{\dot{f}_{\infty,\infty}^{s-\frac{n}{p}}}
\sup_{P\in\mathcal{D}}\left\|\left\{2^{j\frac{n}{p}}2^{(j_P-j)n \delta_1}
\mathbf{1}_P \mathbf{1}_{j\geq j_P}\right\}
_{j\in\mathbb Z}\right\|_{L\dot{A}_{p, q}}.\nonumber
\end{align}
Clearly, by \eqref{eq-t-t/v}, to prove the reverse inequality of
\eqref{eq-u-lambda}, it suffices to show that,
under the assumption on $\delta_1$ in (i), for any $P\in\mathcal{D}$,
\begin{align}\label{eq-P=1}
\left\|\left\{2^{(j_P-j)\delta_1}2^{j_P\frac{n}{p}}
\mathbf{1}_P\mathbf{1}_{j\geq j_P}
\right\}_{j\in\mathbb Z}\right\|_{L\dot{A}_{p, q}}\sim1.
\end{align}
We only give the proof of \eqref{eq-P=1}
in the case where $A=F$ because the proof of the
case where $A=B$ is similar.
From \eqref{eq-LA} and the assumption
on $\delta_1$ in (i), we deduce that, for any $P\in\mathcal{D}$,
\begin{align*}
\left\|\left\{2^{(j_P-j)\delta_1}2^{j_P\frac{n}{p}}
\mathbf{1}_P\mathbf{1}_{j\geq j_P}
\right\}_{j\in\mathbb Z}\right\|_{L\dot{F}_{p, q}}
&=\left\|\left[\sum_{j=j_P}^{\infty}
2^{(j_P-j)\delta_1 q}2^{j_P\frac{n}{p}q}\right]^{\frac{1}{q}}
\mathbf{1}_P\right\|_{L^p}\sim\left\|2^{j_P\frac{n}{p}}
\mathbf{1}_P\right\|_{L^p}=1.
\end{align*}
This finishes the proof of \eqref{eq-P=1} in the case where
$A=F$ and hence (i).

Next, we prove (ii). By Example \ref{examp}(i),
we find that, under the assumptions of (ii),
for any $Q\in\mathcal{D}$, $\upsilon(Q)\sim|Q|^{\frac{1}{p}}$.
Applying this and \eqref{eq-f=f}, we obtain
$\dot{f}_{p,q}^{s,\upsilon}=\dot{f}^{s,\frac{1}{p}}_{p,q}=
\dot{f}_{\infty,q}^{s}$ all with equivalent quasi-norms,
which completes the proof of (ii) and hence Lemma \ref{lem-ident}.
\end{proof}

\begin{remark}
In Lemma \ref{lem-ident},
let $\tau\in[0,\infty)$ and, for any $Q\in\mathcal{D}$,
$\upsilon(Q):=|Q|^{\tau}$. From Example \ref{examp}(i),
it follows that $\upsilon\in\mathcal{G}(\tau,\tau;0)$.
In this case, (i) and (ii) of Lemma \ref{lem-ident}
reduce, respectively, to \cite[Theorem 1]{yy13}
and \cite[Corollary 5.7]{fj90}.
\end{remark}

Combining Lemmas \ref{lem-adbound} and \ref{lem-ident},
we can give the proof of Theorem \ref{thm-bound-ad}.

\begin{proof}[Proof of Theorem \ref{thm-bound-ad}]
We prove the present theorem  by considering,
respectively, supercritical, critical,
and subcritical cases as in \eqref{J_nu}.

For the supercritical case,
let $U:=\{u_{Q,R}\}_{Q, R\in\mathcal{D}}$
be a $(D,E,F)$-almost diagonal operator with
\begin{align*}
D>n+\omega,\
E>\frac{n}{2}+s+n\left(\delta_2-\frac{1}{p}\right),
\text{\ and\ }F>\frac{n}{2}-s-n\left(\delta_1-\frac{1}{p}\right).
\end{align*}
Then we define $\widetilde{U}:=\{\widetilde{u}_{Q,R}
\}_{Q, R\in\mathcal{D}}$ by setting, for any $Q, R\in\mathcal{D}$,
$\widetilde{u}_{Q,R}:=u_{Q,R}[\upsilon(R)/\upsilon(Q)]$.
By this construction, the growth condition of $\upsilon$,
and Definition \ref{DEFadope}, we find that $\widetilde{U}$
is $(D-\omega,E-n \delta_2,F+n \delta_1)$-almost diagonal.
Using this, Lemma \ref{lem-ident}(i), and the boundedness of
$\widetilde{U}$ on $\dot{f}_{\infty,\infty}^{s-\frac{n}{p}}$
(see, for example, \cite[p.\,81]{fj90}), we conclude that,
for any $t:=\{t_R\}_{R\in\mathcal{D}}
\in\dot{a}_{p,q}^{s,\upsilon}$,
\begin{align*}
\|Ut\|_{\dot{a}_{p,q}^{s,\upsilon}}
\sim\left\|\left\{\frac{(Ut)_Q}{\upsilon(Q)}
\right\}_{Q\in\mathcal{D}}\right\|_
{\dot{f}_{\infty,\infty}^{s-\frac{n}{p}}}
=\left\|\widetilde{U}\left(\left\{\frac{t_R}{\upsilon(R)}
\right\}_{R\in\mathcal{D}}\right)\right\|_
{\dot{f}_{\infty,\infty}^{s-\frac{n}{p}}}
\lesssim\left\|\left\{\frac{t_R}{\upsilon(R)}
\right\}_{R\in\mathcal{D}}\right\|_
{\dot{f}_{\infty,\infty}^{s-\frac{n}{p}}}
\sim\|t\|_{\dot{a}_{p,q}^{s,\upsilon}},
\end{align*}
which implies the boundedness of $U$ on $\dot{a}_{p,q}^{s,\upsilon}$
and hence the present theorem in this case.

The critical case directly follows from
Lemma \ref{lem-ident}(ii) and
the boundedness of $\widetilde{U}$ on
$\dot{f}_{\infty,q}^{s}$
(see, for example, \cite[p.\,81]{fj90});
we omit the details.

In the subcritical case, the present
theorem is precisely Lemma \ref{lem-adbound}.
This finishes the proof of the subcritical
case and hence Theorem \ref{thm-bound-ad}.
\end{proof}

Next, we establish the boundedness
of almost diagonal operators
on $\dot{a}_{p,q}^{s,\upsilon}(W)$.

\begin{theorem}\label{a(W)adopebound}
Let $a\in\{b, f\}$, $s\in\mathbb{R}$, $p\in(0,\infty)$,
$q\in(0, \infty]$, $W\in \mathcal{A}_{p,\infty}$, and
$D,E,F\in\mathbb{R}$. Suppose that $\delta_1,\delta_2,
\omega$ satisfy \eqref{eq-delta1<0}
and $\upsilon\in\mathcal{G}(\delta_1, \delta_2; \omega)$. Let
\begin{align*}
&\Delta:=\left[\delta_2-\frac1{p}+
\frac{d^{\operatorname{lower}}_{p, \infty}(W)}{np}\right]_{+},
\ D_{\dot{a}_{p,q}^{s,\upsilon}(W)}:=
J_{\dot{a}_{p,q}^{s,\upsilon}}+
\left[n\Delta\wedge\left(\omega+
\frac{d^{\operatorname{lower}}_{p, \infty}(W)}{p}\right)\right]
+\frac{d^{\operatorname{upper}}_{p, \infty}(W)}{p},\\
&E_{\dot{a}_{p,q}^{s,\upsilon}(W)}:=
\frac{n}{2}+s+n\Delta,\text{\ and\ }
F_{\dot{a}_{p,q}^{s,\upsilon}(W)}:=
J_{\dot{a}_{p,q}^{s,\upsilon}}
-\frac{n}{2}-s-n\left(\delta_1-\frac{1}{p}\right)_+
+\frac{d^{\operatorname{upper}}_{p, \infty}(W)}{p},
\end{align*}
where $J_{\dot{a}_{p,q}^{s,\upsilon}},
d^{\operatorname{lower}}_{p, \infty}(W)$,
and $d^{\operatorname{ upper}}_{p, \infty}(W)$ are
as, respectively, in \eqref{J_nu}, \eqref{eq-low-dim},
and \eqref{eq-upp-dim}. If
\begin{align}\label{DEFa(W)adope}
\displaystyle D>D_{\dot{a}_{p,q}^{s,\upsilon}(W)},\
\displaystyle E>E_{\dot{a}_{p,q}^{s,\upsilon}(W)},\
{\rm and}\
\displaystyle F>F_{\dot{a}_{p,q}^{s,\upsilon}(W)},	
\end{align}	
then any $(D, E, F)$-almost diagonal operator
is bounded on $\dot{a}_{p,q}^{s,\upsilon}(W)$.
\end{theorem}

\begin{proof}
We prove the present theorem by considering the
following two cases for $\delta_2$ and $p$.

\emph{Case (1)} $\frac{1}{p}\leq\delta_2$.
In this case, we find that the conditions
on $D,E$, and $F$ are exactly
\begin{align}\label{eq-DEF-1}
&D>J_{\dot{a}_{p,q}^{s,\upsilon}}+
\left[\omega\wedge n\left(\delta_2-\frac{1}{p}\right)\right]+
\frac{d^{\operatorname{lower}}_{p, \infty}(W)
+d^{\operatorname{upper}}_{p, \infty}(W)}{p},\\
&E>\frac{n}{2}+s+n\left(\delta_2-\frac{1}{p}\right)
+\frac{d^{\operatorname{lower}}_{p, \infty}(W)}{p},
\text{ and }F>J_{\dot{a}_{p,q}^{s,\upsilon}}-\frac{n}{2}-s-
n\left(\delta_1-\frac{1}{p}\right)
+\frac{d^{\operatorname{upper}}_{p, \infty}(W)}{p}.\nonumber
\end{align}
Let $\mathbb{A}:=\{A_Q\}_{Q\in\mathcal{D}}$
be a sequence of reducing operators of order $p$ for $W$.
By Corollary \ref{cor-a(A)=a(W)}, we find that,
to prove the present corollary, it suffices to show
any $(D, E, F)$-almost diagonal operator
is bounded on $\dot{a}_{p,q}^{s,\upsilon}(\mathbb{A})$.
Using \eqref{eq-DEF-1}, we can choose $\beta_1\in[\![d_{p,\infty}^{\mathrm{lower}}(W),\infty)$
and $\beta_2\in[\![d_{p,\infty}^{\mathrm{upper}}(W),\infty)$ such that
$D>J_{\dot{a}_{p,q}^{s,\upsilon}}+
[\omega\wedge n(\delta_2-\frac{1}{p})_+]+
\beta_1+\beta_2$, $E>\frac{n}{2}+s+n(\delta_2-\frac{1}{p})_{+}
+\beta_1$, and $F>J_{\dot{a}_{p,q}^{s,\upsilon}}-\frac{n}{2}-s-
n(\delta_1-\frac{1}{p})_+ +\beta_2$.
Let $U:=\{{u}_{Q,R}\}_{Q, R\in\mathcal{D}}$
be a $(D,E,F)$-almost diagonal operator.
We now define $\widetilde{U}
:=\{\widetilde{u}_{Q,R}\}_{Q,R\in\mathcal{D}}$
by setting, for any $Q,R\in\mathcal{D}$,
$\widetilde{u}_{Q,R}:=\{|{u}_{Q,R}|\|{A}_QA_R^{-1}
\|\}_{Q,R\in\mathcal{D}}$. From this,
Lemma \ref{growEST}, and Definition \ref{DEFadope},
we infer that $\widetilde{U}$ is
$(D-\beta_1-\beta_2,E-\beta_1,F-\beta_2)$-almost diagonal.
Applying this, the definition of $\widetilde{U}$,
Theorem \ref{thm-bound-ad}, and \eqref{eq-a(A)-a},
we obtain, for any $\vec{t}:=\{\vec{t}_R\}_{R\in\mathcal{D}}
\in\dot{a}_{p,q}^{s,\upsilon}(\mathbb{A})$,
\begin{align}\label{eq-aaaa}
\left\|\widetilde{U}\left(\left\{\left|A_R\vec{t}_R
\right|\right\}_{R\in\mathcal{D}}\right)
\right\|_{\dot{a}_{p,q}^{s,\upsilon}}
\lesssim\left\|\left\{\left|A_R
\vec{t}_R\right|\right\}_{R\in\mathcal{D}}
\right\|_{\dot{a}_{p,q}^{s,\upsilon}}=\left\|\vec{t}
\right\|_{\dot{a}_{p,q}^{s,\upsilon}(\mathbb{A})}<\infty
\end{align}
and hence, for any $Q\in\mathcal{D}$,
\begin{align}\label{eq-absoconver}
\sum_{R\in\mathcal{D}}\widetilde{u}_{Q,R}
\left|A_R \vec{t}_R\right|=
\left[\widetilde{U}\left(\left\{\left|A_R\vec{t}_R
\right|\right\}_{R\in\mathcal{D}}\right)\right]_Q
<\infty.
\end{align}
By \eqref{eq-absoconver} and the definition of
$\{\widetilde{u}_{Q,R}\}_{Q,R\in\mathcal{D}}$,
we find that, for any $Q\in\mathcal{D}$ and $\vec{t}:=\{\vec{t}_R\}_{R\in\mathcal{D}}
\in\dot{a}_{p,q}^{s,\upsilon}(\mathbb{A})$,
\begin{align*}
\sum_{R\in\mathcal{D}}\left|u_{Q,R}\vec{t}_R\right|
&=\sum_{R\in\mathcal{D}}|u_{Q,R}|\left|A^{-1}_Q
A _QA_R^{-1}A_R\vec{t}_R\right|\\
&\leq\sum_{R\in\mathcal{D}}|u_{Q,R}|\left\|A^{-1}_Q \right\|
\left\|A _QA_R^{-1}\right\|\left|
A_R\vec{t}_R\right|=
\left\|A^{-1}_Q \right\|\sum_{R\in\mathcal{D}}
\widetilde{u}_{Q,R}\left|A_R\vec{t}_R\right|<\infty.
\end{align*}
This, together with the construction of
$\{\widetilde{u}_{Q,R}\}_{Q,R\in\mathcal{D}}$,
the definition of $U$, \eqref{eq-a(A)-a},
\eqref{eq-absoconver}, and \eqref{eq-aaaa},
further implies that, for any $Q\in\mathcal{D}$ and
$\vec{t}:=\{\vec{t}_R\}_{R\in\mathcal{D}}
\in\dot{a}_{p,q}^{s,\upsilon}(\mathbb{A})$,
\begin{align*}
\left|A_Q \left(U\vec{t}\right)_Q\right|
&=\left|\sum_{R\in\mathcal{D}}u_{Q,R}A_Q
\vec{t}_R\right|\leq
\sum_{R\in\mathcal{D}}\left|u_{Q,R}\right|
\left\|A_QA^{-1}_R\right\|
\left|A_R \vec{t}_R\right|\\
&=\sum_{R\in\mathcal{D}}\widetilde{u}_{Q,R}
\left|A_R \vec{t}_R\right|
=\left[\widetilde{U}\left(\left\{\left|A_R\vec{t}_R
\right|\right\}_{R\in\mathcal{D}}\right)\right]_Q
\end{align*}
and hence
\begin{align*}
\left\|U\vec{t}\right\|_{\dot{a}_{p,q}^{s,\upsilon}(\mathbb{A})}
&=\left\|\left\{\left|A_Q\left(U\vec{t}\right)_Q\right|
\right\}_{Q\in\mathcal{D}}\right\|_{\dot{a}_{p,q}^{s,\upsilon}}
\leq\left\|\widetilde{U}\left(\left\{\left|A_R\vec{t}_R
\right|\right\}_{R\in\mathcal{D}}\right)
\right\|_{\dot{a}_{p,q}^{s,\upsilon}}\lesssim\left\|\vec{t}
\right\|_{\dot{a}_{p,q}^{s,\upsilon}(\mathbb{A})}.
\end{align*}
This gives the boundedness of $U$ on
$\dot{a}_{p,q}^{s,\upsilon}(\mathbb{A})$ and hence
finishes the proof of this case.

\emph{Case (2)} $\delta_2<\frac{1}{p}$.
In this case, we borrow some ideas from the proof
of \cite[Theorem 4.19]{bhyy5}. For brevity, we only present
some key estimates and necessary modifications.
By the argument used in the proof of Lemma \ref{lem-adbound},
we find that, to show the present theorem in this case,
it suffices to consider the case where $s=0$. To this end, let
$\beta_1\in[\![d_{p,\infty}^{\mathrm{lower}}(W),\infty)$
and $\beta_2\in[\![d_{p,\infty}^{\mathrm{upper}}(W),\infty)$
satisfy
\begin{align*}
D>\frac{n}{1\wedge\Gamma_{p,q}}+\widetilde{\Delta}+\frac{\beta_2}{p},\
E>\frac{n}{2}+\widetilde{\Delta},\text{ and }
F>\frac{n}{1\wedge\Gamma_{p,q}}-\frac{n}{2}+\frac{\beta_2}{p},
\end{align*}
where $\widetilde{\Delta}:=(n\delta_2-\frac{n}{p}+\frac{\beta_1}{p})_{+}$.
Assume that $r:=p \wedge q \wedge 1$ and
\begin{align}\label{e10}
\begin{cases}
\varepsilon:=1\wedge\Gamma_{p,q}& \text {when}\ a=b,\\
\varepsilon \in\left(0, 1\wedge\Gamma_{p,q}\right) & \text {when}\ a=f.
\end{cases}
\end{align}
Let $U:=\{{u}_{Q,R}\}_{Q, R\in\mathcal{D}}$
be a $(D,E,F)$-almost diagonal operator.
From \cite[Lemma 4.8]{bhyy5}, it follows that,
for any $P\in\mathcal{D}$ and $\vec{t}\in\dot{a}_{p,q}^{0,\upsilon}(W)$,
\begin{align}\label{e8}
&\left\|\left\{\left|W^{\frac{1}{p}}
\left(U\vec{t}\right)_j\right|
\mathbf{1}_P\mathbf{1}_{j\geq j_P}
\right\}_{j\in\mathbb{Z}}\right\|_{L\dot{A}_{p, q}}^r\\
&\quad\lesssim\sum_{k\in\mathbb{Z}}\sum_{l\in\mathbb{Z}_+}
\left[2^{-(E-\frac{n}{2}) k_{-}}2^{-k_{+}
(F+\frac{n}{2}-\frac{n}{\varepsilon})}
2^{-(D-\frac{n}{\varepsilon})l}\right]^r\nonumber\\
&\quad\quad\times\left\|\left\{\left[
\fint_{B(\cdot, 2^{l+k_{+}-i})}\left| W^{\frac{1}{p}}(\cdot)
\mathbf{1}_P(\cdot)\vec{t}_i(y)\right|^\varepsilon
\mathbf{1}_{i\geq j_P+k}\,dy\right]^{\frac{1}{\varepsilon}}
\right\}_{i\in\mathbb{Z}}\right\|^r_{L\dot{A}_{p, q}},\nonumber
\end{align}
where, for any $i\in\mathbb{Z}$, $\vec{t}_i$
is as in \eqref{vect_j}.
To estimate its right-hand side, let
$\mathbb{A}:=\{A_Q\}_{Q\in\mathcal{D}}$
be a sequence of reducing operators of order $p$ for $W$.
By Lemma \ref{lem-3P-sum1}(iv),
we find that, for any $k\in\mathbb{Z}$, $l\in\mathbb{Z}_+$,
$P\in\mathcal{D}$, and $i\in\{j_P+k, \ldots, j_P+k_{+}+l\}$,
\begin{align}\label{e12}
j_P+k_{+}+l-i \in\left[0, k_{-}+l\right]
\end{align}
and there exists unique $P_0\in\mathcal{D}_{-k_+-\ell+i}$
such that $P\subset P_0$. Applying some arguments similar
to those used in the proof of
\cite[Theorem 4.19]{bhyy5}, we conclude that,
for any $k\in\mathbb{Z}$, $l\in\mathbb{Z}_+$,
$P\in\mathcal{D}$, and
$\vec{t}\in\dot{a}_{p,q}^{0,\upsilon}(W)$,
\begin{align}\label{e9}
&\left\|\left\{\left[\fint_{B(\cdot, 2^{l+k_+-i})}
\left|W^{\frac{1}{p}}(\cdot) \mathbf{1}_P(\cdot)\vec{t}_i(y)
\right|^{\varepsilon}\mathbf{1}_{i\geq j_P+k}\,dy
\right]^{\frac{1}{\varepsilon}}\right\}_{i\in\mathbb{Z}}
\right\|_{L\dot{A}_{p, q}}^{r} \\\nonumber
&\quad\lesssim2^{(l+k_{+})\frac{\beta_2}{p}r}
\left\|\left\{\left|W^{\frac{1}{p}}
\vec{t}_i\right|\mathbf{1}_{3P}\mathbf{1}_{i\geq j_P}
\right\}_{i\in\mathbb{Z}}\right\|^r_{L\dot{A}_{p, q}}\\\nonumber
&\quad\quad+\sum_{i=j_P+k}^{j_P+k_{+}+l}
2^{\left(i-j_P-k_{+}-l\right)(\frac{n}{p}-\frac{\beta_1}{p})r}
2^{\left(k_{+}+l\right)\frac{\beta_2}{p}r}
\left[\int_{3P_0}\left|\mathbb{A}_i(y)
\vec{t}_i(y)\right|^p\,dy\right]^{\frac{r}{p}},\\
&\quad=:(\mathrm{I})^{r}+\sum_{i=j_P+k}^{j_P+k_{+}+l}
\left({\rm J}_i\right)^{r}.\nonumber
\end{align}

We first deal with I.
Applying the quasi-triangle inequality of
$\|\cdot\|_{L\dot{A}_{p, q}}$,
Lemmas \ref{lem-3P-sum1}(iii) and \ref{lem-grow-est}(iii),
and the definition of
$\|\cdot\|_{\dot{a}_{p,q}^{0,\upsilon}(W)}$, we obtain
\begin{align}\label{eq-I}
\mathrm{I}&\lesssim2^{(l+k_{+})\frac{\beta_2}{p}}
\sum_{h\in\mathbb{Z}^n,\, \|h\|_{\infty}\leq1}
\left\|\left\{\left|W^{\frac{1}{p}}
\vec{t}_i\right|\mathbf{1}_{P+h\ell(P)}\mathbf{1}_{i\geq j_P}
\right\}_{i\in\mathbb{Z}}\right\|_{L\dot{A}_{p, q}}\\
&\leq2^{(l+k_{+})\frac{\beta_2}{p}}
\sum_{h\in\mathbb{Z}^n,\, \|h\|_{\infty}\leq1}\upsilon(P+h\ell(P))
\left\|\vec{t}\right\|_{\dot{a}_{p,q}^{0,\upsilon}(W)}
\sim2^{\left(l+k_{+}\right)\frac{\beta_2}{p}}\upsilon(P)
\left\|\vec{t}\right\|_{\dot{a}_{p,q}^{0,\upsilon}(W)}\notag.
\end{align}
This establishes the desired estimate of I.

To estimate ${\rm J}_i$, from Corollary \ref{cor-a(A)=a(W)},
\eqref{e12}, Lemmas \ref{lem-3P-sum1}(iii) and \ref{lem-grow-est}(iii)
with $P$ replaced by $P_0$, Lemma \ref{lem-grow-est}(i) with
$Q$ and $P$ replaced, respectively, by $P$ and $P_0$,
and the definitions of both $\widetilde{\Delta}$
and $\|\cdot\|_{\dot{a}_{p,q}^{0,\upsilon}(\mathbb{A})}$,
it follows that, for any $k\in\mathbb{Z}$, $l\in\mathbb{Z}_+$,
$P\in\mathcal{D}$, and $i\in\{j_P+k, \ldots, j_P+k_{+}+l\}$,
\begin{align}\label{eq-J}
{\rm J}_i&\lesssim 2^{(i-j_P-k_{+}-l)(\frac{n}{p}
-\frac{\beta_1}{p})} 2^{(k_{+}+l)\frac{\beta_2}{p}}
\left[\sum_{h\in\mathbb{Z}^n,\, \|h\|_{\infty}\leq1}
\int_{P_0+h\ell(P_0)}\left|\mathbb{A}_i(y)\vec{t}_i(y)
\right|^p\,dy\right]^{\frac{1}{p}}\\
&\leq 2^{(i-j_P-k_{+}-l)(\frac{n}{p}
-\frac{\beta_1}{p})} 2^{(k_{+}+l)\frac{\beta_2}{p}}
\left\{\sum_{h\in\mathbb{Z}^n,\, \|h\|_{\infty}\leq1}
[\upsilon(P_0+h\ell(P_0))]^p\right\}^{\frac{1}{p}}
\left\|\vec{t}\right\|
_{\dot{a}_{p,q}^{0,\upsilon}(\mathbb{A})}\nonumber\\
&\sim 2^{(i-j_P-k_{+}-l)(\frac{n}{p}
-\frac{\beta_1}{p})} 2^{(k_{+}+l)
\frac{\beta_2}{p}}\upsilon\left(P_0\right)
\left\|\vec{t}\right\|
_{\dot{a}_{p,q}^{0,\upsilon}(\mathbb{A})}\nonumber\\
&\lesssim2^{(j_P+k_{+}+l-i)
(n\delta_2-\frac{n}{p}+\frac{\beta_1}{p})}
2^{(k_{+}+l)\frac{\beta_2}{p}}\upsilon(P)
\left\|\vec{t}\right\|_{\dot{a}_{p,q}^{0,\upsilon}(W)}\notag\\
& \leq 2^{(k_{-}+l)\widetilde{\Delta}}
2^{(k_{+}+l)\frac{\beta_2}{p}}\upsilon(P)
\left\|\vec{t}\right\|_{\dot{a}_{p,q}^{0,\upsilon}(W)}
=2^{k_{-}\widetilde{\Delta}}2^{l(\widetilde{\Delta}+\frac{\beta_2}{p})}
2^{k_{+}\frac{\beta_2}{p}}\upsilon(P)
\left\|\vec{t}\right\|_{\dot{a}_{p,q}^{0,\upsilon}(W)}.\notag
\end{align}
Combining both \eqref{eq-I} and \eqref{eq-J} with \eqref{e9},
we find that, for any $k\in\mathbb{Z}$, $l\in\mathbb{Z}_+$,
$P\in\mathcal{D}$, and $\vec{t}\in\dot{a}_{p,q}^{0,\upsilon}(W)$,
\begin{align}\label{eq-Ji}
&\left\|\left\{\left[\fint_{B(\cdot, 2^{l+k_+-i})}
\left|W^{\frac{1}{p}}(\cdot) \mathbf{1}_P(\cdot)\vec{t}_i(y)
\right|^{\varepsilon}\mathbf{1}_{i\geq j_P+k}\,dy
\right]^{\frac{1}{\varepsilon}}\right\}_{i\in\mathbb{Z}}
\right\|_{L\dot{A}_{p, q}}^{r}\\
&\quad\lesssim(\mathrm{I})^{r}+\sum_{i=j_P+k}^{j_P+k_{+}+l}
\left(J_i\right)^{r}\lesssim\left(2+k_{-}+l\right)
2^{k_{-}\widetilde{\Delta}r}
2^{l(\widetilde{\Delta}+\frac{\beta_2}{p})r}
2^{k_{+}\frac{\beta_2}{p}r}\left[\upsilon(P)\right]^r\left\|\vec{t}\right\|^r
_{\dot{a}_{p,q}^{0,\upsilon}(W)}.\nonumber
\end{align}
Inserting \eqref{eq-Ji} into \eqref{e8} and using the definition of
$\|\cdot\|_{\dot{a}_{p,q}^{0,\upsilon}(W)}$, we obtain,
for any $\vec{t}\in\dot{a}_{p,q}^{0,\upsilon}(W)$,
\begin{align}\label{eq-sum-k-l}
\left\|U\vec{t}\right\|_{\dot{a}_{p,q}^{0,\upsilon}(W)}^r
&=\sup_{P\in\mathcal{D}}\frac{1}{[\upsilon(P)]^r}
\left\|\left\{\left|H_j\left(U\vec{t}\right)_j\right|
\right\}_{j\in\mathbb{Z}}\right\|^r_{L\dot{A}_{p, q}}\\
&\lesssim\sum_{k\in\mathbb{Z}}\sum_{l\in\mathbb{Z}_+}
\left(2+k_{-}+l\right)2^{-k_{-}(E-\frac{n}{2}
-\widetilde{\Delta})r} 2^{-k_{+}
(F+\frac{n}{2}-\frac{n}{\varepsilon}
-\frac{\beta_2}{p})r}2^{-l(D-\frac{n}{\varepsilon}
-\widetilde{\Delta}-\frac{\beta_2}{p})r}\left\|\vec{t}
\right\|^r_{\dot{a}_{p,q}^{0,\upsilon}(W)}.\nonumber
\end{align}

Finally, it remains to determine the conditions on
$D,E,F$ such that the right-hand side of \eqref{eq-sum-k-l} converges.
Obviously, we find that the right-hand side of \eqref{eq-sum-k-l}
converges if and only if
\begin{align}\label{convergecond1}
D>\frac{n}{\varepsilon}
+\widetilde{\Delta} +\frac{\beta_2}{p},\
E>\frac{n}{2}+\widetilde{\Delta},\text{ and }
F>-\frac{n}{2}+\frac{n}{\varepsilon}+\frac{\beta_2}{p}.
\end{align}
Notice that we can choose $\varepsilon$ in \eqref{e10}
to be sufficiently close to $1\wedge\Gamma_{p,q}$
such that \eqref{convergecond1} is satisfied.
Thus, we conclude that, for any $\vec{t}
\in\dot{a}_{p,q}^{0,\upsilon}(W)$, $\|U\vec{t}\|
_{\dot{a}_{p,q}^{0,\upsilon}(W)}\lesssim
\|\vec{t}\|_{\dot{a}_{p,q}^{0,\upsilon}(W)}$.
This finishes the proof of this case and hence
Theorem \ref{a(W)adopebound}.
\end{proof}

Motivated by Theorem \ref{a(W)adopebound}, we
introduce the following concept of
$\dot{a}_{p,q}^{s,\upsilon}(W)$-almost diagonal matrices.
\begin{definition}\label{a(W)adope}
Let $a\in\{b, f\}$, $s\in\mathbb{R}$, $p\in(0,\infty)$,
$q\in(0, \infty]$, and $W\in \mathcal{A}_{p,\infty}$.
Suppose that $\delta_1,\delta_2,\omega$ satisfy
\eqref{eq-delta1<0} and
$\upsilon\in\mathcal{G}(\delta_1, \delta_2; \omega)$.
An infinite matrix $U:=\{u_{Q,R}\}
_{Q,R\in\mathcal{D}}$ in $\mathbb{C}$ is said to be
$\dot{a}_{p,q}^{s,\upsilon}(W)$-\emph{almost diagonal}
if it is $(D, E, F)$-almost diagonal with
$D, E, F$ satisfying \eqref{DEFa(W)adope}.
\end{definition}

Next, we prove that the
class of matrices in Definition \ref{a(W)adope}
is closed under compositions.

\begin{proposition}\label{prop-compadope}
Let $a\in\{b, f\}$, $s\in\mathbb{R}$, $p\in(0,\infty)$,
$q\in(0, \infty]$, and $W\in \mathcal{A}_{p,\infty}$.
Assume that $\delta_1,\delta_2,\omega$ satisfy \eqref{eq-delta1<0} and
$\upsilon\in\mathcal{G}(\delta_1, \delta_2; \omega)$.
Suppose that infinite matrices $u^{(1)}:=\{u^{(1)}_{Q,R}\}
_{Q,R\in\mathcal{D}}$ and $u^{(2)}
:=\{u^{(2)}_{Q,R}\}_{Q,R\in\mathcal{D}}$ are both
$\dot{a}_{p,q}^{s,\upsilon}(W)$-almost diagonal.
Then the infinite matrix
$u:=u^{(1)}\circ u^{(2)}:=\{\sum_{P\in\mathcal{D}}u^{(1)}_{Q,P}
u^{(2)}_{P,R}\}_{Q,R\in\mathcal{D}}$
is also $\dot{a}_{p,q}^{s,\upsilon}(W)$-almost diagonal.
\end{proposition}

\begin{proof}
Since $u^{(1)}$ and $u^{(2)}$ are
$\dot{a}_{p,q}^{s,\upsilon}(W)$-almost diagonal,
there exist $D_1, E_1, F_1, D_2, E_2,F_2$
satisfying \eqref{DEFa(W)adope} such that $u^{(1)}$
is $(D_1, E_1, F_1)$-almost diagonal
and $u^{(2)}$ is $(D_2, E_2, F_2)$-almost diagonal.
Without loss of generality, by \eqref{DEFa(W)adope},
we may assume that $D_1=D_2$, $E_1\neq E_2$, $F_1\neq F_2$,
$E_1+F_2>D_1$, and $E_2+F_1>D_1$.
Repeating the argument used in the proof of
\cite[Theorem D.2]{fj90} with $J+\beta$, $\frac{n+\gamma_1}{2}$,
$\frac{n+\gamma_2}{2}$, $\frac{n+\gamma_1}{2}+J-n$, and
$\frac{n+\gamma_2}{2}+J-n$ replaced, respectively,
by $D_1,E_1,E_2,F_1$, and $F_2$, we conclude that
$U^{D_1E_1F_1}\circ U^{D_2E_2F_2}$
is $(D_1, \min\{E_1, E_2\}, \min\{F_1, F_2\})$-almost diagonal.
From this, Definitions \ref{DEFadope} and \ref{a(W)adope},
the assumptions that $u^{(1)}$ is $(D_1, E_1, F_1)$-almost diagonal
and $u^{(2)}$ is $(D_2, E_2, F_2)$-almost diagonal,
and the above choice of $D_1, E_1, F_1, D_2, E_2, F_2$,
we deduce that, for any $Q, R\in\mathcal{D}$,
\begin{align*}
\left|\sum_{P\in\mathcal{D}}u^{(1)}_{Q,P}
u^{(2)}_{P,R}\right|\leq	\sum_{P\in\mathcal{D}}
\left|u^{(1)}_{Q,P}u^{(2)}_{P,R}\right|
\lesssim\sum_{P\in\mathcal{D}}u^{D_1E_1F_1}_{Q,P}
u^{D_2E_2F_2}_{P,R}\lesssim
u^{D_1\min\{E_1, E_2\}\min\{F_1, F_2\}}_{Q,R},
\end{align*}
which further implies that $u$ is
$(D_1,\min\{E_1, E_2\},\min\{F_1, F_2\})$-almost diagonal
and hence $\dot{a}_{p,q}^{s,\upsilon}(W)$-almost diagonal.
This finishes the proof of Proposition \ref{prop-compadope}.
\end{proof}

At the end of this subsection, we prove that,
if $\upsilon$ is an almost increasing function,
then the boundedness of almost diagonal operators
on $\dot{a}_{p,q}^{s,\upsilon}(W)$
guarantees that $\upsilon$ is a growth function.

\begin{proposition}\label{prop-necessity}
Let $a\in\{b, f\}$, $s, D, E, F\in\mathbb{R}$,
$p\in(0, \infty)$, $q\in(0,\infty]$, and $W \in\mathcal{A}_{p,\infty}$. 	
Suppose that $\upsilon$ is an almost increasing function,
$\beta_1 \in \llbracket d_{p, \infty}^{\mathrm{lower}}(W),\infty)$,
and $\beta_2\in \llbracket d_{p, \infty}^{\mathrm{upper}}(W),\infty)$,
where $d^{\operatorname{lower}}_{p, \infty}(W)$ and
$d^{\operatorname{ upper}}_{p, \infty}(W)$ are as, respectively, in
\eqref{eq-low-dim} and \eqref{eq-upp-dim}.
If any $(D, E, F)$-almost diagonal operator
is bounded on $\dot{a}_{p,q}^{s,\upsilon}(W)$,
then there exists a positive constant $C$ such that,
for any $Q, R\in\mathcal{D}$,
\begin{align}\label{eq-QR-necessity}
\frac{\upsilon(Q)}{\upsilon(R)}\leq C
\left[1+\frac{|x_Q-x_R|}
{\ell(Q)\vee \ell(R)}\right]^{D+\frac{\beta_1+\beta_2}{p}}
\begin{cases}
\displaystyle{\left(\frac{|Q|}{|R|}\right)^{-\frac{s}{n}-
\frac{1}{2}+\frac{1}{p}-\frac{F}{n}-\frac{\beta_1}{np}}}
& \text{if } \ell(Q) \leq \ell(R), \\
\displaystyle{\left(\frac{|Q|}{|R|}\right)^{-\frac{s}{n}-
\frac{1}{2}+\frac{1}{p}+\frac{E}{n}+\frac{\beta_2}{np}}}
& \text{if } \ell(R)<\ell(Q),
\end{cases}
\end{align}
that is, $\upsilon\in\mathcal{G}
(-\frac{s}{n}-\frac{1}{2}+\frac{1}{p}-\frac{F}{n}
-\frac{\beta_1}{np}, -\frac{s}{n}-\frac{1}{2}+\frac{1}{p}
+\frac{E}{n}+\frac{\beta_2}{np}; D+\frac{\beta_1+\beta_2}{p})$.
\end{proposition}

\begin{proof}
Let $\mathbb{A}:=\{A_Q\}_{Q\in\mathcal{D}}$
be a sequence of reducing operators of order $p$ for $W$.
For any given $Q, R\in\mathcal{D}$ and for any
$\vec{e}\in\mathbb{C}^m$ with $|\vec{e}|=1$,
we define the sequence $\vec{t}
:=\{\vec{t}_P\}_{P\in\mathcal{D}}$
by setting, for any $P\in\mathcal{D}$,
$\vec{t}_P:=A^{-1}_{R}\vec{e}$ if $P=Q$
and $\vec{t}_P:=\mathbf{0}$ otherwise.
From Corollary \ref{cor-a(A)=a(W)}, it follows that
$\dot{a}_{p,q}^{s,\upsilon}(W)=\dot{a}_{p,q}^{s,\upsilon}(\mathbb{A})$
with equivalent quasi-norms. By this and the
assumption that the $(D, E, F)$-almost diagonal operator $U^{DEF}$
is bounded on $\dot{a}_{p,q}^{s,\upsilon}(W)$,
we find that $U^{DEF}$ is also bounded on
$\dot{a}_{p,q}^{s,\upsilon}(\mathbb{A})$.
This, together with Definition \ref{DEFadope}, the definition of
$\|\cdot\|_{\dot{a}_{p,q}^{s,\upsilon}(\mathbb{A})}$, and
the assumption that $\upsilon$ is an almost increasing function,
further implies that, for any $\vec{e}\in\mathbb{C}^m$ with $|\vec{e}|=1$,
\begin{align*}
\frac{|R|^{-\frac{s}{n}-\frac{1}{2}+\frac{1}{p}}}
{\upsilon(R)}u_{R,Q}^{D E F}
&\leq\left\|U^{D E F}\vec{t}\right\|
_{\dot{a}_{p,q}^{s,\upsilon}(\mathbb{A})}
\lesssim\left\|\vec{t}\right\|
_{\dot{a}_{p,q}^{s,\upsilon}(\mathbb{A})}\\
&=\sup_{P\in\mathcal{D}, P\supset Q}
\frac{|Q|^{-\frac{s}{n}-\frac{1}{2}
+\frac{1}{p}}}{\upsilon(P)}\left|A_QA^{-1}_{R}\vec{e}\right|
\sim\frac{|Q|^{-\frac{s}{n}-\frac{1}{2}
+\frac{1}{p}}}{\upsilon(Q)}\left|A_Q A^{-1}_{R}\vec{e}\right|,
\end{align*}
where all the implicit positive constants are
independent of $Q, R$, and $\vec{e}$. Taking the supremum over all $\vec{e}\in\mathbb{C}^m$ with $|\vec{e}|=1$
on its right-hand side and applying the definition of
operator norms, we obtain, for any $Q, R\in\mathcal{D}$,
\begin{align*}
\frac{|R|^{-\frac{s}{n}-\frac{1}{2}+\frac{1}{p}}}
{\upsilon(R)}u_{R,Q}^{D E F}
\lesssim\frac{|Q|^{-\frac{s}{n}-\frac{1}{2}+\frac{1}{p}}}
{\upsilon(Q)}\sup_{\vec z\in\mathbb{C}^m, |\vec z|=1}
\left|A_Q A^{-1}_{R}\vec{e}\right|
=\frac{|Q|^{-\frac{s}{n}-\frac{1}{2}
+\frac{1}{p}}}{\upsilon(Q)}\left\|A_Q A^{-1}_{R}\right\|.
\end{align*}
By this, \eqref{DEFmatrix}, and Lemma \ref{growEST},
we find that \eqref{eq-QR-necessity} holds.
This finishes the proof of Proposition \ref{prop-necessity}.
\end{proof}

\begin{remark}
Using Proposition \ref{prop-necessity}, we conclude that,
under the mild assumption that $\upsilon$ is an
almost increasing function, the growth condition
on $\upsilon$ as in Definition \ref{def-grow-func}
is necessary for the boundedness of
almost diagonal operators on $\dot{a}_{p,q}^{s,\upsilon}(W)$.
\end{remark}

\subsection{Molecular and Wavelet Characterizations
of $\dot{A}_{p,q}^{s, \upsilon}(W)$\label{s-cmw-mw}}

Based on Theorems \ref{thm-phitransMWBTL}
and \ref{a(W)adopebound}, we aim to
establish the molecular and the wavelet
characterizations of $\dot{A}^{s,\upsilon}_{p,q}(W)$.
To begin with, we give some notation.
For any $r\in\mathbb{R}$, let
\begin{align}\label{eq-ceil}
\begin{cases}
\lceil\!\lceil r\rceil\!\rceil:=\min\{k\in\mathbb Z:\ k>r\},\
\lceil r\rceil:=\min\{k\in\mathbb Z:\ k\geq r\},\\
\lfloor\!\lfloor r\rfloor\!\rfloor:=\max\{k\in\mathbb Z:\ k< r\},\
\lfloor r\rfloor:=\max\{k\in\mathbb Z:\  k\leq r\}
\end{cases}
\end{align}
and
\begin{align}\label{eq-r**}
r^{**}:=r-\lfloor\!\lfloor r\rfloor\!\rfloor\in(0,1].
\end{align}
For any $K\in[0,\infty)$ and $x\in\mathbb{R}^n$, let
$u_K(x):=(1+|x|)^{-K}$.

Next, we recall the concept of smooth molecules
introduced in \cite[Definition 3.4]{bhyy3},
which is a slight generalization of the conventional
one as in \cite[(3.7)-(3.10)]{fj90}.

\begin{definition}\label{def-mole}
Let $K,M\in[0,\infty)$, $L,N\in\mathbb{R}$, and
$Q\in\mathcal{D}$. A function $m_Q$ on $\mathbb{R}^n$
is called a \emph{(smooth) $(K,L,M,N)$-molecule supported near $Q$}
if, for any $x,y\in\mathbb{R}^n$, it satisfies
\begin{itemize}
\item[{\rm (i)}] $|m_Q(x)|\leq(u_{K})_Q(x)$;
\item[{\rm (ii)}] $\int_{\mathbb R^n} m_Q(x)x^\gamma\,dx=0
\text{ if }\gamma\in\mathbb{Z}_+^n$ and $|\gamma|\leq L$;
\item[{\rm (iii)}] $|\partial^\gamma m_Q(x)|
\leq[\ell(Q)]^{-|\gamma|}(u_{M})_Q(x)\text{ if }
\gamma\in\mathbb{Z}_+^n$ and $|\gamma|<N$;
\item[{\rm (iv)}]
$$\left|\partial^\gamma m_Q(x)-\partial^\gamma m_Q(y)\right|
\leq\left[\ell(Q)\right]^{-|\gamma|}\left[\frac{|x-y|}{\ell(Q)}\right]^{N^{**}}
\sup_{|z|\leq|x-y|}(u_{M})_Q(x+z)$$
if $\gamma\in\mathbb{Z}_+^n$ and $|\gamma|=\lfloor\!\lfloor N\rfloor\!\rfloor$,
where $\lfloor\!\lfloor N\rfloor\!\rfloor$ and $N^{**}$
are as, respectively, in \eqref{eq-ceil} and \eqref{eq-r**}
and $(u_{M})_Q$ is as in \eqref{eq-phi_Q} with $\varphi$
replaced by $u_{M}$.
\end{itemize}
For brevity, we also call $m_Q$ a $(K, L, M, N)$-molecule.
\end{definition}

The following lemma is precisely \cite[Lemma 3.7]{bhyy3}.

\begin{lemma}\label{lem-MGH}
Let $K_m, K_g, M_m, M_g\in(n,\infty)$, $L_m, L_g, N_m, N_g\in\mathbb{R}$,
and $Q, R\in\mathcal{D}$, and let $m_Q$ be a $(K_m, L_m, M_m, N_m)$-molecule
and $g_R$ a $(K_g, L_g, M_g, N_g)$-molecule.
Then, for any $\alpha \in(0, \infty)$,
there exists a positive constant $C$, independent of
$Q$ and $R$, such that $|\langle m_Q, g_R
\rangle|\leq C u_{Q,R}^{MGH}$,
where $\langle\cdot, \cdot\rangle$ is the usual inner product
in $L^2$, $u_{Q,R}^{MGH}$ is as in \eqref{DEFmatrix},
$M:=K_m \wedge M_m \wedge K_g \wedge M_g$,
$G:=\frac{n}{2}+[N_g\wedge\lceil\!\lceil L_m\rceil\!
\rceil\wedge(K_m-n-\alpha)]_{+}$,
and
$H:=\frac{n}{2}+[N_m \wedge\lceil\!\lceil L_g\rceil\!\rceil
\wedge(K_g-n-\alpha)]_{+}$.
\end{lemma}

Motivated by the above lemma,
we introduce two kinds of sequences
of molecules for $\dot{A}_{p,q}^{s,\upsilon}(W)$.

\begin{definition}\label{def-anasynmole}
Let $(A, a)\in\{(B, b), (F, f)\}$, $s\in\mathbb{R}$,
$p\in(0,\infty)$, $q\in(0, \infty]$, and $W\in \mathcal{A}_{p,\infty}$.
Assume that $\delta_1,\delta_2,\omega$ satisfy \eqref{eq-delta1<0},
$\upsilon\in\mathcal{G}(\delta_1, \delta_2; \omega)$,
and $D_{\dot{a}_{p,q}^{s,\upsilon}(W)}, E_{\dot{a}_{p,q}^{s,\upsilon}(W)},
F_{\dot{a}_{p,q}^{s,\upsilon}(W)}$ are as in Theorem \ref{a(W)adopebound}.

We call $m:=\{m_Q\}_{Q\in\mathcal{D}}$ a \emph{family of analysis molecules}
for $\dot{A}_{p,q}^{s,\upsilon}(W)$ if there exist
\begin{align*}
K_m>D_{\dot{a}_{p,q}^{s,\upsilon}(W)}\vee
\left[E_{\dot{a}_{p,q}^{s,\upsilon}(W)}+\frac{n}{2}\right],
\ L_m \geq E_{\dot{a}_{p,q}^{s,\upsilon}(W)}-\frac{n}{2},
\ M_m>D_{\dot{a}_{p,q}^{s,\upsilon}(W)},
\text{ and } N_m>F_{\dot{a}_{p,q}^{s,\upsilon}(W)}-\frac{n}{2}	
\end{align*}
such that, for any $Q\in\mathcal{D}$,
$m_Q$ is a $(K_m, L_m, M_m, N_m)$-molecule.

We call $g:=\{g_Q\}_{Q\in\mathcal{D}}$ a
\emph{family of synthesis molecules}
for $\dot{A}_{p,q}^{s,\upsilon}(W)$ if there exist
\begin{align*}
K_g>D_{\dot{a}_{p,q}^{s,\upsilon}(W)}\vee
\left[F_{\dot{a}_{p,q}^{s,\upsilon}(W)}+\frac{n}{2}\right],
\ L_g\geq F_{\dot{a}_{p,q}^{s,\upsilon}(W)}-\frac{n}{2},
\ M_g>D_{\dot{a}_{p,q}^{s,\upsilon}(W)},
\text{ and } N_g>E_{\dot{a}_{p,q}^{s,\upsilon}(W)}-\frac{n}{2}
\end{align*}
such that, for any $Q\in\mathcal{D}$,
$g_Q$ is a $(K_g, L_g, M_g, N_g)$-molecule.

In particular, for any $Q\in\mathcal{D}$, $m_Q$ (resp. $g_Q$)
is called an \emph{analysis} (resp. a \emph{synthesis}) \emph{molecule}
for $\dot{A}_{p,q}^{s,\upsilon}(W)$.
\end{definition}

\begin{remark}\label{rmk-mole-phi}
Let all the symbols be the same as in
Definition \ref{def-anasynmole}. For any
$\varphi\in\mathcal{S}_{\infty}$,
it is easy to verify that there exists a positive
constant $C$ such that $\{C\varphi_{Q}\}_{Q\in\mathcal{D}}$
is a family of both analysis and synthesis molecules
for $\dot{A}_{p,q}^{s,\upsilon}(W)$, where,
for any $Q\in\mathcal{D}$, $\varphi_{Q}$
is as in \eqref{eq-phi_Q};
we omit the details.
\end{remark}

The following lemma directly follows from Definition \ref{def-anasynmole},
Lemma \ref{lem-MGH}, Theorem \ref{a(W)adopebound},
and Proposition \ref{prop-compadope}; we omit the details.

\begin{lemma}\label{anasynmole}
Let $(A, a)\in\{(B, b), (F, f)\}$, $s\in\mathbb{R}$, $p\in(0, \infty)$,
$q\in(0,\infty]$,  $W\in\mathcal{A}_{p, \infty}$,
and $\varphi, \psi\in\mathcal{S}$ satisfy \eqref{cond1} and \eqref{cond3}.
Suppose that $\delta_1,\delta_2,\omega$ satisfy \eqref{eq-delta1<0} and
$\upsilon\in\mathcal{G}(\delta_1, \delta_2; \omega)$.
Assume that, for any $i\in\{1,2\}$, $\{m_Q^{(i)}\}_{Q\in\mathcal{D}}$ and
$\{g_R^{(i)}\}_{R \in \mathcal{D}}$ are respectively families
of analysis and synthesis molecules for $\dot{A}_{p,q}^{s,\upsilon}(W)$.
Then the following statements hold.
\begin{itemize}
\item[{\rm (i)}] For any $i \in\{1,2\}$,
$\{\langle m_Q^{(i)}, g_R^{(i)}\rangle\}_{Q, R\in\mathcal{D}}$
is $\dot{a}_{p,q}^{s,\upsilon}(W)$-almost diagonal.
\item[{\rm (ii)}] If $\vec{\lambda}:=\{\vec{\lambda}_P
\}_{P\in\mathcal{D}} \in \dot{a}_{p,q}^{s,\upsilon}(W)$,
then, for any $Q\in\mathcal{D}$,
$\vec{t}_Q:=\sum_{P\in\mathcal{D}}\sum_{R\in\mathcal{D}}
\langle m_Q^{(1)}, g_R^{(1)}\rangle
\langle m_R^{(2)}, g_P^{(2)}\rangle\vec{\lambda}_P$
converges absolutely and $\|\{\vec{t}_Q\}
_{Q\in\mathcal{D}}\|_{\dot{a}_{p,q}^{s,\upsilon}(W)}\lesssim
\|\vec{\lambda}\|_{\dot{a}_{p,q}^{s,\upsilon}(W)}$,
where the positive constant is independent of $\vec{\lambda}$.
\end{itemize}
\end{lemma}

In general, for any $\vec{f}\in\dot{A}_{p,q}^{s,\upsilon}(W)$
and $Q\in\mathcal{D}$, an analysis molecule $m_Q$ for
$\dot{A}_{p,q}^{s,\upsilon}(W)$
may not belong to $\mathcal{S}_{\infty}$ and hence
the conventional definition of $\langle\vec{f}, m_Q\rangle$
regarding $\vec{f}\in(\mathcal{S}'_{\infty})^m$ may fail.
However, the following technical lemma
gives a suitable way to define $\langle\vec{f}, m_Q\rangle$.

\begin{lemma}\label{wdfmole}
Let $A\in\{B, F\}$, $s\in\mathbb{R}$, $p\in(0, \infty)$,
$q\in(0,\infty]$, and $W\in\mathcal{A}_{p, \infty}$.
Suppose that $\delta_1,\delta_2,\omega$ satisfy \eqref{eq-delta1<0},
$\upsilon\in\mathcal{G}(\delta_1, \delta_2; \omega)$, and
$\varphi, \psi\in\mathcal{S}$ satisfy \eqref{cond1} and \eqref{cond3}.
If $\{m_Q\}_{Q\in\mathcal{D}}$ is a family of analysis molecules
for $\dot{A}_{p,q}^{s,\upsilon}(W)$, then, for
any $\vec{f}\in\dot{A}_{p,q}^{s,\upsilon}(W)$ and $Q\in\mathcal{D}$,
\begin{align}\label{eq-f-m}
\left\langle\vec{f}, m_Q\right\rangle_*
:=\sum_{R\in\mathcal{D}}\left\langle \psi_R,
m_Q\right\rangle\left\langle\vec{f},\varphi_R\right\rangle	
\end{align}
converges absolutely and its value is independent of
the choice of $\varphi$ and $\psi$.
\end{lemma}

\begin{proof}
We first prove that, for any $\vec{f}\in\dot{A}_{p,q}^{s,\upsilon}(W)$
and $Q\in\mathcal{D}$, \eqref{eq-f-m} converges absolutely.
It follows from Remark \ref{rmk-mole-phi} and
Lemma \ref{anasynmole}(i) that $\{\langle m_Q,
\psi_R\rangle\}_{Q,R\in\mathcal{D}}$
is $\dot{a}_{p,q}^{s,\upsilon}(W)$-almost diagonal.
By this and Theorems \ref{a(W)adopebound}
and \ref{thm-phitransMWBTL}, we find that
$\{\langle\vec{f},\varphi_R\rangle\}_{R\in\mathcal{D}}
\in\dot{a}_{p,q}^{s,\upsilon}(W)$
and hence the summation in \eqref{eq-f-m}
converges absolutely. To verify that \eqref{eq-f-m} is
independent of the choice of $\varphi$ and $\psi$,
assume that another pair $\Phi, \Psi\in\mathcal{S}$
also satisfy \eqref{cond1} and \eqref{cond3}.
Applying \eqref{eq-f-m}, Remark \ref{rmk-mole-phi},
Theorem \ref{thm-phitransMWBTL}, Lemma \ref{lem-Cdreproform},
and Fubini's theorem [Lemma \ref{anasynmole}(ii)
guarantees the absolute convergence of
the following double summations], we obtain,
for any $\vec{f}\in\dot{A}_{p,q}^{s,\upsilon}(W)$
and $Q\in\mathcal{D}$,
\begin{align*}
\left\langle\vec{f}, m_Q\right\rangle_*
&=\sum_{R\in\mathcal{D}}\left\langle \psi_R,
m_Q\right\rangle\left\langle\vec{f},\varphi_R\right\rangle
=\sum_{R\in\mathcal{D}}\sum_{P\in\mathcal{D}}
\left\langle \psi_R,m_Q\right\rangle\left\langle\Psi_P,
\varphi_R\right\rangle\left\langle\vec{f},
\Phi_P\right\rangle\\
&=\sum_{P\in\mathcal{D}}\sum_{R\in\mathcal{D}}
\left\langle \psi_R,m_Q\right\rangle\left\langle\Psi_P,
\varphi_R\right\rangle\left\langle\vec{f},
\Phi_P\right\rangle=\sum_{P\in\mathcal{D}}
\left\langle\Psi_P, m_Q\right\rangle\left\langle\vec{f},
\Phi_P\right\rangle,
\end{align*}
which further implies that \eqref{eq-f-m} is independent
of the choice of $\varphi$ and $\psi$ and hence
completes the proof of Lemma \ref{wdfmole}.
\end{proof}

We now give the molecular characterization
of $\dot{A}_{p,q}^{s,\upsilon}(W)$.

\begin{theorem}\label{moledecomp}
Let $(A, a)\in\{(B, b), (F, f)\}$, $s\in\mathbb{R}$, $p\in(0, \infty)$,
$q\in(0,\infty]$, and $W\in \mathcal{A}_{p,\infty}$.
Suppose that $\delta_1,\delta_2,\omega$ satisfy \eqref{eq-delta1<0} and
$\upsilon\in\mathcal{G}(\delta_1, \delta_2; \omega)$.
Then the following statements hold.
\begin{itemize}
\item[{\rm (i)}]
If $\{m_Q\}_{Q\in\mathcal{D}}$ is a family of analysis molecules
for $\dot{A}^{s,\upsilon}_{p,q}(W)$,
then there exists a positive constant $C$ such that,
for any $\vec{f}\in \dot{A}^{s,\upsilon}_{p,q}(W)$,
$\|\{\langle\vec f,m_Q\rangle_*\}_{Q\in\mathcal{D}}
\|_{\dot{a}^{s,\upsilon}_{p,q}(W)}
\leq C\|\vec{f}\|_{\dot{A}^{s,\upsilon}_{p,q}(W)}$,
where $\langle\cdot, \cdot\rangle_*$ is as in \eqref{eq-f-m}.
\item[{\rm (ii)}]
If $\{g_Q\}_{Q\in\mathcal{D}}$ is a family of synthesis molecules
for $\dot{A}^{s,\upsilon}_{p,q}(W)$,
then, for any $\vec{t}:=\{\vec{t}_Q\}_{Q\in\mathcal{D}}
\in \dot{a}^{s,\upsilon}_{p,q}(W)$,
$\vec{f}:=\sum_{Q\in\mathcal{D}}\vec{t}_Q
g_Q \in (\mathcal{S}'_{\infty})^m$
and there exists a positive constant $C$,
independent of $\vec{\lambda}$, such that
$\|\vec f\|_{\dot{A}^{s,\upsilon}_{p,q}(W)}
\leq C\|\vec{\lambda}\|_{\dot{a}^{s,\upsilon}_{p,q}(W)}$.
\end{itemize}
\end{theorem}

\begin{proof}
We first prove (i). To this end, let $\varphi, \psi\in\mathcal{S}$
satisfy \eqref{cond1} and \eqref{cond3}.
By Remark \ref{rmk-mole-phi} and Lemma \ref{anasynmole},
we find that $\{\langle\psi_R, m_Q\rangle\}_{Q, R\in \mathcal{D}}$ is
$\dot{a}_{p,q}^{s,\upsilon}(W)$-almost diagonal.
From this, \eqref{eq-f-m}, and Theorems
\ref{thm-phitransMWBTL} and \ref{a(W)adopebound}, it follows that,
for any $\vec{f}\in \dot{A}^{s,\upsilon}_{p,q}(W)$,
\begin{align*}
\left\|\left\{\left\langle\vec f, m_Q\right\rangle_*
\right\}_{Q\in\mathcal{D}}\right\|_{\dot{a}^{s,\upsilon}_{p,q}(W)}
&=\left\|\left\{\sum_{R\in\mathcal{D}}\left\langle \psi_R,
m_Q\right\rangle\left\langle\vec{f},\varphi_R\right\rangle	\right\}_{Q\in\mathcal{D}}\right\|_{\dot{a}^{s,\upsilon}_{p,q}(W)}\\
&\lesssim\left\|\left\{\left\langle\vec f, \varphi_R
\right\rangle\right\}_{R\in\mathcal{D}}\right\|
_{\dot{a}^{s,\upsilon}_{p,q}(W)}\lesssim
\left\|\vec{f}\right\|_{\dot{A}^{s,\upsilon}_{p,q}(W)},
\end{align*}
which completes the proof of (i).

Next, we show (ii). To do this, we prove that,
for any $\vec{t}:=\{\vec{t}_Q\}_{Q\in\mathcal{D}}
\in\dot{a}^{s,\upsilon}_{p,q}(W)$,
$\vec f:=\sum_{Q\in\mathcal{D}}\vec{t}_Qg_Q$
is well defined in $(\mathcal{S}'_{\infty})^m$
and $\|\vec{f}\|_{\dot{A}^{s,\upsilon}_{p,q}(W)}
\lesssim\|\vec{t}\|_{\dot{a}^{s,\upsilon}_{p,q}(W)}$.
Let $\vec{t}:=\{\vec{t}_Q\}_{Q\in\mathcal{D}}
\in \dot{a}^{s,\upsilon}_{p,q}(W)$.
From Lemma \ref{lem-Cdreproform}, we infer
that, for any $\phi\in\mathcal{S}_{\infty}$,
\begin{align}\label{eq-f-phi}
\left\langle\vec{f}, \phi\right\rangle:=
\sum_{Q\in\mathcal{D}}\vec{t}_Q
\left\langle g_Q, \phi\right\rangle=
\sum_{Q\in\mathcal{D}}\sum_{R\in\mathcal{D}}
\vec{t}_Q\left\langle g_Q,\varphi_R\right\rangle
\left\langle \psi_R,\phi\right\rangle.
\end{align}
By Remark \ref{rmk-mole-phi}, Theorem \ref{a(W)adopebound},
and Lemma \ref{anasynmole}, we find that
$$\vec{u}:=\{\vec{u}_R\}_{R\in\mathcal{D}}
:=\left\{\sum_{Q\in\mathcal{D}}\vec{t}_Q\langle
g_Q, \varphi_R\rangle\right\}_{R\in\mathcal{D}}
\in\dot{a}^{s,\upsilon}_{p,q}(W)$$
and the right-hand side of \eqref{eq-f-phi} converges absolutely
[because $\phi=\phi_{Q_{0,\mathbf{0}}}$ is a
constant multiple of an analysis molecule
for $\dot{A}_{p,q}^{s,\upsilon}(W)$].
This, combined with \eqref{eq-f-phi}, Theorem \ref{thm-phitransMWBTL},
Fubini's theorem, and the definition of $\vec{t}$,
further implies that $\sum_{R\in\mathcal{D}}\vec{u}_R
\psi_R$ converges in $(\mathcal{S}'_{\infty})^m$ and hence,
for any $\phi\in\mathcal{S}_{\infty}$,
\begin{align*}
\left\langle\vec{f}, \phi\right\rangle
=\sum_{R\in\mathcal{D}}\sum_{Q\in\mathcal{D}}
\vec{t}_Q\left\langle g_Q,\varphi_R\right\rangle
\left\langle \psi_R,\phi\right\rangle
=\sum_{R\in\mathcal{D}}\vec{u}_R
\left\langle \psi_R,\phi\right\rangle
=\left\langle\sum_{R\in\mathcal{D}}\vec{u}_R
\psi_R,\phi\right\rangle.
\end{align*}
Using this and Theorems \ref{thm-phitransMWBTL}
and \ref{a(W)adopebound}, we conclude that
$\vec{f}=\sum_{R\in\mathcal{D}}\vec{u}_R
\psi_R$ in $(\mathcal{S}'_{\infty})^m$ and
\begin{align*}
\left\|\vec{f}\right\|_{\dot{A}^{s,\upsilon}_{p,q}(W)}
=\left\|\sum_{R\in\mathcal{D}}\vec{u}_R\psi_R
\right\|_{\dot{A}^{s,\upsilon}_{p,q}(W)}
\lesssim\left\|\vec{u}\right\|_{\dot{a}^{s,\upsilon}_{p,q}(W)}
\lesssim\left\|\vec{t}\right\|_{\dot{a}^{s,\upsilon}_{p,q}(W)},
\end{align*}
which completes the proof of (ii) and hence Theorem \ref{moledecomp}.
\end{proof}

Based on Theorem \ref{moledecomp}, we
next establish the wavelet characterization
of $\dot{A}_{p,q}^{s,\upsilon}(W)$ in terms of Daubechies wavelets,
which is crucial for obtaining the sufficient and necessary condition
on the Sobolev-type embedding of $\dot{A}_{p,q}^{s,\upsilon}(W)$
in the next section. To this end,
we first present the concept of Daubechies wavelets
(see, for example, \cite{dau88} and \cite[Sections 3.8 and 3.9]{mey92}).
In what follows, for any $k\in\mathbb{N}$,
let $C^{k}$ be the set of all
$k$ times continuously differentiable functions on $\mathbb{R}^n$.

\begin{definition}
Let $k\in\mathbb{N}$. A family of real-valued functions
$\{\theta^{(\lambda)}\}_{\lambda=1}^{2^n-1}$ in
$C^{k}$  with bounded supports
are called \emph{Daubechies wavelets} of class $C^{k}$
if $\{\theta^{(\lambda)}_Q :\ \lambda\in\{1,\ldots,2^n-1\},
\ Q\in\mathcal{D}\}$ is an orthonormal basis of $L^2$.
\end{definition}

Assume that $k\in\mathbb{N}$ and
$\{\theta^{(\lambda)}\}_{\lambda=1}^{2^n-1}$ are
Daubechies wavelets of class $C^{k}$.
From \cite[Corollary 5.5.2]{dau92}, it follows that,
for any $\lambda\in\{1,\ldots,2^n-1\}$ and
$\gamma\in\mathbb{Z}_+^n$ with $|\gamma|\leq k$,
\begin{align}\label{eq-wav-cancel}
\int_{\mathbb{R}^n}\theta^{(\lambda)}(x)x^{\gamma}\,dx=0.
\end{align}
In the following lemma, we establish the relation
between Daubechies wavelets and smooth molecules.

\begin{lemma}\label{lem-Dauwavtomole}
Let $(A, a)\in\{(B, b), (F, f)\}$, $s\in\mathbb{R}$,
$p\in(0,\infty)$, $q\in(0, \infty]$,
and $W\in \mathcal{A}_{p,\infty}$.
Suppose that $\delta_1,\delta_2,\omega$
satisfy \eqref{eq-delta1<0} and
$\upsilon\in\mathcal{G}(\delta_1, \delta_2; \omega)$.
Let $k\in\mathbb{N}$ satisfy
\begin{align}\label{eq-k}
k>\max\left\{E_{\dot{a}_{p,q}^{s,\upsilon}(W)}-\frac{n}{2},
F_{\dot{a}_{p,q}^{s,\upsilon}(W)}-\frac{n}{2}\right\},
\end{align}
where $E_{\dot{a}_{p,q}^{s,\upsilon}(W)}$ and
$F_{\dot{a}_{p,q}^{s,\upsilon}(W)}$
are as in Theorem \ref{a(W)adopebound}.
If $\{\theta^{(\lambda)}\}_{\lambda=1}^{2^n-1}$
are Daubechies wavelets of class $C^{k}$,
then there exists a positive constant $C$ such that,
for any $\lambda\in\{1,\ldots,2^n-1\}$,
$\{C\theta^{(\lambda)}_Q\}_{Q\in\mathcal{D}}$ is
a family of both analysis and synthesis molecules
for $\dot{A}_{p,q}^{s,\upsilon}(W)$.		
\end{lemma}
\begin{proof}
Notice that, for any $\lambda\in\{1,\ldots,2^n-1\}$,
$\theta^{(\lambda)}$ has bounded support and
satisfies \eqref{eq-wav-cancel}. This, combined
with \eqref{eq-k}, Definition \ref{def-anasynmole},
and simple calculations,
further implies the present lemma; we omit the details.
This finishes the proof of Lemma \ref{lem-Dauwavtomole}.
\end{proof}

Finally, we establish the wavelet characterization
of $A^{s,\upsilon}_{p,q}(W)$ via Daubechies wavelets.

\begin{theorem}\label{Dauwav decomp}
Let $(A, a)\in\{(B, b), (F, f)\}$, $s\in\mathbb{R}$,
$p\in(0,\infty)$, $q\in(0, \infty]$, and $W\in\mathcal{A}_{p,\infty}$.
Assume that $\delta_1,\delta_2,\omega$ satisfy \eqref{eq-delta1<0} and
$\upsilon\in\mathcal{G}(\delta_1, \delta_2; \omega)$.
If $k\in\mathbb{N}$ satisfies \eqref{eq-k} and $\{\theta^{(\lambda)}\}_{\lambda=1}^{2^n-1}$ are
Daubechies wavelets of class $C^{k}$, then,
for any $\vec{f}\in\dot{A}^{s,\upsilon}_{p,q}(W)$,
\begin{align}\label{Dauwav decomp eq}
\vec f=\sum_{\lambda=1}^{2^n-1}\sum_{Q\in\mathcal{D}}
\left\langle\vec f,\theta^{(\lambda)}_Q\right\rangle_*\theta^{(\lambda)}_Q
\end{align}
in $(\mathcal{S}'_{\infty})^m$ and
$\|\vec f\|_{\dot{A}^{s,\upsilon}_{p,q}(W)}
\sim\sum_{\lambda=1}^{2^n-1}\|\{\langle
\vec f,\theta^{(\lambda)}_Q \rangle_*\}_{Q\in\mathcal{D}}\|
_{\dot{a}^{s,\upsilon}_{p,q}(W)}$,
where the positive equivalence constants are independent of $\vec{f}$
and $\langle\cdot, \cdot\rangle_*$ is as in \eqref{eq-f-m}.
\end{theorem}

\begin{proof}
To prove the present theorem,
let $\varphi, \psi\in\mathcal{S}$ satisfy
\eqref{cond1} and \eqref{cond3}. By the assumption that	
$\{\theta^{(\lambda)}_Q :\ \lambda\in\{1,\ldots, 2^n-1\},
\ Q\in\mathcal{D}\}$ is an orthonormal basis of $L^2$
and Lemma \ref{lem-Cdreproform},
we find that, for any $\vec{f}\in\dot{A}^{s,\upsilon}_{p,q}(W)$
and $\phi\in\mathcal{S}_{\infty}$,
\begin{align}\label{eq-f-wavelet}
\left\langle\vec{f},\phi\right\rangle=\sum_{R\in\mathcal{D}}
\left\langle\vec f, \varphi_R\right\rangle\left\langle\psi_R,\phi
\right\rangle=\sum_{R\in\mathcal{D}}\left\langle\vec f,
\varphi_R\right\rangle\sum_{\lambda=1}^{2^n-1}\sum_{Q\in\mathcal{D}}
\left\langle\psi_R, \theta^{(\lambda)}_Q\right\rangle
\left\langle\theta^{(\lambda)}_Q,\phi\right\rangle.
\end{align}
Applying Remark \ref{rmk-mole-phi}, Theorem \ref{thm-phitransMWBTL},
and Lemmas \ref{lem-Dauwavtomole} and \ref{anasynmole}(ii),
we conclude that the summations in the right-hand side of
\eqref{eq-f-wavelet} converge absolutely
[because $\phi=\phi_{Q_{0,\mathbf{0}}}$ is a
constant multiple of an analysis molecule
for $\dot{A}_{p,q}^{s,\upsilon}(W)$].
This, combined with \eqref{eq-f-wavelet},
Fubini's theorem, and \eqref{eq-f-m},
further implies that, for any $\vec{f}\in
\dot{A}^{s,\upsilon}_{p,q}(W)$ and $\phi\in\mathcal{S}_{\infty}$,
\begin{align*}
\left\langle\vec{f},\phi\right\rangle=
\sum_{\lambda=1}^{2^n-1}\sum_{Q\in\mathcal{D}}
\sum_{R\in\mathcal{D}}\left\langle\vec f, \varphi_R\right\rangle
\left\langle\psi_R,\theta^{(\lambda)}_Q
\right\rangle\left\langle\theta^{(\lambda)}_Q,\phi\right\rangle
=\sum_{\lambda=1}^{2^n-1}\sum_{Q\in\mathcal{D}}
\left\langle\vec f,\theta^{(\lambda)}_Q\right\rangle_*
\left\langle\theta^{(\lambda)}_Q,\phi\right\rangle
\end{align*}
and hence \eqref{Dauwav decomp eq} holds
in $(\mathcal{S}'_{\infty})^m$.
From Lemma \ref{lem-Dauwavtomole} and
Theorem \ref{moledecomp}(i), it follows that,
for any $\lambda\in\{1,\dots,2^n-1\}$ and
$\vec{f}\in\dot{A}^{s,\upsilon}_{p,q}(W)$,
$\|\{\langle\vec f, \theta^{(\lambda)}_Q\rangle_*
\}_{Q\in\mathcal{D}}\|_{\dot{a}^{s,\upsilon}_{p,q}(W)}
\lesssim\|\vec{f}\|_{\dot{A}^{s,\upsilon}_{p,q}(W)}$
and hence
\begin{align}\label{eq-wavelet-norm}
\sum_{\lambda=1}^{2^n-1}\left\|\left\{\left\langle\vec f,
\theta^{(\lambda)}_Q\right\rangle_*\right\}_{Q\in\mathcal{D}}
\right\|_{\dot{a}^{s,\upsilon}_{p,q}(W)}
\lesssim\left\|\vec{f}\right\|_{\dot{A}^{s,\upsilon}_{p,q}(W)}.
\end{align}

Finally, we show the reverse estimate of \eqref{eq-wavelet-norm}.
To achieve this, by the just proved \eqref{Dauwav decomp eq}
and \eqref{eq-wavelet-norm}, Lemma \ref{lem-Dauwavtomole},
Theorem \ref{moledecomp}(ii), and the quasi-triangle inequality of
$\|\cdot\|_{\dot{A}^{s,\upsilon}_{p,q}(W)}$, we find that,
for any $\lambda\in\{1,\dots,2^n-1\}$ and
$\vec{f}\in\dot{A}^{s,\upsilon}_{p,q}(W)$,
$\sum_{Q\in\mathcal{D}}\langle\vec f,
\theta^{(\lambda)}_Q\rangle_*\theta^{(\lambda)}_Q
\in\dot{A}^{s,\upsilon}_{p,q}(W)$ and
\begin{align*}
\left\|\vec{f}\right\|_{\dot{A}^{s,\upsilon}_{p,q}(W)}
\lesssim\sum_{\lambda=1}^{2^n-1}\left\|\sum_{Q\in\mathcal{D}}
\left\langle\vec f,\theta^{(\lambda)}_Q\right\rangle_*
\theta^{(\lambda)}_Q\right\|_{\dot{A}^{s,\upsilon}_{p,q}(W)}
\lesssim\sum_{\lambda=1}^{2^n-1}\left\|\left\{\left\langle\vec f,
\theta^{(\lambda)}_Q\right\rangle_*\right\}_{Q\in\mathcal{D}}
\right\|_{\dot{a}^{s,\upsilon}_{p,q}(W)}.
\end{align*}
This finishes the proof of the reverse estimate of
\eqref{eq-wavelet-norm} and hence Theorem \ref{Dauwav decomp}.
\end{proof}

\begin{remark}
Let all the symbols be the same as in Theorem \ref{Dauwav decomp}.
For any $\lambda\in\{1,\dots,2^n-1\}$, $Q\in\mathcal{D}$, and
$\vec{f}:=(f_1,\dots,f_m)^{T}\in(L^2)^{m}$,
it is well known that $\langle\vec{f}, \theta^{(\lambda)}_Q\rangle_*
=(\langle f_1, \theta^{(\lambda)}_Q\rangle,\dots,
\langle f_m, \theta^{(\lambda)}_Q\rangle)^{T}$,
where $\langle\cdot,\cdot\rangle$ denotes the conventional
inner product in $L^2$ (see, for example, \cite[Remark 8.5]{syy24} or \cite[Remark 4.9]{bhyy3}).
\end{remark}

\section{Applications\label{s-app}}

This section contains two subsections.
In Subsection \ref{s5.1}, we find the
sufficient and necessary conditions
for the invariances of $\dot{f}_{p,q}^{s,\upsilon_{1/p, W}}(W)$
and $\dot{F}_{p,q}^{s,\upsilon_{1/p, W}}(W)$
on the integrable index $p$, which answers
Question \ref{q1} in the introduction.
Moreover, we also show these invariances generally do not hold
for $\dot{b}_{p,q}^{s,\frac{1}{p}}$ and $\dot{B}_{p,q}^{s,\frac{1}{p}}$,
which also answers an open question posed in \cite[p.\,464]{yy10}.
In Subsection \ref{s5.2},  we establish the Sobolev-type embedding of $\dot{A}^{s,\upsilon}_{p,q}(W)$.

\subsection{Invariances of $\dot{f}_{p,q}^{s, \upsilon_{1/p, W}}(W)$
and $\dot{F}_{p,q}^{s,\upsilon_{1/p, W}}(W)$
on Integrable Index $p$\label{s5.1}}

To give the scalar-weighted version of \eqref{eq-f=f},
we begin with recalling the concept of doubling
weights. A scalar weight $w$ is said to be \emph{doubling}
if there exists a positive constant $C$ such that,
for any $x\in\mathbb{R}^n$ and $r\in(0,\infty)$,
$w(B(x, 2r))\leq C w(B(x,r))$.
Suppose that $s\in\mathbb{R}$, $p\in(0, \infty)$,
$q\in(0, \infty]$, and $w$ is a scalar weight.
In \cite[(2.16)]{bow08},
the space $\dot{f}_{\infty, q}^{s}(w)$
is defined to be the set of all
$t:=\{t_Q\}_{Q\in\mathcal{D}}$ in $\mathbb{C}$ such that
\begin{align}\label{eq-f_3-w}
\|t\|_{\dot{f}_{\infty, q}^{s}(w)}
:=\sup_{P\in\mathcal{D}}\left\{\frac{1}{w(P)}
\int_P\sum_{Q\in\mathcal{D}, Q\subset P}
\left[|Q|^{-\frac{s}{n}}\left|t_Q\right|
\widetilde{\mathbf{1}}_Q(x)\right]^q w(x)\,dx\right\}^{\frac{1}{q}}
\end{align}
is finite (with the usual modification made if $q=\infty$).
In \cite[Definition 2.4]{lbyy12},
the space $\dot{f}^{s,\frac{1}{p}}_{p,q}(w)$
is defined to be the set of all
$t:=\{t_Q\}_{Q\in\mathcal{D}}$ in $\mathbb{C}$ such that
\begin{align}\label{eq-f_4-w}
\|t\|_{\dot{f}^{s,\frac{1}{p}}_{p,q}(w)}
:=\sup_{P\in\mathcal{D}}\left\{\frac{1}{w(P)}
\int_P\left(\sum_{Q\in\mathcal{D}, Q\subset P}
\left[|Q|^{-\frac{s}{n}}\left|t_Q\right|
\widetilde{\mathbf{1}}_Q(x)\right]^q\right)^{\frac{p}{q}}
w(x)\,dx\right\}^{\frac{1}{p}}
\end{align}
is finite (with the usual modification made if $q=\infty$).
Clearly, when $q\in(0, \infty)$,
the space $\dot{f}_{\infty, q}^{s}(w)$ is exactly
$\dot{f}^{s,\frac{1}{q}}_{q,q}(w)$.
Moreover, the space $\dot{f}_{\infty, \infty}^{s}(w)$
should be interpreted as
$\dot{b}_{\infty, \infty}^{s}$
as in Definition \ref{averBTLseq}.

The following lemma is precisely a particular
case of \cite[Theorem 3.6]{bow08} in which
Bownik studied the function spaces based on $\mathbb{R}^n$
associated with general expansive dilations and
corresponding doubling measures (see \cite{bow05,bow07,bow08,bh06}
for more studies of function spaces in this setting).

\begin{lemma}\label{lem-infty-3}
Let $s\in\mathbb{R}$, $p\in(0, \infty)$,
$q\in(0, \infty]$, and $w$ be a scalar doubling weight.
Then $\dot{f}_{\infty, q}^{s}(w)=\dot{f}^{s,\frac{1}{p}}_{p,q}(w)$
with equivalent quasi-norms.
\end{lemma}

Observe that Lemma \ref{lem-infty-3} establishes the equivalence
between \eqref{eq-f_3-w} and \eqref{eq-f_4-w}.
We aim to find the counterpart of
Lemma \ref{lem-infty-3} in the matrix-weighted setting.
To this end, we first present some concepts and notation.
Recall that two matrix weights $W$ and $V$
are said to be \emph{equivalent}
(see, for example, \cite[p.\, 272]{tv97}),
denoted by $W\sim V$, if there exists a
positive constant $C$ such that, for almost every
$x\in\mathbb{R}^n$ and any $\vec{e}\in\mathbb{C}^m$,
\begin{align*}
\frac{1}{C}\left(V(x)\vec{e}, \vec{e}\right)
\leq\left(W(x)\vec{e}, \vec{e}\right)
\leq C\left(V(x)\vec{e}, \vec{e}\right),
\end{align*}
where $(\cdot,\cdot)$ is the conventional inner product in $\mathbb{C}^m$.
For any matrix weight $W$ and any $x\in\mathbb{R}^n$,
let $E_W(x):=\lambda_{\mathrm{max}}(x)=\|W(x)\|$
and $e_W(x):=\lambda_{\mathrm{min}}(x)$,
where $\lambda_{\mathrm{max}}(x)$ and
$\lambda_{\mathrm{min}}(x)$ are respectively
the maximal and the minimal eigenvalues of $W(x)$.
We next give a lemma to characterize the
equivalence of matrix weights.

\begin{lemma}\label{lem-W=I_m}
If $W$ is a matrix weight,
then the following three statements are
mutually equivalent.
\begin{itemize}
\item[{\rm (i)}] $W\sim E_W I_m$, where
$I_m$ is the identity matrix of order $m$.
\item[{\rm (ii)}] For almost every $x\in\mathbb{R}^n$,
\begin{align}\label{cond-3=4}
E_W(x)\sim e_W(x),
\end{align}
where the positive equivalence constants
are independent of $x$.
\item[{\rm (iii)}] $W\sim e_W I_m$.
\end{itemize}
\end{lemma}
\begin{proof}
We first prove (ii) $\Longrightarrow$ (i)
and (ii) $\Longrightarrow$ (iii). If (ii) holds,
by \eqref{eq-W^a}, we find that, for almost every
$x\in\mathbb{R}^n$ and any $\vec{e}\in\mathbb{C}^m$,
\begin{align*}
e_W(x)\left|\vec{e}\right|^2
\leq\left(W(x)\vec{e}, \vec{e}\right)
\leq E_W(x)\left|\vec{e}\right|^2,
\end{align*}
which, together with \eqref{cond-3=4}, further implies that
\begin{align*}
\left(e_W(x)I_m\vec{e}, \vec{e}\right)
=e_W(x)\left|\vec{e}\right|^2\sim
\left(W(x)\vec{e}, \vec{e}\right)
\sim E_W(x)\left|\vec{e}\right|^2=
\left(E_W(x)I_m\vec{e}, \vec{e}\right).
\end{align*}
Thus, both (i) and (iii) hold.

Next, we show (i) $\Longrightarrow$ (ii).
Applying \eqref{eq-W^a}, we obtain,
for almost every $x\in\mathbb{R}^n$,
\begin{align}\label{eq-maxengi}
e_W(x)=\min_{\vec{e}\in\mathbb{C}^m\setminus\{\mathbf{0}\}}
\frac{(W(x)\vec{e}, \vec{e})}{|\vec{e}|^2},
\end{align}
which, combined with the assumption that $W\sim E_W I_m$,
further implies \eqref{cond-3=4}.
This finishes the proof of (i) $\Longrightarrow$ (ii).

Finally, we prove (iii) $\Longrightarrow$ (ii).
It suffices to repeat the argument used in the proof
of (i) $\Longrightarrow$ (ii) with \eqref{eq-maxengi}
replaced by the fact that, for almost every $x\in\mathbb{R}^n$,
\begin{align*}
E_W(x)=\max_{\vec{e}\in\mathbb{C}^m\setminus\{\mathbf{0}\}}
\frac{(W(x)\vec{e}, \vec{e})}{|\vec{e}|^2}.
\end{align*}
This finishes the proof of (iii) $\Longrightarrow$ (ii)
and hence Lemma \ref{lem-W=I_m}.
\end{proof}

The following lemma
follows from the spectral theorem
(see, for example, \cite[Theorem 2.5.6]{hj13});
we omit the details.
\begin{lemma}\label{lem-Ttheta}
Let $W$ be a matrix weight. Then, for any $x\in\mathbb{R}^n$,
$0\leq e_W(x)\leq E_W(x)<\infty$. Moreover,
for any $\alpha\in(0, \infty)$, $x\in\mathbb{R}^n$,
and $\vec{z}\in\mathbb{C}^m$, $[e_W(x)]^{\alpha}|\vec{z}|\leq|W^{\alpha}(x)\vec{z}|
\leq[E_W(x)]^{\alpha}|\vec{z}|$.
\end{lemma}

We next extends Lemma \ref{lem-infty-3}
to the matrix-weighted setting and hence answers Question \ref{q1}.

\begin{theorem}\label{thm-3=4}
Let $s\in\mathbb{R}$, $p\in(0, \infty)$,
and $W$ be a matrix weight with $E_W$
being a scalar doubling weight.
Then the following statements hold.
\begin{itemize}
\item[{\rm (i)}] If $q\in(0, \infty)$
with $p\neq q$, then
\begin{align}\label{eq-f_q=f_p}
\dot{f}_{q,q}^{s, \upsilon_{1/q, W}}(W)
=\dot{f}_{p,q}^{s, \upsilon_{1/p, W}}(W)
\end{align}
with equivalent quasi-norms if and only if
$W\sim E_W I_m$, where $I_m$ is the identity matrix of order $m$
and $\upsilon_{1/q, W}$ and $\upsilon_{1/p, W}$
are growth functions as in \eqref{eq-tau_W}.
\item[{\rm (ii)}]
\begin{align}\label{eq-f_q=b_p}
\dot{b}_{\infty, \infty}^{s}(\mathbb{C}^m)
=\dot{f}_{p,\infty}^{s, \upsilon_{1/p, W}}(W)
\end{align}
with equivalent quasi-norms if and only if
$W\sim E_W I_m$, where $\dot{b}_{\infty,
\infty}^{s}(\mathbb{C}^m)$ is as in Remark \ref{rmk-a(A)-a}.
\end{itemize}
\end{theorem}

\begin{proof}		
We first prove the sufficiency of (i).
To this end, applying Lemma \ref{lem-W=I_m} with
the assumption $W\sim E_W I_m$
and Lemma \ref{lem-Ttheta}, we obtain,
for almost every $x\in\mathbb{R}^n$
and any $\vec{z}\in\mathbb{C}^{m}$,
\begin{align}\label{eq-W-lamda}
\left|W^{\frac{1}{p}}(x)\vec{z}\right|\sim
\left[E_W(x)\right]^{\frac{1}{p}}\left|\vec{z}\right|
\text{\ and\ }\left|W^{\frac{1}{q}}(x)\vec{z}\right|
\sim\left[E_W(x)\right]^{\frac{1}{q}}\left|\vec{z}\right|.
\end{align}
From \eqref{eq-W-lamda}, \eqref{eq-f_3-w},
\eqref{eq-f_4-w}, and the definitions of
$\|\cdot\|_{\dot{f}_{q,q}^{s,\upsilon_{1/q, W}}(W)}$
and $\|\cdot\|_{\dot{f}_{p,q}^{s, \upsilon_{1/p, W}}(W)}$,
it follows that, for any $\vec{t}:=
\{\vec{t}_Q\}_{Q\in\mathcal{D}}$ in $\mathbb{C}^m$,
\begin{align*}
\left\|\vec{t}\right\|_{\dot{f}_{q,q}^{s,\upsilon_{1/q, W}}(W)}
\sim\left\|\,\left|\vec{t}\right|
\,\right\|_{\dot{f}_{\infty, q}^{s}(E_W)}
\text{\ and\ }
\left\|\vec{t}\right\|_{\dot{f}_{p,q}^{s, \upsilon_{1/p, W}}(W)}
\sim\left\|\,\left|\vec{t}\right|\,
\right\|_{\dot{f}^{s,\frac{1}{p}}_{p,q}(E_W)},
\end{align*}
where $|\vec{t}|:=\{|\vec{t}_Q|\}_{Q\in\mathcal{D}}$.
Using this and Lemma \ref{lem-infty-3} together with
the assumption that the scalar weight $E_W$ is doubling, we conclude that,
for any $\vec{t}:=\{\vec{t}_Q\}
_{Q\in\mathcal{D}}$ in $\mathbb{C}^m$,
\begin{align*}
\left\|\vec{t}\right\|_{\dot{f}_{q,q}^{s,\upsilon_{1/q, W}}(W)}
\sim\left\|\,\left|\vec{t}\right|\,\right\|_{\dot{f}_{\infty, q}^{s}(E_W)}
\sim\left\|\,\left|\vec{t}\right|\,
\right\|_{\dot{f}^{s,\frac{1}{p}}_{p,q}(E_W)}
\sim\left\|\vec{t}\right\|_{\dot{f}_{p,q}^{s, \upsilon_{1/p, W}}(W)}	
\end{align*}
and hence \eqref{eq-f_q=f_p} holds.
This finishes the proof of the sufficiency of (i).

Next, we show the necessity of (i). To do this,
for any $Q,R\in\mathcal{D}$,
let $\mathbf{1}_{Q=R}$ be as in \eqref{eq-e-QR}.
If \eqref{eq-f_q=f_p} holds, by Lemma \ref{lem-onepoint},
Remark \ref{rmk-onepoint}, and Example \ref{exam-tau_W},
we find that, for any $Q\in\mathcal{D}$ and
$\vec{z}\in\mathbb{C}^{m}$,
\begin{align}\label{eq-t-t}
2^{j_Q(s+\frac{n}{2})}\left[\frac{
\fint_Q |W^{\frac1q}(x)\vec z|^q\,dx}
{\fint_Q E_W(x)\,dx}\right]^{\frac1q}
&\sim\left\|\left\{\mathbf{1}_{Q=R}\vec{z}\right\}_{R\in\mathcal{D}}
\right\|_{\dot{f}_{q,q}^{s, \upsilon_{1/q, W}}(W)}\\
&\sim\left\|\left\{\mathbf{1}_{Q=R}\vec{z}\right\}_{R\in\mathcal{D}}
\right\|_{\dot{f}_{p,q}^{s, \upsilon_{1/p, W}}(W)}
\sim2^{j_Q(s+\frac{n}{2})}\left[\frac{
\fint_Q|W^{\frac1p}(x)\vec z|^p\,dx}
{\fint_Q E_W(x)\,dx}\right]^{\frac1p},\nonumber
\end{align}
where the positive equivalence constants are
independent of $Q$ and $\vec{z}$. Fixing some
$\vec{z}\in\mathbb{C}^{m}$ and applying
Lebesgue's differentiation theorem
on the both sides of \eqref{eq-t-t},
we obtain, for almost every $x\in\mathbb{R}^n$,
\begin{align}\label{cond-W-w}
\left[\frac{|W^{\frac1{q}}(x)\vec z|^q}
{E_W(x)}\right]^{\frac1{q}}
\sim\left[\frac{|W^{\frac1{p}}(x)\vec z|^p}
{E_W(x)}\right]^{\frac1{p}},
\end{align}
where the positive equivalence constants
are independent of $x$ and $\vec{z}$.
This, combined with the definition of matrix weights and
the facts that $\mathbb{C}^{m}$ has a countable dense subset
and any matrix $M\in M_m(\mathbb{C})$ is bounded on $\mathbb{C}^{m}$,
further implies that, for almost every $x\in\mathbb{R}^n$
and any $\vec{z}\in\mathbb{C}^{m}$,
$W(x)$ is positive definite, $e_W(x)\in(0,\infty)$,
and \eqref{cond-W-w} holds. By \eqref{eq-W^a},
for any $x\in\mathbb{R}^n$ satisfying
$W(x)$ is positive definite,
we can choose some $\vec{z}\in\mathbb{C}^{m}$
such that $|W^{\frac1{q}}(x)\vec{z}|^q
=|W^{\frac1{p}}(x)\vec z|^p=
\lambda_{\mathrm{min}}(x)=e_W(x)$.
From this, \eqref{cond-W-w}, the assumption $p\neq q$,
and Lemma \ref{lem-W=I_m}, we infer that,
for almost every $x\in\mathbb{R}^n$,
$E_W(x)\sim e_W(x)$ and hence $W\sim E_W I_m$.
This finishes the proof of the necessity of (i).

The sufficiency of (ii) follows from the same argument as that used in
the proof of the sufficiency of (i) with $\dot{f}_{q,q}^{s,\upsilon_{1/q, W}}(W),
\dot{f}_{\infty, q}^{s}(E_W), \dot{f}_{p,q}^{s, \upsilon_{1/p, W}}(W)$,
and $\dot{f}^{s,\frac{1}{p}}_{p,q}(E_W)$ therein replaced,
respectively, by $\dot{b}_{\infty, \infty}^{s}(\mathbb{C}^m),
\dot{f}_{\infty, \infty}^{s}(E_W)$,
$\dot{f}_{p,\infty}^{s, \upsilon_{1/p, W}}(W)$,
and $\dot{f}^{s,\frac{1}{p}}_{p,\infty}(E_W)$;
we omit the details.

Finally, we show the necessity of (ii).
For any $Q,R\in\mathcal{D}$,
let $\mathbf{1}_{Q=R}$ be as in \eqref{eq-e-QR}.
If \eqref{eq-f_q=b_p} holds, applying the definition
of $\|\cdot\|_{\dot{b}_{\infty, \infty}^{s}(\mathbb{C}^m)}$,
Lemma \ref{lem-onepoint}, Remark \ref{rmk-onepoint},
and Example \ref{exam-tau_W},
we obtain, for any $Q\in\mathcal{D}$ and $\vec{z}\in\mathbb{C}^{m}$,
\begin{align}\label{eq-z-t}
2^{j_Q(s+\frac{n}{2})}|\vec{z}|
&\sim\left\|\left\{\mathbf{1}_{Q=R}\vec{z}
\right\}_{R\in\mathcal{D}}\right\|
_{\dot{b}_{\infty, \infty}^{s}(\mathbb{C}^m)}
\sim\left\|\left\{\mathbf{1}_{Q=R}\vec{z}
\right\}_{R\in\mathcal{D}}\right\|
_{\dot{f}_{p,\infty}^{s, \upsilon_{1/p, W}}(W)}\\
&\sim2^{j_Q(s+\frac{n}{2})}\left[\frac{
\fint_Q|W^{\frac1p}(x)\vec z|^p\,dx}
{\fint_Q E_W(x)\,dx}\right]^{\frac1p},\nonumber
\end{align}
where the positive equivalence constants are
independent of $Q$ and $\vec{z}$. Fixing some
$\vec{z}\in\mathbb{C}^{m}$ and using
Lebesgue's differentiation theorem
on the right-hand side of \eqref{eq-z-t},
we obtain, for almost every $x\in\mathbb{R}^n$,
\begin{align}\label{cond-z-w}
|\vec{z}|\sim\left[\frac{|W^{\frac1{p}}(x)
\vec{z}|^p}{E_W(x)}\right]^{\frac1{p}},
\end{align}
where the positive equivalence constants
are independent of $x$ and $\vec{z}$.
By the same argument used to prove the necessity of (i),
we find that, for almost every $x\in\mathbb{R}^n$
and any $\vec{z}\in\mathbb{C}^{m}$, \eqref{cond-z-w}
holds and hence, for almost every $x\in\mathbb{R}^n$,
$E_W(x)\sim e_W(x)$. From this and Lemma \ref{lem-W=I_m},
we infer that $W\sim E_W I_m$.
This finishes the proof of the necessity of (ii)
and hence Theorem \ref{thm-3=4}.
\end{proof}			

To further discuss Theorem \ref{thm-3=4}, we
recall that a scalar weight $w\in A_1$ if
\begin{align*}
[w]_{A_1}:=\sup_{\operatorname{cube} Q\subset\mathbb{R}^n}
\fint_Qw(x)\,dx\left\|w^{-1}\right\|_{L^\infty(Q)}<\infty;
\end{align*}
(see, for example, \cite[Definition 7.1.1]{gra14a}).

\begin{remark}
\begin{itemize}
\item[{\rm (i)}] When $m=1$, \eqref{cond-3=4}
is naturally satisfied. In this case, Theorem \ref{thm-3=4}
coincides with Lemma \ref{lem-infty-3}.
Furthermore, when $m=1$ and $W\equiv1$,
Theorem \ref{thm-3=4} reduces to \eqref{eq-f=f}.
\item[{\rm (ii)}] We next give an example
to show that, even in some simple cases,
$W\sim E_W I_m$ does not hold.
To this end, for any $x\in\mathbb{R}^n$, let					
\begin{align*}
W(x):=\left(\begin{array}{cc}
|x|^\alpha & 0 \\
0 & |x|^\beta
\end{array}\right),
\end{align*}
where $-n<\alpha<\beta\leq0$. It is well known that $|\cdot|^{\alpha},
|\cdot|^{\beta}\in A_1$ (see, for example, \cite[Example 7.1.7]{gra14a}).
Next, we claim that $W\in\bigcap_{r\in(0, \infty)}\mathcal{A}_{r, \infty}$.
To see this, let $r\in(0,\infty)$. By Jensen's inequality,
we find that, for any cube $Q\subset\mathbb{R}^n$,
\begin{align*}
&\exp\left(\fint_Q\log\left(\fint_{Q}\left\|W^{\frac{1}{r}}(x)
W^{-\frac{1}{r}}(y)\right\|^r\,dx\right)\,dy\right)\\
&\quad\leq\fint_Q\fint_{Q}\left\|W^{\frac{1}{r}}(x)
W^{-\frac{1}{r}}(y)\right\|^r\,dx\,dy
\sim\fint_Q\fint_{Q}
|x|^{\alpha}|y|^{-\alpha}\,dx\,dy+\fint_Q\fint_{Q}
|x|^{\beta}|y|^{-\beta}\,dx\,dy\\
&\quad\leq\left\|\,|\cdot|^{-\alpha}\,\right\|_{L^{\infty}(Q)}
\fint_{Q}|x|^{\alpha}\,dx+\left\|\,|\cdot|^{-\beta}\,\right\|_{L^{\infty}(Q)}
\fint_{Q}|x|^{\beta}\,dx\leq\left[|\cdot|^{\alpha}\right]_{A_1}
+\left[|\cdot|^{\beta}\right]_{A_1},
\end{align*}
which further implies that
\begin{align*}
[W]_{\mathcal{A}_{r,\infty}}=\sup_{\operatorname{cube} Q\subset\mathbb{R}^n}
\exp\left(\fint_Q\log\left(\fint_{Q}\left\|W^{\frac{1}{r}}(x)
W^{-\frac{1}{r}}(y)\right\|^r\,dx\right)\,dy\right)
\lesssim\left[|\cdot|^{\alpha}\right]_{A_1}
+\left[|\cdot|^{\beta}\right]_{A_1}<\infty
\end{align*}
and hence $W\in\mathcal{A}_{r, \infty}$.
From \cite[Lemma 5.3]{bhyy4} and \cite[Proposition 7.2.8]{gra14a},
it follows that $E_W=\|W\|\in A_{\infty}$
and hence it is a scalar doubling weight.
However, using the construction of $W$,
we conclude that, for any $x\in\mathbb{R}^n$ with $|x|>1$,
$e_W(x)=|x|^{\alpha}$ and $E_W(x)=|x|^{\beta}$ and hence
\eqref{cond-3=4} does not hold. By these and
Lemma \ref{lem-W=I_m}, we find that $W\sim E_W I_m$
does not hold. This phenomenon indicates that,
as the range dimension increases (from 1 to $m$),
matrix-weighted function spaces have distinctive properties
compared to scalar-weighted function spaces.			
\end{itemize}				
\end{remark}

The following theorem is a corollary of Theorem \ref{thm-3=4}.

\begin{theorem}\label{thm-3=4-F}
Let $s\in\mathbb{R}$ and $p\in(0, \infty)$.
Then the following statements hold.
\begin{itemize}
\item[{\rm (i)}] If $q\in(0, \infty)$
with $p\neq q$ and if $W\in\mathcal{A}_{p\wedge q, \infty}$, then
$\dot{F}_{q,q}^{s, \upsilon_{1/q, W}}(W)
=\dot{F}_{p,q}^{s, \upsilon_{1/p, W}}(W)$
with equivalent quasi-norms if and only if $W\sim E_W I_m$,
where $\upsilon_{1/q, W}$ and $\upsilon_{1/p, W}$ are
growth functions as in \eqref{eq-tau_W}.
\item[{\rm (ii)}] If $W\in\mathcal{A}_{p, \infty}$,
$\dot{B}_{\infty, \infty}^{s}(\mathbb{C}^m)
=\dot{F}_{p,\infty}^{s, \upsilon_{1/p, W}}(W)$
with equivalent quasi-norms if and only if
$W\sim E_W I_m$, where $\dot{B}_{\infty,
\infty}^{s}(\mathbb{C}^m)$ is as in Remark \ref{rmk-a(A)-a}.
\end{itemize}
\end{theorem}	

\begin{proof}
From Example \ref{exam-tau_W}, it follows that
all the results for $\dot{A}_{p,q}^{s, \upsilon}(W)$
established in this article also hold for both
$\dot{F}_{q,q}^{s, \upsilon_{1/q, W}}(W)$ and
$\dot{F}_{p,q}^{s, \upsilon_{1/p, W}}(W)$.
Recall that \cite[Remark 3.49]{tri14} also gives the wavelet characterization
of $\dot{B}_{\infty, \infty}^{s}(\mathbb{C}^m)$ via Daubechies wavelets.
By this and Theorem \ref{Dauwav decomp}, we find that,
to prove the present theorem, it suffices to show the corresponding results
at the level of sequence spaces. Applying the definition of
$E_W$ and \cite[Lemma 5.3]{bhyy4}, we obtain
$E_W=\|W\|\in A_{\infty}$. Using this and \cite[Proposition 7.2.8]{gra14a},
we conclude that $E_W$ is a scalar doubling weight.
This, together with Theorem \ref{thm-3=4}, further implies that
Theorem \ref{thm-3=4-F} holds.
\end{proof}

\begin{remark}
In the proof of Theorem \ref{thm-3=4-F}, to use the
wavelet characterizations
of $\dot{F}_{q,q}^{s, \upsilon_{1/q, W}}(W)$ and
$\dot{F}_{p,q}^{s, \upsilon_{1/p, W}}(W)$
as in Theorem \ref{Dauwav decomp},
we need stronger assumptions on $W$ than Theorem \ref{thm-3=4}.
It is still unclear whether the wavelet characterization
as in Theorem \ref{Dauwav decomp} holds
for both $\dot{F}_{q,q}^{s, \upsilon_{1/q, W}}(W)$ and
$\dot{F}_{p,q}^{s, \upsilon_{1/p, W}}(W)$ with $W$
as in Theorem \ref{thm-3=4}.
\end{remark}

Based on the above discussions, a natural question
is whether Lemma \ref{lem-infty-3} holds
for Besov-type spaces.  It is worth pointing out that,
even for unweighted Besov-type spaces,
Lemma \ref{lem-infty-3} generally does not hold.
To be precise, we have the following proposition,
which also answers an open question posed in \cite[p.\,464]{yy10}.

\begin{proposition}\label{prop-3=4-b}
Let $s\in\mathbb{R}$ and $p\in(0, \infty)$.
Then the following statements hold.
\begin{itemize}
\item[{\rm (i)}]
$\dot{B}_{\infty, \infty}^{s}=\dot{B}_{p,\infty}^{s, \frac{1}{p}}$
with equivalent quasi-norms.
\item[{\rm (ii)}] If $q\in(p, \infty)$,
then $\dot{B}_{q,q}^{s, \frac{1}{q}}
\subsetneqq\dot{B}_{p,q}^{s, \frac{1}{p}}$;
if $q\in(0, p)$, then $\dot{B}_{p,q}^{s, \frac{1}{p}}
\subsetneqq\dot{B}_{q,q}^{s, \frac{1}{q}}$.
\end{itemize}
\end{proposition}

\begin{proof}
Notice that (i) is exactly a particular case of \cite[Theorem 1(ii)]{yy13};
we omit the details. We next prove (ii) in the case where $q\in(p, \infty)$.
For simplicity, we may assume that $s=0$.
By Theorem \ref{Dauwav decomp}, it suffices to show
$\dot{b}_{q,q}^{0, \frac{1}{q}}\subsetneqq\dot{b}_{p,q}^{0, \frac{1}{p}}$.
From H\"{o}lder's inequality, we infer that
$\dot{b}_{q,q}^{0, \frac{1}{q}}\subset
\dot{b}_{p,q}^{0, \frac{1}{p}}$.
We now prove $\dot{b}_{q,q}^{0, \frac{1}{q}}
\subsetneqq\dot{b}_{p,q}^{0, \frac{1}{p}}$
by considering the following sequence
$t:=\{t_{Q}\}_{Q\in\mathcal{D}}$
defined by setting, for any $j\in\mathbb{Z}$
and $k:=(k_1,\dots,k_n)\in\mathbb{Z}^n$,
$$
t_{Q_{j,k}}:=
\begin{cases}
|Q_{j,k}|^{\frac{1}{2}}
& \text{if }
j\in\mathbb{Z}_+\text{ and }
\frac{k_1}{1+j}\in\mathbb{Z},\\
0
& \text{otherwise}.
\end{cases}
$$
We next claim that $\|t\|_{\dot{b}_{q,q}^{0, \frac{1}{q}}}=\infty$
and $\|t\|_{\dot{b}_{p,q}^{0, \frac{1}{p}}}<\infty$, which
further implies that $t\in\dot{b}_{p,q}^{0, \frac{1}{p}}
\backslash\dot{b}_{q,q}^{0, \frac{1}{q}}$
and hence $\dot{b}_{q,q}^{0, \frac{1}{q}}
\subsetneqq\dot{b}_{p,q}^{0, \frac{1}{p}}$.
To this end, using some basic calculations and
the definitions of $t$ and
$\|\cdot\|_{\dot{b}_{q,q}^{0, \frac{1}{q}}}$, we conclude that
\begin{align*}
\|t\|_{\dot{b}_{q,q}^{0, \frac{1}{q}}}
&=\sup_{P\in\mathcal{D}}\left\{\frac{1}{|P|}
\sum_{Q\in\mathcal{D}, Q\subset P}
\left[\left(\left|t_{Q}\right|
\left|Q\right|^{-\frac{1}{2}}\right)^q
|Q|\right]\right\}^{\frac{1}{q}}\\
&\geq\left\{\frac{1}{|Q_{0,\mathbf{0}}|}
\sum_{Q\in\mathcal{D}, Q\subset Q_{0,\mathbf{0}}}
\left[\left(\left|t_{Q}\right|
\left|Q\right|^{-\frac{1}{2}}\right)^q
|Q|\right]\right\}^{\frac{1}{q}}
\geq\left\{\sum_{j=0}^{\infty}\frac{1}{j+1}
\right\}^{\frac{1}{q}}=\infty.
\end{align*}
We now estimate $\|t\|_{\dot{b}_{p,q}^{0, \frac{1}{p}}}$
as follows. By some basic calculations,
the assumption that $q\in(p, \infty)$,
and the definition of $t$ again,
we find that, for any $P\in\mathcal{D}$,
\begin{align}\label{eq-P}
\left\{\sum_{j=j_P}^{\infty}
\left[\sum_{Q\in\mathcal{D}_j, Q\subset P}
\left(\left|t_{Q}\right|
\left|Q\right|^{-\frac{1}{2}}\right)^p
\frac{|Q|}{|P|}\right]^{\frac{q}{p}}\right\}^{\frac{1}{q}}
&\leq\left\{\sum_{j=j_P\vee0}^{\infty}
\left(\left\lceil\frac{2^{j-j_P}}{j+1}\right
\rceil2^{j_P-j}\right)^{\frac{q}{p}}\right\}^{\frac{1}{q}}\\
&\lesssim\left\{\sum_{j=0}^{\infty}\left(\frac{1}{j+1}
\right)^{\frac{q}{p}}\right\}^{\frac{1}{q}}\sim1.\nonumber
\end{align}
Taking the supremum over all $P\in\mathcal{D}$ on the
both sides of \eqref{eq-P} and applying the definition of
$\|\cdot\|_{\dot{b}_{p,q}^{0, \frac{1}{p}}}$, we obtain
$\|t\|_{\dot{b}_{p,q}^{0, \frac{1}{p}}}<\infty$,
which completes the proof of the above claim
and hence the proof of the case where $q\in(p, \infty)$.
The case where $q\in(0, p)$ follows from a similar argument;
we omit the details. This finishes the proof of (ii)
and hence Proposition \ref{prop-3=4-b}.
\end{proof}

\subsection{Sobolev-Type Embedding\label{s5.2}}

In this subsection we give the sufficient and
necessary condition for the Sobolev-type embedding
of $\dot{A}_{p,q}^{s, \upsilon}(W)$.
Let the symbol ``$\hookrightarrow$'' stand for the
continuous embedding. We start with some basic
embeddings of $\dot{A}_{p,q}^{s, \upsilon}(W)$.

\begin{proposition}\label{prop-embed}
Let $A\in\{B, F\}$, $s\in\mathbb{R}$, $p\in(0, \infty)$,
$q, q_1, q_2\in(0,\infty]$, and $W\in\mathcal{A}_{p, \infty}$.
Suppose that $\delta_1,\delta_2,\omega$ satisfy \eqref{eq-delta1<0} and
$\upsilon\in\mathcal{G}(\delta_1, \delta_2; \omega)$.
Then the following statements hold.
\begin{itemize}
\item[{\rm (i)}] If $q_1\leq q_2$, then
$\dot{A}_{p,q_1}^{s, \upsilon}(W)
\hookrightarrow\dot{A}_{p,q_2}^{s, \upsilon}(W)$.
\item[{\rm (ii)}] $\dot{B}_{p,\min\{p, q\}}^{s, \upsilon}(W)
\hookrightarrow\dot{F}_{p,q}^{s, \upsilon}(W)
\hookrightarrow\dot{B}_{p,\max\{p, q\}}^{s, \upsilon}(W)$.
\end{itemize}	
\end{proposition}

\begin{proof}
The proof of (i) follows from the monotonicity on $q$ of
the sequence space $l^q$; we omit the details.

Next, we prove (ii). To this end,
assume first $0<p \leq q\leq\infty$. In this case,
by the monotonicity on $q$ of the sequence space $l^q$ again,
we find that the first embedding in (ii) holds.
Applying (generalized) Minkowski's inequality,
we obtain the second embedding in (ii), which completes the proof of (ii)
in the case where $0<p \leq q\leq\infty$. The case where $0<q<p<\infty$
follows from the above argument by exchanging the application of
the monotonicity on $q$ of the sequence space $l^q$ with
the application of (generalized) Minkowski's inequality;
we omit the details. This finishes the proof of (ii)
and hence Proposition \ref{prop-embed}.
\end{proof}

\begin{proposition}\label{prop-SAS'}
Let $A\in\{B, F\}$, $s\in\mathbb{R}$, $p\in(0, \infty)$,
$q\in(0,\infty]$, and $W\in \mathcal{A}_{p,\infty}$.
Suppose that $\delta_1,\delta_2,\omega$
satisfy \eqref{eq-delta1>0} and
$\upsilon\in\mathcal{G}(\delta_1, \delta_2; \omega)$.
Then $(\mathcal{S}_{\infty})^m\hookrightarrow
\dot{A}^{s,\upsilon}_{p,q}(W)
\hookrightarrow(\mathcal{S}'_{\infty})^m$.
Moreover, there exist $N\in\mathbb{N}$
and a positive constant $C$ such that,
for any $\vec{f}:=(f_1, f_2, \dots, f_m)^{T}
\in(\mathcal{S}_{\infty})^m$,
\begin{align*}
\left\|\vec{f}\right\|_{\dot{A}^{s,\upsilon}_{p,q}(W)}
\leq C\sum_{i=1}^{m}\|f_i\|_{\mathcal{S}_N},
\end{align*}
where $\|\cdot\|_{\mathcal{S}_N}$ is as in \eqref{eq-S_N}.
\end{proposition}

\begin{proof}
We begin with proving the first embedding.
Let $N\in\mathbb{N}$ satisfy $$N>\max\left\{D_{\dot{a}_{p,q}^{s,\upsilon}(W)}-n,
E_{\dot{a}_{p,q}^{s,\upsilon}(W)}-\frac{n}{2},
F_{\dot{a}_{p,q}^{s,\upsilon}(W)}-\frac{n}{2}\right\},$$
where $D_{\dot{a}_{p,q}^{s,\upsilon}(W)},
E_{\dot{a}_{p,q}^{s,\upsilon}(W)}$,
and $F_{\dot{a}_{p,q}^{s,\upsilon}(W)}$
are as in Theorem \ref{a(W)adopebound},
By the definition of $\mathcal{S}_{\infty}$
and Definitions \ref{def-mole} and \ref{def-anasynmole},
it is not hard to verify that,
for any $\vec{f}:=(f_1, f_2, \dots, f_m)^{T}
\in(\mathcal{S}_{\infty})^m$ and $i\in\{1,\dots,m\}$,
$[\sqrt{n}\|f_i\|_{\mathcal{S}_N}]^{-1}f_i$
is an $(n+N,N,n+N,N)$ molecule supported near
$Q_{0,\mathbf{0}}$ and hence a synthesis
for $\dot{A}_{p,q}^{s,\upsilon}(W)$. Observe that,
for any $\vec{f}:=(f_1, f_2, \dots, f_m)^{T}
\in(\mathcal{S}_{\infty})^m$,
$\vec{f}=\sum_{i=1}^{m}\vec{e}_if_i$,
where, for any $i\in\{1,\dots,m\}$,
$\vec{e}_i$ denotes the vector $(0,\dots,1,\dots,0)^{T}$
with 1 in the $i$th entry and 0 elsewhere.
Let $\mathbf{1}_{Q_{0,\mathbf{0}}=R}$
be as in \eqref{eq-e-QR} with $Q$ replaced by $Q_{0,\mathbf{0}}$.
From Lemma \ref{lem-onepoint} and Remark \ref{rmk-onepoint},
it follows that, for any $i\in\{1,\dots,m\}$,
$\{\mathbf{1}_{Q_{0,\mathbf{0}}=R}\vec{e}_i\}_{R\in\mathcal{D}}
\in\dot{a}^{s,\upsilon}_{p,q}(W)$.
Combining the above arguments, the quasi-triangle inequality
of $\|\cdot\|_{\dot{A}^{s,\upsilon}_{p,q}(W)}$,
and Theorem \ref{moledecomp}(ii),
we conclude that, for any $\vec{f}:=(f_1, f_2, \dots, f_m)^{T}
\in(\mathcal{S}_{\infty})^m$,
\begin{align*}
\left\|\vec{f}\right\|_{\dot{A}^{s,\upsilon}_{p,q}(W)}
&=\left\|\sum_{i=1}^{m}\vec{e}_if_i
\right\|_{\dot{A}^{s,\upsilon}_{p,q}(W)}
\lesssim\sum_{i=1}^{m}\left\|\vec{e}_if_i
\right\|_{\dot{A}^{s,\upsilon}_{p,q}(W)}\\
&\lesssim\sum_{i=1}^{m}\|f_i\|_{\mathcal{S}_N}
\left\|\left\{\mathbf{1}_{Q_{0,\mathbf{0}}=R}
\vec{e}_i\right\}_{R\in\mathcal{D}}
\right\|_{\dot{a}^{s,\upsilon}_{p,q}(W)}
\sim\sum_{i=1}^{m}\|f_i\|_{\mathcal{S}_N},
\end{align*}
which implies that the first embedding holds.
The second embedding directly follows from Lemma \ref{well-define},
Theorem \ref{thm-phitansaverMWBTL}, and Corollary \ref{cor-a(A)=a(W)};
we omit the details. This finishes the proof of Proposition \ref{prop-SAS'}.
\end{proof}

We next establish the Sobolev-type embedding of
$\dot{A}_{p,q}^{s, \upsilon}(W)$.
The key idea in the following proof is the application of
Corollary \ref{cor-a(A)=a(W)}, which gives the
equivalence between matrix-weighted spaces and
averaging spaces.

\begin{theorem}\label{thm-sobolev-B}
Let $s_0, s_1\in\mathbb{R}$, $p_0, p_1\in(0, \infty)$,
$q\in(0, \infty]$, $W_0\in\mathcal{A}_{p_0, \infty}$,
and $W_1\in\mathcal{A}_{p_1, \infty}$.
Assume that $\delta_1,\delta_2,\omega$
satisfy \eqref{eq-delta1>0} and
$\upsilon\in\mathcal{G}(\delta_1, \delta_2; \omega)$.
Then the following statements hold.
\begin{itemize}
\item[{\rm (i)}] If $p_0\leq p_1$, then
$\dot{B}_{p_0,q}^{s_0, \upsilon}(W_0)
\hookrightarrow\dot{B}_{p_1,q}^{s_1, \upsilon}(W_1)$ if
and only if there exists a positive constant $C$
such that, for any $Q\in\mathcal{D}$ and $\vec{z}\in\mathbb{C}^m$,
\begin{align}\label{eq-sobolev-B}
2^{j_Q s_1}\left\|w_{1,\vec{z}}\right\|_{L^{p_1}(Q)}
\leq C2^{j_Q s_0}\left\|w_{0,\vec{z}}\right\|_{L^{p_0}(Q)},
\end{align}
where, for any $i\in\{0, 1\}$, $\vec{z}\in\mathbb{C}^m$,
and $x\in\mathbb{R}^n$,
$w_{i,\vec{z}}(x):=|W_{i}^{\frac{1}{p_i}}(x)\vec{z}|$.
\item[{\rm (ii)}] If $p_0< p_1$, then
$\dot{F}_{p_0, \infty}^{s_0, \upsilon}(W_0)
\hookrightarrow\dot{F}_{p_1,q}^{s_1, \upsilon}(W_1)$
if and only if \eqref{eq-sobolev-B} holds.
\end{itemize}
\end{theorem}

\begin{proof}
By Theorems \ref{Dauwav decomp}, we find that, to prove
the present theorem, it suffices to show the corresponding
results for the related sequence spaces.
We first prove the necessity of both (i) and (ii).
To this end, for any $Q,R\in\mathcal{D}$,
let $\mathbf{1}_{Q=R}$ be as in \eqref{eq-e-QR}.
Applying Lemma \ref{lem-onepoint} and Remark \ref{rmk-onepoint},
we obtain, for any $Q\in\mathcal{D}$ and $\vec{z}\in\mathbb{C}^{m}$,
\begin{align*}
\left\|\left\{\mathbf{1}_{Q=R}\vec{z}\right\}_{R\in\mathcal{D}}
\right\|_{\dot{b}_{p_0,q}^{s_0, \upsilon}(W_0)}
\sim\frac{2^{j_Q(s_0+\frac{n}{2})}}{\upsilon(Q)}
\left\|w_{0,\vec{z}}\right\|_{L^{p_0}(Q)}\sim
\left\|\left\{\mathbf{1}_{Q=R}\vec{z}\right\}_{R\in\mathcal{D}}
\right\|_{\dot{f}_{p_0, \infty}^{s_0, \upsilon}(W_0)}
\end{align*}
and
\begin{align*}
\left\|\left\{\mathbf{1}_{Q=R}\vec{z}\right\}_{R\in\mathcal{D}}
\right\|_{\dot{b}_{p_1,q}^{s_1, \upsilon}(W_1)}
\sim\frac{2^{j_Q(s_1+\frac{n}{2})}}{\upsilon(Q)}
\left\|w_{1,\vec{z}}\right\|_{L^{p_1}(Q)}\sim
\left\|\left\{\mathbf{1}_{Q=R}\vec{z}\right\}_{R\in\mathcal{D}}
\right\|_{\dot{f}_{p_1,q}^{s_1, \upsilon}(W_1)},
\end{align*}
Using these and the fact $\upsilon(Q)\in(0,\infty)$,
we conclude that, if the embedding in (i) or (ii) holds,
then \eqref{eq-sobolev-B} holds. This finishes the proof
of the necessity of both (i) and (ii).

Next, we prove the sufficiency of (i). To achieve this, from
the definitions of $\|\cdot\|_{\dot{b}_{p_1,q}^{s_1, \upsilon}(W_1)}$
and $\|\cdot\|_{\dot{b}_{p_0,q}^{s_0, \upsilon}(W_0)}$,
\eqref{eq-sobolev-B}, and the monotonicity on $q$ of the sequence space $l^q$,
we infer that, for any $\vec{t}:=\{\vec{t}_Q\}_{Q\in\mathcal{D}}
\in\dot{b}_{p_0,q}^{s_0, \upsilon}(W_0)$,
\begin{align*}
\left\|\vec{t}\right\|_{\dot{b}_{p_1,q}^{s_1, \upsilon}(W_1)}
&=\sup_{P\in\mathcal{D}}\frac{1}{\upsilon(P)}
\left\{\sum_{j=j_P}^{\infty}2^{j(s_1+\frac{n}{2})q}
\left[\sum_{Q\in\mathcal{D}_j, Q\subset P}
\left\|w_{1,\vec{t}_Q}\right\|^{p_1}_{L^{p_1}(Q)}
\right]^{\frac{q}{p_1}}\right\}^{\frac{1}{q}}\\
&\leq\sup_{P\in\mathcal{D}}\frac{1}{\upsilon(P)}
\left\{\sum_{j=j_P}^{\infty}2^{j(s_0+\frac{n}{2})q}
\left[\sum_{Q\in\mathcal{D}_j, Q\subset P}
\left\|w_{0,\vec{t}_Q}\right\|^{p_0}_{L^{p_0}(Q)}
\right]^{\frac{q}{p_0}}\right\}^{\frac{1}{q}}
=\left\|\vec{t}\right\|_{\dot{b}_{p_0,q}^{s_0, \upsilon}(W_0)},
\end{align*}
which further implies that $\dot{b}_{p_0,q}^{s_0, \upsilon}(W_0)
\hookrightarrow\dot{b}_{p_1,q}^{s_1, \upsilon}(W_1)$ and hence
completes the proof of the sufficiency of (i).

Finally, we show the sufficiency of (ii).
To this end,  let $\mathbb{A}^{(0)}
:=\{A_{Q}^{(0)}\}_{Q\in\mathcal{D}}$
and $\mathbb{A}^{(1)}:=\{A_{Q}^{(1)}\}_{Q\in\mathcal{D}}$
be sequences of reducing operators of orders
$p_0$ and $p_1$, respectively, for $W_0$ and $W_1$.
Applying Definition \ref{def-red-ope}, we find that
the condition \eqref{eq-sobolev-B} is equivalent to that,
for any $Q\in\mathcal{D}$ and $\vec{z}\in\mathbb{C}^m$,
\begin{align}\label{eq-sobolev-A}
2^{j_Q(s_1-\frac{n}{p_1})}\left|A_{Q}^{(1)}\vec{z}\right|
\lesssim2^{j_Q(s_0-\frac{n}{p_0})}\left|A_{Q}^{(0)}\vec{z}\right|.
\end{align}
By Corollary \ref{cor-a(A)=a(W)}, we obtain,
to prove $\dot{f}_{p_0, \infty}^{s_0, \upsilon}(W_0)
\hookrightarrow\dot{f}_{p_1,q}^{s_1, \upsilon}(W_1)$
under the assumption \eqref{eq-sobolev-B},
it suffices to show that, under the assumption \eqref{eq-sobolev-A},
$\dot{f}_{p_0, \infty}^{s_0, \upsilon}(\mathbb{A}^{(0)})\hookrightarrow
\dot{f}_{p_1, q}^{s_1, \upsilon}(\mathbb{A}^{(1)})$ holds.
For this purpose, let $\vec{t}:=\{\vec{t}_{Q}\}_{Q\in\mathcal{D}}\in
\dot{f}_{p_0, \infty}^{s_0, \upsilon}(\mathbb{A}^{(0)})$ with
$\|\vec{t}\|_{\dot{f}_{p_0, \infty}^{s_0, \upsilon}(\mathbb{A}^{(0)})}=1$.
Notice that, for any given $P\in\mathcal{D}$
and for any $\lambda\in(0, \infty)$,
there exists $L\in\mathbb{Z}$ such that
\begin{align}\label{eq-L}
2^{\frac{n}{p_1}L}\leq\frac{\lambda}
{2\upsilon(P)}<2^{\frac{n}{p_1}(L+1)}.
\end{align}
Suppose first that $L\in[j_P, \infty)\cap\mathbb{Z}$.
From the quasi-triangle inequality
of $l^q$ and \eqref{eq-sobolev-A}, we deduce that,
for any $x\in P$,
\begin{align}\label{eq-G_P}
G_P(x):&=\left\{\sum_{j=j_P}^{\infty} 2^{js_1 q}
\left[\sum_{Q\in\mathcal{D}_j, Q\subset P}
\widetilde{\mathbf{1}}_Q(x)\left|A_{Q}^{(1)}\vec{t}_Q\right|
\right]^q\right\}^{\frac{1}{q}}\\
&\lesssim\left\{\sum_{j=j_P}^{L}2^{js_1 q}
\left[\sum_{Q\in\mathcal{D}_j, Q\subset P}
\widetilde{\mathbf{1}}_Q(x)\left|A_{Q}^{(1)}\vec{t}_Q\right|
\right]^q\right\}^{\frac{1}{q}}
+\left\{\sum_{j=L+1}^{\infty} \cdots\right\}^{\frac{1}{q}}\nonumber\\
&\lesssim\left\{\sum_{j=j_P}^{L}
2^{j(s_0-\frac{n}{p_0}+\frac{n}{p_1})q}
\left[\sum_{Q\in\mathcal{D}_j, Q\subset P}
\widetilde{\mathbf{1}}_Q(x)\left|A_{Q}^{(0)}\vec{t}_Q\right|
\right]^q\right\}^{\frac{1}{q}}\nonumber
+\left\{\sum_{j=L+1}^{\infty}
\cdots\right\}^{\frac{1}{q}}\nonumber
=:\operatorname{I}(x)+\operatorname{II}(x).\nonumber
\end{align}
Let $K_1$ be the implicit positive constant in \eqref{eq-G_P}
and notice that $K_1$ is independent of $P$, $\lambda$, and $x$.
Using the definition of $\|\cdot\|_{\dot{f}_{p_0,
\infty}^{s_0, \upsilon}(\mathbb{A}^{(0)})}$,
the assumption $\|\vec{t}\|_{\dot{f}_{p_0, \infty
}^{s_0, \upsilon}(\mathbb{A}^{(0)})}=1$,
and Lemma \ref{lem-grow-est}(i) together with
the assumption $\delta_1\in[0, \infty)$,
we conclude that, for any $Q\in\mathcal{D}$ with $Q\subset P$,
\begin{align}\label{eq-est-f}
|Q|^{-\frac{s_0}{n}-\frac{1}{2}+\frac{1}{p_0}}
\left|A_{Q}^{(0)}\vec{t}_Q\right|\leq\upsilon(Q)
\left\|\vec{t}\right\|_{\dot{f}_{p_0,
\infty}^{s_0, \upsilon}(\mathbb{A}^{(0)})}
=\upsilon(Q)\lesssim\upsilon(P).
\end{align}
Applying \eqref{eq-est-f} and \eqref{eq-L},
we obtain, for any $x\in P$,
\begin{align}\label{eq-I-I}
\operatorname{I}(x)\lesssim\upsilon(P)\left(\sum_{j=j_P}^{L}
2^{j\frac{n}{p_1}q}\right)^{\frac{1}{q}}
\sim\upsilon(P)2^{\frac{n}{p_1}L}\leq\frac{\lambda}{2}.
\end{align}
Let $K_2$ be the implicit positive constant in \eqref{eq-I-I},
which is independent of $P$, $\lambda$, and $x$.
For any $x\in P$, let
\begin{align*}
H_P(x):=\left[\sup_{Q\in\mathcal{D}, Q\subset P}|Q|^{-\frac{s_0}{n}}
\left|A_{Q}^{(0)}\vec{t}_Q\right|\widetilde{\mathbf{1}}_Q(x)\right].
\end{align*}
By this, the assumption that $p_0<p_1$, and \eqref{eq-L},
we find that, for any $x\in P$,
\begin{align}\label{eq-II}
\operatorname{II}(x)\leq\left[\sum_{j=L+1}^{\infty}
2^{j(-\frac{n}{p_0}+\frac{n}{p_1})q}\right]^{\frac{1}{q}}H_P(x)
\sim2^{(L+1)(-\frac{n}{p_0}+\frac{n}{p_1})}H_P(x)
\lesssim\lambda^{(1-\frac{p_1}{p_0})}
[\upsilon(P)]^{-(1-\frac{p_1}{p_0})}H_P(x),
\end{align}
where the implicit positive constant in \eqref{eq-II},
independent of $P$, $\lambda$, and $x$, is denoted by $K_3$.
From \eqref{eq-G_P}, \eqref{eq-I-I}, \eqref{eq-II},
and the obvious fact that $\{x\in P:\ G_P(x)>K_1K_2\lambda\}$
is the union of both
$\{x\in P:\ \operatorname{I}(x)>K_2\frac{\lambda}{2}\}$
and $\{x\in P:\ \operatorname{II}(x)>K_2\frac{\lambda}{2}\}$,
it follows that
\begin{align}\label{eq-dist}
\left|\left\{x\in P:\ G_P(x)>K_1K_2\lambda\right\}\right|
&\leq\left|\left\{x\in P:\ \operatorname{I}(x)>
K_2\frac{\lambda}{2}\right\}\right|+\left|
\left\{x\in P:\ \operatorname{II}(x)>K_2
\frac{\lambda}{2}\right\}\right|\\\nonumber
&\leq\left|\left\{x\in P:\ H_P(x)>\frac{K_2}{2K_3}
[\upsilon(P)]^{(1-\frac{p_1}{p_0})}
\lambda^{\frac{p_1}{p_0}}\right\}\right|\nonumber.
\end{align}
When $L\in(-\infty, j_P)\cap\mathbb{Z}$,
using \eqref{eq-G_P} and \eqref{eq-II},
we conclude that, for any $x\in P$,
$G_P(x)\leq K_1\operatorname{II}(x)$ and hence
\begin{align*}
\left|\left\{x\in P:\ G_P(x)>K_1K_2\lambda\right\}\right|
&\leq\left|\left\{x\in P:\ \operatorname{II}(x)>
K_2\frac{\lambda}{2}\right\}\right|\\
&\leq\left|\left\{x\in P:\ H_P(x)>\frac{K_2}{2K_3}
[\upsilon(P)]^{(1-\frac{p_1}{p_0})}
\lambda^{\frac{p_1}{p_0}}\right\}\right|,
\end{align*}
which further implies that \eqref{eq-dist} in this case
also holds. Applying \eqref{eq-dist}, the layer-cake formula
(see, for example, \cite[Proposition 1.1.4]{gra14a}),
and a change of variables, we obtain, for any $P\in\mathcal{D}$,
\begin{align*}
\frac{1}{[\upsilon(P)]^{p_1}}\left\|G_P\right\|^{p_1}_{L^{p_1}(P)}
&=\frac{1}{[\upsilon(P)]^{p_1}}\int_0^{\infty}
\lambda^{p_1-1}\left|\left\{x\in P:\ G_P(x)
>\lambda\right\}\right|\, d\lambda\\
&\sim\frac{1}{[\upsilon(P)]^{p_1}}\int_0^{\infty}
\lambda^{p_1-1}\left|\left\{x\in P:\ G_P(x)>K_1K_2
\lambda\right\}\right|\, d\lambda\\
&\leq\frac{1}{[\upsilon(P)]^{p_1}}
\int_0^{\infty}\lambda^{p_1-1}
\left|\left\{x\in P:\ H_P(x)>\frac{K_2}{2K_3}
[\upsilon(P)]^{(1-\frac{p_1}{p_0})}
\lambda^{\frac{p_1}{p_0}}\right\}\right|\, d\lambda\\
&\sim\frac{1}{[\upsilon(P)]^{p_0}}
\int_0^{\infty}\lambda^{p_0-1}\left|
\left\{x\in P:\ H_P(x)>\lambda\right\}\right|\, d\lambda
=\frac{1}{[\upsilon(P)]^{p_0}}\left\|H_P\right\|^{p_0}_{L^{p_0}(P)},
\end{align*}
where all the positive equivalence constants are independent of $P$.
Taking the supremum over all $P\in\mathcal{D}$
on its both sides and using the definitions of both
$\|\cdot\|_{\dot{f}_{p_1,q}^{s_1, \upsilon}(\mathbb{A}^{(1)})}$
and $\|\cdot\|_{\dot{f}_{p_0,\infty}^{s_0, \upsilon}(\mathbb{A}^{(0)})}$,
we conclude that $\|\vec{t}\|^{p_1}_{\dot{f}_{p_1,
q}^{s_1, \upsilon}(\mathbb{A}^{(1)})}
\lesssim\|\vec{t}\|^{p_0}_{\dot{f}_{p_0,
\infty}^{s_0, \upsilon}(\mathbb{A}^{(0)})}$.
This, together with the assumption $\|\vec{t}\|_{\dot{f}_{p_0,
\infty}^{s_0, \upsilon}(\mathbb{A}^{(0)})}=1$,
further implies $\dot{f}_{p_0, \infty}^{s_0,\upsilon}
(\mathbb{A}^{(0)})\hookrightarrow\dot{f}_{p_1, q}^{s_1,
\upsilon}(\mathbb{A}^{(1)})$. This finishes the
proof of the sufficiency of (ii) and hence
Theorem \ref{thm-sobolev-B}.
\end{proof}

\begin{remark}
In Theorem \ref{thm-sobolev-B}, let $\tau\in[0, \infty)$ and,
for any $Q\in\mathcal{D}$, $\upsilon(Q):=|Q|^{\tau}$.
Then the spaces in Theorem \ref{thm-sobolev-B}
are matrix-weighted BTL-type spaces as in
\cite{bhyy1,bhyy2,bhyy3,bhyy5}.
Even for these spaces, Theorem \ref{thm-sobolev-B}
is completely new. In particular, when $\tau=0$,
the spaces in Theorem \ref{thm-sobolev-B}
reduce to classical matrix-weighted BTL spaces
as in \cite{fr04,fr21,rou03}. In this case,
Theorem \ref{thm-sobolev-B} is also new.
\end{remark}

In the scalar-weighted setting, an important case of
Theorem \ref{thm-sobolev-B} is as follows.

\begin{corollary}\label{cor-sobolev-w}
Let $0<p_0<p_1<\infty$, $d\in(0, \infty)$,
$s_0, s_1\in\mathbb{R}$ with $s_0-\frac{d}{p_0}
=s_1-\frac{d}{p_1}$, $q\in(0, \infty]$,
and the scalar weight $w\in A_{\infty}$.
Assume that $\delta_1,\delta_2,\omega$
satisfy \eqref{eq-delta1>0} and $\upsilon\in\mathcal{G}(\delta_1,
\delta_2; \omega)$. Then the following statements are mutually equivalent.
\begin{itemize}
\item[{\rm (i)}] $\dot{B}_{p_0,q}^{s_0, \upsilon}(w)
\hookrightarrow\dot{B}_{p_1,q}^{s_1, \upsilon}(w)$.
\item[{\rm (ii)}] $\dot{F}_{p_0, \infty}^{s_0, \upsilon}(w)
\hookrightarrow\dot{F}_{p_1,q}^{s_1, \upsilon}(w)$.
\item[{\rm (iii)}] There exists a positive constant $C$
such that, for any $Q\in\mathcal{D}$,
\begin{align}\label{eq-sobolev-w-d}
w(Q)\geq C 2^{-j_Q d}.	
\end{align}
\end{itemize}
\end{corollary}

\begin{proof}
By the assumption $s_0-\frac{d}{p_0}=s_1-\frac{d}{p_1}$
and Theorem \ref{thm-sobolev-B} with $W_0=W_1=w\in A_{\infty}$,
we find that the condition \eqref{eq-sobolev-B} is equivalent to
the condition \eqref{eq-sobolev-w-d} and hence (i), (ii), and (iii) of
Corollary \ref{cor-sobolev-w} are mutually equivalent.
This finishes the proof of Corollary \ref{cor-sobolev-w}.
\end{proof}

\begin{remark}
\begin{itemize}
\item[{\rm (i)}] In \eqref{eq-sobolev-w-d} of
Corollary \ref{cor-sobolev-w}, if $d=n$,
then \eqref{eq-sobolev-w-d}
is equivalent to the condition that,
for almost every $x\in\mathbb{R}^n$,
\begin{align}\label{eq-sobolev-w-n}
w(x)\geq C.	
\end{align}
Indeed, when $d=n$,
\eqref{eq-sobolev-w-d} is equivalent to the condition that,
for any $Q\in\mathcal{D}$, $w(Q)\geq C|Q|$.
Applying this and Lebesgue's differentiation theorem,
we conclude that \eqref{eq-sobolev-w-d} is
equivalent to \eqref{eq-sobolev-w-n}.

Next, we present a class of examples satisfying \eqref{eq-sobolev-w-n}.
For any scalar weight $w\in A_{\infty}$ and any $E\in(0, \infty)$,
let $w_E:=\max\{w, E\}$. By \eqref{eq-w-const} and some basic
calculations, we find that
\begin{align*}
\left[w_E\right]_{A_{\infty}}
&\leq\sup_{\operatorname{cube} Q\subset\mathbb{R}^n}\left[\fint_Qw(x)\,dx+E\right]
\min\left\{\exp\left(\fint_Q \log\left(\left[w(x)
\right]^{-1}\right)\, dx\right), \frac{1}{E}\right\}\\
&\leq[w]_{A_{\infty}}+1<\infty,
\end{align*}
which further implies that $w_E\in A_{\infty}$.
Obviously, $w_E$ satisfies \eqref{eq-sobolev-w-n}
with $C$ replaced by $E$.

\item[{\rm (ii)}] Condition \eqref{eq-sobolev-w-d} is
called the \emph{lower bound condition}, which often appears
in the study of weighted Sobolev-type embeddings
(see, for example, \cite{bui82,hhhlp21,sal02}).
\end{itemize}
\end{remark}

The following conclusion gives
the Sobolev-type embedding of $\dot{A}_{p, q}^{s, \tau}$.
\begin{corollary}\label{cor-sobolev-1}
Let $s_0, s_1\in\mathbb{R}$, $0<p_0<p_1<\infty$,
$q\in(0, \infty]$, and $\tau\in[0, \infty)$.
Then the following statements are mutually equivalent.
\begin{itemize}
\item[{\rm (i)}] $\dot{B}_{p_0,q}^{s_0, \tau}
\hookrightarrow\dot{B}_{p_1,q}^{s_1, \tau}$.
\item[{\rm (ii)}] $\dot{F}_{p_0, \infty}^{s_0, \tau}
\hookrightarrow\dot{F}_{p_1,q}^{s_1, \tau}$.
\item[{\rm (iii)}] $s_0-\frac{n}{p_0}=s_1-\frac{n}{p_1}$.
\end{itemize}
\end{corollary}

\begin{proof}
To prove the present corollary, applying Theorem \ref{thm-sobolev-B}
with $m=1$ and $W_0=W_1\equiv1$,
we only need to show condition \eqref{eq-sobolev-B} is equivalent to
$s_0-\frac{n}{p_0}=s_1-\frac{n}{p_1}$. In Corollary \ref{cor-sobolev-1},
\eqref{eq-sobolev-B} is precisely the condition that,
for any $Q\in\mathcal{D}$,
$2^{j_Q(s_1-s_0-\frac{n}{p_1}+\frac{n}{p_0})}\lesssim1$.
By the arbitrariness of $j_Q\in\mathbb{Z}$, we find that
\eqref{eq-sobolev-B} is equivalent to
$s_0-\frac{n}{p_0}=s_1-\frac{n}{p_1}$.
This finishes the proof of Corollary \ref{cor-sobolev-1}.
\end{proof}

\subsection{Comparison with Known Results}

Finally, we compare the results obtained
in this article with corresponding known ones.
To begin with, we clarify the
relation of spaces $\dot{A}_{p,q}^{s, \upsilon}(W)$
with spaces introduced in \cite{lbyy12}. Recall that,
to study the duality of weighted anisotropic
Besov--Triebel–Lizorkin spaces,
Li et al. \cite[Definitions 2.4 and 2.5]{lbyy12}
also introduced weighted Besov--Triebel--Lizorkin-type spaces
on $\mathbb{R}^n$ associated with general expansive dilations.
The spaces in \cite[Definitions 2.4 and 2.5]{lbyy12}
defined on $\mathbb{R}^n$ associated with the
standard dilation are as follows.

\begin{definition}\label{def-lbyy-BTLtype}
Let $s\in\mathbb{R}$, $p,q\in(0, \infty]$, $\tau\in[0, \infty)$,
the scalar weight $w\in A_{\infty}$, and $\varphi\in{\mathcal{S}}$
satisfy \eqref{cond1}.
\begin{itemize}
\item[{\rm (i)}] The \emph{weighted Besov-type space} $\widetilde{\dot{B}_{p,q}^{s,\tau}}(w)$
is defined to be the set of all $f \in \mathcal{S}_{\infty}^{\prime}$
such that
\begin{align*}
\|f\|_{\widetilde{\dot{B}_{p,q}^{s,\tau}}(w)}
:=\left(\sum_{j \in \mathbb{Z}}\left[\int_{\mathbb{R}^n}
\left\{\sum_{Q\in \mathcal{D}_j}2^{js}
\left|\varphi_j * f(x)\right| \frac{|Q| }{[w(Q)]^\tau}
\mathbf{1}_Q(x)\right\}^p w(x)\,d x\right]^{\frac{q}{p}}\right)^{\frac{1}{q}}<\infty.
\end{align*}
The \emph{weighted Besov-type sequence space} $\widetilde{\dot{b}_{p,q}^{s,\tau}}(w)$
is defined to be the set of all
$t:=\{t_Q\}_{Q\in\mathcal{D}}$ in $\mathbb{C}$ such that
\begin{align*}
\|t\|_{\widetilde{\dot{b}_{p,q}^{s,\tau}}(w)}
:=\left(\sum_{j \in \mathbb{Z}}\left[\int_{\mathbb{R}^n}
\left\{\sum_{Q\in \mathcal{D}_j}2^{js}
\left|t_Q\right| \frac{|Q| }{[w(Q)]^\tau}
\widetilde{\mathbf{1}}_Q(x)\right\}^p w(x)\,d x\right]^{\frac{q}{p}}\right)^{\frac{1}{q}}<\infty
\end{align*}
(with the usual modification made if $p=\infty$ or $q=\infty$).
\item[{\rm (ii)}] If $p\in(0, \infty)$, the \emph{weighted Triebel--Lizorkin-type space}
$\widetilde{\dot{F}_{p,q}^{s,\tau}}(w)$ is defined to be the set of
all $f \in \mathcal{S}_{\infty}^{\prime}$ such that
\begin{align*}
\|f\|_{\widetilde{\dot{F}_{p,q}^{s,\tau}}(w)}
:=\sup_{P\in\mathcal{D}} \frac{1}{[w(P)]^\tau}
\left[\int_P\left(\sum_{Q \in\mathcal{D}}
\left[|Q|^{-\frac{s}{n}}\left|\varphi_{j_Q}*f(x)\right|
\frac{|Q|}{w(Q)}\mathbf{1}_Q(x)\right]^q
\right)^{\frac{p}{q}} w(x)\,dx\right]^{\frac{1}{p}}<\infty.
\end{align*}
If $p\in(0, \infty)$, the
\emph{weighted Triebel--Lizorkin-type sequence space}
$\widetilde{\dot{f}_{p,q}^{s,\tau}}(w)$ is defined to be the set of
all $t:=\{t_Q\}_{Q\in\mathcal{D}}$ in $\mathbb{C}$ such that
\begin{align*}
\|t\|_{\widetilde{\dot{f}_{p,q}^{s,\tau}}(w)}
:=\sup_{P\in\mathcal{D}} \frac{1}{[w(P)]^\tau}
\left[\int_P\left(\sum_{Q \in\mathcal{D}}
\left[|Q|^{-\frac{s}{n}}\left|t_Q\right|
\frac{|Q|}{w(Q)}\widetilde{\mathbf{1}}_Q \right]^q
\right)^{\frac{p}{q}} w(x)\,d x\right]^{\frac{1}{p}}<\infty
\end{align*}
(with the usual modification made if $q=\infty$).
Moreover,
the \emph{weighted Triebel--Lizorkin-type space}
$\widetilde{\dot{F}_{\infty, \infty}^{s,\tau}}(w)$
is defined to be the set of
all $f\in\mathcal{S}_{\infty}^{\prime}$ such that
\begin{align*}
\|f\|_{\widetilde{\dot{F}_{\infty, \infty}^{s,\tau}}(w)}
:=\sup_{P\in\mathcal{D}} \frac{1}{[w(P)]^\tau}
\sup_{j\geq j_P}2^{js}\left\|\sum_{Q\in\mathcal{D}_j, Q\subset P}\left|\varphi_j *f\right|\frac{|Q|}{w(Q)}\mathbf{1}_Q
\right\|_{L^{\infty}(P)}<\infty.
\end{align*}
The \emph{weighted Triebel--Lizorkin-type sequence space} $\widetilde{\dot{f}_{\infty, \infty}^{s,\tau}}(w)$
is defined to be the set of all
$t:=\{t_Q\}_{Q\in\mathcal{D}}$ in $\mathbb{C}$ such that
\begin{align*}
\|t\|_{\widetilde{\dot{f}_{\infty, \infty}^{s,\tau}}(w)}
:=\sup_{P\in\mathcal{D}} \frac{1}{[w(P)]^\tau}
\sup_{Q\in\mathcal{D}, Q\subset P}|Q|^{-\frac{s}{n}-\frac{1}{2}}
|t_Q|\frac{|Q|}{w(Q)}<\infty.
\end{align*}
\end{itemize}
\end{definition}
\begin{remark}
Let all the symbols be the same as in Definition
\ref{def-lbyy-BTLtype} and Example \ref{exam-BTL}. The spaces
$\widetilde{\dot{B}_{p,q}^{s,\tau}}(w)$ and
$\widetilde{\dot{F}_{p,q}^{s,\tau}}(w)$ can be respectively
regarded as the weighted variants of the Besov-type space
$\dot{B}_{p,q}^{s,\tau}$ and the Triebel--Lizorkin-type space
$\dot{F}_{p,q}^{s,\tau}$ as in Example \ref{exam-BTL}.
\end{remark}

In \cite[Theorems 2.1 and 2.2]{lbyy12}, Li et al. proved
that the dual spaces of weighted Besov--Triebel--Lizorkin spaces
can be expressed in terms of the spaces in Definition
\ref{def-lbyy-BTLtype} as follows.

\begin{proposition}\label{prop-lbyy-dual}
Let $s\in\mathbb{R}$, $p,q\in(0, \infty)$,
$\tau_0=\frac{1}{p}+\frac{1}{q'}-1$, $\tau_1=\max\{\frac{1}{p},1\}$,
and the scalar weight $w\in A_{\infty}$. Then
\begin{align*}
\left[\dot{B}_{p,q}^{s}(w)\right]^*
=\widetilde{\dot{B}^{-s,\tau_1}_{p',q'}}(w)
\end{align*}
and
\begin{align*}
\left[\dot{F}_{p,q}^{s}(w)\right]^*
=\begin{cases}
\displaystyle\widetilde{\dot{F}_{q', q'}^{-s, \tau_0}}(w)
& \text{ if } p \in(0,1], \\
\displaystyle\widetilde{\dot{F}_{p', q'}^{-s, 0}}(w)
& \text{ if } p \in(1, \infty),
\end{cases}
\end{align*}
where $[\dot{B}_{p,q}^{s}(w)]^*$ and $[\dot{F}_{p,q}^{s}(w)]^*$
are respectively the dual spaces of $\dot{B}_{p,q}^{s}(w)$
and $\dot{F}_{p,q}^{s}(w)$.
\end{proposition}

By Definition \ref{def-red-ope} and Lemma \ref{growEST},
we find that, for any $p\in(0,\infty)$ and
any scalar weight $w\in A_{\infty}$,
$\mathbb{A}_{w,p}:=\{[\fint_{Q}w(x)\,dx
]^{\frac{1}{p}}\}_{Q\in\mathcal{D}}$
is a sequence of reducing operators of order $p$
for $w$ and hence strongly doubling of order
$(\beta_1,\beta_2)$ for some $\beta_1,\beta_2\in[0, \infty)$.
Obviously, $\mathbb{A}^{-1}_{w,p}:=\{[\frac{|Q|}{w(Q)}
]^{\frac{1}{p}}\}_{Q\in\mathcal{D}}$ is
strongly doubling of order $(\beta_2,\beta_1)$.
Based on these facts, we next
show that the weighted Besov--Triebel--Lizorkin-type
spaces in Proposition \ref{prop-lbyy-dual}
are exactly averaging spaces as in Definition \ref{averBTL}.

\begin{proposition}\label{prop-lbyy-compare}
Let $s\in\mathbb{R}$, $p,q\in(0, \infty)$,
$\tau_0=\frac{1}{p}+\frac{1}{q'}-1$,
$\tau_1=\max\{\frac{1}{p},1\}$,
and the scalar weight $w\in A_{\infty}$.
Then the following statements hold.
\begin{itemize}
\item[{\rm (i)}] $\widetilde{\dot{B}^{-s,\tau_1}_{p',q'}}(w)=
\dot{B}^{-s+n(\frac{1}{p}-1)_+}_{p',q'}(\mathbb{A}^{-1}_{w,p})$
with equivalent norms.
\item[{\rm (ii)}] If $p \in(0,1]$, then
\begin{align*}
\widetilde{\dot{F}_{q', q'}^{-s, \tau_0}}(w)=\begin{cases}
\displaystyle\dot{B}_{q', q'}^{-s, \upsilon}(\mathbb{A}^{-1}_{w,1})
& \text{ if } q\in(0,1], \\
\displaystyle\dot{B}_{q', q'}^{-s, \upsilon}(\mathbb{A}^{-1}_{w,q})
& \text{ if } q\in(1, \infty)
\end{cases}
\end{align*}
all with equivalent norms, where, for any $Q\in\mathcal{D}$, $\upsilon(Q):=[w(Q)]^{\tau_0}$.
\item[{\rm (iii)}] If $p \in(1, \infty)$, then $\widetilde{\dot{F}_{p', q'}^{-s, 0}}(w)=\dot{F}_{p', q'}^{-s}(\mathbb{A}^{-1}_{w,p})$
with equivalent norms.
\end{itemize}
\end{proposition}
\begin{proof}
Recall that \cite[Lemma 3.1]{lbyy12} gives the
$\varphi$-transform characterizations of
$\widetilde{\dot{B}^{-s,\tau_1}_{p',q'}}(w)$,
$\widetilde{\dot{F}_{q', q'}^{-s, \tau_0}}(w)$,
and $\widetilde{\dot{F}_{p', q'}^{-s, 0}}(w)$ via their
corresponding sequence spaces as in Definition \ref{def-lbyy-BTLtype}.
By this and Theorem \ref{thm-phitansaverMWBTL}, to prove the
present proposition, we only need to show the
corresponding results at the level of sequence spaces.
We first prove (i) by considering the following
two cases for $p$. If $p\in(0,1]$, using the definition of
$\|\cdot\|_{\widetilde{\dot{b}^{-s,\tau_1}_{p',q'}}(w)}$,
\eqref{eq-a(A)-a}, and basic calculations,
we conclude that, for any $t:=\{t_Q\}
_{Q\in\mathcal{D}}$ in $\mathbb{C}$,
\begin{align*}
\|t\|_{\widetilde{\dot{b}^{-s,\tau_1}_{p',q'}}(w)}
&=\left\|\left\{\frac{|Q|}{[w(Q)]^{\frac{1}{p}}}t_Q\right\}
_{Q\in\mathcal{D}}\right\|_{\dot{b}^{-s}_{p',q'}}
=\left\|\left\{\frac{|Q|^{\frac{1}{p}-(\frac{1}{p}-1)}}{[w(Q)]^{\frac{1}{p}}}
t_Q\right\}_{Q\in\mathcal{D}}\right\|_{\dot{b}^{-s}_{p',q'}}
=\left\|t\right\|_{\dot{b}^{-s+n(\frac{1}{p}-1)}_{p',q'}
(\mathbb{A}^{-1}_{w,p})}.
\end{align*}
This finishes the proof of (i) in this case. If $p\in(1, \infty)$,
from the definition of
$\|\cdot\|_{\widetilde{\dot{b}^{-s,\tau_1}_{p',q'}}(w)}$,
the fact that $\mathbb{A}_{w,p'}$ is a sequence of
reducing operators of order $p'$ for $w$,
Corollary \ref{cor-a(A)=a(W)}, and \eqref{eq-a(A)-a},
we infer that, for any $t:=\{t_Q\}
_{Q\in\mathcal{D}}$ in $\mathbb{C}$,
\begin{align*}
\|t\|_{\widetilde{\dot{b}^{-s,\tau_1}_{p',q'}}(w)}
&=\left\|\left\{\frac{|Q|}{w(Q)}t_Q\right\}
_{Q\in\mathcal{D}}\right\|_{\dot{b}^{-s}_{p',q'}(w)}\sim
\left\|\left\{\frac{|Q|}{w(Q)}t_Q\right\}
_{Q\in\mathcal{D}}\right\|_{\dot{b}^{-s}_{p',q'}(\mathbb{A}_{w,p'})}\\
&=\left\|\left\{\left[\frac{|Q|}
{w(Q)}\right]^{\frac{1}{p}}t_Q\right\}
_{Q\in\mathcal{D}}\right\|_{\dot{b}^{-s}_{p',q'}}
=\left\|t\right\|_{\dot{b}^{-s}_{p',q'}
(\mathbb{A}^{-1}_{w,p})},
\end{align*}
which completes the proof of (i) in this case and hence (i).

Next, we show (ii) by considering the following
two cases for $q$. If $q\in(0,1]$, in this case,
applying the definition of $\|\cdot\|_{\widetilde{\dot{f}^{-s,\tau_0}_{q',q'}}(w)}$
and \eqref{eq-a(A)-a},
we obtain, for any $t:=\{t_Q\}
_{Q\in\mathcal{D}}$ in $\mathbb{C}$,
\begin{align*}
\|t\|_{\widetilde{\dot{f}^{-s,\tau_0}_{q',q'}}(w)}
&=\left\|\left\{\frac{|Q|}{w(Q)}t_Q\right\}
_{Q\in\mathcal{D}}\right\|_{\dot{b}^{-s, \upsilon}_{q',q'}}
=\left\|t\right\|_{\dot{b}^{-s, \upsilon}_{q',q'}(\mathbb{A}^{-1}_{w,1})}.	
\end{align*}
This finishes the proof of (ii) in this case.
If $q\in(1, \infty)$, by the definition of
$\|\cdot\|_{\widetilde{\dot{f}^{-s,\tau_0}_{q',q'}}(w)}$,
the fact that $\mathbb{A}_{w,q'}$ is a sequence of
reducing operators of order $q'$ for $w$,
Corollary \ref{cor-a(A)=a(W)}, and \eqref{eq-a(A)-a},
we find that, for any $t:=\{t_Q\}
_{Q\in\mathcal{D}}$ in $\mathbb{C}$,
\begin{align*}
\|t\|_{\widetilde{\dot{f}^{-s,\tau_0}_{q',q'}}(w)}
&=\left\|\left\{\frac{|Q|}{w(Q)}t_Q\right\}
_{Q\in\mathcal{D}}\right\|_{\dot{b}^{-s, \upsilon}_{q',q'}(w)}\sim
\left\|\left\{\frac{|Q|}{w(Q)}t_Q\right\}
_{Q\in\mathcal{D}}\right\|_{\dot{b}^{-s, \upsilon}_{q',q'}(\mathbb{A}_{w,q'})}\\
&=\left\|\left\{\left[\frac{|Q|}{w(Q)}\right]^{\frac{1}{q}}t_Q\right\}
_{Q\in\mathcal{D}}\right\|_{\dot{b}^{-s, \upsilon}_{q',q'}}
=\left\|t\right\|_{\dot{b}^{-s, \upsilon}_{q',q'}(\mathbb{A}^{-1}_{w,q})},
\end{align*}
which completes the proof of (ii) in this case and hence (ii).

Finally, we prove (iii). Using the definition of
$\|\cdot\|_{\widetilde{\dot{f}^{-s,0}_{p',q'}}(w)}$,
the fact that $\mathbb{A}_{w,p'}$ is a sequence of
reducing operators of order $p'$ for $w$,
Corollary \ref{cor-a(A)=a(W)}, and \eqref{eq-a(A)-a},
we conclude that, for any $t:=\{t_Q\}
_{Q\in\mathcal{D}}$ in $\mathbb{C}$,
\begin{align*}
\|t\|_{\widetilde{\dot{f}^{-s,0}_{p',q'}}(w)}
&=\left\|\left\{\frac{|Q|}{w(Q)}t_Q\right\}
_{Q\in\mathcal{D}}\right\|_{\dot{f}^{-s}_{p',q'}(w)}\sim
\left\|\left\{\frac{|Q|}{w(Q)}t_Q\right\}
_{Q\in\mathcal{D}}\right\|_{\dot{f}^{-s}_{p',q'}(\mathbb{A}_{w,p'})}\\
&=\left\|\left\{\left[\frac{|Q|}{w(Q)}\right]^{\frac{1}{p}}t_Q\right\}
_{Q\in\mathcal{D}}\right\|_{\dot{f}^{-s}_{p',q'}}
=\left\|t\right\|_{\dot{f}^{-s}_{p',q'}(\mathbb{A}^{-1}_{w,p})}.
\end{align*}
This finishes the proof of (iii) and
hence Proposition  \ref{prop-lbyy-compare}.
\end{proof}

\begin{remark}
In Proposition \ref{prop-lbyy-compare}, if $p=q$, (i) and (ii)
were obtained in \cite[Corollary 2.2]{lbyy12}.
\end{remark}

Next, we compare the Peetre-type maximal function
characterization of $\dot{A}^{s,\upsilon}_{p,q}(W)$
established in Theorem \ref{thm-Pee-cha}
with some known results. To this end,
we first show the relation of \eqref{eq-beta-p-weak}
with some known indices. Let $p\in(0, \infty)$,
$W\in\mathcal{A}_{p,\infty}$, and $\beta_{p}(W)$ and $\alpha_{p}(W)$
be as in Definition \ref{def-double-W}.
From \cite[Lemma 2.2]{fr21} and the definition of
$\alpha_{p}(W)$ [see \eqref{eq-beta-p-weak}],
we deduce that
\begin{align}\label{eq-alpha-beta}
\alpha_{p}(W)\leq\frac{\beta_{p}(W)}{p}.
\end{align}
Using Lemma \ref{growEST} and the definition of
$\alpha_{p}(W)$ again, we conclude that
\begin{align}\label{eq-beta-low-upp}
\alpha_{p}(W)\leq
\frac{d^{\operatorname{lower}}_{p, \infty}(W)
+d^{\operatorname{upper}}_{p, \infty}(W)}{p},
\end{align}
where $d^{\operatorname{lower}}_{p, \infty}(W)$
and $d^{\operatorname{ upper}}_{p, \infty}(W)$
are as, respectively, in \eqref{eq-low-dim} and \eqref{eq-upp-dim}.
In particular, if $W\equiv1$, it follows from Definitions \ref{def-red-ope}
and \ref{def-doub-seq} that $\{I_m\}_{Q\in\mathcal{D}}$
is a sequence of reducing operators of order $p$ for $W$ and is weakly
doubling of order $0$, where $I_m$ is the identity matrix of order $m$.
Applying this and \eqref{eq-beta-p-weak}, we obtain, if $W\equiv1$,
then $\alpha_{p}(W)=0$. For any scalar weight $w\in A_{\infty}$,
let
\begin{align*}
r_w:=\inf\left\{r\in[1, \infty):\ w\in A_{r}\right\}
\end{align*}
be the well-known \emph{critical index} (see, for example,
\cite[Definitions 7.1.1 and 7.1.3]{gra14a}
for the definition of the scalar $A_r$ class of Muckenhoupt
for any $r\in[1, \infty)$). Let $w\in A_{\infty}$.
By \eqref{eq-low-dim} and \cite[Theorem 4.28(i)]{bhyy5},
we find that $d^{\operatorname{lower}}_{p, \infty}(w)<n$ and
$d^{\operatorname{upper}}_{p, \infty}(w)\leq n(r_w-1)$.
Using this and \eqref{eq-beta-low-upp}, we conclude that
\begin{align}\label{eq-beta-rw}
\alpha_{p}(w)\leq
\frac{d^{\operatorname{lower}}_{p, \infty}(w)
+d^{\operatorname{upper}}_{p, \infty}(w)}{p}
<\frac{n+d^{\operatorname{upper}}_{p, \infty}(w)}{p}
\leq n\frac{r_w}{p}.
\end{align}

\begin{remark} Let all the symbols be the same as in Theorem \ref{thm-Pee-cha}.
\begin{itemize}
\item[{\rm (i)}]
Let $\tau\in[0, \infty)$ and, for any $Q\in\mathcal{D}$,
$\upsilon(Q):=|Q|^{\tau}$. By Example \ref{examp}(i), we find that
$\upsilon\in\mathcal{G}(\tau,\tau;0)$.
Then the space $\dot{A}^{s,\upsilon}_{p,q}(W)$ reduces
to the matrix-weighted BTL-type space
$\dot{A}^{s,\tau}_{p, q}(W)$. In this case,
Theorem \ref{thm-Pee-cha} is also new.
In particular, $\dot{F}^{s,0}_{p,q}(W)$ is precisely
the matrix-weighted Triebel--Lizorkin space
$\dot{F}^{s}_{p, q}(W)$. In this case,
Theorem \ref{thm-Pee-cha} improves \cite[Theorem 3.1]{wyy23}
in which only the case where
$W$ is a matrix $\mathcal{A}_p$ weight was showed
(see, for example, \cite[p.\,490]{fr21}
for the definition of the matrix $\mathcal{A}_p$ class).
It is well known that, for any $p\in(0, \infty)$,
$\mathcal{A}_{p, \infty}\supsetneqq\mathcal{A}_p$
(see, for example, \cite[Proposition 4.2]{bhyy4}).
Moreover, by \eqref{eq-alpha-beta},
we find that the restriction on $\eta$ in Theorem \ref{thm-Pee-cha},
namely $\eta\in(\frac{n}{p\wedge q}+
\alpha_{p}(W),\infty)$,
is also better than the corresponding one in \cite[Theorem 3.1]{wyy23}
that $\eta\in(\frac{n}{1\wedge p\wedge q}
+\frac{\beta_{p}(W)}{p},\infty)$.
\item[{\rm (ii)}]
When $m=1$ and $W:=w\in A_{\infty}$, the space
$\dot{A}^{s,0}_{p,q}(W)$ reduces to the classical
weighted BTL space $\dot{A}^{s}_{p, q}(w)$.
Compared to \cite[Theorem 2.2]{bui82},
the restriction on $\eta$ in Theorem \ref{thm-Pee-cha},
namely $\eta\in(\frac{n}{\Gamma_{p,q}}
+\alpha_{p}(w),\infty)$,
is different from the corresponding one in \cite[Theorem 2.2]{bui82}
that $\eta\in(\frac{nr_w}{p}\vee\frac{n}{q},\infty)$ if $A=F$ and
$\eta\in(\frac{nr_w}{p},\infty)$ if $A=B$
(see Remark \ref{rmk-pee} for the reason).
When $w\equiv1$, where $r_w=1$ and $\alpha_{p}(w)=0$,
the restriction on $\eta$ in \cite[Theorem 2.2]{bui82}
coincides with the one in Theorem \ref{thm-Pee-cha}.
We need to point out that Theorem \ref{thm-Pee-cha}
and \cite[Theorem 2.2]{bui82} have their own advantages
on the the restriction on $\eta$.
For example, let $A=B$, $\alpha\in(1,\infty)$,
and, for any $x:=(x_1,x_2)\in\mathbb{R}^2$,
$w_{\alpha}(x):=|x_1||x_2|^{\alpha}$.
From \cite[(i) and (v) of Lemma 4.30]{bhyy5}, it follows
that $w_{\alpha}\in A_{\infty}$
and $d^{\operatorname{upper}}_{p, \infty}(w_{\alpha})
=r_{w_{\alpha}}=1+\alpha$. If $\alpha>3$,
by \eqref{eq-beta-rw}, we obtain
\begin{align*}
\frac{2}{p}+\alpha_{p}(w_{\alpha})<
\frac{4+d^{\operatorname{upper}}_{p, \infty}(w_{\alpha})}{p}
=\frac{5+\alpha}{p}<\frac{2+2\alpha}{p}=\frac{2r_{w_{\alpha}}}{p},
\end{align*}
which implies that the restriction on $\eta$ in
Theorem \ref{thm-Pee-cha} is better than
the corresponding one in \cite[Theorem 2.2]{bui82}.
On the other hand, let $A=B$ and,
for any $x\in\mathbb{R}$, $w(x):=|x|^{-\frac{1}{2}}$.
By \cite[Example 7.1.7]{gra14a}, we find that $w\in A_1$
and hence $r_w=1$. In this case,
using Definition \ref{def-red-ope},
we then conclude that $\mathbb{A}:=\{A_Q\}_{Q\in\mathcal{D}}
:=\{[\fint_{Q}w(x)\,dx]^{\frac{1}{p}}\}_{Q\in\mathcal{D}}$
is a sequence of reducing operators of order $p$ for $w$.
We next claim that $\alpha_{p}(w)=\frac{1}{2p}$.
From \cite[Corollary 2.41]{bhyy1}, we infer that, for any
$Q\in\mathcal{D}$, $\fint_{Q}w(x)\,dx\sim[|x_Q|+\ell(Q)]^{-\frac{1}{2}}$.
By this, the construction of $\mathbb{A}$, and the
triangle inequality of $|\cdot|$, we find that,
for any $Q,R\in\mathcal{D}$ with $\ell(Q)=\ell(R)$,
\begin{align}\label{eq-A}
\left\|A_QA^{-1}_R\right\|&=\left[\frac{\fint_{Q}w(x)\,dx}{\fint_{R}w(x)\,dx}
\right]^{\frac{1}{p}}\sim\left[\frac{|x_R|+\ell(R)}{|x_Q|+\ell(Q)}
\right]^{\frac{1}{2p}}\\
&\leq\left[\frac{|x_Q|+|x_Q-x_R|+\ell(R)}{|x_Q|+\ell(Q)}
\right]^{\frac{1}{2p}}\leq
\left\{1+[\ell(Q)]^{-1}|x_Q-x_R|\right\}^{\frac{1}{2p}}.\nonumber
\end{align}
This, together with Definition \ref{def-doub-seq}(ii),
further implies that $\mathbb{A}$ is weakly doubling of
order $\frac{1}{2p}$. We now show that, for
any $\beta\in[0, \frac{1}{2p})$, $\mathbb{A}$ is not
weakly doubling of order $\beta$. Let $\beta\in[0, \frac{1}{2p})$.
Suppose that $\mathbb{A}$ is weakly doubling of order $\beta$.
Using \eqref{eq-A} and Definition \ref{def-doub-seq}(ii),
we conclude that, for any $Q,R\in\mathcal{D}$
with $\ell(Q)=\ell(R)$ and $x_Q=\mathbf{0}$,
\begin{align*}
\left\{1+[\ell(Q)]^{-1}|x_R|\right\}^{\frac{1}{2p}}&=
\left[\frac{|x_R|+\ell(R)}{\ell(Q)}\right]^{\frac{1}{2p}}
=\left[\frac{|x_R|+\ell(R)}{\ell(Q)}\right]^{\frac{1}{2p}}\\
&\sim\left\|A_QA^{-1}_R\right\|\lesssim
\left\{1+[\ell(Q)]^{-1}|x_R|\right\}^{\beta},
\end{align*}
which induces a contradiction as $|x_R|\to\infty$ and hence
shows that $\mathbb{A}$ is not weakly doubling of order $\beta$.
From the above arguments and \eqref{eq-beta-p-weak},
we deduce that $\alpha_{p}(w)=\frac{1}{2p}$ and hence
prove the claim. Combining the above discussions, we conclude that
\begin{align*}
\frac{r_{w}}{p}=\frac{1}{p}<\frac{1}{p}
+\frac{1}{2p}=\frac{1}{p}+\alpha_{p}(w),
\end{align*}
which further implies that, in this case,
the restriction on $\eta$ in \cite[Theorem 2.2]{bui82}
is better than the corresponding one in Theorem \ref{thm-Pee-cha}.
\end{itemize}	
\end{remark}

We next illustrate that the Lusin area function and the
Littlewood--Paley $g^{*}_{\lambda}$-function characterizations of
$\dot{A}^{s,\upsilon}_{p,q}(W)$ established in Theorem \ref{thm-G-L-cha}
also improves some known results.

\begin{remark}
Let all the symbols be the same as in Theorem \ref{thm-G-L-cha}.
Let $\tau\in[0, \infty)$ and, for any $Q\in\mathcal{D}$,
$\upsilon(Q):=|Q|^{\tau}$.
By Example \ref{examp}(i), we find that
$\upsilon\in\mathcal{G}(\tau,\tau;0)$.
Then the space $\dot{A}^{s,\upsilon}_{p,q}(W)$ reduces
to the matrix-weighted BTL-type space
$\dot{A}^{s,\tau}_{p, q}(W)$. In this case,
Theorem \ref{thm-G-L-cha} is also new.
In particular, the space $\dot{F}^{s,0}_{p,q}(W)$ is
exactly the matrix-weighted Triebel--Lizorkin space
$\dot{F}^{s}_{p, q}(W)$. In this case,
the Lusin area function characterization
of $\dot{F}^{s}_{p, q}(W)$ in Theorem \ref{thm-G-L-cha}
improves \cite[Theorem 3.11]{wyy23}
in which only the case where $\alpha=1$, $r=q$,
and $W$ is a matrix $\mathcal{A}_p$ weight was proved.
On the other hand, the Littlewood--Paley $g^{*}_{\lambda}$-function
characterization of $\dot{F}^{s}_{p, q}(W)$
in Theorem \ref{thm-G-L-cha} also improves
\cite[Theorem 3.14]{wyy23} in which only
the case where $r=q$ and $W$ is a matrix $\mathcal{A}_p$ weight
was showed. By \eqref{eq-alpha-beta}, we find that
the restriction on $\lambda$ in Theorem \ref{thm-G-L-cha},
namely $\lambda\in(\frac{n}{p\wedge q}+\alpha_{p}(W),\infty)$,
is better than the corresponding
one in \cite[Theorem 3.14]{wyy23} that $\lambda\in(\frac{n}{1\wedge p\wedge q}+\frac{\beta_{p}(W)}{p},\infty)$.
Furthermore, when $m=1$ and $W\equiv1$, Theorem \ref{thm-G-L-cha}
coincides with \cite[Theorem 3.2]{cho10}.
\end{remark}

We present the following remark to discuss
the results on the boundedness of almost diagonal operators
on $\dot{a}_{p,q}^{s,\upsilon}$ and
$\dot{a}_{p,q}^{s,\upsilon}(W)$.

\begin{remark}\label{rmk-com-ad}
\begin{itemize}
\item[{\rm (i)}] Let all the symbols be the same as in Theorem \ref{thm-bound-ad}. Let $\tau\in[0,\infty)$ and,
for any $Q\in\mathcal{D}$, $\upsilon(Q):=|Q|^{\tau}$.
From Example \ref{examp}(i), it follows that
$\upsilon\in\mathcal{G}(\tau, \tau; 0)$.
In this case, the space $\dot{a}_{p,q}^{s,\upsilon}$
reduces to the BTL-type sequence space $\dot{a}_{p,q}^{s,\tau}$,
and the conditions on $D,E,F$ in Theorem \ref{thm-bound-ad} are precisely
\begin{align*}
D>J_{\dot{a}_{p,q}^{s,\tau}},\
E>\frac{n}{2}+s+n\left(\tau-\frac{1}{p}\right)_{+},
\text{ and }F>J_{\dot{a}_{p,q}^{s,\tau}}-\frac{n}{2}-s-n
\left(\tau-\frac{1}{p}\right)_+.
\end{align*}
Thus, Theorem \ref{thm-bound-ad} coincides with
\cite[Theorem 4.4(ii)]{bhyy2}. In particular,
when $\tau=0$, the space $\dot{a}_{p,q}^{s,0}$
is exactly the classical BTL sequence space $\dot{a}_{p,q}^{s}$
and, in this case, Theorem \ref{thm-bound-ad} coincides with
\cite[Theorems 7.1 and Lemma 9.1]{bhyy2} in which
the sharpness on
$D,E,F$ for $\dot{b}_{p,q}^{s}$ and for
$\dot{f}_{p,q}^{s}$ with $q\in[1\wedge p, \infty]$ was also proved.
\item[{\rm (ii)}] Let all the symbols be the same as in Theorem \ref{a(W)adopebound}.
Let $\tau\in[0,\infty)$ and, for any $Q\in\mathcal{D}$,
$\upsilon(Q):=|Q|^{\tau}$.
In this case, the space $\dot{a}_{p,q}^{s,\upsilon}(W)$
is precisely the matrix-weighted BTL-type sequence space
$\dot{a}_{p,q}^{s,\tau}(W)$ studied in \cite{bhyy1,bhyy2,bhyy3,bhyy5}.
Moreover, Theorem \ref{a(W)adopebound}
coincides with \cite[Theorem 4.5]{bhyy5}, which gives
the boundedness of almost diagonal operators
on $\dot{a}_{p,q}^{s,\tau}(W)$.
For the comparison of \cite[Theorem 4.5]{bhyy5}
with some known results on the
boundedness of almost diagonal operators,
we refer to \cite[Remark 4.6 and Subsection 4.2]{bhyy5}.
\end{itemize}
\end{remark}

Finally, we compare the Sobolev-type embedding
of $\dot{A}_{p,q}^{s, \upsilon}(W)$ obtained
in Subsection \ref{s5.2} with some known results.

\begin{remark}
\begin{itemize}
\item[{\rm (i)}]In Corollary \ref{cor-sobolev-w},
for any $Q\in\mathcal{D}$, let $\upsilon(Q):=1$.
In this case, the embeddings in Corollary \ref{cor-sobolev-w}
were obtained in \cite[(iv) and (v) of Theorem 2.6]{bui82}.
Here, we also prove that condition \eqref{eq-sobolev-w-d}
in Corollary \ref{cor-sobolev-w} is necessary.
However, it is worth pointing out that,
for scalar weights satisfying \eqref{eq-sobolev-w-d},
both the embeddings (i) and (ii) of Corollary \ref{cor-sobolev-w}
were proved in \cite[(iv) and (v) of Theorem 2.6]{bui82}
without assuming $w\in A_{\infty}$.
Moreover, Corollary \ref{cor-sobolev-w}
also contains \cite[Theorem 5.1]{hhhlp21} as a special case.

\item[{\rm (ii)}] In Corollary \ref{cor-sobolev-1},
let $\tau\in[0, \infty)$ and, for any $Q\in\mathcal{D}$,
$\upsilon(Q):=|Q|^{\tau}$. In this case,
the embeddings in Corollary \ref{cor-sobolev-1}
coincide with \cite[Proposition 3.3]{yy10}.
Here, we also prove the condition
$s_0-\frac{n}{p_0}=s_1-\frac{n}{p_1}$
in Corollary \ref{cor-sobolev-1} is necessary.
\end{itemize}
\end{remark}

\noindent\textbf{Acknowledgements}\quad
The last author would like to thank Yiqun Chen
for providing the idea to construct the
sequence in the proof of Proposition \ref{prop-3=4-b}.

\section*{Data availability}

No data was used for the research described in the article.

\bigskip

\noindent Fan Bu, Dachun Yang (Corresponding author),
Wen Yuan and Mingdong Zhang

\medskip

\noindent Laboratory of Mathematics and Complex Systems
(Ministry of Education of China),
School of Mathematical Sciences, Beijing Normal University,
Beijing 100875, The People's Republic of China

\smallskip

\noindent{\it E-mails:}
\texttt{fanbu@mail.bnu.edu.cn} (F. Bu)

\noindent\phantom{{\it E-mails:}}
\texttt{dcyang@bnu.edu.cn} (D. Yang)

\noindent\phantom{{\it E-mails:}}
\texttt{wenyuan@bnu.edu.cn} (W. Yuan)

\noindent\phantom{{\it E-mails:}}
\texttt{mdzhang@mail.bnu.edu.cn} (M. Zhang)
\end{document}